\documentclass[12pt]{amsart}
\usepackage{color}
\usepackage{amssymb}
\usepackage{comment}
\usepackage{pdfpages}
\usepackage{graphicx}
\usepackage{dcolumn}
\usepackage{algorithmicx}
\usepackage{algorithm,caption}
\usepackage{algpseudocode}
\usepackage{bbm}

\RequirePackage[numbers]{natbib}
\RequirePackage[colorlinks=true, pdfstartview=FitV, linkcolor=blue,
  citecolor=blue, urlcolor=blue]{hyperref}
\usepackage{paralist}

\usepackage{tikz} 
\usepackage{dcolumn}
\usetikzlibrary{arrows,positioning}
\tikzset{
    >=stealth',
    pil/.style={
           ->,
           thick,
           shorten <=2pt,
           shorten >=2pt,}
}

\usepackage{graphicx}
\graphicspath{{./figures/}}
\usepackage{amsfonts}
\usepackage{amsmath}
\usepackage{amsthm}
\usepackage{amssymb}
\usepackage{amsbsy}
\usepackage{epsfig}
\usepackage{fullpage}
\usepackage{natbib, mathrsfs} 
\usepackage{verbatim}
\usepackage[latin1]{inputenc}
\usepackage{mhequ}
\numberwithin{equation}{section}

\def \be{\begin{equs}}
\def \ee{\end{equs}}
\def \P{\mathbb{P}}
\def \E{\mathbb{E}}

\def \TV{\mathrm{TV}}

\def \d{\mathrm{d}}

\newtheorem{theorem}{Theorem}[section]
\newtheorem{lemma}[theorem]{Lemma}

\newtheorem{assumptions}[theorem]{Assumptions}
\newtheorem{assumption}[theorem]{Assumption}

\newtheorem{cond}[theorem]{Condition}

\theoremstyle{plain}
\newtheorem{thm}{Theorem}
\newtheorem*{thm-non}{Theorem}

\theoremstyle{definition}
\newtheorem{defn}[theorem]{Definition}
\newtheorem{remark}[theorem]{Remark}

\begin{document}

\title[Mixing of HMC]{Rapid Mixing of Hamiltonian Monte Carlo on Strongly Log-Concave Distributions}


\author{Oren Mangoubi$^{\ddag}$}
\thanks{$^{\ddag}$omangoubi@gmail.com, 
   \'{E}cole Polytechnique F\'{e}d\'{e}rale de Lausanne (EPFL),
 IC IINFCOM THL3, Station 14, 1015 Lausanne, Switzerland}

\author{Aaron Smith$^{\sharp}$}
\thanks{$^{\sharp}$smith.aaron.matthew@gmail.com, 
   Department of Mathematics and Statistics,
University of Ottawa, 585 King Edward Avenue, Ottawa
ON K1N 7N5, Canada}

\maketitle






\begin{abstract}
We obtain several quantitative bounds on the mixing properties of the Hamiltonian Monte Carlo (HMC) algorithm for a strongly log-concave target distribution $\pi$ on $\mathbb{R}^{d}$, showing that HMC mixes quickly in this setting. One of our main results is a dimension-free bound on the mixing of an ``ideal" HMC chain, which is used to show that the usual leapfrog  implementation of HMC can sample from $\pi$ using only $\mathcal{O}(d^{\frac{1}{4}})$ gradient evaluations. This dependence on dimension is sharp, and our results significantly extend and improve previous quantitative bounds on the mixing of HMC.
\end{abstract}

\section{Introduction} \label{sec:intro}

Markov chain Monte Carlo (MCMC) algorithms are ubiquitous in Bayesian statistics and other areas, and Hamiltonian Monte Carlo (HMC) algorithms are some the most widely-used MCMC algorithms \cite{girolami2011riemann, cheung2009bayesian, mehlig1992hybrid} . Despite the popularity of HMC and the widespread belief that HMC outperforms other algorithms in high-dimensional statistical problems (see \textit{e.g.} \cite{HMC_optimal_tuning}), its theoretical properties are not as well-understood as some of its older cousins, such as the Metropolis-Hastings algorithm (MH) or Metropolis-Adjusted Langevin Algorithm (MALA). This lack of theoretical results can make it harder to optimize HMC algorithms, and it means we do not have a good understanding of when HMC is better than other popular algorithms.

Several recent papers have begun to bridge this gap, most noteably by proving geometric ergodicity of HMC under general conditions \cite{livingstone2016geometric} and establishing some quantitative bounds on the rate of convergence for Gaussian target distributions \cite{seiler2014positive}.  
In this paper, we extend this work by obtaining rapid mixing bounds for HMC in an important general class of target distributions: those which are strongly log-concave. In this regime, we show upper bounds on the mixing rate of HMC that are better than those of many competitor algorithms, including the Langevin algorithm\cite{durmus2016sampling2, durmus2016sampling}. Our work is particularly close to that of \cite{seiler2014positive}, which is to our knowledge the only other paper giving quantitative non-asymptotic bounds on the mixing of HMC. We improve on their conclusions by greatly improving the dependence of their bounds on the dimension of the target distribution, extending their analysis from Gaussian to general strongly log-concave targets, proving convergence in stronger norms, and providing rates for numerical implementations of HMC algorithms rather than merely ``ideal" versions of HMC. \\

Although our results are far from providing a complete understanding of HMC, the strongly log-concave distributions are an important special case. Recall that a distribution $\pi$ is strongly log-concave if the Hessian matrix of $-\log(\pi)$ has eigenvalues bounded above and below by positive numbers $M_2$ and $m_2$, respectively. Many important posterior distributions in statistics, including the ``ridge regression" posterior associated with Gaussian priors for logistic regression, are strongly log-concave  \cite{durmus2016sampling2, durmus2016sampling}. In addition to this, we expect most MCMC algorithms to perform well for strongly log-concave targets. For these reasons, the performance of many Monte Carlo algorithms has been studied extensively in the strongly log-concave setting \cite{durmus2016sampling} . This has the added advantage of allowing us to give a sensible comparison of the performance of HMC to its competitors, such as the Langevin algorithm and the ball walk.

Following the work of \cite{seiler2014positive}, our paper begins by studying an `ideal' HMC algorithm - one that has no numerical error.  We show that the HMC Markov chain has mixing time $\mathcal{O}^*((\frac{M_2}{m_2})^2)$ on strongly log-concave $\pi$ (see Theorem \ref{ThmMainConcave}).  In particular, our mixing time bound does not depend explicitly on the dimension, improving on the best-previous bound of $\mathcal{O}^*(d^{2} \, C(m_{2},M_{2}))$ obtained in \cite{seiler2014positive}, which has quadratic dependence on dimension and no quantitative dependence on $m_{2},M_{2}$.
 
Using this result, we bound the computational costs of various numerical implementations of HMC, showing that HMC can be much faster than competing algorithms (see Theorems \ref{ThmMainApprox} through \ref{ThmMainApprox_leapfrog2}). Our first result shows that a simple numerical implementation of HMC can approximately sample from the stationary distribution with a number of gradient evaluations that grows at rate $\mathcal{O}_{d}(d^{\frac{1}{2}})$ in the dimension (see Theorem \ref{ThmMainApprox}). Under an additional seperability assumption, we show that an unadjusted numerical implementation of HMC based on a $k$'th-order numerical integrator allows one to approximately sample the stationary distribution in $\mathcal{O}_{d}(d^{\frac{1}{2k}})$ steps (see Theorem \ref{ThmMainApprox_leapfrog_unadjusted}).  Most implementations of HMC use second-order ($k=2$) integrators, giving a dimension dependence of $d^{\frac{1}{4}}$.  For comparison, the best available mixing time bound for the unadjusted Langevin algorithm on strongly logconcave $\pi$ grows like $d^{\frac{1}{2}}$, a much larger dependence on dimension than our bounds for unadjusted HMC implementations with integrators of order $k\geq 2$ \cite{durmus2016sampling2, durmus2016sampling}. These comparisons are summarized in Table \ref{table:cost}.  Our bounds also compare favorably to the ball walk, whose best available mixing time bound is roughly $\mathcal{O}(d^{2} \frac{M_{2}^{2}}{m_{2}^{2}})$ \cite{vempala2005geometric} (see Table \ref{table:cost_ball} for further comparisons; note that the particular assumptions made in different papers are slightly different).

\begin {table}[t]
\begin{center}
\begin{tabular}{ | m{3.5em} | m{1.8cm}| m{2.0cm}| m{3cm} || m{3cm} | }
\hline
 & Stochastic limit & Running Time (Ideal Integrator)
 & Cost (Unadjusted) &  Conjectured cost (Metropolis-Adjusted)  \\ 
\hline
\vspace{2mm} HMC \vspace{2mm} & \vspace{-10mm} \begin{flushright} $^\dagger$ \end{flushright} \vspace{-0mm}& \vspace{-6mm} \begin{flushright} $^\bigstar$ \end{flushright} \vspace{-0mm} \center $\frac{M_2^2}{m_2^2}$ \vspace{2mm} &\vspace{1mm} \begin{flushright} $^\bigstar$ \end{flushright} \vspace{-1mm} \vspace{2mm} \center$d^{\frac{1}{2k}}\times C'_k(m_2, M_2)$ \vspace{2mm} &\vspace{2mm}  $d^{\frac{1}{2k}} \times \frac{M_2}{m_2} \times M_{2}^{-\frac{1}{k}}$  \vspace{2mm}  \\ 
\cline{1-1} \cline{3-5}
\vspace{8mm} Langevin \vspace{2mm}& \vspace{-20mm} \center $\frac{M_2}{m_2}$  & \vspace{-8mm}  \begin{flushright} $^\ddagger$ \end{flushright} $d^{\frac{1}{2}} \times \frac{M_2}{m_2}  $ & \vspace{-8mm} \begin{flushright} $^\ddagger$ \end{flushright} \vspace{-4mm}   \center $d^{\frac{1}{2}} \times \frac{M_2}{m_2}  $ &\vspace{2mm}  $d^{\frac{1}{3}} \times \frac{M_2}{m_2}$ \vspace{2mm}  \\ 
\hline

\end{tabular}
\\
$\bigstar=$ \emph{This paper},
$\dagger=$  \cite{eberle2015reflection},
$\ddagger=$ \cite{durmus2016sampling2, durmus2016sampling},
\end{center}
\caption{Best available upper bounds on the running time and computational cost of various HMC and Langevin Markov chains on strongly logconcave targets, in terms of dimension $d$ and strong convexity bounds $M_2$ and $m_2$. The dimension dependence in the conjectures are widely-accepted from numerical evidence and the theoretical diffusion limit results of \cite{HMC_optimal_tuning, pillai2012optimal, breyer2004optimal}.  We propose the conjectured dependences on $m_2$ and $M_2$ here based on the stochastic limit \cite{eberle2015reflection} and the closely-related geodesic walk  \cite{mangoubi2016rapid}. Note that  while we prove the cost bound for numerical integrators of order $k=1$ for general strongly logconcave targets, the cost bounds we proved for $k\geq 2$ only apply under the additional assumption that the target is separable.} \label{table:cost} 
\end{table}

\noindent

\begin {table}[t]
\begin{center}
\begin{tabular}{ | m{5em} | m{4cm}  | m{4cm}  || m{2.2cm} | m{1cm} m{0cm} | }
\hline
 & \center Cost \\ (ignoring gradient)&  \center Cost \\ (gradient costs $d$ evaluations) & momentum term? & force term?& \\ 
\hline
\center  & \begin{flushright} $^\bigstar$ \end{flushright} & \begin{flushright} $^\bigstar$ \end{flushright} &  &  &\\
\vspace{-10mm} \center HMC (unadjusted) & \vspace{-10mm} \center$d^{\frac{1}{2k}}$  & \vspace{-10mm} \center $d^{1+\frac{1}{2k}}$  &\vspace{-10mm}\center yes &\vspace{-10mm} \center yes &\\ 
\hline
\center  & \begin{flushright} $^\ddagger$ \end{flushright} & \begin{flushright} $^\ddagger$ \end{flushright} & &  &\\
\vspace{-10mm} \center Langevin (unadjusted) & \vspace{-10mm} \center$d^{\frac{1}{2}}$  & \vspace{-10mm} \center $d^{1+\frac{1}{2}}$ &\vspace{-10mm} \center no &\vspace{-10mm}\center  yes &\\ 
\hline
\center  & \begin{flushright} $^\S$ \end{flushright} & \begin{flushright} $^\S$ \end{flushright} & &  &\\
\vspace{-10mm} \center Ball Walk & \vspace{-10mm} \center$d^{2}$  & \vspace{-10mm} \center $d^{2}$  &\vspace{-10mm}\center no &\vspace{-10mm}\center  no &\\ 

\hline

\end{tabular}
\\
$\bigstar=$ \emph{This paper},
$\ddagger=$ \cite{durmus2016sampling2, durmus2016sampling},
$\S=$ \cite{vempala2005geometric}
\end{center}
\caption{The best available bounds on the computational complexity for the ball-walk and unadjusted implementations of HMC and Langevin on strongly logconcave targets, in terms of the dimension $d$.  
}\label{table:cost_ball}
\end{table}

Our main techniques in this paper are explicit comparisons of ODEs and probabilistic coupling bounds. Recall that HMC is thought to explore many target distributions faster than the ball walk or the Langevin algorithm because the momentum and force terms in Hamiltonian dynamics allow HMC to take longer steps while remaining within a high-density region of the target distribution \cite{betancourt2017geometric} (Langevin has a force term but no momentum, and the ball walk has neither term; see Table \ref{table:cost_ball}).  Most existing methods for analyzing ``geometric" MCMC algorithms such as hit-and-run use conductance bounds \cite{vempala2005geometric}. However, conductance bounds generally cannot capture improvements obtained from momentum, force, or curvature (see \textit{e.g.} \cite{chen1999lifting,ramanan2017bounds} for quantitative theorems related to this heuristic and \cite{mangoubi2017bottleneck} for related results on HMC).  To take advantage of positive curvature, momentum, and force, our analysis of HMC instead uses probabilistic coupling bounds, obtained via comparison theorems for ODEs.  Our analysis suggests that the long-term momentum of Hamiltonian trajectories allows HMC to take much larger steps than the ball walk or Langevin algorithms, allowing HMC to explore the target distribution with a much faster mixing time than either the ball walk or Langevin.

\subsection{Preliminary Notation}

For any function $f: \mathbb{R}^a \rightarrow \mathbb{R}$, we use the shorthand $f' := \nabla f$, and denote by  $D_v f := \langle v, \nabla f \rangle$ the directional derivative in the direction $v$.  For a vector-valued function $g=(g_1,\ldots, g_b)^\top$, we define the coordinate-wise directional derivative $D_v g:= (D_v g_1, \dots, D_v g_b)$.

Throughout the paper, our goal is to sample from a stationary distribution $\pi(q)$ on $\mathbb{R}^{d}$, which we will write as $\pi(q) \propto e^{-U(q)}$ for  some potential function $U \, : \, \mathbb{R}^{d} \mapsto \mathbb{R}^{+}$.  We always assume that $U$ is \textit{strongly convex:}

\begin{defn} [Strong Convexity]
Say that a differentiable function $U \, : \, \mathbb{R}^{d} \mapsto \mathbb{R}$ is \textit{strongly convex with parameter} $m > 0$ if it satisfies the inequality 
\be 
\langle U'(x) - U'(y), x-y \rangle \geq m \|x - y \|^{2}
\ee  
for all $x, y \in \mathbb{R}^{d}$.
\end{defn}

Recall that any strongly convex function has a unique minimizer. Throughout this paper, we assume WLOG that this minimum occurs at 0 in order to simplify notation.

Throughout the paper, we make a few small abuses of notation. For any function $f \, : \, X \mapsto Y$ between two sets, and any $S \subset X$, we  define 
\be 
f(S) = \{ f(x) \, : \, x \in S\}.
\ee 
In addition, we will generally write $x$ for the single-element set $\{x \}$ when this does not result in any ambiguity.

\subsubsection{Distributions and Mixing}

For measures $\nu_1, \nu_2$ on a measure space $(\Omega, \mathcal{F})$, the \textit{total variation distance} between $\nu_1, \nu_2$ is given by
\be 
\| \nu_1 - \nu_2 \|_{\TV} = \sup_{A \in \mathcal{F}} (\nu_1(A) - \nu_2(A)).
\ee 
We denote the distribution of a random variable $X$ by $\mathcal{L}(X)$ and write $X \sim \nu$ as a shorthand for $\mathcal{L}(X) = \nu$. 

For two probability measures $\nu_1, \nu_2$ on $\mathbb{R}^d$, define the \textit{Wasserstein-$k$ distance}
\be
W_{k}(\nu_1,\nu_2)^{k} = \inf_{(X,Y) \in \mathcal{C}( \nu_1, 
 \nu_2)} \E[ \|X-Y\|^{k}],
\ee
where $\mathcal{C} (\nu_{1},\nu_{2})$ is the set of all random variables on $\mathbb{R}^d\times \mathbb{R}^d$ with marginal distributions $\nu_1$ and $\nu_2$. Similarly, define the \textit{Prokhorov distance} 
\be 
\mathsf{Prok}(\nu_{1},\nu_{2}) = \inf \{ \epsilon > 0 \, : \, \forall \text{ measurable } A \subset \mathbb{R}^{d}, \,  \nu_{1}(A) \leq \nu_{2}(A_{\epsilon}) + \epsilon \text{ and }  \nu_{2}(A) \leq \nu_{1}(A_{\epsilon}) + \epsilon  \},
\ee 

where  $A_{\epsilon} = \{ x \in \mathbb{R}^{d} \, : \, \inf_{y \in A} \, \| x-y\| < \epsilon \}$.

Recall that the \textit{spectral gap} of a transition kernel $K$ is given by 
\be 
1 - \sup \{ |\lambda| \, : \, \lambda \in (-1,1), \,  (K - \lambda \, \mathrm{Id}) \, \text{ is not invertible.} \}.
\ee 

We define the \textit{relaxation time} $\tau_{\mathrm{rel}}(K)$ of a transition kernel $K$ to be the reciprocal of its spectral gap.

\subsubsection{Big-O Notation}
For two nonnegative functions or sequences $f,g$, we write $f = O(g)$ as shorthand for the statement: there exists a constant $0 < C < \infty$ so that for all $x_{1},\ldots,x_{n}$, we have $f(x_{1},\ldots,x_{n}) \leq C \, g(x_{1},\ldots,x_{n})$. We write $f = \Omega(g)$ for $g = O(f)$, and we write $f = \Theta(g)$ if both $f= O(g)$ and $g=O(f)$. Relatedly, we write $f = o(g)$ as shorthand for the statement: $\lim_{x_{1},\ldots,x_{n} \rightarrow \infty} \frac{f(x_{1},\ldots,x_{n})}{g(x_{1},\ldots,x_{n})} = 0$. Finally, we use ``$f$ is polynomial in $x$" as shorthand for ``there exist $0 < C_{1},C_{2}, C_{3}< \infty$ such that $f(x) \leq C_{1} x^{C_{2}}$ for all $x>C_{3}$."

We give two small extensions of this notation; all of these modifications apply in the obvious way to $\Omega(\cdot)$, $\Theta(\cdot)$ and $o(\cdot)$. First, when we wish to view a function as depending on only a subset of its arguments, we indicate the arguments of interest using a subscript. For example, if $f(x,y) = \frac{x}{1 + y^{2}}$, we write $f(x,y) = O(x)$ but also $f(x,y) = O_{x}(x)$ or $f(x,y) = O_{y}(1)$. Second, we use a superscript $*$ to indicate that the relationship holds ``up to logarithmic factors." For example, we write $f= O^{*}(g)$ if there exist constants $0 < C_{1},C_{2} < \infty$ so that, for all $x_{1},\ldots,x_{n}$, we have $f(x_{1},\ldots,x_{n}) \leq C_{1} \, g(x_{1},\ldots,x_{n}) \, \log(g(x_{1},\ldots,x_{n}))^{C_{2}}$.

\subsubsection{Ideal HMC Dynamics}

A Hamiltonian of a simple system is written as
\be \label{EqHamilDef}
H(q,p) = U(q) + \frac{\mathsf{1}}{2}\|p\|^{2}, 
\ee 
where $q$ represents `position', $p$ represents `momentum,' $U$ represents `potential energy,' and $\frac{\mathsf{1}}{2}\|p\|^{2}$ represents  `kinetic energy.'

For fixed $\mathbf{q} \in \mathbb{R}^d,\, \mathbf{p} \in \mathbb{R}^d$, we denote by $\{q_t(\mathbf{q},\mathbf{p})\}_{t \geq 0}$, $\{ p_t(\mathbf{q},\mathbf{p}) \}_{t \geq 0}$ the solutions to Hamilton's equations:
\be \label{EqHamiltonEquations}
\frac{\mathrm{d}q_t(\mathbf{q},\mathbf{p})}{\mathrm{d}t} = p_t(\mathbf{q},\mathbf{p}), \qquad \frac{\mathrm{d}p_t(\mathbf{q},\mathbf{p})}{\mathrm{d}t} = -U'(q_t(\mathbf{q},\mathbf{p})), 
\ee 
with initial conditions 
\be 
q_0(\mathbf{q},\mathbf{p}) = \mathbf{q}, \qquad p_0(\mathbf{q},\mathbf{p}) = \mathbf{p}.
\ee 
When  the initial conditions $(\mathbf{q},\mathbf{p})$ are clear from the context, we write $q_t$, $p_t$ in place of $q_t(\mathbf{q},\mathbf{p})$ and $p_t(\mathbf{q},\mathbf{p})$, respectively. The dependence of these solutions on the Hamiltonian $H$ is always suppressed in our notation, as it will always be clear from the context.

For a fixed integration time $T \in \mathbb{R}^{+}$ and starting point $\mathbf{q} \in \mathbb{R}^{d}$, we define the solution map $\mathcal{Q}_{T}^{\mathbf{q}} \, : \, \mathbb{R}^{d} \mapsto \mathbb{R}^{d}$ by
\be  \label{EqForwardMappingRep}
\mathcal{Q}_{T}^{\mathbf{q}}(\mathbf{p}) := q_T(\mathbf{q},\mathbf{p}).
\ee

For $V > 0$, denote by $\Phi_{V}(\cdot)$ the normal distribution on $\mathbb{R}^{d}$ with mean 0 and variance $V$ times the identity matrix. We study the simplest Hamiltonian Monte Carlo algorithm (see Figure \ref{fig:walk}):

\begin{algorithm}[H]
\caption{Simple HMC \label{DefSimpleHMC}}
\flushleft
\textbf{parameters:} Potential $U$, integration time $T>0$, running time $i_{\max} \in \mathbb{N}$.\\
 \textbf{input:}   Initial point $X_0 \in \mathbb{R}^d$.\\
 \textbf{output:} Markov chain $X_0, X_1, \ldots, X_{i_{\max}}$.
\begin{algorithmic}[1]
\For{$i=0$ to $i_{\mathrm{max}}-1$}
\\Sample $\mathbf{p}_i \sim \Phi_{1}$.
\\ Set $X_{i+1} = \mathcal{Q}_{T}^{X_{i}}(\mathbf{p}_i)$.
\EndFor
\end{algorithmic}
\end{algorithm}

In the context of HMC, we refer to $q_t$ as the \textit{position} variable and $p_t$ as the \textit{momentum} variable. In this context, we call $\mathbb{R}^{d}$ the \textit{state space} of the algorithm and $\mathbb{R}^{2d}$ the \textit{phase space} of the HMC algorithm.

Note that the sequence $\{X_{i}\}_{i \geq 0}$ is a deterministic function of its initial value $X_{0}$ and the i.i.d. sequence $\{\mathbf{p}_{i}\}_{i \geq 0}$ of momentum updates sampled during this algorithm. In the Markov chain literature, this fact is summarized by saying that this algorithm defines a \textit{random mapping representation} of  $\{X_{i}\}_{ i \geq 0}$ with \textit{update sequence} $\{\mathbf{p}_{i}\}_{i \geq 0}$ (see Chapter 1.2 of \cite{LPW09}). In particular, the fact that this algorithm gives a random mapping representation means that it is possible to define a coupling of two Markov chains evolving according to this algorithm by defining a coupling of the momentum updates. All of the HMC-based algorithms defined in this paper will also have this property, and we will use it throughout the paper to construct couplings of Markov chains. Finally, note that this algorithm also naturally defines the nonreversible Markov chains  $\{ (X_i, \mathbf{p}_i)\}_{i \geq 0}$, which we call the \textit{phase-space} chain on $\mathbb{R}^{2d}$.

\begin{figure}[t]
\begin{center}
\includegraphics[trim={0 42mm 0 31mm}, clip, scale=0.4]{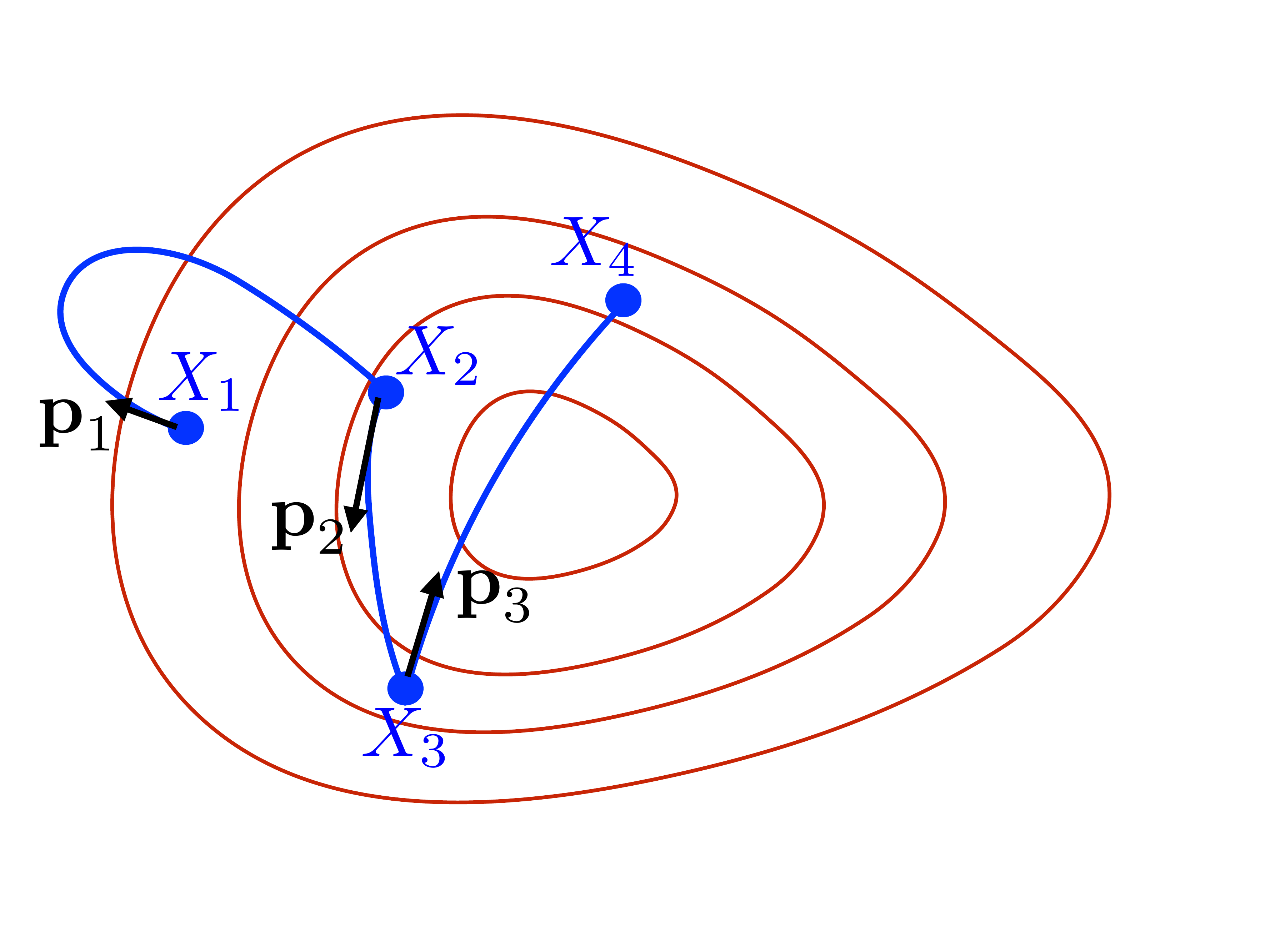}
\end{center}
\caption{The Hamiltonian Monte Carlo Markov chain $X_1, X_2, \ldots$ with momentum $\textbf{p}_{1}, \textbf{p}_{2},\ldots$.  }\label{fig:walk}
\end{figure}

\subsection{Assumptions for Ideal HMC Dynamics}

We will prove that HMC mixes quickly if the potential $U$ satisfies the following strong convexity assumption on the bulk of the target distribution: 

\begin{assumptions} [Strong Convexity Assumptions] \label{AssumptionsConvexity}

Assume that there exist constants $0 < m_{2}, M_{2} < \infty$ and a measurable set $\mathcal{X} \subset \mathbb{R}^{d}$ so that  $ -\langle D_v U'(q),v \rangle \frac{1}{\|v\|^2} \leq -m_2$ and $\|D_v U'(q) \| \leq M_2$ for all $q \in \mathcal{X}$ and $v \in \mathbb{R}^{d}$. Also assume that $U(0) = 0$ and $U(x) \geq 0 $ for all $x \in \mathbb{R}^{d}$.
\end{assumptions}

It is often convenient to assume that HMC also satisfies a drift condition (see Lemma \ref{ThmDriftHMC} for some simple and easily-verified sufficient conditions, as well as Remark \ref{RemGenDriftBounds} for related computations):

\begin{assumptions} [Drift Condition] \label{AssumptionsDrift}
Let $K$ be the transition kernel of interest. There exists a function $V \, : \, \mathbb{R}^{d} \mapsto \mathbb{R}^{+}$ and constants $0 < a \leq 1$, $1 \leq b < \infty$ so that, for all $x \in \mathbb{R}^{d}$ and $X \sim K(x,\cdot)$, 
\be 
\E[V(X) ] \leq (1 - a) V(x) + b.
\ee 
\end{assumptions}

To obtain mixing bounds in the strong Total Variation metric, we also assume:

\begin{assumption} \label{AssumptionSecondDerivative}
There exists $0 \leq M_{3} < \infty$ so that, for all $u,v, q \in \mathbb{R}^d$,
\be 
\| D_u D_v U'(q) \|\leq M_3. 
\ee
\end{assumption}

\subsection{Main Results for Ideal HMC Dynamics} \label{SecMainResApproxHmc}

We state the main results for the ideal HMC dynamics.  For simplicity, all  main results are stated for the largest integration time $T= \frac{1}{2 \sqrt{2}} \frac{\sqrt{m_{2}}}{M_{2}}$ allowed by our proofs.  We observe that these results can still be applied for any $0<T\leq \frac{1}{2 \sqrt{2}} \frac{\sqrt{m_{2}}}{M_{2}}$, since a potential $U$ that satisfies Assumption \ref{AssumptionsConvexity} for constants $m_{2}, M_{2}$ will also satisfy it for any pair $m_{2}',M_{2}'$ satisfying $0 < m_{2}' < m_{2} \leq M_{2} < M_{2}' < \infty$.

\begin{thm} [Mixing For Strongly Log-Concave Targets] \label{ThmMainConcave}
Let $K$ be the transition kernel defined by Algorithm \ref{DefSimpleHMC} with parameter $T= \frac{1}{2 \sqrt{2}} \frac{\sqrt{m_{2}}}{M_{2}}$. Let Assumption \ref{AssumptionsConvexity} hold with $\mathcal{X} = \mathbb{R}^{d}$. Then $K$ satisfies the contraction bound 
\be \label{MixingLogConcaveMainConc2}
\sup_{x,y \in \mathbb{R}^{d}} \frac{W_{k}(K(x,\cdot),K(y,\cdot))}{\| x- y \|} \leq 1 - \frac{1}{64} \left( \frac{m_{2}}{M_{2}} \right)^{2} \quad \quad \forall k \in \mathbb{N}
\ee 
and the spectral bound
\be \label{MixingLogConcaveMainConc3}
\tau_{\mathrm{rel}}(K) \leq 64  \left( \frac{M_{2}}{m_{2}} \right)^{2}.
\ee 

If Assumption \ref{AssumptionSecondDerivative} is also satisfied, then $K$  also satisfies the mixing bound
\be \label{MixingLogConcaveMainConc1}
\| K^{t}(x,\cdot) - \pi \|_{\mathrm{TV}} \leq 25^{-\lfloor \frac{\kappa t}{100} \rfloor} + e^{-\frac{\kappa \delta t}{2}}  e^{\|x\| + \frac{d}{4 m_{2}}}
\ee  
for all $x \in \mathbb{R}^{n}$ and $t \in \mathbb{N}$, for constants
\be \label{EqThm1Consts}
\kappa = \frac{1}{64} \frac{m_{2}^{2}}{M_{2}^{2}}, \, \delta = \Theta \left( \min \left( \sqrt{\frac{m_{2}}{M_{2}}}, d^{-1.5}, \frac{\sqrt{m_{2}M_{2}}}{M_{3} d}, \frac{d M_{3}}{M_{2}^{2}} \right) \right).
\ee 
\end{thm}

\begin{proof}
The proof is deferred to Section \ref{SecMainRes}.
\end{proof}

\begin{remark}
The ratio $\frac{m_{2}}{M_{2}}$ that appears prominently in the conclusions of Theorem \ref{ThmMainConcave} can be made much smaller for realistic examples by the use of appropriate preconditioning steps. See Section \ref{SecPreproc} for details.
\end{remark}

 It is common for posterior distributions to be strongly log-concave in a large region around their mode, but not log-concave on the entire state space; for example, the $t$ distributions have this property. When the target distribution is strongly log-concave only on the bulk, we obtain the slightly weaker guarantee:

\begin{thm} [Mixing For Targets With Strongly Log-Concave Bulk] \label{ThmMainConcaveBulk}
Let $K$ be the transition kernel defined by Algorithm \ref{DefSimpleHMC}  with parameter  $T= \frac{1}{2 \sqrt{2}} \frac{\sqrt{m_{2}}}{M_{2}}$. Let Assumptions \ref{AssumptionsConvexity}, \ref{AssumptionsDrift} and \ref{AssumptionSecondDerivative} hold, with  $\mathcal{X}$ satisfying
\be \label{MixingBulkContainmentCond1}
\{ x \, : \, U(x) \leq  d \, \log(200m) +  \sup_{x \, : \, V(x) \leq \frac{4000mb}{a}} U(x)  \} \subset \mathcal{X}.
\ee 

Then $K$  satisfies the mixing bound 
\be \label{MixingLogConcaveBulkMainConc1}
\| K^{t}(x,\cdot) - \pi \|_{\mathrm{TV}} \leq \beta_{1}^{\lfloor\frac{t}{m} \rfloor} + \beta_{2}^{\lfloor\frac{t}{m} \rfloor} \, (1 + \frac{b}{a} + V(x)), \,
\ee 
where the constants are given by
\be 
\log(\beta_{1}) = \frac{\log(1 + \frac{a}{3})}{\log(\frac{80b}{a^{2}})} \log(0.9) < 0, \quad \beta_{2} = (1 - \frac{1}{3}a) < 1
\ee

and 

\be \label{EqThm2Consts}
\delta = \Theta \left( \min \left( \sqrt{\frac{m_{2}}{M_{2}}}, d^{-1.5}, \frac{\sqrt{m_{2}M_{2}}}{M_{3} d}, \frac{d M_{3}}{M_{2}^{2}} \right) \right), \, m = \big\lceil \frac{\log(\delta)}{\log(1 - \frac{m_{2}^{2}}{64 M_{2}^{2}})} \big\rceil.
\ee 

Furthermore, under the same conditions $K$ satisfies the spectral bound 
\be \label{SpectralLogConcaveBulkMainConc1}
\tau_{\mathrm{rel}}(K) = O \left(\frac{m}{1 - \max(\beta_{1},\beta_{2})} \right).
\ee 

\end{thm}

\begin{proof}
The proof is deferred to Appendix \ref{SecThmMainConcaveBulk}.
\end{proof}

\subsection{Approximate HMC Dynamics}

It is difficult to solve Hamilton's equations \eqref{EqHamiltonEquations} for most Hamiltonians of interest. In practice, one uses numerical integrators such as the Euler or Verlet integrator in order to approximate solutions (see \cite{neal2011mcmc}). We set some generic notation for approximations of HMC dynamics and recall some standard numerical integration algorithms. 

Throughout the paper, by ``numerical integrator" we mean any parameterized family $\dagger = \{ \dagger_{d, U, \theta, T} \}$ of functions $\dagger_{d,U,\theta,T} \, : \,  \mathbb{R}^{2d} \mapsto \mathbb{R}^{2d}$, where the parameters take the following values:
\begin{itemize}
\item The dimension $d \in \mathbb{N}$.
\item The potential function $U$ is any twice-differentiable function $U \, : \, \mathbb{R}^{d} \mapsto \mathbb{R}^{+}$.
\item The accuracy level $\theta \in \mathbb{R}^{+}$ is any nonnegative real number.
\item The integration time $T \in \mathbb{R}^{+}$ is any nonnegative real number.
\end{itemize}

We assume that all numerical integrators $\dagger$ satisfy the following separability condition:

\begin{defn} [Separability Condition] \label{DefSepCond}
For all $\theta, T > 0$, all $a, b \in \mathbb{N}$ and all potential functions $U_{a} \, : \, \mathbb{R}^{a} \mapsto \mathbb{R}$ and $U_{b} \, : \, \mathbb{R}^{b} \mapsto \mathbb{R}$, we have 
\be 
(\dagger_{a,U_{a},\theta,T}, \dagger_{b,U_{b},\theta,T}) = \dagger_{a+b,U_{a} \otimes U_{b},\theta,T},
\ee 
where $U_{a} \otimes U_{b} \, : \, \mathbb{R}^{a+b} \mapsto \mathbb{R}$ is given by the formula
\be 
U_{a} \otimes U_{b}  (x_{1},\ldots,x_{a+b}) = U_{a}(x_{1},\ldots,x_{a}) + U_{b}(x_{a+1},\ldots,x_{a+b}).
\ee 
\end{defn}

To make the notation for approximate HMC dynamics similar to that for the ideal HMC dynamics, for any numerical integrator $\dagger$ and potential function $U \, : \, \mathbb{R}^{d} \mapsto \mathbb{R}^{+}$ we define 
\be 
(q^{\dagger \theta}_{T}(\mathbf{q},\mathbf{p}), p^{\dagger \theta}_{T}(\mathbf{q},\mathbf{p})) &= \dagger_{d,U,\theta,T}(\mathbf{q},\mathbf{p}). \\
\ee 
Note that, as in the notation for exact solutions to Hamilton's dynamics, the dependence on $U$ is suppressed.

Numerical integrators generally work by repeatedly taking very small steps. To distinguish between the full numerical integrator and its component small steps, we refer to a parameterized collection of functions $ \bigstar = \{ \bigstar_{d,U,\theta} \}$  with $\bigstar_{d,U,\theta} \, : \, \mathbb{R}^{2d} \mapsto \mathbb{R}^{2d}$ as an \textit{oracle}, where the range of $d,U, \theta$ is the same as for numerical integrators. Again, to make the notation similar to our notation for HMC we define 
\be 
(q^\star_{\theta}(\mathbf{q},\mathbf{p}), p^\star_{\theta}(\mathbf{q},\mathbf{p})) = \bigstar_{d, U,\theta}(\mathbf{q},\mathbf{p})
\ee 

and suppress the dependence on the potential $U$. We now recall some popular numerical integrators in the literature. The standard first-order Euler integrator is given by Algorithm \ref{alg:Approx_integrator} with oracle given by Algorithm \ref{alg:Integrator_first_order} and $k=1$:

\begin{algorithm}[H]
\caption{First-order Euler integrator \label{alg:Integrator_first_order}}
\flushleft
\textbf{parameter:} Potential $U$ and accuracy $\theta>0$.\\
 \textbf{input:}   Initial point $\mathbf{q} \in \mathbb{R}^d$, initial momentum $\mathbf{p} \in \mathbb{R}^d$.\\
 \textbf{output:}  Approximate position $q^\star_{\theta}(\mathbf{q},\mathbf{p})$ and momentum $p^\star_{\theta}(\mathbf{q},\mathbf{p})$.\\
\begin{algorithmic}[1]
\State Set $q^\star_{\theta}(\mathbf{q},\mathbf{p}) = \mathbf{q}+\mathbf{p}\theta$.
\\ Set $p^\star_{\theta}(\mathbf{q},\mathbf{p}) = \mathbf{p} - \theta U'(\mathbf{q})$.
\end{algorithmic}
\end{algorithm}

The standard second-order integrator is given by Algorithm \ref{alg:Approx_integrator} with oracle given by Algorithm \ref{alg:Integrator_leapfrog} and $k=2$:

\begin{algorithm}[H]
\caption{Second-order leapfrog integrator \label{alg:Integrator_leapfrog}}
\flushleft
\textbf{parameter:} Potential $U$ and accuracy $\theta>0$.\\
 \textbf{input:}   Initial point $\mathbf{q} \in \mathbb{R}^d$, initial momentum $\mathbf{p} \in \mathbb{R}^d$.\\
 \textbf{output:}  Approximate position $q^\star_{\theta}(\mathbf{q},\mathbf{p})$ and momentum $p^\star_{\theta}(\mathbf{q},\mathbf{p})$.\\
\begin{algorithmic}[1]
\State Set $\hat{p}^\star_{\frac{\theta}{2}}(\mathbf{q},\mathbf{p}) = \mathbf{p} - \frac{\sqrt{\theta}}{2} U'(\mathbf{q})$.
\\Set $q^\star_{\theta}(\mathbf{q},\mathbf{p}) = \mathbf{q}+\sqrt{\theta} \hat{p}^\star_{\frac{\theta}{2}}(\mathbf{q},\mathbf{p}) $.
\\ Set $p^\star_{\theta}(\mathbf{q},\mathbf{p}) = \hat{p}^\star_{\frac{\theta}{2}}(\mathbf{q},\mathbf{p}) - \frac{\sqrt{\theta}}{2} U'(q^\star_{\theta}(\mathbf{q},\mathbf{p}))$.
\end{algorithmic}
\end{algorithm}

Algorithm \ref{alg:Approx_integrator} gives the standard Hamiltonian integrator algorithm based on a generic oracle:

\begin{algorithm}[H]
\caption{Numerical Approximation of Hamiltonian Trajectory \label{alg:Approx_integrator}}
\flushleft
\textbf{parameters:} Oracle $\bigstar$, potential $U$, integration time $T>0$, order $k \in \mathbb{N}$ and accuracy $\theta>0$.\\
 \textbf{input:}   Initial point $\mathbf{q} \in \mathbb{R}^d$, initial momentum $\mathbf{p} \in \mathbb{R}^d$.\\
 \textbf{output:} $q^{\dagger \theta}_{T}(\mathbf{q},\mathbf{p})$.
\begin{algorithmic}[1]
\State Set $q^0_\dagger=\mathbf{q}$ and $p^0_\dagger = \mathbf{p}$. 
\For{$i=0$ to $\lceil \frac{T}{\theta^{\frac{1}{k}}}-1 \rceil$}
\\ Call the oracle $\bigstar$ to compute $q^{i+1}_\dagger = q^\star_{\theta}(q^i_\dagger,p^i_\dagger)$ and $p^{i+1}_\dagger = p^\star_{\theta}(q^i_\dagger,p^i_\dagger)$.
\EndFor
\\ Set $q^{\dagger \theta}_{T}(\mathbf{q},\mathbf{p}) = q^{i+1}_\dagger$.
\end{algorithmic}
\end{algorithm}

There are two popular algorithms for defining an HMC Markov chain based on a numerical integrator. The first is the \textit{unadjusted} HMC algorithm:

\begin{algorithm}[H]
\caption{Unadjusted HMC \label{alg:Unadjusted}}
\flushleft
\textbf{parameters:} Number of steps $i_{\max} \in \mathbb{N}$, numerical integrator $\dagger$, potential $U$, integration time $T>0$, and accuracy $\theta>0$.\\
 \textbf{input:}   Initial point $X_0 \in \mathbb{R}^d$.\\
 \textbf{output:} Markov chain $X_0, X_1, \ldots, X_{i_{\max}}$.
\begin{algorithmic}[1]
\For{$i=0$ to $i_{\mathrm{max}}-1$}
\\Sample $\mathbf{p}_i \sim \Phi_{1}$.
\\ Set $X_{i+1} =  q^{\dagger \theta}_{T}(X_i,\mathbf{p}_i)$.
\EndFor
\end{algorithmic}
\end{algorithm}

Algorithm \ref{alg:Unadjusted} also naturally defines the phase-space unadjusted HMC Markov chain $\{(X_i, \mathbf{p}_i) \}_{i \geq 0}$.

The second is the \textit{Metropolis-adjusted} HMC Markov chain. Letting $\mathrm{Unif}([0,1])$ be the uniform distribution on $[0,1]$, we write:

\begin{algorithm}[H]
\caption{Metropolis-adjusted HMC \label{alg:Metropolis}} 
\flushleft
\textbf{parameters:} Number of steps $i_{\max} \in \mathbb{N}$, numerical integrator $\dagger$, potential $U$, integration time $T>0$, and accuracy $\theta>0$.\\
 \textbf{input:}   Initial point $X_0 \in \mathbb{R}^d$.\\
 \textbf{output:} Markov chain $X_0, X_1, \ldots, X_{i_{\max}}$.
\begin{algorithmic}[1]
\For{$i=0$ to $i_{\mathrm{max}}-1$}
\\Sample $\mathbf{p}_i \sim \Phi_{1}$.
\\ Set $q_{i+1} =  q^{\dagger \theta}_{T}(X_i,\mathbf{p}_i)$ and  $p_{i+1} =  p^{\dagger \theta}_{T}(X_i,\mathbf{p}_i)$.
\\ Sample $b_{i} \sim \mathrm{Unif}([0,1])$.
\\ Set  $X_{i+1} = q_{i+1}$ if $b_{i} < \min\{1, e^{H(q_{i},p_{i})-H(q_{i+1},p_{i+1})}\}$; else set  $X_{i+1} =X_i$
\EndFor
\end{algorithmic}
\end{algorithm}

Algorithm \ref{alg:Metropolis} also naturally defines the phase-space unadjusted HMC Markov chain $\{(X_i, \mathbf{p}_i) \}_{i \geq 0}$.

\subsection{Main Assumptions for Approximate HMC Dynamics}

We will obtain bounds on the mixing of several approximations to the ideal HMC dynamics. When analyzing the first-order Euler approximation of Algorithms \ref{alg:Approx_integrator} and \ref{alg:Integrator_first_order}, we will not need any additional assumptions. However, when analyzing higher-order approximations, we will make additional assumptions on the potential $U$ and numerical integrator $\dagger$. These assumptions follow in the tradition of previous papers on scaling limits, such as \cite{HMC_optimal_tuning, roberts1997weak}. 

Fix $\mathsf{m} \in \mathbb{N}$ for this section. Also set the notation $x^{(i)} := (x[\mathsf{m}(i-1)+1], \ldots, x[\mathsf{m}i])$ for $x \in  \mathbb{R}^{d}$. We make the following seperability assumption on $U$:

\begin{assumption} \label{assumption:product_measure_potential}
There exist functions $U_1,U_2, \ldots, U_{\frac{d}{\mathsf{m}}}$ with  $U_i \, : \, \mathbb{R}^{\mathsf{m}}\rightarrow \mathbb{R}^+$ for every $i \in \{1, \ldots \frac{d}{\mathsf{m}}\}$, such that
\be
U(q) = \sum_{i=1}^\frac{d}{\mathsf{m}} U_i(q^{(i)}) \quad \quad \forall q \in \mathbb{R}^d.
\ee
In this case we define $H_i(q^{(i)},p^{(i)}) = U_i(q^{(i)}) + \frac{1}{2}\|p^{(i)}\|^2$ for every $ q,p \in \mathbb{R}^d$ and all $i \in \{1, \ldots \frac{d}{\mathsf{m}}\}$.
\end{assumption}

We will assume that our numerical integrator is very accurate for well-behaved potentials. In particular, we consider the class of potentials:

\begin{assumption} (Assumption 5.1 of \cite{HMC_optimal_tuning}) \label{assumption:leapfrog} 
There exists a constant $\mathsf{B}>0$, and numbers $\mathsf{a},\mathsf{b} \in \mathbb{N}$, such that for every $i$ the potential $U_i$ satisfies:
\begin{itemize}
\item $U_i \in C^{\mathsf{a}}( \mathbb{R}^\mathsf{m} \rightarrow \mathbb{R}^{+})$.
\item The first $\mathsf{b}$ Fr\'echet derivatives of $U'_i$ are uniformly bounded by the constant $\mathsf{B}$.
\end{itemize}
\end{assumption}

On this class of potentials, we make the following assumption about the integrator $\dagger$:

\begin{cond} \label{condition:kth_order}
There exists some constants $\mathsf{B}, \mathsf{c}>0$, integers $\mathsf{a},\mathsf{b} \in \mathbb{N}$, and a constant $\mathsf{K} = \mathsf{K}_{T} > 0$  that depends only on $T$, such that for every potential $U: \mathbb{R}^{\mathsf{m}} \rightarrow \mathbb{R}^+$ satisfying Assumption \ref{assumption:leapfrog} with these constants $\mathsf{a}, \mathsf{b}, \mathsf{B}$ we have
\be
\|q_T^{ \dagger \theta }(\mathbf{q}, \mathbf{p}) - q_T(\mathbf{q}, \mathbf{p})\| \leq \theta \mathsf{K} (H^{\mathsf{c}}(\mathbf{q}, \mathbf{p}) +1)
\ee
and
\be
|H(q_T^{ \dagger \theta}(\mathbf{q},\mathbf{p}), p_T^{\dagger \theta}(\mathbf{q}, \mathbf{p}))-  H(\mathbf{q},\ \mathbf{p})| \leq \theta \mathsf{K} (H^{2\mathsf{c}}(\mathbf{q},\mathbf{p}) +1).
\ee
\end{cond}

Finally, we formalize the notion of ``computational cost" of a numerical implementation of HMC as follows:

\begin{defn} [Computational Cost]

Let $Q$ be the transition kernel of a Markov chain given by either Algorithm \ref{alg:Unadjusted} or \ref{alg:Metropolis}, with notation as in those algorithms. Assume that the numerical integrator $\dagger$ used to define $Q$ is given by Algorithm \ref{alg:Approx_integrator}, with order $k$ as in that algorithm. We then define the \textit{computational cost of taking $s$ steps from $Q$} by the function 
\be
N(Q,s) = \theta^{-\frac{1}{k}} \, T \, s 
\ee
for $s \geq 0$.

\end{defn}

When the oracle used by Algorithm \ref{alg:Approx_integrator} is given by either the Algorithm \ref{alg:Integrator_first_order} or  \ref{alg:Integrator_leapfrog}, $N(Q,s)$ is a bound on the number of times that the gradient of $U$ is evaluated while taking $s$ steps of the HMC Markov chain. More generally, this is close to the number of gradient evaluations that a typical $k$'th-order implementation of HMC will require. In practice, the number of gradient evaluations made is a good proxy for the running time of an HMC chain, which motivates the definition of $N(Q,s)$.

\subsection{Main Results for Approximate HMC Dynamics}

Using the bounds on the mixing of ideal HMC dynamics in Section \ref{SecMainResApproxHmc}, we can obtain similar results for implementations of HMC with various integrators.  All proofs are deferred to Appendix \ref{AppProofsMainApproxRes}. Note that, in each theorem, we give a range on the value of the parameter $\theta$ for which the conclusion holds. 

Our first result says that a simple numerical integrator can give approximate samples from $\pi$ with  $\mathcal{O}_{d}^{*}(\sqrt{d})$ computational cost:

\begin{thm} [Mixing of first-order Unadjusted HMC] \label{ThmMainApprox}

Fix $0 < \epsilon < e^{-1}$ and let $U$ satisfy Assumption \ref{AssumptionsConvexity} with $\mathcal{X} = \mathbb{R}^{d}$. Let $\dagger$ be the integrator given by Algorithm \ref{alg:Approx_integrator} with oracle given by Algorithm \ref{alg:Integrator_first_order} and let $T= \frac{1}{2\sqrt{2}} \frac{\sqrt{m_2}}{M_2}$. For $\theta > 0$, let $Q_{\theta}$ be the transition kernel defined in Algorithm \ref{alg:Unadjusted} with these parameters, and let $\nu_{\theta}$ be its stationary measure. 

Then there exists some constant $\theta_{0} = \theta_{0}(m_{2},M_{2},d,\epsilon)$ satisfying 

\be
\theta_{0} =\Omega^{\ast}\left(d^{-\frac{1}{2}} \times \epsilon \times (\frac{M_2}{m_2})^{-4.5}\right)
\ee
and some universal constant $0 < c < \infty$ so that, for all $0 < \theta \leq \theta_{0}$, the following hold:

For all $\mathcal{I} \geq c\, \frac{M_{2}^{2}}{m_{2}^{2}} \log(\frac{M_{2}}{m_{2} \epsilon})$ and all $x$ satisfying $\|x\|  \leq \frac{\sqrt{d}}{\sqrt{m_2}}$,
\be \label{WassContConc_Unadjusted1}
W_{1}(Q_{\theta}^{\mathcal{I}}(x,\cdot), \pi) \leq \epsilon.
\ee 
Furthermore, if $Q_{\theta}$ is ergodic,
\be \label{WassContConc_Unadjusted2}
W_{1}( \nu_{\theta}, \pi) \leq \epsilon.
\ee 

Finally, if we choose $\theta = \Omega^{\ast}(\theta_{0})$, then the computational cost to compute the first $s= c\, \frac{M_{2}^{2}}{m_{2}^{2}} \log(\frac{M_{2}}{m_{2} \epsilon})$ steps of the Markov chain is
\be
N(Q_{\theta},s) = O^{\ast}\left(d^{\frac{1}{2}} \times \epsilon^{-1} \times (\frac{M_2}{m_2})^{6.5}\right).
\ee
\end{thm}

\begin{remark} [Starting Point]
Under the assumption that $U$ is strongly log-concave, it is straightforward to find a starting point $x \in \mathbb{R}^{d}$ satisfying $\|x \| \leq \sqrt{\frac{d}{m_{2}}}$ using algorithms from convex optimization.
\end{remark}


Our next result shows that $k$'th-order HMC algorithms can give approximate samples from $\pi$ with  $\mathcal{O}_{d}^{*}(d^{\frac{1}{2k}})$ computational cost, under stronger assumptions:

\begin{thm} [Prokhorov Mixing of Higher-Order unadjusted HMC] \label{ThmMainApprox_leapfrog_unadjusted}

Fix $0 < \epsilon < e^{-1}$, let $U$ be a potential satisfying Assumptions \ref{assumption:product_measure_potential}, \ref{assumption:leapfrog} and \ref{AssumptionsConvexity} with $\mathcal{X} = \mathbb{R}^{d}$,  let $\dagger$ be a numerical integrator of order $k\in \mathbb{N}$ satisfying Condition \ref{condition:kth_order}, and set $T =\frac{1}{2\sqrt{2}} \frac{\sqrt{m_2}}{M_2}$. For $\theta > 0$, let $Q_{\theta}$ be the transition kernel  defined by Algorithm \ref{alg:Unadjusted} with these parameters, and denote by $\nu_{\theta}$ its stationary distribution. Then there exists some universal constant $0 < c < \infty$ and, for all $\mathcal{I} \geq c \, \frac{M_{2}^{2}}{m_{2}^{2}} \log(\frac{M_{2}}{m_{2} \epsilon})$, there exists some function $\theta_{0}$ depending on all parameters, with the following properties:

For all $0 < \theta < \theta_{0}$ and all $x \in \mathbb{R}^{d}$ satisfying $\|x\|  \leq \frac{\sqrt{d}}{\sqrt{m_2}}$, 
\be \label{WassContConcUA}
\mathsf{Prokh}(Q_{\theta}^{\mathcal{I}}(x,\cdot), \pi) \leq \epsilon.
\ee

Furthermore, $\theta_{0}$ satisfies

\be
\theta_{0}  =\Omega_{\epsilon, d}^*\left( \frac{\epsilon^2}{\sqrt{d}}\right).
\ee

Finally, if $\dagger$ is given by Algorithm \ref{alg:Approx_integrator} with order $k$ and we choose $\mathcal{I} = \Omega^*(c\, \frac{M_{2}^{2}}{m_{2}^{2}} \log(\frac{M_{2}}{m_{2} \epsilon}))$ and $\theta = \Omega^{\ast}(\theta_{0})$, the computational cost to compute the first $\mathcal{I}$ steps of the Markov chain is
\be
N(Q_{\theta}, \mathcal{I}) = O_{\epsilon, d}^{\ast}\left(d^{\frac{1}{2k}} \times \epsilon^{-\frac{2}{k}}\right).
\ee

\end{thm}

Our remaining result gives mixing bounds for Metropolis-adjusted HMC, which is guaranteed to have the correct stationary measure (see \cite{neal2011mcmc}):

\begin{thm} [Prokhorov Mixing of higher-order Metropolis-adjusted HMC] \label{ThmMainApprox_leapfrog2}

Fix $0 < \epsilon < e^{-1}$, let $U$ satisfy Assumptions \ref{assumption:product_measure_potential}, \ref{assumption:leapfrog} and \ref{AssumptionsConvexity} with $\mathcal{X} = \mathbb{R}^{d}$, let $\dagger$ satisfy Condition \ref{condition:kth_order}, and let $ T = \frac{1}{2\sqrt{2}} \frac{\sqrt{m_2}}{M_2}$. For fixed $\theta > 0$, let $Q_{\theta}$ be the transition kernel defined by Algorithm \ref{alg:Metropolis} with these parameters and let $\nu_{\theta}$ be its stationary distribution. 

Then there exists some universal constant $0 < c < \infty$ and, for all $\mathcal{I} \geq c \, \frac{M_{2}^{2}}{m_{2}^{2}} \log(\frac{M_{2}}{m_{2} \epsilon})$, there exists some function $\theta_{0}$ depending on all parameters, with the following properties:

For all $0 < \theta < \theta_{0}$ and all $x \in \mathbb{R}^{d}$ satisfying $\|x\|  \leq \frac{\sqrt{d}}{\sqrt{m_2}}$,

\be \label{ProkhContConcMH}
\mathsf{Prokh}(Q_{\theta}^{\mathcal{I}}(x,\cdot), \pi) \leq \epsilon.
\ee 

Furthermore, $\theta_{0}$ satisfies

\be
\theta_{0}  = \Omega_{\epsilon, d}^*\left( \min(\frac{\epsilon}{d}, \, \frac{\epsilon^2}{\sqrt{d}})\right).
\ee

Finally, if $\dagger$ is given by Algorithm \ref{alg:Approx_integrator} with order $k$ and we choose $\mathcal{I} = \Omega^*(c\, \frac{M_{2}^{2}}{m_{2}^{2}} \log(\frac{M_{2}}{m_{2} \epsilon}))$ and $\theta = \Omega^{\ast}(\theta_{0})$, then the computational cost to compute the first $\mathcal{I}$ steps of the Markov chain is
\be
N(Q_{\theta}, \mathcal{I}) = O_{\epsilon, d}^{\ast}\left(\max(d^{\frac{1}{k}} \epsilon^{-\frac{1}{k}}, \, \, d^{\frac{1}{2k}} \epsilon^{-\frac{2}{k}})\right).
\ee

\end{thm}

\subsection{Literature Review}

There is a large literature on obtaining quantitative bounds on the convergence rates of Markov chains (see \cite{jones2001honest} for an introduction to the statistical literature on the topic, and \cite{diaconis2009markov} for connections in other fields, including computer science). In general, it is difficult to obtain good quantitative bounds for large classes of chains. As such, the literature focuses on either finding very tight bounds for specific chains (see \textit{e.g.} \cite{diaconis2008gibbs}) or on quantatative bounds on the running time of the algorithm as a function of the problem complexity (see \textit{e.g.} \cite{borgs1999torpid} or essentially any paper in the large computer science literature on the subject). Our work falls in the latter category.

Although there are many papers that obtain quantitative bounds on the convergence of Markov chains, very little has been written in the context HMC. To our knowledge, \cite{seiler2014positive} is the only paper that focuses on obtaining quantitative bounds on the convergence rates of HMC, and it served as an inspiration for this work. A number of other papers, most prominently \cite{HMC_optimal_tuning}, have also worked on the problem of calculating the computational complexity of HMC algorithms by computing the rate at which certain proxies for the mixing or relaxation time of HMC increase with the dimension of the target distribution under reasonable conditions (see \cite{roberts2016complexity} for a general discussion relating results similar to \cite{HMC_optimal_tuning} to the usual notions of complexity). Several other papers give calculations that imply or suggest quantitative bounds, though we are not aware of any that are close to tight (see \textit{e.g.} the discussion in Section 7.5 of \cite{livingstone2016geometric}).

Our main results most closely resemble those of \cite{durmus2016sampling}, which studies the non-asymptotic mixing properties of the Langevin algorithm on strongly log-concave distributions. Their main results hold under essentially the same conditions as our Theorems \ref{ThmMainConcave}, \ref{ThmMainConcaveBulk} and \ref{ThmMainApprox}, and are compared to our results in Table \ref{table:cost}. See also \cite{dalalyan2016theoretical}, which studies very similar conditions to \cite{durmus2016sampling}.

\subsection{Paper Overview}

In Section \ref{SecTechBound} we state the main lemmas required for our main theorems and give proofs of the simplest and most important bounds (the proofs of the remaining technical bounds are deferred to the appendix). In Section \ref{SecMainRes}, we give the proof of our main result for ideal HMC dynamics. In Section \ref{SecPreproc}, we discuss preconditioning steps that may be taken before running HMC in order to improve mixing. We also explain the relationship between these preprocessing steps and our main results. Finally, in Section \ref{SecDisc} we discuss several open problems related to this work. The Appendix contains the proofs of our remaining main results.

\section{Technical Results} \label{SecTechBound}

In this section, we give some useful bounds related to the solutions of Hamilton's equations \eqref{EqHamiltonEquations}. We begin with some notation and estimates that will be used for the remainder of the section.

\subsection{Definitions and notation}

Throughout this section, we assume that the potential $U$ satisfies Assumption \ref{AssumptionsConvexity} with $\mathcal{X} = \mathbb{R}^{d}$, and we consider two solutions $(q^{(1)}_t, p^{(1)}_t)$ and $(q^{(2)}_t, p^{(2)}_t)$ of Equation \eqref{EqHamiltonEquations}. Denote by  $\tilde{q}_t := q^{(2)}_t - q^{(1)}_t$ and $\tilde{p}_t := p^{(2)}_t - p^{(1)}_t$ the differences between these solutions, and denote by $\hat{q}_t := \|\tilde{q}_t\|_{2}$ and $\hat{p}_t := \|\tilde{p}_t\|_{2}$ the magnitudes of these differences.

Hamilton's equations give 
\be\label{eq:Galilean}
\frac{\mathrm{d}\tilde{q}_t}{\mathrm{d}t} = \tilde{p}_t, \quad \quad \frac{\mathrm{d}\tilde{p}_t}{\mathrm{d}t} = -\bigg(U'(q^{(2)}_t) - U'(q^{(1)}_t)\bigg),
\ee

so we have
\be \frac{\mathrm{d}\hat{q}_t}{\mathrm{d}t} =\frac{\mathrm{d}}{\mathrm{d}t}\|\tilde{q}_t\| = \frac{\langle \tilde{p}_t, \tilde{q}_t \rangle}{\|\tilde{q}_t\|} \leq \|\tilde{p}_t\| = \hat{p}_t
\ee
and
\be
\frac{\mathrm{d}\hat{p}_t}{\mathrm{d}t} = \frac{\mathrm{d}}{\mathrm{d}t} \|\tilde{p}_t\| = \frac{\langle \frac{\d}{\d t}\tilde{p}_t, \tilde{p}_t \rangle}{\|\tilde{p}_t\|} \leq \|\frac{\mathrm{d}}{\mathrm{d}t} \tilde{p}_t\| \stackrel{{\scriptsize \textrm{Eq. }}\ref{eq:Galilean}}{=}  \| U'(q^{(2)}_t) - U'(q^{(1)}_t) \|.
\ee
This implies the following system of differential inequalities
\be\label{eq:chaplygin1}
\frac{\mathrm{d}\hat{q}_t}{\mathrm{d}t} \leq \hat{p}_t, \quad \quad \frac{\mathrm{d}\hat{p}_t}{\mathrm{d}t} \leq \bigg \| U'(q^{(2)}_t) - U'(q^{(1)}_t) \bigg\|,
\ee
with initial conditions $\hat{q}_0$ and $\hat{p}_0$.

Finally, for two points $u,v \in \mathbb{R}^{d}$, we define the unit-speed parametrization of the line connecting $u$ and $v$ to be $\ell_s(u,v) := u+s\frac{v-u}{\|v-u\|}$ for all $0\leq s \leq \|v-u\|$. We keep this last notation for the remainder of the paper.

\subsection{ODE Comparison Theorem}

We make frequent use of the following comparison theorem for systems of ordinary differential equations, a generalization of Gr{\"o}nwall's inequality originally stated in Proposition 1.4 of \cite{comparison_ODE}:

\begin{lemma} [ODE comparison theorem, Proposition 1.4 of \cite{comparison_ODE}] \label{LemmaOdeComp}
Let $U \subset \mathbb{R}^{n}$ and $I \subset \mathbb{R}$ be open, nonempty and connected. Let $f, \, g \, : \, I \times U \mapsto \mathbb{R}^{n}$ be continuous and locally Lipschitz maps. Then the following are equivalent:
\begin{enumerate}
\item For each pair $(t_{0},y)$, $(t_{0}, \overline{y})$ with $t_{0} \in I$ and $y, \overline{y} \in U$, the inequality $y \leq \overline{y}$ implies $z(t) \leq \overline{z}(t)$ for all $t \geq t_{0}$, where 
\be 
\frac{d}{dt} z &= f(t,z), \qquad z(t_{0}) = y \\
\frac{d}{dt} \overline{z} &= g(t,\overline{z}), \qquad \overline{z}(t_{0}) = \overline{y}. \\
\ee 
\item For all $i \in \{1,2,\ldots,n\}$ and all $t \geq t_{0}$, the inequality 
\be 
g(t, (\overline{x}[1], \ldots, \overline{x}[i-1], x[i], \overline{x}[i+1],\ldots,\overline{x}[n]))[i] \geq f(t, (x[1],\ldots,x[i-1],x[i],x[i+1],\ldots,x[n]))[i]
\ee 
holds whenever $\overline{x}[j] \geq x[j]$ for every $j \neq i$.
\end{enumerate}
\end{lemma}

\subsection{Error Bounds for HMC}

We give a simple estimate, showing that solutions to Equation \eqref{EqHamiltonEquations} don't diverge by very much on a small timescale:

\begin{lemma}\label{thm:bounds}
With notation as above:
\begin{enumerate}
\item  For $t \geq 0$ we have
\be \label{thmboundsconc1}
\hat{q}_t &\leq k_1e^{t\sqrt{M_2}} + k_2e^{-t\sqrt{M_2}}\\
\hat{p}_t &\leq k_1\sqrt{M_2}e^{t\sqrt{M_2}} - k_2\sqrt{M_2}e^{-t\sqrt{M_2}},
\ee 
where $k_1 = \frac{1}{2}(\hat{q}_0 + \frac{\hat{p}_0}{\sqrt{M_2}})$, $k_2 = \frac{1}{2}(\hat{q}_0 - \frac{\hat{p}_0}{\sqrt{M_2}})$.

\item Suppose that $\hat{q}_0=0$, and that $\tau > 0$ is any positive real number. Then
\be \label{thmboundsconc4}
\hat{q}_t \geq    \hat{p}_0(t - \sqrt{M_2}\sinh(\tau \sqrt{M_2})\times t^{2} ) \quad \quad \forall t\in[0,\tau].
\ee

\item Suppose instead that $\hat{p}_0=0$ and that $\tau>0$ is such that $\hat{q}_t \leq 2\hat{q}_0$ on all $t\in[0,\tau]$.  Then 
\be \label{thmboundsconc3}
\hat{q}_t \geq  \hat{q}_0 (1 - 2M_{2} t^{2}) \quad \quad \forall t\in[0,\tau].
\ee
\end{enumerate}
\end{lemma}

\begin{proof}
Recalling $\|D_v U' \|\leq M_2 \, \|v \|$ for all $v \in \mathbb{R}^{d}$, we have the initial estimate:
\be \label{eq:gradient}
\left\|\frac{\d \tilde{p}_t}{\d t} \right\| &= \bigg\|U'(q^{(2)}_t) -U'(q^{(1)}_t)  \bigg\|\\
&=  \bigg \| \int_0^{\|\tilde{q}_t\|} D_\frac{\tilde{q}_t}{\|\tilde{q}_t\|} U' \big|_{\ell_s(q^{(1)}_t, q^{(2)}_t)} \mathrm{d}s \bigg \|\\
&\leq  \int_0^{\|\tilde{q}_t\|} \bigg \| D_\frac{\tilde{q}_t}{\|\tilde{q}_t\|} U' \big|_{\ell_s(q^{(1)}_t, q^{(2)}_t)}  \bigg \| \mathrm{d}s \\
&\leq \int_0^{\|\tilde{q}_t\|} M_2 \mathrm{d}s = \|\tilde{q}_t\| M_2. 
\ee

We now prove our three conclusions in order:

\begin{enumerate}

\item  Equations \eqref{eq:chaplygin1} and \eqref{eq:gradient} together give the system of differential inequalities
\be
\frac{\mathrm{d}\hat{q}_t}{\mathrm{d}t} \leq \hat{p}_t := f_1(\hat{q}_t,\hat{p}_t), \quad \quad \frac{\mathrm{d}\hat{p}_t}{\mathrm{d}t} \leq M_2 \hat{q}_t:= f_2(\hat{q}_t,\hat{p}_t).
\ee

Define $\hat{q}^\star_t$ and $\hat{p}^\star_t$ to be the solution to the system of differential equations
\be \label{eq:chaplygin2}
\frac{\mathrm{d}\hat{q}^\star_t}{\mathrm{d}t} = \hat{p}^\star_t = f_1(\hat{q}^\star_t,\hat{p}^\star_t), \quad \quad \frac{\mathrm{d}\hat{p}^\star_t}{\mathrm{d}t} = M_2 \hat{q}^\star_t = f_2(\hat{q}^\star_t,\hat{p}^\star_t)
\ee
 with initial conditions $\hat{q}^\star_0 = \hat{q}_0$ and $\hat{p}^\star_0 = \hat{p}_0$. We now compute  $\hat{q}^\star_t$.  Turning the system of equations \ref{eq:chaplygin2} into a single second-order equation gives
\[ \frac{\mathrm{d}^2\hat{q}^\star_t}{\mathrm{d}t^2} = M_2 \hat{q}^\star_t\]
which has solution
\[\hat{q}^\star_t = k_1e^{t\sqrt{M_2}} + k_2e^{-t\sqrt{M_2}}\]
for some constants $k_1, k_2$.  Using the initial conditions to solve for the constants, 
\[k_1 = \frac{1}{2}(\hat{q}_0 + \frac{\hat{p}_0}{\sqrt{M_2}}), \quad \quad k_2 = \frac{1}{2}(\hat{q}_0 - \frac{\hat{p}_0}{\sqrt{M_2}}).\]

Noting that
\[\hat{q}_t \leq \hat{q}^\star_t,  \quad \quad \hat{p}_t \leq \hat{p}^\star_t, \quad \quad t \in [0,\infty)\]
by Lemma \ref{LemmaOdeComp} completes the proof of Inequality \eqref{thmboundsconc1}.

\item[(2),(3)] Define $z_t := \frac{\mathrm{d}}{\mathrm{d}t}\hat{q}_t$. Fix $\tau>0$.  Suppose that $\mathsf{C}>0$ is a number such that $\hat{q}_t \leq \mathsf{C}$ for all $t \in [0,\tau]$, to be fixed later in the proof.  Now,
\be \label{eq:z1}
|\frac{\d}{\d t}z_t| = |\frac{\langle \frac{\d}{\d t}\tilde{p}_t, \tilde{q}_t \rangle}{\|\tilde{q}_t\|}|  \leq \| \frac{\mathrm{d}}{\mathrm{d}t}\tilde{p}_t\| \leq M_2 \hat{q}_t,
\ee
which implies
\be
\frac{\mathrm{d}\hat{q}_t}{\mathrm{d}t} = z_t := g_1(\hat{q}_t,z_t), \quad \quad \frac{\mathrm{d}z_t}{\mathrm{d}t} \geq -M_2 \hat{q}_t \geq -M_2\mathsf{C} := g_2(\hat{q}_t,z_t) \quad \quad \forall  t \in [0,\tau].
\ee

 Define $\hat{q}^\dagger_t$ and $z^\dagger_t$ to be the solution to the system of differential equations
\be \label{eq:chaplygin3}
\frac{\mathrm{d}\hat{q}^\dagger_t}{\mathrm{d}t} = \hat{z}^\dagger_t = g_1(\hat{q}^\dagger_t,\hat{z}^\dagger_t), \quad \quad \frac{\mathrm{d}\hat{z}^\dagger_t}{\mathrm{d}t} = -M_2 \mathsf{C} = g_2(\hat{q}^\dagger_t,\hat{z}^\dagger_t)
\ee
 with initial conditions $\hat{q}^\dagger_0 = \hat{q}_0$ and $z^\dagger_0= z_0$.
 
Since $g_1$ and $g_2$ are nondecreasing in each variable, Lemma \ref{LemmaOdeComp} implies
\be 
\hat{q}_t \geq \hat{q}^\dagger_t,  \quad \quad \hat{p}_t \geq z^\dagger_t, \quad \quad t \in [0,\tau].
\ee

Hence, all we need to do is solve for $z^\dagger_t$. 
We can turn the system of equations \eqref{eq:chaplygin3} into the following second-order equation:
\[ \frac{\mathrm{d}^2\hat{q}^\dagger_t}{\mathrm{d}t^2} = -M_2 \mathsf{C}\],
whose solutions are of the form
\[\hat{q}^\dagger_t = -M_2\mathsf{C}t^2 +z_0t+\hat{q}_0.\]

Therefore, we have
\be
\hat{q}_t \geq \hat{q}^\dagger_t = -M_2\mathsf{C}t^2 +z_0t + \hat{q}_0 \quad \quad \forall \, t\in[0,\tau].
\ee

In the special case that $\hat{q}_0 = 0$, we have that $z_0 = \hat{p}_0$.  By Inequality \eqref{thmboundsconc1}, for any $\tau >0$ we also have in this special case that $\hat{q}_t \leq \frac{\hat{p}_0}{\sqrt{M_2}}\sinh(\tau \sqrt{M_2})$ for all $t \in [0,\tau]$.  So setting $\mathsf{C} =  \frac{\hat{p}_0}{\sqrt{M_2}}\sinh(\tau \sqrt{M_2})$ completes the proof of Inequality \eqref{thmboundsconc4}.
\\

In the special case that $\hat{p}_0 = 0$, we have $|z_0| \leq |\hat{p}_0| = 0$.  Suppose that $\tau>0$ is such that $\hat{q}_t \leq 2\hat{q}_0$ holds for  all $t\in[0,\tau]$.  Then setting $\mathsf{C} = 2\hat{q}_0$ completes the proof of Inequality \eqref{thmboundsconc3}.
\end{enumerate}
\end{proof}

\subsection{Contraction for Strongly Log-Concave Targets}

In this section, we show that two solutions to Hamilton's equations with the same initial momenta will tend to move closer to each other over a moderate time interval (see Figure \ref{fig:coupling}).
\\
\\
\begin{figure}[h]
\begin{center}
\includegraphics[trim={0 42mm 0 31mm}, scale=0.4]{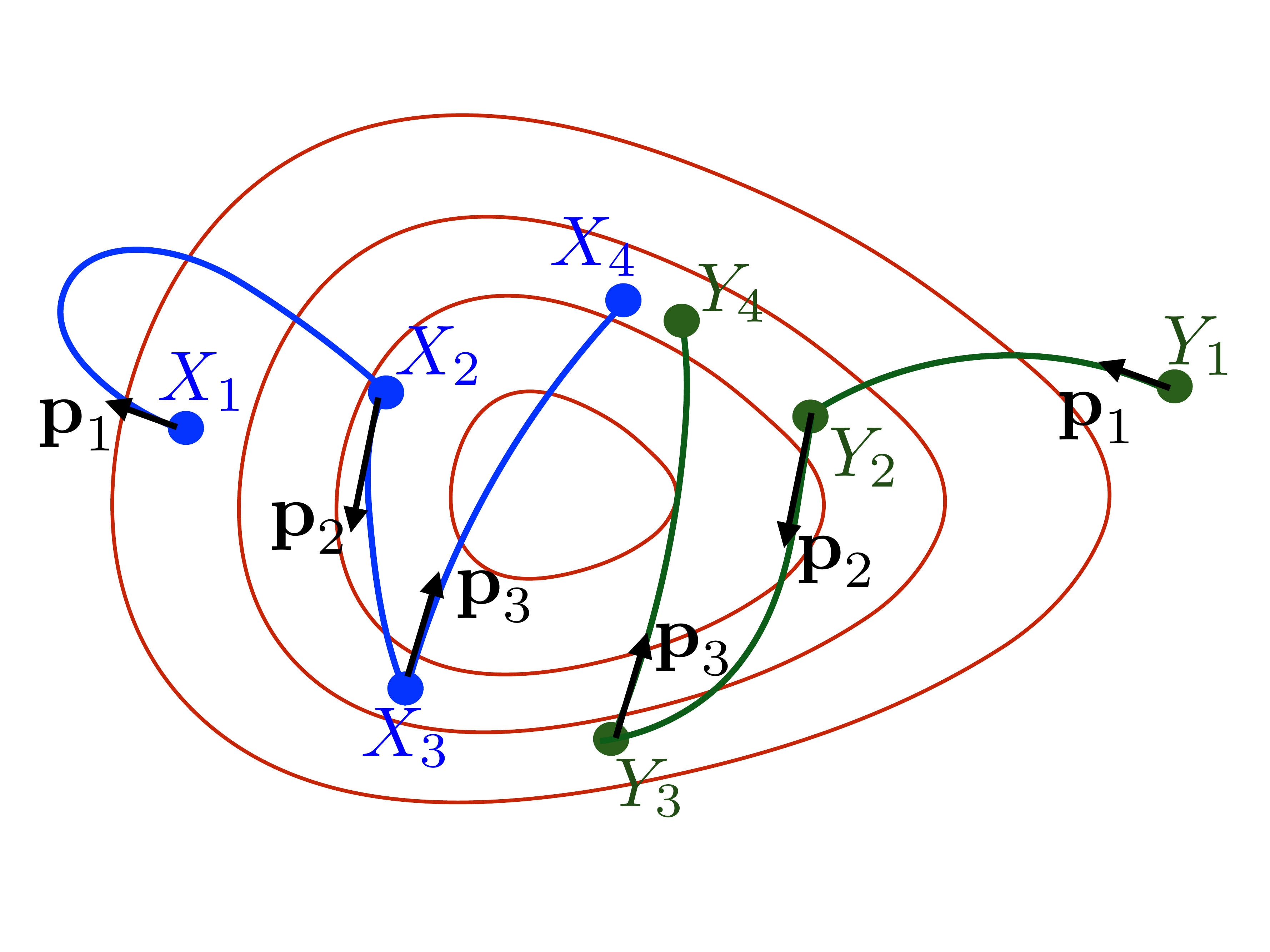}
\end{center}
\caption{Coupling two copies $X_1,X_2,\ldots$ (blue) and $Y_1,Y_2,\ldots$ (green) of HMC by choosing same momentum $\textbf{p}_i$ at every step. Note that contraction is guaranteed at each step, despite the fact that the trajectory from $X_1$ to $X_2$ is not the shortest path between those two points. This is in contrast to typical comparison theorems from differential geometry, such as the Rauch comparison theorem, which require the assumption that paths are length-minimizing. }\label{fig:coupling}
\end{figure}

For this section, we define the error function
\be
\mathcal{ERR}(t) := \frac{\sinh^2(t)}{1-2t^2} \\ 
\ee
for all $t \geq 0$. We also recall:

\begin{lemma}\label{thm:strong_convexity}(Theorem 2.1.9 of \cite{nesterov2013introductory})
Fix $m > 0$. If $U \, : \, \mathbb{R}^{d} \mapsto \mathbb{R}^{+}$ satisfies 
\be 
\langle D_{v} U'(q), v \rangle \leq \|v \|^{2} \, m 
\ee 
for all $q, v \in \mathbb{R}^{d}$, then it is strongly convex with parameter $m$.

\end{lemma}

We can now prove the following contraction estimate:

\begin{lemma}\label{lemma:contraction}
Define 
\be 
\Psi_T(t) &:=  \bigg(-\frac{1}{2} + \frac{M_2}{m_2}\mathcal{ERR}(T\sqrt{M_2})\bigg)\cdot \frac{1}{2} (\sqrt{m}_2t)^2 +1  \\
\psi(t) &:= -2M_2 t^2 + 1. \\
\ee

Suppose that $0 < T \leq \frac{1}{2\sqrt{2}} \frac{\sqrt{m_2}}{M_{2}}$.  Then if $\tilde{p}_0 = 0$,
\be\label{eq:contraction}
\hat{q}_0\psi(t) \leq \hat{q}_t \leq \hat{q}_0\Psi_T(t)
\ee
for all $t \in [0,T]$.
\end{lemma}

\begin{proof}
By Lemma \ref{thm:strong_convexity}, we have:
\be\label{eq:InnerProduct1}
\bigg \langle -\bigg( U'(q^{(2)}_t) - U'(q^{(1)}_t)\bigg), q^{(2)}_t - q^{(1)}_t \bigg \rangle \leq -m_2 \|q^{(2)}_t - q^{(1)}_t\|^2.
\ee

Note that
\be \label{eq:a1'}
\frac{\mathrm{d}}{\mathrm{d}t}\|\tilde{q}_t\| = \frac{\langle \tilde{p}_t, \tilde{q}_t \rangle}{\|\tilde{q}_t\|},
\ee 
so that
\be\label{eq:a1}
\|\tilde{q}_t\| \frac{\mathrm{d}}{\mathrm{d}t}\|\tilde{q}_t\| = \langle \tilde{p}_t, \tilde{q}_t\rangle.
\ee
\\

Taking derivatives,
\be \label{eq:a2} 
\frac{\mathrm{d}}{\mathrm{d}t} \Bigg ( \|\tilde{q}_t\|\frac{\mathrm{d}}{\mathrm{d}t}\|\tilde{q}_t\| \Bigg) &= \frac{\mathrm{d}}{\mathrm{d}t} \langle \tilde{p}_t, \tilde{q}_t \rangle\\
&=\langle \tilde{p}_t, \frac{\mathrm{d} \tilde{q}_t}{\mathrm{d}t} \rangle +  \langle \frac{\mathrm{d} \tilde{p}_t}{\mathrm{d}t}, \tilde{q}_t\rangle\\
&=\|\tilde{p}_t\|^2 + \langle \frac{\mathrm{d} \tilde{p}_t}{\mathrm{d}t}, \tilde{q}_t \rangle\\
&\stackrel{{\scriptsize \textrm{Eq. }}\ref{eq:Galilean}}{=}\|\tilde{p}_t\|^2 + \bigg \langle -\bigg(U'(q^{(2)}_t) - U'(q^{(1)}_t)\bigg), q^{(2)}_t - q^{(1)}_t \bigg \rangle\\
&\stackrel{{\scriptsize \textrm{Eq. }}\ref{eq:InnerProduct1}}{\leq} \|\tilde{p}_t\|^2  -m_2 \|q^{(2)}_t - q^{(1)}_t\|^2 \\
&\stackrel{{\scriptsize \textrm{Pythagorean theorem }}}{=}\bigg\|\frac{\langle \tilde{q}_t, \tilde{p}_t \rangle}{\|\tilde{q}_t\|}\bigg\|^2 + \bigg\| \tilde{p}_t - \frac{\tilde{q}_t}{\|\tilde{q}_t\|}\frac{\langle \tilde{q}_t, \tilde{p}_t \rangle}{\|\tilde{q}_t\|}\bigg\|^2    -m_2 \|q^{(2)}_t - q^{(1)}_t\|^2 \\
&\stackrel{{\scriptsize \textrm{Eq. }}\ref{eq:a1'}}{=} \bigg(\frac{\mathrm{d}}{\mathrm{d}t}\|\tilde{q}_t\|\bigg) ^2 + \bigg\| \tilde{p}_t - \frac{\tilde{q}_t}{\|\tilde{q}_t\|}\frac{\langle \tilde{q}_t, \tilde{p}_t \rangle}{\|\tilde{q}_t\|}\bigg\|^2    -m_2 \|q^{(2)}_t - q^{(1)}_t\|^{2},
\ee
where the sixth line follows from the fact that $\frac{\tilde{q}_t}{\|\tilde{q}_t\|}\frac{\langle \tilde{q}_t, \tilde{p}_t \rangle}{\|\tilde{q}_t\|}$ is the projection of $\tilde{p}_t$ onto the unit vector $\frac{\tilde{q}_t}{\|\tilde{q}_t\|}$.

Applying the chain rule to the LHS of Equation \eqref{eq:a1},
\be \label{eq:a3} 
\frac{\mathrm{d}}{\mathrm{d}t} \Bigg ( \|\tilde{q}_t\| \frac{\mathrm{d}}{\mathrm{d}t}\|\tilde{q}_t\| \Bigg) &= \|\tilde{q}_t\| \frac{\mathrm{d}^2}{\mathrm{d}t^2}\|\tilde{q}_t\| + \bigg(\frac{\mathrm{d}}{\mathrm{d}t} \|\tilde{q}_t\|\bigg ) \times \bigg( \frac{\mathrm{d}}{\mathrm{d}t}\|\tilde{q}_t\|\bigg)\\
&= \|\tilde{q}_t\| \frac{\mathrm{d}^2}{\mathrm{d}t^2}\|\tilde{q}_t\| + \bigg(\frac{\mathrm{d}}{\mathrm{d}t} \|\tilde{q}_t\|\bigg)^2.
\ee

Combining Equality \eqref{eq:a3} with Inequality \eqref{eq:a2}, we get
\[
\|\tilde{q}_t\| \frac{\mathrm{d}^2}{\mathrm{d}t^2}\|\tilde{q}_t\| + \bigg(\frac{\mathrm{d}}{\mathrm{d}t} \|\tilde{q}_t\|\bigg)^2 \leq  \bigg(\frac{\mathrm{d}}{\mathrm{d}t}\|\tilde{q}_t\|\bigg) ^2 + \bigg\| \tilde{p}_t - \frac{\tilde{q}_t}{\|\tilde{q}_t\|}\frac{\langle \tilde{q}_t, \tilde{p}_t \rangle}{\|\tilde{q}_t\|}\bigg\|^2    -m_2 \|\tilde{q}_t\|^2,
\]
and hence
\be\label{eq:a4}
\frac{\mathrm{d}^2}{\mathrm{d}t^2}\|\tilde{q}_t\| &\leq \bigg\| \tilde{p}_t - \frac{\tilde{q}_t}{\|\tilde{q}_t\|}\frac{\langle \tilde{q}_t, \tilde{p}_t \rangle}{\|\tilde{q}_t\|}\bigg\|^2/ \|\tilde{q}_t\|    -m_2\|\tilde{q}_t\| \\
&\leq  \frac{\| \tilde{p}_t \|^2}{ \|\tilde{q}_t\|}    -m_2\|\tilde{q}_t\|,
\ee
where again we use the fact that $\frac{\tilde{q}_t}{\|\tilde{q}_t\|}\frac{\langle \tilde{q}_t, \tilde{p}_t \rangle}{\|\tilde{q}_t\|}$ is the projection of $\tilde{p}_t$ onto the unit vector $\frac{\tilde{q}_t}{\|\tilde{q}_t\|}$.


Recall that we assume $\tilde{p}_0= 0$ in the statement of this Lemma. Inequality \eqref{eq:contraction} is obviously true if $\hat{q}_0=0$, so without loss of generality we may also assume that $\hat{q}_0 >0$. By the fact that solutions to Hamilton's equations are continuous (first shown in \cite{poincare1899celeste}) and our assumptions that $\tilde{p}_0= 0$ and $\|\tilde{q}_0\| =\hat{q}_0 \neq 0$, there must exist an $\epsilon>0$ such that $\frac{\| \tilde{p}_t \|^2}{ \|\tilde{q}_t\|} -m_2\|\tilde{q}_t\| < 0$ for every $t \in (0,\epsilon]$.  Recall that $\hat{q}_t:= \|\tilde{q}_t\|$ and define
\be
\tau_1 := \max\{ \tau \in [0,T] :\hat{q}_t \leq 2\hat{q}_0 \, \, \forall \, t  \in[0,\tau]\},
\ee
where a maximum value exists by continuity of $\hat{q}_t$.
By parts 3 and 1, respectively, of Lemma \ref{thm:bounds}, we have
\be\label{eq:a5}
\|\tilde{q}_t\| \geq  -2M_2\hat{q}_0t^2 + \hat{q}_0, \quad \quad \forall t \in [0,\tau_1]
\ee
and
\be\label{eq:a6}
\|\tilde{p}_t\| \leq \frac{1}{2}\hat{q}_0\sqrt{M_2}\bigg(e^{t\sqrt{M_2}} - e^{-t\sqrt{M_2}} \bigg).
\ee

Hence, Inequalities \eqref{eq:a4}, \eqref{eq:a5}, and \eqref{eq:a6} together imply that for all $t \in [0,\tau_1]$ we have,
\be \label{eq:a9}
\frac{\mathrm{d}^2}{\mathrm{d}t^2}\|\tilde{q}_t\| &\leq \frac{\| \tilde{p}_t \|^2}{ \|\tilde{q}_t\|}    -m_2\|\tilde{q}_t\| \\
&\leq \frac{\bigg[ \frac{1}{2}\hat{q}_0\sqrt{M_2}\bigg(e^{t\sqrt{M_2}} - e^{-t\sqrt{M_2}} \bigg) \bigg]^2}{-2M_2\hat{q}_0t^2 + \hat{q}_0} - m_2\|\tilde{q}_t\| \\
&=  - m_2\|\tilde{q}_t\| + \frac{\frac{1}{4}\hat{q}_0 M_2\bigg(e^{t\sqrt{M_2}} - e^{-t\sqrt{M_2}} \bigg)^2}{-2(t\sqrt{M_2})^2 + 1}.
\ee
But $T \leq \frac{1}{2\sqrt{2}} \frac{\sqrt{m_2}}{\sqrt{M_2}} \frac{1}{\sqrt{M_2}}$ implies that $\mathcal{ERR}(t\sqrt{M_2})$ is nondecreasing on $t \in [0, T]$, so by Inequality \eqref{eq:a9} we have

\be \label{eq:a16}
\frac{\mathrm{d}^2}{\mathrm{d}t^2}\|\tilde{q}_t\| \leq -m_2\|\tilde{q}_t\| + \hat{q}_0M_2\mathcal{ERR}(T\sqrt{M_2}) \quad \quad \forall t \in [0, \tau_1].
\ee
 
Define 
\be
\tau_2 := \max\{ \tau  \in [0,T] :\hat{q}_t \geq \frac{1}{2}\hat{q}_0 \, \, \forall \, t  \in[0,\tau]\},
\ee
where a maximum value exists by continuity of $\hat{q}_t$. Then by Inequality \eqref{eq:a16},
 \be\label{eq:a7}
 \frac{\mathrm{d}^2}{\mathrm{d}t^2}\|\tilde{q}_t\| \leq -\frac{m_2}{2}\hat{q}_0 + \hat{q}_0M_2\mathcal{ERR}(T\sqrt{M_2}) \quad \quad \forall t \in  [0, \tau_1] \cap[0,\tau_2]
 \ee
 
Let  $\hat{q}^\ddagger_t$ be the solution to the differential equation
\be
\frac{\mathrm{d}^2}{\mathrm{d}t^2} \hat{q}^\ddagger_t = -\frac{m_2}{2}\hat{q}_0 + \hat{q}_0M_2\mathcal{ERR}(T\sqrt{M_2}),
\ee
with initial conditions $\hat{q}^\ddagger_0 = \hat{q}_0$ and $\frac{\mathrm{d}}{\mathrm{d}t} \hat{q}^\ddagger_0=\hat{p}_0=0$. Since the RHS of the differential inequality \eqref{eq:a7} is nondecreasing in both variables $\hat{q}_t$ and $\frac{\mathrm{d}}{\mathrm{d}t}\hat{q}_t$, we have by Lemma \ref{LemmaOdeComp} that
\be \label{IneqContLemmaApplyCompAgain}
\hat{q}_t \leq \hat{q}^\ddagger_t \quad \quad \forall t \in  [0, \tau_1] \cap[0,\tau_2].
\ee
Solving the differential equation for $\hat{q}^{\ddagger}_t$ gives
\be 
\hat{q}^\ddagger_t &=\bigg(-\frac{m_2}{2}\hat{q}_0 + \hat{q}_0M_2\mathcal{ERR}(T\sqrt{M_2})\bigg)\cdot \frac{1}{2} t^2 + \hat{p}_0t+\hat{q}_0\\
&=\bigg(-\frac{m_2}{2}\hat{q}_0 + \hat{q}_0M_2\mathcal{ERR}(T\sqrt{M_2})\bigg)\cdot \frac{1}{2} t^2 +\hat{q}_0\\
&=\hat{q}_0\Bigg[\bigg(-\frac{1}{2} + \frac{M_2}{m_2}\mathcal{ERR}(T\sqrt{M_2})\bigg)\cdot \frac{1}{2} (\sqrt{m}_2t)^2 +1\Bigg],
\ee
where the second line uses the fact that $\hat{p}_0=0$. Therefore, by Inequality \eqref{IneqContLemmaApplyCompAgain}, this implies
\be\label{eq:a8}
 \hat{q}_t \leq \hat{q}_0\Bigg[\bigg(-\frac{1}{2} + \frac{M_2}{m_2}\mathcal{ERR}(T\sqrt{M_2})\bigg)\cdot \frac{1}{2} (\sqrt{m}_2t)^2 +1\Bigg] \quad \quad \forall t \in  [0, \tau_1] \cap[0,\tau_2].
 \ee

Note that $\frac{\sqrt{m_2}}{\sqrt{M_2}} \leq 1$, so
\be\label{eq:b1}
T \leq \frac{1}{2\sqrt{2}} \frac{\sqrt{m_2}}{\sqrt{M_2}} \frac{1}{\sqrt{M_2}} \leq \frac{1}{2\sqrt{2}}\frac{1}{\sqrt{M_2}}.
\ee
Hence,
\be\label{eq:b2'}  \psi(t) = -2M_2t^2 + 1 \geq \frac{3}{4} \quad \quad \forall t \in [0,T]. \ee

We calculate,

\be \label{eq:b3}
\frac{M_2}{m_2}\mathcal{ERR}(T\sqrt{M_2}) &= \frac{M_2}{m_2}\frac{\sinh^2(\sqrt{M_2}T)}{1 - 2(\sqrt{M_2}T)^2}\\
&\leq \frac{M_2}{m_2}\frac{\sinh^2(\sqrt{M_2}T)}{1 - 2(\sqrt{M_2}T)^2} \leq \frac{1}{4},
\ee
where both inequalities use the fact that $T \leq \frac{1}{2\sqrt{2}} \frac{\sqrt{m_2}}{\sqrt{M_2}} \frac{1}{\sqrt{M_2}}$ and the first also uses the fact that $\sinh^2(t) \leq 1.2 t^2$ for all $t \in [0,\frac{1}{2}]$. This bound implies
\be\label{eq:b4'} \Psi_T(t) \leq 1- \frac{1}{4} \cdot \frac{1}{2} (\sqrt{m}_2t)^2 \leq 1 \quad \quad \forall t \in [0,T].  \ee

Therefore, Inequalities \eqref{eq:b2'} and \eqref{eq:b4'}, respectively, imply that
\be \label{eq:a14}
\hat{q}_0\psi(t) \geq \frac{3}{4}\hat{q}_0 \quad \textrm{ and } \quad \hat{q}_0\Psi_T(t)\leq \hat{q}_0 \quad \forall t\in [0,T].
\ee

We now define $\tau_{3}$ to be supremum of all values of $0 \leq \tau \leq T$ so that the following inequalities hold:
\be \label{eq:a10}
 \frac{3}{4}\hat{q}_0 \leq \hat{q}_0\psi(t) \leq \hat{q}_t \leq \hat{q}_0\Psi_T(t)\leq \hat{q}_0  \quad \quad \forall t \in [0,\tau].
\ee
Observe that $\Psi_T(0) = \psi(0) = 1$, which implies $\tau_{3} \geq 0$. Moreover, since $\hat{q}_t$, $\psi(t)$, and $\Psi_T(t)$ are continuous on $t\in [0,T]$, Inequalities \eqref{eq:a10} are satisfied for $t = \tau_{3}$. We now prove by contradiction that in fact $\tau_{3} = T$.

\begin{itemize}[\qquad \qquad]
\item Suppose (towards a contradiction) that $\tau_3<T$.
Since $\hat{q}_t$ is continuous on $t\in \mathbb{R}$, Equation \eqref{eq:a10} (which is satisfied for $\tau= \tau_3$) implies that there exists a number $\delta$ with $0< \delta \leq T-\tau_3$ such that
\be \label{eq:a11}
\frac{1}{2} \hat{q}_0 \leq \hat{q}_t \leq 2\hat{q}_0 \quad \quad \forall t\in[0, \tau_3+\delta].
\ee
Equation \eqref{eq:a11} implies that $\tau_3+\delta \leq \tau_1$ and $\tau_3 +\delta \leq \tau_2$.  Therefore, by Equation \eqref{eq:a8}, we have
\be \label{eq:a12}
\hat{q}_t \leq \hat{q}_0\Psi_T(t) \quad \quad \forall t \in [0, \tau_3+\delta]
\ee
and by part 3 of Lemma \ref{thm:bounds} we have
\be \label{eq:a13}
\hat{q}_0\psi(t) \leq \hat{q}_t \quad \quad \forall t \in [0, \tau_3+\delta].
\ee
Therefore, Equations  \eqref{eq:a14}, \eqref{eq:a12}, and \eqref{eq:a13} together imply that
\be \label{eq:a15}
\frac{3}{4}\hat{q}_0 \leq \hat{q}_0\psi(t) \leq \hat{q}_t \leq \hat{q}_0\Psi_T(t)\leq \hat{q}_0  \quad \quad \forall t \in [0,\tau_3+\delta].
\ee
But Equation \eqref{eq:a15} implies that Equation \eqref{eq:a10} is satisfied for $\tau= \tau_3+\delta$, which contradicts the fact that $\tau=\tau_3$ is the largest value of $\tau$ that satisfies equation $\ref{eq:a10}$.  Therefore, by contradiction, our assumption that $\tau_3<T$ must be false.\\
\end{itemize}

Therefore, $\tau_3 = T$, and so Equation \eqref{eq:a10} is satisfied for $\tau=T$:
\be
\frac{3}{4}\hat{q}_0 \leq \hat{q}_0\psi(t) \leq \hat{q}_t \leq \hat{q}_0\Psi_T(t)\leq \hat{q}_0  \quad \quad \forall t \in [0,T].
\ee
This completes the proof of the Lemma.
\end{proof}

This bound quickly implies the main result of this section:

\begin{thm} [Contraction For Hamiltonian Mechanics with Convex Potentials] \label{ThmContractionConvexMainResult}
For $0 \leq T \leq \frac{1}{2\sqrt{2}} \frac{\sqrt{m_2}}{M_2}$, 
\be \label{IneqContractLemmaMain}
\hat{q}_T \leq \left[1- \frac{1}{8} (\sqrt{m}_2T)^2 \right] \times \hat{q}_0.
\ee

In particular, if $T = \frac{1}{2\sqrt{2}} \frac{\sqrt{m_2}}{M_2}$, then
\be \label{IneqContractLemmaCorr}
\hat{q}_T \leq \left[1- \frac{1}{64} \left(\frac{m_2}{M_2}\right)^2 \right] \times \hat{q}_0.
\ee
\end{thm}

\begin{proof}
By Inequality  \eqref{eq:b3},
\be \label{eq:b3'}
\frac{M_2}{m_2}\mathcal{ERR}(T\sqrt{M_2}) \leq \frac{1}{4}.
\ee

This implies
\be\label{eq:b4} \Psi_T(t) \leq 1- \frac{1}{4} \cdot \frac{1}{2} (\sqrt{m}_2t)^2 \quad \quad \forall t \in [0,T].  \ee

Applying Lemma \ref{lemma:contraction}, Equation \eqref{eq:b4} implies that
\be
\hat{q}_T \leq \left[1- \frac{1}{8}(\sqrt{m}_2T)^2 \right] \times \hat{q}_{0}.
\ee
This completes the proof of Inequality \eqref{IneqContractLemmaMain}. Inequality \eqref{IneqContractLemmaCorr} is an immediate consequence of Inequality \eqref{IneqContractLemmaMain}.
\end{proof}

\subsection{Bounds Leading to Total Variation Mixing}

To prove that HMC mixes quickly in the strong Total Variation metric, not merely the weaker Wasserstein metric, we need the following continuity estimate:

\begin{lemma} [Kernel Continuity] \label{LemmaMinStrongLog}
Fix $\epsilon' > 0$, let Assumption \ref{AssumptionsConvexity} hold with $\mathcal{X} = \mathbb{R}^{d}$, and let Assumption \ref{AssumptionSecondDerivative} hold. Let $K$ be the transition kernel defined by Algorithm \ref{DefSimpleHMC} with parameter $0 \leq T \leq \frac{1}{2\sqrt{2}} \frac{\sqrt{m_{2}}}{M_{2}}$.

For all $\mathbf{q}^{(1)}, \mathbf{q}^{(2)} \in \mathbb{R}^{d}$ with 
\be 
\| \mathbf{q}^{(1)} - \mathbf{q}^{(2)} \| \leq \epsilon \equiv \epsilon' \times \min \left( [\frac{40}{9T\sqrt{M_2}}]^{-1} \, \, , \, \, [8 d\sqrt{d}]^{-1} \, \,  , \, \, \left[16e\sqrt{2} d\right]^{-1} \,\, , \, \, [d\frac{M_3}{(M_2)^{3/2} T}]^{-1} \, \, ,\, \, [10\frac{(M_2)^2}{d M_3}]^{-1}\right),
\ee  
we have 
\be 
\| K(\mathbf{q}^{(1)},\cdot) - K(\mathbf{q}^{(2)},\cdot) \|_{\mathrm{TV}} \leq 5 \epsilon'.
\ee 
\end{lemma}

The proof of this lemma is deferred to Appendix \ref{SecAppendixStrong}.

\subsection{Drift Condition} \label{SecDriftCond}

Although this paper focuses on mixing bounds, we feel it is worth mentioning that the strong log-concavity assumption will also imply a quantitatively useful drift condition:

\begin{thm} [Drift Conditions for HMC: Quadratic Tails] \label{ThmDriftHMC}
Fix $1 < C < \infty$ and define
\be 
S = \{x \in \mathbb{R}^{d} \, : \,  \| x \| \leq C\}.
\ee 

Let Assumption \ref{AssumptionsConvexity} hold for $\mathcal{X} = S^{c}$. Let $\{X_{t}\}_{t \geq 0}$ be the Markov chain given defined by Algorithm \ref{DefSimpleHMC} with parameter $T = \frac{\sqrt{m_{2}}}{2\sqrt{2 }M_{2}}$. Then

\be \label{IneqDriftConcMain}
\E[e^{\| X_{1}\|} \, | \, X_{0}] \leq e^{-1} e^{\|X_{0}\|} + A,
\ee 

where the constant $0 < A < \infty$ satisfies

\be \label{IneqDriftConcConst}
\log(A) = O \left( \frac{M_{2}^{2.5}}{m_{2}^{4}}, \frac{M_{2}^{3}}{m_{2}^{3.5}}, C \frac{M_{2}}{m_{2}}, \frac{\sqrt{M_{2}}}{m_{2}^{1.5}} \right).
\ee 
\end{thm}

\begin{proof}
The proof is given in Appendix \ref{AppDriftCond}.
\end{proof}

\begin{remark} [Other Drift Bounds In the HMC Literature] \label{RemGenDriftBounds}
Quantitative drift conditions are required to extend the quantiative mixing bounds in the present paper to more general distributions. Finding such bounds is not generally easy, and is beyond the scope of this paper, but we mention some existing related work.  Although there is very little published work obtaining general quantitative drift conditions in the HMC literature, there has been substantial work on finding drift conditions without explicit quantitative bounds, most noteably in \cite{livingstone2016geometric} \cite{durmus2017convergence}. Of course it is possible to obtain quantitative bounds by following the calculations in \textit{e.g.} Theorem 5.4 of \cite{livingstone2016geometric}. However, in order to obtain useful results for the ideal HMC integrator, additional assumptions are required. We briefly sketch the difficulty.

In the last clause in the last sentence of the proof of Proposition 5.10, they assert that a certain sum is strictly positive because all of its terms are nonnegative and at least one is strictly positive. The continuous analogue to this argument fails: the analogous integral cannot be bounded away from 0 simply by noting that the integrand is strictly positive at at least one point. In order to avoid this problem, it is sufficient to add an assumption that $U'(q)$ changes slowly with $q$ - for example, our Assumption \ref{AssumptionSecondDerivative}. 
\end{remark}

\section{Proofs of Main Result for Ideal Integrator} \label{SecMainRes}

We prove Theorem \ref{ThmMainConcave}: 

\begin{proof} [Proof of Theorem \ref{ThmMainConcave}]
Inequality \eqref{MixingLogConcaveMainConc2} follows immediately from Theorem \ref{ThmContractionConvexMainResult}.  Inequality \eqref{MixingLogConcaveMainConc3} follows immediately from   Inequality \eqref{MixingLogConcaveMainConc2} and Proposition 30 of \cite{ollivier2009ricci}.

Finally, we prove Inequality \eqref{MixingLogConcaveMainConc1}  by applying Lemma \ref{LemmaDriftMinRep}. Define 
\be 
\delta = \frac{1}{125} \, \min \left( [\frac{40}{9T\sqrt{M_2}}]^{-1} \, \, , \, \, [8 d\sqrt{d}]^{-1} \, \,  , \, \, \left[16e\sqrt{2} d\right]^{-1} \,\, , \, \, [d\frac{M_3}{(M_2)^{3/2} T}]^{-1} \, \, ,\, \, [10\frac{(M_2)^2}{d M_3}]^{-1}\right),
\ee 
noting that this agrees with Equation \eqref{EqThm1Consts}. In the notation of Lemma \ref{LemmaDriftMinRep}, we will set 
\be 
h(x,y) &= e^{\|x-y\|}, \, k_{0} = 1, \\
S &= \{ (x, y) \in \mathbb{R}^{2d} \, : \, \|x - y \| \leq  \delta \}.
\ee 
Applying Theorem \ref{ThmContractionConvexMainResult} for both bounds, the constants $\alpha, A$ that appear in Lemma \ref{LemmaDriftMinRep} may be taken to be 
\be 
\alpha^{-1} &= e^{-\frac{ \delta}{64} \left( \frac{m_{2}}{M_{2}} \right)^{2} }\\
A &= e^{2 \delta}.
\ee 

By Lemma \ref{LemmaMinStrongLog}, the constant $\epsilon$ that appears in the statement of Lemma \ref{LemmaDriftMinRep} may be taken to be
\be 
\epsilon = \frac{24}{25}.
\ee 
Let $Y \sim \pi$. Applying Lemma \ref{LemmaDriftMinRep} with these constants, we find 
\be \label{IneqStrongLongConcReallyAlmostDone}
\| \mathcal{L}(X_{t}) - \pi \|_{\mathrm{TV}} \leq \inf_{0 \leq j \leq t} \left( (1 - \epsilon)^{j} + \alpha^{-t + j -1} \, A^{j-1} \, \E[h(x,Y)] \right).
\ee 

Finally, we bound $\E[h(x,Y)]$. Let $Z \sim \Phi_{1}$ be a standard univariate Gaussian. By the strong log-concavity of $\pi$, we have 
\be 
\P[ \| Y \| > y] \leq \P[ | Z | > \frac{\sqrt{2 m_{2}} y}{\sqrt{d}}] 
\ee  
for all $y > 0$, so 
\be \label{SillyMeanSize}
\E[e^{\|Y-x\|}] \leq e^{\|x \|} \,\E[e^{\|Y\|}] \leq  e^{\|x\| + \frac{d}{4 m_{2}}}.
\ee 
Applying this with Inequality \eqref{IneqStrongLongConcReallyAlmostDone},
\be 
\| \mathcal{L}(X_{t}) - \pi \|_{\mathrm{TV}} &\leq \inf_{0 \leq j \leq t} \left( (1 - \epsilon)^{j} + \alpha^{-t + j +1} \, A^{j-1} \, e^{\|x\| + \frac{d}{4 m_{2}}} \right) \\
&\leq  \inf_{0 \leq j \leq t} \left( 25^{-j} + e^{-(t-j+1)\frac{ \delta}{64} \left( \frac{m_{2}}{M_{2}} \right)^{2} } e^{2 \delta j} e^{\|x\| + \frac{d}{4 m_{2}}} \right).
\ee 
Defining $\kappa = \frac{1}{64} \frac{m_{2}^{2}}{M_{2}^{2}}$ and fixing $j = \lfloor \frac{\kappa  t}{100} \rfloor$, we have 
\be 
\| \mathcal{L}(X_{t}) - \pi \|_{\mathrm{TV}} &\leq 25^{-\lfloor \frac{\kappa t}{100} \rfloor} + e^{-t (1 - \frac{\kappa}{100}) \delta \kappa} e^{2 \delta \frac{\kappa t}{100}}  e^{\|x\| + \frac{d}{4 m_{2}}} \\
&\leq 25^{-\lfloor \frac{\kappa t}{100} \rfloor} + e^{-\frac{\kappa \delta t}{2}}  e^{\|x\| + \frac{d}{4 m_{2}}},
\ee  

completing the proof of the theorem.

\end{proof}

\section{Preconditioning and Optimization} \label{SecPreproc}

Our main results are stated in terms of the ratio $\frac{M_{2}}{m_{2}}$ of upper and lower bounds $m_{2},M_{2}$ on the Hessian of the potential $U$. This ratio can be quite large for many target distributions, such as the posterior distribution of a regression problem in which different coefficients have very different sizes. In this section, we show that simple preprocessing steps can make this ratio much smaller, thus making our bounds much better in practice than they might first appear. We note that this preprocessing is common in other ``geometric" Markov chain applications (see \textit{e.g.} the survey \cite{vempala2005geometric}). The basic idea is to find a linear transformation of the potential for which this ratio is small on the bulk of the target distribution. Fix a  probability distribution $\pi$ given by $\pi(x) =  \frac{1}{\int_{\mathbb{R}^d}e^{-U(x)} \mathrm{d}x}e^{-U(x)}$ for some potential function $U$. We consider the assumption:

\begin{defn}[Rounding Matrix] \label{DefRounding}
Fix a constant $\epsilon > 0$. Define the \textit{bulk level set} $\mathfrak{S}_{\epsilon}$ of $\pi$ by the pair of equations
\be
L_{\epsilon} &= \inf\{ C \, : \, \pi( \{x \, : \, \pi(x) \geq C\}) \leq 1 - \epsilon \} \\
\mathfrak{S}_{\epsilon} &= \{x \, : \, \pi(x) \geq L_{\epsilon} \}.
\ee 
Call a matrix $A$ a \textit{rounding matrix} for $\pi$ with constants $\epsilon, m_{2}, M_{2}$ if the eigenvalues of the Hessian of $\hat{U}(x):=U(A^{-1}x)$ are bounded below and above by $m_2$ and $M_2$, respectively, at every point $x \in A \mathfrak{S}_\epsilon$.
\end{defn}

For every $x \in \mathbb{R}^d$, define $H_x$ to be the Hessian of $U$ evaluated at $x$. The following theorem says that, if there \textit{exists} a linear transformation with associated ratio $\frac{M_{2}}{m_{2}}$, then it is \textit{easy to find} a linear transformation with associated ratio $\frac{M_{2}^{2}}{m_{2}^{2}}$:

\begin{thm}
Fix a probability distribution $\pi(x) = \frac{1}{\int_{\mathbb{R}^d}e^{-U(x)} \mathrm{d}x} e^{-U(x)}$. Suppose that there exists a rounding matrix $A$ for $\pi$ with constants $\epsilon, m_2, M_2 > 0 $. Define $\tilde{U}(z):= U(\sqrt{H_x}^{-1}z)$ for $z\in \mathbb{R}^d$. For every $\zeta \in \mathbb{R}^{d}$, let $\tilde{H}_{\zeta}$ be the Hessian of $\tilde{U}$ evaluated at the point $z= \sqrt{H_x}\zeta$.  Then for every $y \in \mathfrak{S}_\epsilon$, $\tilde{H}_y$ has all its eigenvalues bounded below and above by $\frac{m_2}{M_2}$ and $\frac{M_2}{m_2}$, respectively.
\end{thm}

\begin{proof}
As in Definition \ref{DefRounding}, let $\hat{U}(z):=U(A^{-1}z)$ for every $z \in \mathbb{R}^d$

For every $\zeta \in \mathbb{R}^d$, let $\hat{H}_\zeta$ be the Hessian of $\hat{U}(z)$ evaluated at the point $z=A\zeta$. 
By Definition \ref{DefRounding},
\be\label{eq:r0}
m_2 u^\top u \leq u^\top \hat{H}_{\zeta} u \leq M_2 u^\top u,
\ee
for all $\zeta \in \mathfrak{S}_\epsilon$ and all $u\in \mathbb{R}^d$ . Fixing $x \in \mathfrak{S}_\epsilon$ and applying Inequality \eqref{eq:r0} twice gives
\be\label{eq:r1}
\frac{m_2}{M_2} u^\top \hat{H}_x u \leq m_2 u^\top u \leq u^\top \hat{H}_y u \leq M_2 u^\top u \leq \frac{M_2}{m_2} u^\top \hat{H}_x u.
\ee

Since $\sqrt{H}_z = A \sqrt{\hat{H}_z}$ for every $z \in \mathbb{R}^d$, applying Equation \eqref{eq:r1} with $v = Au$ gives
\be\label{eq:r2}
\frac{m_2}{M_2}v^\top H_x v \leq v^\top H_y v \leq \frac{M_2}{m_2} v^\top H_x v
\ee
 for any $u \in \mathbb{R}^d$ (and hence for any $v \in \mathbb{R}^d$, since $A$ is invertible).

Now $\sqrt{\tilde{H}_z} = \sqrt{H_x}^{-1} \sqrt{H_z}$ for every $z \in \mathbb{R}^d$.  Therefore, Equation \eqref{eq:r2} implies that:
\be\label{eq:r3}
\frac{m_2}{M_2} v^\top v \leq v^\top \tilde{H}_y v \leq \frac{M_2}{m_2} v^\top v
\ee
for every $y \in \mathfrak{S}_\epsilon$ and every $v \in \mathbb{R}^d$.
Therefore, by the minimax theorem for eigenvalues, $\tilde{H}_y$ has all its eigenvalues bounded between $\frac{m_2}{M_2}$ and $\frac{M_2}{m_2}$  for all $y \in \mathfrak{S}_\epsilon$.
\end{proof}

\section{Discussion} \label{SecDisc}

In this paper, we provide useful bounds on the convergence rate of HMC under rather strong assumptions of strong log-concavity. These bounds improve on several earlier results, and in particular give mixing bounds with near-optimal dependence on dimension for certain implementable variants of HMC, but we leave many important questions open. In this section, we mention some that seem most interesting.

\subsection{Relationship to the Jacobi Metric}
The biggest difference between the approach of  \cite{seiler2014positive} and our paper is as follows. \cite{seiler2014positive} uses concentration of measure to analyze contraction of ``\emph{typical}" Hamiltonian trajectories with parallel initial momenta on strongly convex potentials by expressing them as geodesic trajectories on a positively curved manifold under the Jacobi metric, via the Rauch comparison theorem from differential geometry.  Our paper instead proves contraction of \emph{all} Hamiltonian trajectories with parallel initial momenta by applying comparison theorems for ordinary differential equations (ODEs) directly to the Hamilton's equations that define the Hamiltonian trajectories.  Because the Jacobi manifold is never defined on the entire state space of an HMC algorithm, it does not seem possible to extend an approach based on the Jacobi manifold to obtain uniform contraction estimates for all Hamiltonian trajectories. As a result, we are able to achieve bounds that do not grow explicitly with the dimension $d$, while the Jacobi metric approach in \cite{seiler2014positive} yields bounds that grow like $d^2$. 

We found this slightly dissapointing: the Jacobi metric is a natural tool for analyzing HMC, and it is far from clear to us if the technical difficulties that appear in \cite{seiler2014positive} can be overcome. We leave as an open problem the question of whether the Jacobi metric approach of \cite{seiler2014positive} can be refined to obtain bounds that do not grow explicitly with $d$, as well as to possibly further strengthen the relaxation time bound in our Theorem \ref{ThmMainConcave} from $\mathcal{O}^*(\frac{M_2^2}{m_2^2})$ to the conjectured value of $\mathcal{O}^*(\frac{M_2}{m_2})$.

\subsection{Riemannian HMC}
This paper analyzes one of the simplest possible HMC algorithms. However, many other variants exist. Riemannian HMC, introduced in \cite{girolami2011riemann}, is one of the most popular. This approach seems to obviate the need for the preconditioning step discussed in Section \ref{SecPreproc}, but there are very few rigorous results on the performance of this algorithm. It would be interesting to check that Riemannian HMC does have this property, and that no additional problems arise.

\subsection{Quantitative Drift Conditions}

As discussed in \cite{livingstone2016geometric}, it is more difficult to obtain a Lyapunov condition for HMC than for MH. Obtaining much more general quantitative drift conditions for HMC would allow us to check that bad behaviour ``in the tails" of the target distribution does not greatly influence mixing.

\section*{Acknowledgements}
We are grateful to Natesh Pillai and Alain Durmus for helpful discussions.  Oren Mangoubi was supported by a Canadian Statistical Sciences
Institute (CANSSI) Postdoctoral Fellowship, and by an NSERC
Discovery grant. Aaron Smith was supported by an NSERC Discovery
grant.

\bibliographystyle{plain}
\bibliography{HmcBib}

\newpage

\appendix

\section{Bounds Related to Total Variation Mixing} \label{SecAppendixStrong}

In this section, we will prove Lemmas \ref{LemmaMinStrongLog} and \ref{LemmaDriftMinRep}. Before proving Lemma \ref{LemmaMinStrongLog} in Section \ref{SubsecProofLemmaMinStrongLog}, we need a technical result from analysis that is not directly related to the study of HMC. We give this result in Section \ref{SubsecCoverIntersection}, using notation that is independent of the remainder of the paper. Readers not interested in these details can easly skip to Section \ref{SubsecProofLemmaMinStrongLog}.

\subsection{Coverings and Intersections}\label{SubsecCoverIntersection}

In this section, we prove a small extension of the Jordan-Brouwer curve theorem. Roughly speaking, the main result of this section is the intuitively obvious fact that it is impossible to go from the interior of the image of a ``nice" function to the exterior of the image without passing through the boundary of the image. We suspect that this result is well-known, but could not find it in the literature.

For a function $f \, : \, \Omega_{1} \mapsto \Omega_{1}$ and a set $S \subset \Omega_{1}$, we denote by $f|_{S}$ the \textit{restriction of $f$ to $S$}. This is a function $f|_{S} \, : \, S \mapsto \Omega_{2}$ with values $f|_{S}(x) = f(x)$ for all $x \in S$. We define:

\begin{defn} [Local Injectivity Constant]
Fix $\Omega \subset \mathbb{R}^{d}$ and function $f \, : \, \Omega \mapsto \mathbb{R}^{d}$. For $x \in \mathbb{R}^{d}$ and $r > 0$, denote by $B_{r}(x)$ the ball of radius $r$ with center $x$. For $x \in \Omega$, define the local injectivity constant $C(x)$ by 
\be 
C(x) = \sup \{ r \geq 0 \, : \, f|_{B_{r}(x) \cap \Omega} \text{ is injective.} \}.
\ee 
Say that $f$ if \textit{locally injective} if $C(x) > 0$ for all $x \in \Omega$.
\end{defn}

Throughout this section, we denote by $\sigma_{1}(A)$ the smallest singular value of a matrix $A$.

\begin{lemma}\label{thm:homotopy1}
Let $S = [0,1]^{d} \subset \mathbb{R}^{d}$, and let $f \, : \, S \mapsto \mathbb{R}^{d}$ be continuous. Let $J(f)$ and $C$ be the Jacobian and local injectivity constant of $f$. Assume there exists $\eta > 0$ so that $\sigma_{1}(J(f)(x)), \, C(x) > \eta > 0$ for all $x \in S$. Fix a smooth function $\gamma \, : \, [0,1] \mapsto \mathbb{R}^{d}$ satisfying $\gamma(0) \in f(S)$,  $\gamma(1) \notin f(S)$, and $\| \gamma'(t) \| = v > 0$ for all $t \in [0,1]$. Then there exists $0 \leq t \leq 1$ such that $\gamma(t) \in \partial f(S)$.
\end{lemma}

\begin{proof}
Fix a triple $(M, \{Q_{i}\}_{i=1}^{M}, \delta)$ so that $M \in \mathbb{N}$, $\{ Q_{i} \}_{i =1}^{M}$ is a covering of $S$, and $\delta > 0$ is a real number, satisfying the properties 
\begin{enumerate}
\item for all $i$, $Q_{i}$ is homotopic to a ball.
\item For all $i$, $f|_{Q_{i}}$ is injective.
\item For all $x \in f(S)$, there exists $1 \leq i \leq M$ so that
\be \label{EqSillyLargeContainmentCondition}
B_{2 \delta}(x) \cap f(S)  \subset f(Q_{i}).
\ee
\end{enumerate}

The existence of a covering with this property is guaranteed by the lower bounds on $\sigma_{1}(J(f)(x))$ and $C(x)$.

Define $w^{-1} = \inf \{ k \in \mathbb{N} \, : \, k^{-1} \leq \frac{\delta}{v}\} \leq \frac{3v}{\delta}$ and define $t_{n} = n w$ for $0 \leq n \leq w^{-1}$. Inductively define $\{ i_{n} \}_{n \geq 0}$ by:

\begin{enumerate}
\item If $\gamma([0,t_{n}]) \subset f(S)$, let $i_{n}$ be any integer in $\{1,2,\ldots,M\}$ that satisfies $B_{2 \delta}(\gamma(t_{n})) \cap f(S) \subset Q_{i_{n}}$ (this is possible by Equation \eqref{EqSillyLargeContainmentCondition}).
\item Otherwise, let $i_{n} = M+1$.
\end{enumerate}

Define $N = \max \{n \, : \, \gamma([0,t_{n}]) \subset f(S)\}$. Since $\gamma(0) \in f(S)$ but $\gamma(1) \notin f(S)$, we have $N \in \{0,\ldots,w^{-1} - 1\}$. Let $t^\star = \inf \{ t \, : \, \gamma(t) \notin f(Q_{i_{N}}) \} \in [t_{N},t_{N+1}]$. By the Jordan-Brouwer separation theorem and the fact that $f|_{Q_{i_{N}}}$ is injective, we must have $\gamma(t^\star) \in \partial f|_{Q_{i_{N}}}(S)$ (Figure \ref{fig:homotopy}) - that is, $\gamma(t^\star)$ is in the boundary of the image of the \textit{restriction} of $f$ to $Q_{i_{N}}$. Since $\| \gamma(t^\star) - \gamma(t_{N}) \| \leq \delta$ by definition and $B_{2 \delta}(\gamma(t_{N})) \cap f(S)  \subset f(Q_{i_{N}})$, this implies that $\gamma(t^\star) \in \partial f(S)$ - that is, $\gamma(t^\star)$ is in the boundary of the image of $f$ itself. This completes the proof.
\end{proof}

\begin{figure}[t]
\begin{center}
\includegraphics[trim={0 5mm 0 5mm}, clip, scale=0.3]{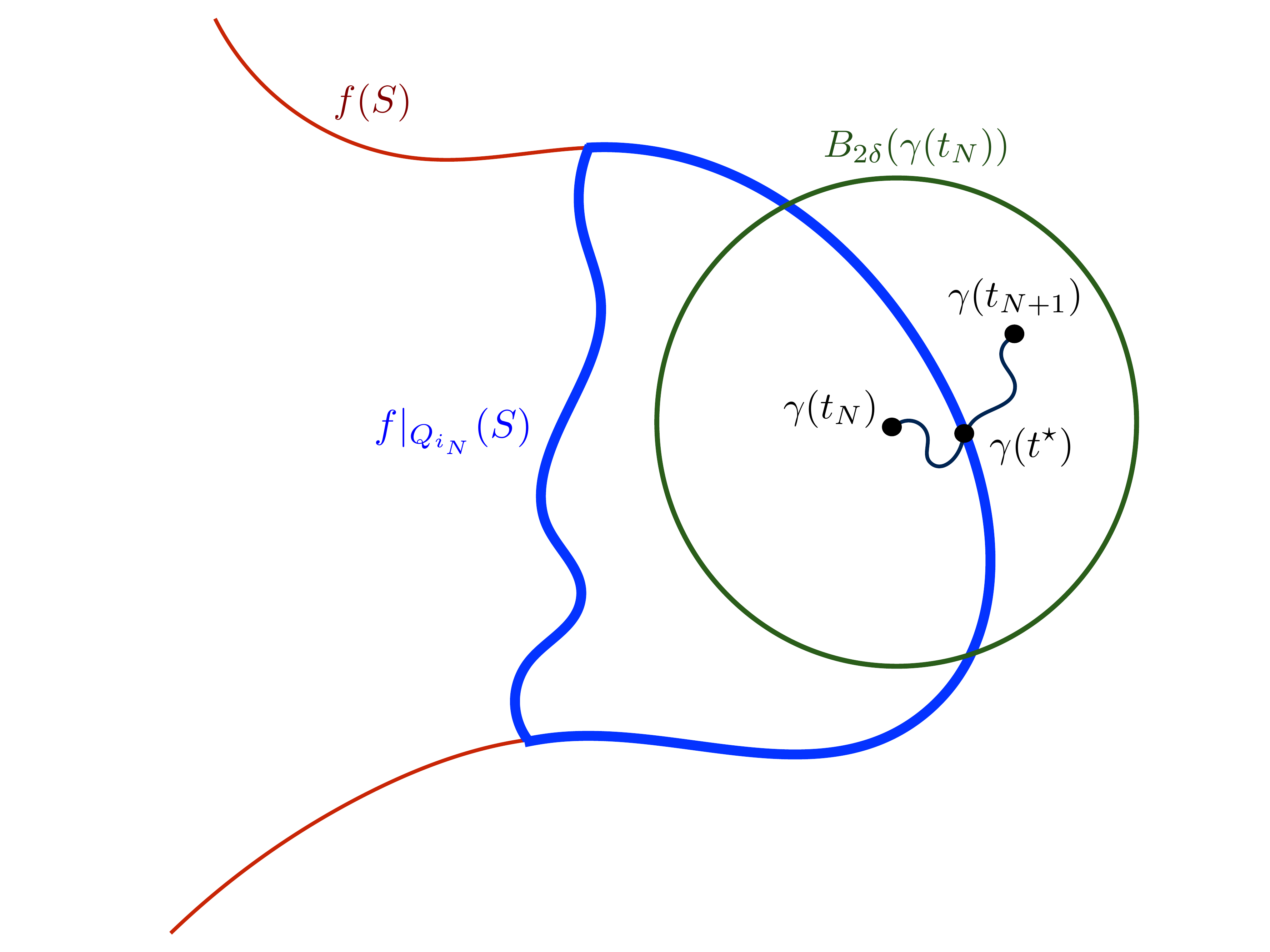}
\end{center}
\caption{Illustration of the proof of Lemma \ref{thm:homotopy1}.}\label{fig:homotopy}
\end{figure}

\subsection{Proof of Lemma \ref{LemmaMinStrongLog}} \label{SubsecProofLemmaMinStrongLog}

 Let $\Phi$ and $\phi$ denote the probability measure and probability density function of the multivariate Gaussian distribution on $\mathbb{R}^{d}$ with mean 0 and covariance matrix equal to the identity. We now prove Lemma \ref{LemmaMinStrongLog} as an immediate corollary of the following slightly more complicated lemma: 

\begin{lemma} \label{LemmaKernelContinuityMainLemma}
Fix all notation as in Lemma \ref{LemmaMinStrongLog}. Set $\delta = \min(\epsilon, \,\, \frac{\epsilon}{d} \cdot \frac{(M_2)^2 T}{M_3}$). Let $L_\delta$ be the lattice $L_\delta := \{x = 2 \delta i :  i \in \mathbb{Z}^d\}$; for $x \in L_{\delta}$, let $C_x^\delta$ be a cube with side length $2 \delta$ and center $x$.

Let $\mathsf{p}^\star_0$ be a random variable supported on the lattice $L_\delta$ distributed according to the law $\mathbb{P}(\mathsf{p}^\star_0 = x) = \Phi(C_x^\delta)$ for all $x\in L_\delta$.  Let $\overline{\mathfrak{p}_0}(x)$ be distributed according to the standard Gaussian distribution, conditioned on the event $\{ \overline{\mathfrak{p}_0}(x) \in C_x^\delta \}$.  Define the random variable $\overline{\mathfrak{p}_0}$ by $\overline{\mathfrak{p}_0} := \overline{\mathfrak{p}_0}(\mathsf{p}^\star_0)$ (Note that $\overline{\mathfrak{p}_0}$ is itself a standard Gaussian on $\mathbb{R}^d$).  Then 

\be 
\mathbb{E}\left[\, \,  \big \| \mathcal{L}(\mathcal{Q}^{\mathbf{q}^{(2)}}_T(\overline{\mathfrak{p}_0}) | \, \,  \mathsf{p}^\star_0) - \mathcal{L}(\mathcal{Q}^{\mathbf{q}^{(1)}}_T(\overline{\mathfrak{p}_0})  | \, \,  \mathsf{p}^\star_0) \big \|_{\mathrm{TV}} \right] \leq  5\epsilon'.
\ee
\end{lemma}

In the remainder of this section of the appendix, we will prove a sequence of lemmas leading up to the proof Lemma \ref{LemmaKernelContinuityMainLemma}. Unless otherwise noted, all notation from the statement of Lemma \ref{LemmaKernelContinuityMainLemma} will be retained throughout the section. We will also use the following notation: 

\begin{itemize}
\item We write $q_t(\mathbf{q},\mathbf{p})$ for the solution to Hamilton's equation \eqref{EqHamiltonEquations} for the position variable at time $t$, with initial position $\mathbf{q}$ and initial momentum $\mathbf{p}$ (in previous sections $q_t(\mathbf{q},\mathbf{p})$ was often denoted instead by $q(t)$, but in this section we emphasize the dependence on initial conditions). 

\item For any multivariate function $f: \mathbb{R}^d \rightarrow \mathbb{R}^d$, denote its Jacobian by $J(f)$, and denote its Jacobian at $p \in \mathbb{R}^{d}$ by $J(f)|_{p}$.
\item For a $d$ by $d$ matrix $H$, let $\sigma_{1}(H) \leq \ldots \leq \sigma_{d}(H)$ be the ordered singular values of $H$; let $\sigma_{\mathrm{max}}(H) = \sigma_{d}(H)$ and $\sigma_{\mathrm{min}}(H)=\sigma_{1}(H)$ be the largest and smallest singular values, respectively.
\item  We often use $\mathbf{p}^\star \in \mathbb{R}^{d}$ as a placeholder for a generic momentum. We define three quantities in terms of $\mathbf{p}^{\star}$: the set
\be \label{EqDefP_0ByP_Star}
P_0: = \mathbf{p}^\star + \frac{\epsilon}{d} \cdot \frac{(M_2)^2 T}{M_3} \cdot[-1,1]^{d}
\ee
and the two families of sets 
\be
S_{q} &:= \mathcal{Q}^{q}_T(\mathbf{p}^\star) + (1-\frac{\epsilon}{d})\cdot\left[J\big(\mathcal{Q}^{q}_T\big)\big|_{\mathbf{p}^\star}(P_0 -\mathbf{p}^\star)\right] \\
S'_{q} &:= \mathcal{Q}^{q}_T(\mathbf{p}^\star) + (1+\frac{\epsilon}{d})\cdot\left[J\big(\mathcal{Q}^{q}_T\big)\big|_{\mathbf{p}^\star}(P_0 -\mathbf{p}^\star)\right] \\
\ee
parameterized by an initial position $q \in \mathbb{R}^{d}$.
\end{itemize}

Most of the proof of Lemma \ref{LemmaKernelContinuityMainLemma} consists of bounding bounding the dependence of a solution to Hamilton's equations on the \textit{initial point} $\mathbf{h}:= (\mathbf{q},\mathbf{p})$. Our first bound on this dependence is:

\begin{lemma}\label{thm:differential_operator}
Let Assumption \ref{AssumptionSecondDerivative} hold. Let $\zeta = (u, \omega) \in \mathbb{R}^{2d}$ and $\eta = (v,\xi) \in \mathbb{R}^{2d}$ be any two vectors in the phase space. Then for all $\mathbf{h} \in \mathbb{R}^{d}\times \mathbb{R}^{d}$,
\be\label{eq:TV7}
\left\| D_\zeta q_t\big|_{\mathbf{h}}\right\| \leq \kappa_1e^{t\sqrt{M_2}} + \kappa_2e^{-t\sqrt{M_2}},
\ee 
where $\kappa_1 :=  \frac{1}{2}\|u\|+  \frac{1}{2\sqrt{M_2}}\|\omega\|$ and $\kappa_2 := \frac{1}{2}\|u\| -  \frac{1}{2\sqrt{M_2}}\|\omega\|$. Moreover, \\

\be\label{eq:TV8}
\left\| D_\eta D_\zeta q_t\big|_{\mathbf{h}}\right\| &\leq \frac{M_3}{M_2} \left(\frac{\kappa_1 \kappa'_1}{3}e^{2t\sqrt{M_2}} +  \frac{\kappa_2\kappa'_2}{3}e^{-2t\sqrt{M_2}} - \kappa_1\kappa'_2-\kappa'_1\kappa_2\right)\\
&\quad \quad + \kappa_3e^{\sqrt{M_2}t}+ \kappa_4e^{-\sqrt{M_2}t},
\ee
where
\be 
\kappa'_1 &=  \frac{1}{2}\|v\|+  \frac{1}{2\sqrt{M_2}}\|\xi\| \\
\kappa'_2 &= \frac{1}{2}\|v\| -  \frac{1}{2\sqrt{M_2}}\|\xi\| \\
\kappa_3 &=  -\frac{M_3}{2M_2} \left(\frac{\kappa_1 \kappa'_1}{3} +  \frac{\kappa_2\kappa'_2}{3} - \kappa_1\kappa'_2-\kappa'_1\kappa_2\right)  -\frac{M_3}{M_2} \left(\frac{\kappa_1 \kappa'_1}{3} -  \frac{\kappa_2\kappa'_2}{3}\right) \\
\kappa_4 &= -\frac{M_3}{M_2} \left(\frac{\kappa_1 \kappa'_1}{3} +  \frac{\kappa_2\kappa'_2}{3} - \kappa_1\kappa'_2-\kappa'_1\kappa_2\right) - \kappa_3.
\ee

\end{lemma}

\begin{proof}
Define $V:= U'[l]$ for each $l \in \{1,\ldots, d\}$.  Then for every $x\in \mathbb{R}^d$ we have
\be \label{eq:TV1'}
D_\zeta U'(q_t)[l]\big|_{x} &= \sum_{j=1}^{2d} \zeta_j \frac{\partial}{\partial x_j} V(q_t)\\
&\stackrel{\textrm{chain rule}}{=} \sum_{j=1}^{2d} \zeta_j \sum_{i=1}^d \frac{\partial V}{\partial y_i} \big|_{q_t(x)}  \frac{\partial q_t}{\partial x_j}[i]\\
&= \sum_{i=1}^d \frac{\partial V}{\partial y_i} \big|_{q_t(x)} \sum_{j=1}^{2d} \zeta_j  \frac{\partial q_t}{\partial x_j}[i]\\
&= \sum_{i=1}^d \frac{\partial V}{\partial y_i} \big|_{q_t(x)} \left( D_\zeta q_t |_{x}[i]\right)\\
&= D_{D_\zeta q_t |_{x}}V\big|_{q_t(x)},
\ee
so  
\be \label{eq:TV1}
\bigg\|D_\zeta U'(q_t)\big|_{\mathbf{h}}\bigg\| &\stackrel{\textrm{Eq. } \ref{eq:TV1'}}{=}\left \| D_{D_\zeta q_t |_{\mathbf{h}}} U'\big|_{q_t(\mathbf{h})}\right\|\\
&=\left\| D_{\frac{D_\zeta q_t |_{\mathbf{h}}}{\|D_\zeta q_t |_{\mathbf{h}}\|}} U'\big|_{q_t(\mathbf{h})}\right\| \times \left\|D_\zeta q_t\big|_{\mathbf{h}} \right\| \\
&\leq M_2 \times \left\|D_\zeta q_t\big|_{\mathbf{h}}\right\|. \\
\ee 

By a similar calculation,

\be \label{eq:TV2'}
D_\eta D_\zeta U'(q_t)[l]\big|_{\mathbf{h}} &= D_\eta D_\zeta V(q_t)\big|_{\mathbf{h}}\\
&\stackrel{\textrm{Eq. } \ref{eq:TV1'}}{=}  D_\eta \left( D_{D_\zeta q_t |_{x}} V\big|_{q_t(x)} \right) \bigg|_{x=\mathbf{h}} \\
&=   D_\eta \left(\left(D_\zeta q_t |_{x} \right)^\top \nabla V \bigg |_{q_t(x)} \right) \bigg|_{x=\mathbf{h}}\\
&=   D_\eta \left( \sum_{i=1}^{d} \left(D_\zeta q_t[i] |_{x} \right) \frac{\partial V}{\partial y_i} \bigg |_{q_t(x)} \right) \bigg|_{x=\mathbf{h}}\\
&=   D_\eta \left( \sum_{i=1}^{d} \left(\sum_{j=1}^{2d} \zeta_j \frac{\partial q_t}{\partial x_j}[i]\bigg |_x \right) \frac{\partial V}{\partial y_i} \bigg |_{q_t(x)} \right) \bigg|_{x=\mathbf{h}}\\
&=   \eta^\top \nabla \left[ \sum_{i=1}^{d} \left(\sum_{j=1}^{2d} \zeta_j \frac{\partial q_t}{\partial x_j}[i]\bigg |_x \right) \frac{\partial V}{\partial y_i} \bigg |_{q_t(x)} \right] \Bigg|_{x=\mathbf{h}}\\
&=   \sum_{k=1}^{2d} \eta_k \frac{\partial}{\partial x_k} \left[ \sum_{i=1}^{d} \left(\sum_{j=1}^{2d} \zeta_j \frac{\partial q_t}{\partial x_j}[i]\bigg |_x \right) \frac{\partial V}{\partial y_i} \bigg |_{q_t(x)} \right] \Bigg|_{x=\mathbf{h}}\\
&=   \sum_{k=1}^{2d} \eta_k \sum_{i=1}^{d} \left[ \left(\sum_{j=1}^{2d} \zeta_j \frac{\partial^2 q_t}{\partial x_k \partial x_j}[i]\bigg |_x \right) \frac{\partial V}{ \partial y_i} \bigg |_{q_t(x)}     
+ \left(\sum_{j=1}^{2d} \zeta_j \frac{\partial q_t}{\partial x_j}[i]\bigg |_x \right) \frac{\partial}{\partial x_k} \left(\frac{\partial V}{ \partial y_i} \bigg |_{q_t(x)}\right) \right] \Bigg|_{x=\mathbf{h}}\\
&\stackrel{\textrm{chain rule}}{=}   \sum_{k=1}^{2d} \eta_k \sum_{i=1}^{d} \bigg[ \left(\sum_{j=1}^{2d} \zeta_j \frac{\partial^2 q_t}{\partial x_k \partial x_j}[i]\bigg |_x \right) \frac{\partial V}{ \partial y_i} \bigg |_{q_t(x)}\\     
&+ \left(\sum_{j=1}^{2d} \zeta_j \frac{\partial q_t}{\partial x_j}[i]\bigg |_x \right) \left(\sum_{m=1}^d \frac{\partial^2 V}{ \partial y_m \partial y_i} \bigg |_{q_t(x)} \times \frac{\partial q_t}{\partial x_k}[m]\right) \bigg] \Bigg|_{x=\mathbf{h}}\\
&=   \sum_{k,i,j} \eta_k \zeta_j \frac{\partial^2 q_t}{\partial x_k \partial x_j}[i]\bigg |_{\mathbf{h}} \times \frac{\partial V}{ \partial y_i} \bigg |_{q_t(\mathbf{h})}\\     
&+ \sum_{i=1}^{d} \left[ \left(\sum_{j=1}^{2d} \zeta_j \frac{\partial q_t}{\partial x_j}[i]\bigg |_{\mathbf{h}} \right) \left(\sum_{m=1}^d \frac{\partial^2 V}{ \partial y_m \partial y_i} \bigg |_{q_t(\mathbf{h})} \times \bigg( \sum_{k=1}^{2d} \eta_k \frac{\partial q_t}{\partial x_k}[m]\Bigg|_{\mathbf{h}} \bigg) \right) \right]\\
&=    (D_\eta D_\zeta q_t  |_{\mathbf{h}})^\top \nabla V  |_{q_t(\mathbf{h})} + \sum_{i=1}^{d} \left[ \left(D_\zeta q_t |_{\mathbf{h}}[i] \right) \left(\sum_{m=1}^d \frac{\partial^2 V}{ \partial y_m \partial y_i} \bigg |_{q_t(\mathbf{h})} \times \bigg( D_\eta q_t |_{\mathbf{h}}[m] \bigg) \right) \right]\\
&=    (D_\eta D_\zeta q_t  |_{\mathbf{h}})^\top \nabla V  |_{q_t(\mathbf{h})} + D_{D_\eta q_t |_{\mathbf{h}}} D_{D_\zeta q_t |_{\mathbf{h}}} V |_{q_t(\mathbf{h})},\\
\ee
so
\be \label{eq:TV2}
\|D_\eta D_\zeta U'(q_t)\big|_{\mathbf{h}} \| &\stackrel{{\scriptsize \textrm{Eq. }}\ref{eq:TV2'}}{\leq} \|(D_\eta D_\zeta q_t  |_{\mathbf{h}})^\top \nabla U'  |_{q_t(\mathbf{h})}\| + \|D_{D_\eta q_t |_{\mathbf{h}}} D_{D_\zeta q_t |_{\mathbf{h}}} U' |_{q_t(\mathbf{h})}\|\\
&= \|(D_\eta D_\zeta q_t  |_{\mathbf{h}})^\top \nabla U'  |_{q_t(\mathbf{h})}\| + \|D_{\frac{D_\eta q_t |_{\mathbf{h}}}{\|D_\eta q_t |_{\mathbf{h}}\|}} D_{\frac{D_\zeta q_t |_{\mathbf{h}}}{\|D_\zeta q_t |_{\mathbf{h}}\|}} U' |_{q_t(\mathbf{h})}\| \times \|D_\eta q_t |_{\mathbf{h}}\| \times \|D_\zeta q_t |_{\mathbf{h}}\| \\
&\stackrel{{\scriptsize \textrm{Assumptions }}\ref{AssumptionsConvexity}, \, \ref{AssumptionSecondDerivative}}{\leq} M_2\|(D_\eta D_\zeta q_t  |_{\mathbf{h}})\| + M_3\|D_\eta q_t |_{\mathbf{h}}\| \times \|D_\zeta q_t |_{\mathbf{h}}\|.
\ee

We denote by $\mathcal{D}$ either of the differential operators  $D_\zeta$ or  $D_\eta D_\zeta$. Recall Equation \eqref{eq:Galilean}:

\be 
\frac{\mathrm{d}q}{\mathrm{d}t} = p, \quad \quad \frac{\mathrm{d}p}{\mathrm{d}t} = -U'(q).
\ee
Applying $\mathcal{D}$ to both sides gives
\be 
\mathcal{D} \left(\frac{\mathrm{d}q_t}{\mathrm{d}t} \right)\bigg|_{\mathbf{h}} = \mathcal{D} p_t\big|_{\mathbf{h}}, \quad \quad \mathcal{D} \left(\frac{\mathrm{d}p_t}{\mathrm{d}t} \right)\bigg|_{\mathbf{h}}= -\mathcal{D}U'(q_t) \big|_{\mathbf{h}},
\ee 
and exchanging derivatives gives
\be 
\frac{\mathrm{d}}{\mathrm{d}t} \mathcal{D} q_t\big|_{\mathbf{h}} = \mathcal{D} p_t\big|_{\mathbf{h}}, \quad \quad \frac{\mathrm{d}}{\mathrm{d}t} \mathcal{D} p_t\big|_{\mathbf{h}} = -\mathcal{D}U'(q_t)\big|_{\mathbf{h}}. 
\ee
Thus, 
\be \label{eq:TV4}
 \frac{\mathrm{d}}{\mathrm{d}t}\| \mathcal{D} q_t \big|_{\mathbf{h}}\| & \leq \|\frac{\mathrm{d}}{\mathrm{d}t} \mathcal{D} q_t \big|_{\mathbf{h}}\| = \|\mathcal{D} p_t \big|_{\mathbf{h}}\| \\
 \frac{\mathrm{d}}{\mathrm{d}t} \|\mathcal{D} p_t\big|_{\mathbf{h}}\| &\leq \|\frac{\mathrm{d}}{\mathrm{d}t} \mathcal{D} p_t\big|_{\mathbf{h}}\| = \|\mathcal{D}U'(q_t) \big|_{\mathbf{h}}\|.
\ee

We now prove Inequality \eqref{eq:TV7}. In the case where $\mathcal{D} = D_\zeta$, Inequalities \eqref{eq:TV4} and \eqref{eq:TV1} give the following system of differential inequalities:
\be \label{eq:TV26}
\frac{\mathrm{d}}{\mathrm{d}t}\| \mathcal{D} q_t \big|_{\mathbf{h}}\| \leq \|\mathcal{D} p_t \big|_{\mathbf{h}}\|, \quad \quad \frac{\mathrm{d}}{\mathrm{d}t} \|\mathcal{D} p_t \big|_{\mathbf{h}}\| \leq M_2 \, \|\mathcal{D} q_t \big|_{\mathbf{h}}\|. 
\ee

Since $D_\zeta q_0 \big|_{\mathbf{h}} = u$ and  $D_\zeta p_0 \big|_{\mathbf{h}} = \omega$, the constants $\kappa_1$ and $\kappa_2$ satisfy
\be 
\kappa_1+ \kappa_2 &= \|\mathcal{D} q_0\big|_{\mathbf{h}}\| \\
\kappa_1\sqrt{M_2} - \kappa_2\sqrt{M_2} &= \frac{\mathrm{d}}{\mathrm{d}t} \|\mathcal{D} q_0\big|_{\mathbf{h}}\| = \|\mathcal{D} p_0\big|_{\mathbf{h}}\|,
\ee
so applying Lemma \ref{LemmaOdeComp} to the system of equations \eqref{eq:TV26} we have
\begin{equation}\label{eq:TV3}
\| \mathcal{D} q_t \big|_{\mathbf{h}}\| \leq \kappa_1e^{t\sqrt{M_2}} + \kappa_2e^{-t\sqrt{M_2}}.
\end{equation}

This completes the proof of Inequality \eqref{eq:TV7}.

We now prove Inequality \eqref{eq:TV8}. In the case where $\mathcal{D} \equiv D_\eta D_\zeta$,

\be \label{eq:TV5}
\|D_\eta D_\zeta U'(q_t&)\big|_{\mathbf{h}}\| \stackrel{{\scriptsize \textrm{Eq. }}\ref{eq:TV2}}{\leq} M_2 \cdot  \|D_\eta D_\zeta q_t \big|_{\mathbf{h}}\| + M_3 \cdot \|D_\eta q_t \big|_{\mathbf{h}}\| \cdot \|D_\zeta q_t \big|_{\mathbf{h}}\| \\
&\stackrel{{\scriptsize \textrm{Eq. }}\ref{eq:TV3}}{\leq} M_2 \times  \|D_\eta D_\zeta q_t \big|_{\mathbf{h}}\| + M_3 \times \left(\kappa_1e^{t\sqrt{M_2}} + \kappa_2e^{-t\sqrt{M_2}}\right) \, \left(\kappa'_1e^{t\sqrt{M_2}} + \kappa'_2e^{-t\sqrt{M_2}}\right) \\
&= M_2 \times  \|D_\eta D_\zeta q_t \big|_{\mathbf{h}}\| + M_3 \times \left(\kappa_1\kappa'_1e^{2t\sqrt{M_2}} + \kappa_2\kappa'_2e^{-2t\sqrt{M_2}} + \kappa_1\kappa'_2+\kappa'_1\kappa_2\right), \ee

where $\kappa'_1 =  \frac{1}{2}\|v\| +  \frac{1}{2\sqrt{M_2}}\|\xi\|$  and $\kappa'_2 := \frac{1}{2}\|v\| -  \frac{1}{2\sqrt{M_2}}\|\xi\|$. Therefore, Equations \eqref{eq:TV4} and \eqref{eq:TV5} give the following system of differential inequalities when $\mathcal{D} \equiv D_\eta D_\zeta$:
\be \label{eq:TV27}
\frac{\mathrm{d}}{\mathrm{d}t}\| \mathcal{D} q_t \big|_{\mathbf{h}}\| &\leq \|\mathcal{D} p_t \big|_{\mathbf{h}}\|,\\
\frac{\mathrm{d}}{\mathrm{d}t} \|\mathcal{D} p_t \big|_{\mathbf{h}}\| &\leq \left(\kappa_1\kappa'_1e^{2t\sqrt{M_2}} + \kappa_2\kappa'_2e^{-2t\sqrt{M_2}} + \kappa_1\kappa'_2+\kappa'_1\kappa_2\right)\times M_3  +   \| \mathcal{D}q_t \big|_{\mathbf{h}} \|  \times M_2.
\ee

Since $\| \mathcal{D} q_0 \big| =\| D_\eta D_\zeta q_0 \big|_{\mathbf{h}}\| =0$ and $\| \mathcal{D} p_0 \big|_{\mathbf{h}}\| = \| D_\eta D_\zeta p_0 \big|_{\mathbf{h}}\| =0$, the constants $\kappa_3$ and $\kappa_4$ satisfy
\be 
\frac{M_3}{M_2} \left(\frac{\kappa_1 \kappa'_1}{3} +  \frac{\kappa_2\kappa'_2}{3} - \kappa_1\kappa'_2-\kappa'_1\kappa_2\right) + \kappa_3+ \kappa_4 & =\| \mathcal{D} q_0 \big|_{\mathbf{h}}\| \\
\frac{2M_3}{\sqrt{M_2}} \left(\frac{\kappa_1 \kappa'_1}{3} -  \frac{\kappa_2\kappa'_2}{3}\right) + \kappa_3\sqrt{M_2}- \kappa_4\sqrt{M_2} &=\| \mathcal{D} p_0 \big|_{\mathbf{h}}\|.
\ee

Thus, applying Lemma \ref{LemmaOdeComp} to the system of equations \eqref{eq:TV27}, we get

\be\label{eq:TV6}
\| \mathcal{D} q_t \big|_{\mathbf{h}}\| &\leq \frac{M_3}{M_2} \left(\frac{\kappa_1 \kappa'_1}{3}e^{2t\sqrt{M_2}} +  \frac{\kappa_2\kappa'_2}{3}e^{-2t\sqrt{M_2}} - \kappa_1\kappa'_2-\kappa'_1\kappa_2\right)\\
&\quad \quad + \kappa_3e^{\sqrt{M_2}t}+ \kappa_4e^{-\sqrt{M_2}t}.
\ee
This completes the proof of the Lemma. 

\end{proof}

\begin{lemma}\label{thm:TV2}
We have
\be \label{IneqThmTV2Conc1}
\frac{9}{10} T \leq \sigma_{\mathrm{min}}\big(J(\mathcal{Q}^{\mathbf{q}}_T)\big|_{\mathbf{p}}\big) \leq \sigma_{\mathrm{max}}\big(J(\mathcal{Q}^{\mathbf{q}}_T)\big|_{\mathbf{p}}\big) \leq \frac{2}{\sqrt{M_2}}.
\ee
Moreover, fix any $v, \omega, \xi \in \mathbb{R}^{d}$. Then
\be \label{IneqThmTV2Conc2}
\|D_{(0,\xi)} D_{(0,\omega)} q_T \big|_{\mathbf{h}}\| \leq  \frac{1}{40} \frac{M_3}{(M_2)^2} \cdot \|\omega\| \cdot \|\xi\|
\ee
and
\be \label{IneqThmTV2Conc3}
\| D_{(v,0)} D_{(0,\omega)} q_T \big|_{\mathbf{h}}\| \leq \frac{1}{10}\frac{M_3}{(M_2)^{\frac{3}{2}}}\cdot \|\omega\|\cdot \|v\|.
\ee

\end{lemma}

\begin{proof}

We begin by proving Inequality \eqref{IneqThmTV2Conc1}. By Inequality \eqref{thmboundsconc4} of Lemma \ref{thm:bounds},

\be  \label{EqInterJacDDiffOp}
\frac{1}{\|\omega\|}D_{(0,\omega)} q_T \big|_{\mathbf{h}}  & \geq    T - \sqrt{M_2}\sinh(T \sqrt{M_2})\times T^{2}\\
&\geq T - 1.2\sqrt{M_2} (T \sqrt{M_2})^2\times T^{2}\\
&\geq \frac{9}{10}T,\\
\ee 
where we use the fact $T \leq \frac{1}{2\sqrt{2}} \frac{\sqrt{m_2}}{M_{2}}$ in  both inequalities, and the additional fact that $\sinh^2(t) \leq 1.2 t^2$ for all $t \in [0,\frac{1}{2}]$ in the first inequality. By Lemma \ref{thm:differential_operator},

\be 
\frac{1}{\|\omega\|}D_{(0,\omega)} q_T \big|_{\mathbf{h}}  &\leq \frac{1}{2\sqrt{M_2}}e^{T\sqrt{M_2}} \\
&\leq \frac{1}{2\sqrt{M_2}}e^{\frac{1}{\sqrt{2}}} \leq \frac{2}{\sqrt{M_2}}.
\ee 

Combining this with Inequality \eqref{EqInterJacDDiffOp} gives
\be 
\frac{9}{10} T \leq D_{(0,\omega)} q_T \big|_{\mathbf{h}} \leq \frac{2}{\sqrt{M_2}},
\ee
and so
\be 
\frac{9}{10}T \leq \sigma_{\mathrm{min}}\big(J(\mathcal{Q}^{\mathbf{q}}_T)\big|_{\mathbf{p}}\big) \leq \sigma_{\mathrm{max}}\big(J(\mathcal{Q}^{\mathbf{q}}_T)\big|_{\mathbf{p}}\big) \leq \frac{2}{\sqrt{M_2}}. 
\ee
This completes the proof of Inequality \eqref{IneqThmTV2Conc1}. 

 Next, we prove Inequality \eqref{IneqThmTV2Conc2}. By Lemma \ref{thm:differential_operator}, we have

\be \label{IneqThmTv2Int2}
\hspace{-15mm} \| D_{(0,\xi)} D_{(0,\omega)} q_T \big|_{\mathbf{h}}\| \leq \frac{M_3}{M_2} \left(\frac{\kappa_1 \kappa'_1}{3}e^{\sqrt{2}} +  \frac{\kappa_2\kappa'_2}{3}e^{-\sqrt{2}} - \kappa_1\kappa'_2-\kappa'_1\kappa_2\right) + \kappa_3e^{\frac{1}{\sqrt{2}}}+ \kappa_4e^{-\frac{1}{\sqrt{2}}},
\ee

where  
\be 
\kappa_1 &= - \kappa_{2} =  \frac{1}{2\sqrt{M_2}}\|\omega\| \\
\kappa'_1 &= -\kappa'_{2} = \frac{1}{2\sqrt{M_2}}\|\xi\| \\
\kappa_3 &=  -\frac{M_3}{2M_2} \left(\frac{\kappa_1 \kappa'_1}{3} +  \frac{\kappa_2\kappa'_2}{3} - \kappa_1\kappa'_2-\kappa'_1\kappa_2\right)  -\frac{M_3}{M_2} \left(\frac{\kappa_1 \kappa'_1}{3} -  \frac{\kappa_2\kappa'_2}{3}\right) \\
\kappa_4 &= -\frac{M_3}{M_2} \left(\frac{\kappa_1 \kappa'_1}{3} +  \frac{\kappa_2\kappa'_2}{3} - \kappa_1\kappa'_2-\kappa'_1\kappa_2\right) - \kappa_3.
\ee
Observing
\be 
- \kappa_1\kappa'_2-\kappa'_1\kappa_2 &= 2\kappa_1\kappa'_1 = 2\times \frac{1}{2\sqrt{M_2}}\|\omega\|\cdot\frac{1}{2\sqrt{M_2}}\|\xi\|= \frac{1}{2M_2}\|\omega\| \cdot \|\xi\| \\
\kappa_1\kappa'_2-\kappa'_1\kappa_2 &= -\kappa_1\kappa'_1+\kappa'_1\kappa_1 =0  \\
\kappa_1 \kappa'_1 +  \kappa_2\kappa'_2 &= 2\kappa_1\kappa'_1 = \frac{1}{2M_2}\|\omega\| \cdot \|\xi\| \\
\kappa_1 \kappa'_1 -  \kappa_2\kappa'_2 &= \kappa_1 \kappa'_1 -  \kappa_1\kappa'_1 =0, 
\ee 

we have $\kappa_3 = \kappa_4 =  -\frac{M_3}{2M_2} \left(\frac{4}{3}\times\frac{1}{2M_2}\|\omega\| \cdot \|\xi\|\right) = -\frac{M_3}{3(M_2)^2}\|\omega\| \cdot \|\xi\|$.

By Inequality \eqref{IneqThmTv2Int2}, this gives
\be 
\|D_{(0,\xi)} D_{(0,\omega)} q_T\| &\leq  \frac{M_3}{M_2} \left(\frac{\kappa_1 \kappa'_1}{3}e^{\sqrt{2}} +  \frac{\kappa_2\kappa'_2}{3}e^{-\sqrt{2}} - \kappa_1\kappa'_2-\kappa'_1\kappa_2\right) + \kappa_3e^{\frac{1}{\sqrt{2}}}+ \kappa_4e^{-\frac{1}{\sqrt{2}}}\\
&=  \frac{M_3}{M_2} \left(\frac{\kappa_1 \kappa'_1}{3}\left(e^{\sqrt{2}} +  e^{-\sqrt{2}}\right) + \frac{1}{2M_2}\|\omega\| \cdot \|\xi\|\right) + \kappa_3\left(e^{\frac{1}{\sqrt{2}}}+ e^{-\frac{1}{\sqrt{2}}}\right)\\
&=  \frac{M_3}{M_2} \left(\frac{1}{12M_2}\|\omega\| \cdot \|\xi\|\left(e^{\sqrt{2}} +  e^{-\sqrt{2}}\right) + \frac{1}{2M_2}\|\omega\| \cdot \|\xi\|\right)  -\frac{M_3}{3(M_2)^2}\|\omega\| \cdot \|\xi\|\cdot\left(e^{\frac{1}{\sqrt{2}}}+ e^{-\frac{1}{\sqrt{2}}}\right)\\
&=  \frac{M_3}{(M_2)^2} \left[\frac{1}{12}\left(e^{\sqrt{2}} +  e^{-\sqrt{2}}\right) + \frac{1}{2}  -\frac{1}{3}\cdot\left(e^{\frac{1}{\sqrt{2}}}+ e^{-\frac{1}{\sqrt{2}}}\right)\right]\cdot \|\omega\| \cdot \|\xi\|\\
&\leq \frac{1}{40}  \frac{M_3}{(M_2)^2} \cdot \|\omega\| \cdot \|\xi\|.
\ee

This completes the proof of Inequality \eqref{IneqThmTV2Conc2}.

Finally, we prove Inequality \eqref{IneqThmTV2Conc3}. 
By Lemma \ref{thm:differential_operator} we have:

\be \label{IneqThmTv2Int3}
\left\| D_{(v,0)} D_{(0,\omega)} q_T \big|_{\mathbf{h}}\right\| &\leq \frac{M_3}{M_2} \left(\frac{\kappa_1 \kappa'_1}{3}e^{2T\sqrt{M_2}} +  \frac{\kappa_2\kappa'_2}{3}e^{-2T\sqrt{M_2}} - \kappa_1\kappa'_2-\kappa'_1\kappa_2\right)\\
& \quad \quad+ \kappa_3e^{\sqrt{M_2}T}+ \kappa_4e^{-\sqrt{M_2}T}
\ee
where  
\be 
\kappa_1 &= - \kappa_{2} = \frac{1}{2\sqrt{M_2}}\|\omega\| \\
\kappa'_1 &= \kappa'_{2} =  \frac{1}{2}\|v\| \\
\kappa_3 &=  -\frac{M_3}{2M_2} \left(\frac{\kappa_1 \kappa'_1}{3} +  \frac{\kappa_2\kappa'_2}{3} - \kappa_1\kappa'_2-\kappa'_1\kappa_2\right)  -\frac{M_3}{M_2} \left(\frac{\kappa_1 \kappa'_1}{3} -  \frac{\kappa_2\kappa'_2}{3}\right) \\
\kappa_4 &= -\frac{M_3}{M_2} \left(\frac{\kappa_1 \kappa'_1}{3} +  \frac{\kappa_2\kappa'_2}{3} - \kappa_1\kappa'_2-\kappa'_1\kappa_2\right) - \kappa_3.
\ee

Observing
\be 
- \kappa_1\kappa'_2-\kappa'_1\kappa_2 &= - \kappa_1\kappa'_1 +\kappa'_1\kappa_1 = 0 \\
\kappa_1\kappa'_2-\kappa'_1\kappa_2 &= 2\kappa_1\kappa'_1 =  2\times \frac{1}{2\sqrt{M_2}}\|\omega\|\cdot \frac{1}{2}\|v\| = \frac{1}{2\sqrt{M_2}}\|\omega\|\cdot \|v\| \\
\kappa_1 \kappa'_1 +  \kappa_2\kappa'_2 &= \kappa_1\kappa'_1 - \kappa_1\kappa'_1 = 0\\
\kappa_1 \kappa'_1 -  \kappa_2\kappa'_2 &= 2\kappa_1 \kappa'_1= \frac{1}{2\sqrt{M_2}}\|\omega\|\cdot \|v\|, 
\ee
we have the simpler relationships
\be
\kappa_3 =  -\frac{M_3}{2M_2} \left(\frac{\kappa_1 \kappa'_1}{3} -  \frac{\kappa_1\kappa'_1}{3} - 0\right)  -\frac{M_3}{M_2} \left(\frac{\kappa_1 \kappa'_1}{3} +  \frac{\kappa_1\kappa'_1}{3}\right) = 0 - \frac{2}{3}\frac{M_3}{M_2} \kappa_1 \kappa'_1 = - \frac{M_3}{6(M_2)^\frac{3}{2}}\|\omega\|\cdot \|v\|
\ee 
and $\kappa_4 = - \kappa_3$.

Applying Inequality \eqref{IneqThmTv2Int3}, this gives
\be  
\|D_{(v,0)} D_{(0,\omega)} q_T \big|_{\mathbf{h}}\| &\leq  \frac{M_3}{M_2} \left(\frac{\kappa_1 \kappa'_1}{3}e^{\sqrt{2}} +  \frac{\kappa_2\kappa'_2}{3}e^{-\sqrt{2}} - \kappa_1\kappa'_2-\kappa'_1\kappa_2\right) + \kappa_3e^{\frac{1}{\sqrt{2}}}+ \kappa_4e^{-\frac{1}{\sqrt{2}}}\\
&=  \frac{M_3}{M_2} \left(\frac{\kappa_1 \kappa'_1}{3}e^{\sqrt{2}} -  \frac{\kappa_1\kappa'_1}{3}e^{-\sqrt{2}}-0\right) + \kappa_3e^{\frac{1}{\sqrt{2}}} - \kappa_3e^{-\frac{1}{\sqrt{2}}}\\
&=  \frac{M_3}{M_2} \frac{1}{12\sqrt{M_2}}\|\omega\|\cdot \|v\| \left(e^{\sqrt{2}} -  e^{-\sqrt{2}}\right)  - \frac{M_3}{6(M_2)^\frac{3}{2}} \left(e^{\frac{1}{\sqrt{2}}} - e^{-\frac{1}{\sqrt{2}}}\right)\cdot \|\omega\|\cdot \|v\|\\
&=  \frac{M_3}{(M_2)^{\frac{3}{2}}} \left[\frac{1}{12} \left(e^{\sqrt{2}} -  e^{-\sqrt{2}}\right)  - \frac{1}{6}\left(e^{\frac{1}{\sqrt{2}}} - e^{-\frac{1}{\sqrt{2}}}\right) \right]\cdot \|\omega\|\cdot \|v\|\\
&\leq \frac{1}{10}\frac{M_3}{(M_2)^{\frac{3}{2}}}\cdot \|\omega\|\cdot \|v\|.
\ee

This completes the proof of the lemma.
\end{proof}

\begin{lemma} \label{thm:Jacobian}
For any $p, w \in \mathbb{R}^d$ we have
\be
\left \| \left(J(\mathcal{Q}^{\mathbf{q}^{(2)}}_T)\big|_{p}\right)\omega - \left(J(\mathcal{Q}^{\mathbf{q}^{(1)}}_T)\big|_{p}\right)\omega \right \| \leq \frac{1}{10}\frac{M_3}{(M_2)^{\frac{3}{2}}}\cdot \|\omega\|\cdot \|\mathbf{q}^{(2)} - \mathbf{q}^{(1)}\|.
\ee

\end{lemma}
\begin{proof}
By Inequality \eqref{IneqThmTV2Conc3} of Lemma \ref{thm:TV2}, we have
\be \label{eq:J1}
\| D_{(v,0)} \left(J(\mathcal{Q}^{x}_T)\big|_{y}\right)\omega \big|_{(x,y)=(q,p)}\| = \| D_{(v,0)} D_{(0,\omega)} q_T \big|_{(q,p)}\| \leq \frac{1}{10}\frac{M_3}{(M_2)^{\frac{3}{2}}}\cdot \|\omega\|\cdot \|v\|
\ee
for all $q,v\in \mathbb{R}^d$. Therefore,

\be
\left \| \left(J(\mathcal{Q}^{\mathbf{q}^{(2)}}_T)\big|_{p}\right)\omega - \left(J(\mathcal{Q}^{\mathbf{q}^{(1)}}_T)\big|_{p}\right)\omega \right \| &=  \bigg \| \int_0^{\|\mathbf{q}^{(2)} - \mathbf{q}^{(1)}\|} D_\frac{\mathbf{q}^{(2)} - \mathbf{q}^{(1)}}{\|\mathbf{q}^{(2)} - \mathbf{q}^{(1)}\|} \left(J(\mathcal{Q}^{x}_T)\big|_{p}\right)\omega \big|_{\ell_s(\mathbf{q}^{(1)}, \mathbf{q}^{(2)})} \mathrm{d}s \bigg \|\\
&\leq  \int_0^{\|\mathbf{q}^{(2)} - \mathbf{q}^{(1)}\|} \bigg \|D_\frac{\mathbf{q}^{(2)} - \mathbf{q}^{(1)}}{\|\mathbf{q}^{(2)} - \mathbf{q}^{(1)}\|} \left(J(\mathcal{Q}^{x}_T)\big|_{p}\right)\omega \big|_{x =\ell_s(\mathbf{q}^{(1)}, \mathbf{q}^{(2)})} \bigg \| \mathrm{d}s\\
&\stackrel{{\scriptsize \textrm{Eq. }}\eqref{eq:J1}}{\leq}  \int_0^{\|\mathbf{q}^{(2)} - \mathbf{q}^{(1)}\|} \frac{1}{10}\frac{M_3}{(M_2)^{\frac{3}{2}}}\cdot \|\omega\|\cdot \|\mathbf{q}^{(2)} - \mathbf{q}^{(1)}\| \mathrm{d}s\\
&\leq \frac{1}{10}\frac{M_3}{(M_2)^{\frac{3}{2}}}\cdot \|\omega\|\cdot \|\mathbf{q}^{(2)} - \mathbf{q}^{(1)}\|.
\ee

This completes the proof of the lemma.
\end{proof}

\begin{lemma}\label{thm:determinant}
For any two momenta $\mathbf{p}, \mathbf{p}^{\star} \in \mathbb{R}^d$, we have:
\be\label{eq:det1}
 e^{-  \frac{d M_{3}}{30 \, M_{2}^{2} \, T} \|\mathbf{p}^\star - \mathbf{p}\| }  \leq \left| \frac{\det\big(J(\mathcal{Q}^{\mathbf{q}}_T)\big|_{\mathbf{p}}\big)}{ \det \big( J(\mathcal{Q}^{\mathbf{q}}_T)\big|_{\mathbf{p}^\star}\big)} \right | \leq e^{\frac{d M_{3} }{30 M_{2}^{2} \, T} \|\mathbf{p}^\star - \mathbf{p}\|   }.
\ee

For any two positions $\mathbf{q}, \mathbf{q}^{\star} \in \mathbb{R}^d$, we have:
\be\label{eq:det2} 
e^{- \frac{d M_{3} }{9 M_{2}^{3/2} T} \|\mathbf{q}^\star - \mathbf{q}\|  }  \leq \left| \frac{\det\big(J(\mathcal{Q}^{\mathbf{q}}_T)\big|_{\mathbf{p}}\big)}{ \det \big(J(\mathcal{Q}^{\mathbf{q}^\star}_T)\big|_{\mathbf{p}}\big)} \right | \leq e^{ \frac{d M_{3} }{9 M_{2}^{3/2} T} \|\mathbf{q}^\star - \mathbf{q}\| }.
\ee
\end{lemma}

\begin{proof}
By Inequality \eqref{IneqThmTV2Conc2} in Lemma \ref{thm:TV2} and Weyl's theorem from singular value perturbation theory, we have:

\be\label{eq:TV16}
\left | D_{(0,\xi)} \sigma_i \big(J(\mathcal{Q}^{q}_T)\big|_{p} \big) \big|_{(q,p)= \mathbf{h}}\right | \leq \frac{1}{40} \frac{M_3}{(M_2)^2} \cdot \|\xi\|
\ee
and
\be\label{eq:TV17}
\left | D_{(v,0)} \sigma_i \big(J(\mathcal{Q}^{q}_T)\big|_{p}\big) \big|_{(q,p)= \mathbf{h}}\right | \leq \frac{1}{10}\frac{M_3}{(M_2)^{\frac{3}{2}}}\cdot \|v\|
\ee

for every $i \in \{1,\ldots,d\}$. If $\|\xi\|=1$, we have:
\be 
\left |D_{(0,\xi)} \det\left(J(\mathcal{Q}^{q}_T)\big|_{p}\right) \big|_{(q,p)= \mathbf{h}} \right | &= \left | D_{(0,\xi)}  \prod_{i=1}^d  \sigma_i \big(J(\mathcal{Q}^{q}_T)\big|_{p}\big) \big|_{(q,p)= \mathbf{h}} \right |\\
&= \left |\sum_{i=1}^d  \left[ D_{(0,\xi)} \sigma_i \big(J(\mathcal{Q}^{q}_T)\big|_{p}\big) \big|_{(q,p)= \mathbf{h}} \right ] \cdot \prod_{j\neq i}  \sigma_j \big(J(\mathcal{Q}^{\mathbf{q}}_T)\big|_{\mathbf{p}}\big) \right | \\
&= \left |\sum_{i=1}^d  \left[ \frac{D_{(0,\xi)} \sigma_i \big(J(\mathcal{Q}^{q}_T)\big|_{p}\big) \big|_{(q,p)= \mathbf{h}}}{\sigma_i \big(J(\mathcal{Q}^{\mathbf{q}}_T)\big|_{\mathbf{p}}\big)} \right ]\cdot \prod_{j=1}^d  \sigma_j \big(J(\mathcal{Q}^{\mathbf{q}}_T)\big|_{\mathbf{p}}\big) \right |\\
&= \left |\det \big(J(\mathcal{Q}^{\mathbf{q}}_T)\big|_{\mathbf{p}}\big) \cdot \sum_{i=1}^d  \frac{D_{(0,\xi)} \sigma_i \big(J(\mathcal{Q}^{\mathbf{q}}_T)\big|_{\mathbf{p}}\big)}{\sigma_i \big(J(\mathcal{Q}^{\mathbf{q}}_T)\big|_{\mathbf{p}}\big)} \right |\\
&\leq \left |\det \big(J(\mathcal{Q}^{\mathbf{q}}_T)\big|_{\mathbf{p}}\big) \cdot \sum_{i=1}^d \frac{\frac{1}{40} \frac{M_3}{(M_2)^2}}{\frac{9}{10}T}  \right |\\
&= \left |\det \big(J(\mathcal{Q}^{\mathbf{q}}_T)\big|_{\mathbf{p}}\big) \cdot d \cdot \frac{ M_{3} }{36 T M_{2}^{2}}  \right |,
\ee
where the inequality holds by Inequality \eqref{eq:TV16} and Lemma \ref{thm:TV2}. Therefore, by Lemma \ref{LemmaOdeComp}, we have  
\be\label{eq:TV15}
\left| \det\big(J(\mathcal{Q}^{\mathbf{q}}_T)\big|_{\mathbf{p}}\big) \right | \leq \left | \det \left(J(\mathcal{Q}^{\mathbf{q}}_T)\big|_{\mathbf{p}^\star} \right)\right | \cdot e^{\|\mathbf{p}^\star - \mathbf{p}\|   \times d\times 
\frac{1}{30}\frac{M_3}{(M_2)^2 T}}.
\ee
Since Equation \eqref{eq:TV15} also holds if we exchange $\mathbf{p}$ and $\mathbf{p}^\star$, it must be true that

\[ e^{-\|\mathbf{p}^\star - \mathbf{p}\|   \times \frac{d}{30}\frac{M_3}{(M_2)^2 T}}  \leq \left| \frac{\det\big(J(\mathcal{Q}^{\mathbf{q}}_T)\big|_{\mathbf{p}}\big)}{ \det \big(J(\mathcal{Q}^{\mathbf{q}}_T)\big|_{\mathbf{p}^\star}\big)} \right | \leq e^{\|\mathbf{p}^\star - \mathbf{p}\|   \times  \frac{d}{30}\frac{M_3}{(M_2)^2 T}}.\]

By the same reasoning as above, applying Equation \eqref{eq:TV17} in place of Equation \eqref{eq:TV16}, we get
\[  e^{-\|\mathbf{q}^\star - \mathbf{q}\|   \times  \frac{d}{9}\frac{M_3}{(M_2)^{3/2} T}}  \leq \left| \frac{\det\big(J(\mathcal{Q}^{\mathbf{q}}_T)\big|_{\mathbf{p}}\big)}{ \det \big(J(\mathcal{Q}^{\mathbf{q}^\star}_T)\big|_{\mathbf{p}}\big)} \right | \leq e^{\|\mathbf{q}^\star - \mathbf{q}\|   \times  \frac{d}{9}\frac{M_3}{(M_2)^{3/2} T}}.\]
This completes the proof.
\end{proof}

For any $r \geq 0$ we denote by $\mathbb{B}_r$ the Euclidean ball of radius $r$ centered at the origin.  Also, for any two sets $A,B \subseteq \mathbb{R}^d$ denote by $A+B := \{a+b : a\in A, b \in B\}$ the Minkowski sum of $A$ and $B$.

Recalling the definitions of $P_{0}$, $S_{\mathbf{q}}$,  and $S'_{\mathbf{q}}$ from Equation \eqref{EqDefP_0ByP_Star} and the following lines, we have the following containment bound:

\begin{lemma}\label{thm:homotopy2}\label{thm:cube}
For any $\mathbf{q} \in \mathbb{R}^d$, we have

\[S_{\mathbf{q}} \subseteq \mathcal{Q}_T^{\mathbf{q}}(P_0) \subseteq S'_{\mathbf{q}}.  \]
\end{lemma}

\begin{proof}

Define
\be
W^r_{q} &:= \mathcal{Q}^{q}_T(\mathbf{p}^\star) + r\times\left[J\big(\mathcal{Q}^{q}_T\big)\big|_{\mathbf{p}^\star}(P_0 -\mathbf{p}^\star)\right] \quad \quad \forall r\geq 0, \\
\ee
so that $S_{\mathbf{q}} = W^{1-\frac{\epsilon}{d}}_{\mathbf{q}}$ and $S'_{\mathbf{q}} = W^{1+\frac{\epsilon}{d}}_{\mathbf{q}}$.

Now, for every $\mathbf{p} \in \partial P_0$,
\be 
\|\mathcal{Q}_T^{\mathbf{q}}(\mathbf{p})  - \left[\mathcal{Q}^{\mathbf{q}}_T(\mathbf{p}^\star) + J\big(\mathcal{Q}^{\mathbf{q}}_T\big)\big|_{\mathbf{p}^\star} \mathbf{p}\right]\| &=  \bigg \| \int_0^{\|\mathbf{p}- \mathbf{p}^\star \|} \int_0^s D_\frac{\mathbf{p}- \mathbf{p}^\star}{\|\mathbf{p}- \mathbf{p}^\star\|} D_\frac{\mathbf{p}- \mathbf{p}^\star}{\|\mathbf{p}- \mathbf{p}^\star\|} \mathcal{Q}_T^{\mathbf{q}}\big|_{\ell_\alpha(\mathbf{p}^\star, s(\mathbf{p}^\star - \mathbf{p}))} \mathrm{d}\alpha \mathrm{d}s \bigg \|\\
&\leq \int_0^{\|\mathbf{p}- \mathbf{p}^\star \|} \int_0^s \left \| D_\frac{\mathbf{p}- \mathbf{p}^\star}{\|\mathbf{p}- \mathbf{p}^\star\|} D_\frac{\mathbf{p}- \mathbf{p}^\star}{\|\mathbf{p}- \mathbf{p}^\star\|} \mathcal{Q}_T^{\mathbf{q}}\big|_{\ell_\alpha(\mathbf{p}^\star, s(\mathbf{p}^\star - \mathbf{p}))} \right \| \mathrm{d}\alpha \mathrm{d}s \\
&\stackrel{{\scriptsize \textrm{Lemma }}\ref{thm:TV2}}{\leq}  \int_0^{\|\mathbf{p}- \mathbf{p}^\star \|} \int_0^s  \frac{1}{40} \frac{M_3}{(M_2)^2} \mathrm{d}\alpha \mathrm{d}s \\
&= \frac{1}{2}\|\mathbf{p}- \mathbf{p}^\star \|^2  \frac{1}{40} \frac{M_3}{(M_2)^2}\\
&\leq \frac{\epsilon^2}{d^2}\cdot \frac{(M_2)^2 T}{M_3}\\
& \leq \frac{T\epsilon}{10d}, \label{eq:h6}
\ee  
where the first inequality holds since $\|\mathbf{p}^\star - \mathbf{p}\| \leq \frac{\epsilon}{d} \cdot \frac{(M_2)^2 T}{M_3}$ for all $\mathbf{p}\in P_0$, and the last inequality holds since $\epsilon \leq [10\frac{(M_2)^2}{d M_3}]^{-1}$.

Therefore, Equation \eqref{eq:h6} and Lemma \ref{thm:homotopy1} together imply that

\be \label{eq:h7}
W_{\mathbf{q}}^1\backslash [(W_{\mathbf{q}}^1)^c + \mathbb{B}_{\frac{T\epsilon}{10d}}] \subseteq \mathcal{Q}_T^{\mathbf{q}}(P_0) \subseteq W_{\mathbf{q}}^1 + \mathbb{B}_{\frac{T\epsilon}{10d}}.
\ee

But by Inequality \eqref{thm:TV2} of Lemma \ref{thm:TV2},
\be \label{eq:h8}
W_{\mathbf{q}}^{1-\frac{\epsilon}{d}} &\subseteq W_{\mathbf{q}}^1\backslash [(W_{\mathbf{q}}^1)^c + \mathbb{B}_{\frac{T\epsilon}{10d}}],\\
W_{\mathbf{q}}^{1+\frac{\epsilon}{d}} &\supseteq W_{\mathbf{q}}^1 + \mathbb{B}_{\frac{T\epsilon}{10d}}.
\ee

Therefore Equations \eqref{eq:h7} and \eqref{eq:h8} together imply that
\be \label{eq:h5}
 W^{1-\frac{\epsilon}{d}}_{\mathbf{q}} &\subseteq \mathcal{Q}_T^{\mathbf{q}}(P_0) \subseteq  W^{1+\frac{\epsilon}{d}}_{\mathbf{q}}.\\ 
\ee

But $S_{\mathbf{q}} = W^{1-\frac{\epsilon}{d}}_{\mathbf{q}}$ and $S'_{\mathbf{q}} = W^{1+\frac{\epsilon}{d}}_{\mathbf{q}}$, so Equation \ref{eq:h5} implies that $S_{\mathbf{q}} \subseteq \mathcal{Q}_T^{\mathbf{q}}(P_0) \subseteq S'_{\mathbf{q}}$.
\end{proof}

Recall that $\mathbf{q}^{(1)}, \mathbf{q}^{(2)}, T, \epsilon$ and $\epsilon'$ are defined in the statement of Lemma \ref{LemmaMinStrongLog}. For the rest of this section, denote by $\mathrm{Vol}(\Lambda)$ the $h$-dimensional Lebesgue measure of a set, where $h$ is the Hausdorff dimension of $\Lambda$.

\begin{lemma}\label{thm:overlap}
With notation as above,

\[\mathrm{Vol}\left(\mathcal{Q}^{\mathbf{q}^{(1)}}_T(P_0) \cap \mathcal{Q}^{\mathbf{q}^{(2)}}_T(P_0)\right) \geq (1-\epsilon')\times \mathrm{Vol}(S_{\mathbf{q}^{(1)}}).\]
\end{lemma}
\begin{proof}

To lighten notation, set $A := S_{\mathbf{q}^{(1)}}$ and $B:= S_{\mathbf{q}^{(2)}}$. 
By Theorem \ref{ThmContractionConvexMainResult}, we have 
\be \label{eq:parallelipiped1}
\|\mathcal{Q}^{\mathbf{q}^{(2)}}_T(p) - \mathcal{Q}^{\mathbf{q}^{(1)}}_T(p)\| \leq \|\mathbf{q}^{(2)} - \mathbf{q}^{(1)}\| \leq \epsilon  \quad \quad \forall p\in \mathbb{R}^{d}.
\ee 
Moreover, for all $q \in \mathbb{R}^d$
we have 
\be\label{eq:parallelipiped2}
\frac{\mathrm{Vol}(\partial S_q)}{\mathrm{Vol}( S_q)} \leq \frac{\sigma_{\mathrm{max}}(J\big(\mathcal{Q}^{q}_T\big)\big|_{\mathbf{p}^\star})}{\sigma_{\mathrm{min}}(J\big(\mathcal{Q}^{q}_T\big)\big|_{\mathbf{p}^\star})} \stackrel{{\scriptsize \textrm{Lemma }}\ref{thm:TV2}}{\leq} \frac{20}{9T\sqrt{M_2}},
\ee
where the first inequality holds because $S_{q} = \mathcal{Q}^{q}_T(\mathbf{p}^\star) + (1-\frac{\epsilon}{d})\cdot\left[J\big(\mathcal{Q}^{q}_T\big)\big|_{\mathbf{p}^\star}(P_0 -\mathbf{p}^\star)\right]$ is a parellelopiped.

Applying Inequality \eqref{eq:parallelipiped1} and then Lemma \ref{thm:Jacobian}, we have for every $x\in B$, 

\be
\min_{ y \in A} \|x-y\| &\leq \|\mathcal{Q}^{\mathbf{q}^{(2)}}_T(\mathbf{p}^\star) - \mathcal{Q}^{\mathbf{q}^{(1)}}_T(\mathbf{p}^\star)\| + (1-\frac{\epsilon}{d})\sup_{\omega \in (P_0 - p_\star)}\left \| \left(J(\mathcal{Q}^{\mathbf{q}^{(2)}}_T)\big|_{\mathbf{p}^\star}\right)\omega - \left(J(\mathcal{Q}^{\mathbf{q}^{(1)}}_T)\big|_{\mathbf{p}^\star}\right)\omega \right \|\\
&\leq \epsilon + \sup_{\omega \in (P_0 - p_\star)} (1-\frac{\epsilon}{d})\frac{1}{10}\frac{M_3}{(M_2)^{\frac{3}{2}}}\cdot \|\omega\|\cdot \|\mathbf{q}^{(2)} - \mathbf{q}^{(1)}\| \\
&\leq \epsilon + (1-\frac{\epsilon}{d})\frac{1}{10}\frac{M_3}{(M_2)^{\frac{3}{2}}}\cdot   2\sqrt{d}\frac{\epsilon}{d} \cdot \frac{(M_2)^2 T}{M_3}   \cdot \epsilon \\
&\leq 2 \epsilon
\ee
and hence that
\be \label{eq:parallelipiped3}
A \backslash (\partial A + \mathbb{B}_{2\epsilon}) \subset B,
\ee
since $A$ and $B$ are parallelopipeds.
But $S_{q}$ is convex, so 
\be \label{eq:parallelipiped5}
\mathrm{Vol}\big((\partial S_q + \mathbb{B}_{2\epsilon}) \cap S_q\big) \leq 2\epsilon \times \mathrm{Vol}(\partial S_q) \stackrel{{\scriptsize \textrm{Eq. }}\ref{eq:parallelipiped2}}{\leq} 2\epsilon \times \frac{20}{9T\sqrt{M_2}} \times \mathrm{Vol}(S_q) \quad \quad \forall q\in \mathbb{R}^d.
\ee

Therefore, 
\be \label{eq:parallelipiped4}
\mathrm{Vol}(A\cap B) &\stackrel{{\scriptsize \textrm{Eq. }}\ref{eq:parallelipiped3}}{\geq} \mathrm{Vol}(A \backslash (\partial A + \mathbb{B}_{2\epsilon}))\\
&= \mathrm{Vol}(A)- \mathrm{Vol}\big((\partial A + \mathbb{B}_{2\epsilon}) \cap A\big)\\
&\stackrel{{\scriptsize \textrm{Eq. }}\ref{eq:parallelipiped5}}{\geq} \mathrm{Vol}(A) - 2\epsilon \times \frac{20}{9T\sqrt{M_2}}\times \mathrm{Vol}(A)\\
&\geq (1-\epsilon') \times \mathrm{Vol}(A).
\ee

By Lemma \ref{thm:cube}, $A \subseteq \mathcal{Q}^{\mathbf{q}^{(1)}}_T(P_0)$ and $B \subseteq \mathcal{Q}^{\mathbf{q}^{(2)}}_T(P_0)$, and so $A \cap B \subseteq \mathcal{Q}^{\mathbf{q}^{(1)}}_T(P_0) \cap \mathcal{Q}^{\mathbf{q}^{(2)}}_T(P_0)$.  Therefore,

\[\mathrm{Vol}\left(\mathcal{Q}^{\mathbf{q}^{(1)}}_T(P_0)) \cap \mathcal{Q}^{\mathbf{q}^{(2)}}_T(P_0))\right) \geq \mathrm{Vol}(A\cap B)\]
\[\stackrel{{\scriptsize \textrm{Eq. }}\ref{eq:parallelipiped4}}{\geq} (1-\epsilon')\times \mathrm{Vol}(A).\]
\end{proof}

\begin{lemma}\label{Thm:cube_overlap}
Set notation as above. Let $\mathfrak{p}_0$ be uniformly distributed (with respect to the Lebesgue measure) on $P_0$.  Then

\be\label{eq:TV18}
\| \mathcal{L}(\mathcal{Q}^{\mathbf{q}^{(2)}}_T(\mathfrak{p}_0)) - \mathcal{L}(\mathcal{Q}^{\mathbf{q}^{(1)}}_T(\mathfrak{p}_0)) \|_{\mathrm{TV}} \leq  1- (1-\epsilon')^{4}.
\ee

\end{lemma}

\begin{proof}

Since $\|\mathbf{p}^\star - \mathbf{p}\| \leq \frac{\epsilon}{d} \cdot \frac{(M_2)^2 T}{M_3}$ for all $\mathbf{p}\in P_0$, Inequality \eqref{eq:det1} of Lemma \ref{thm:determinant} gives:

\be\label{eq:TV20}
 1-\epsilon< e^{-\epsilon} \leq \left| \frac{\det\left(J\big(\mathcal{Q}^{\mathbf{q}^{(i)}}_T)\big|_{\mathbf{p}}\right)}{ \det \left(J\big(\mathcal{Q}^{\mathbf{q}^{(i)}}_T)\big|_{\mathbf{p}^\star} \right)} \right | \leq e^{\epsilon} \leq \frac{1}{1-\epsilon}
 \ee
for $i \in \{1,2\}$ and $\mathbf{p} \in P_{0}$.

Since $\|\mathbf{q}^{(2)} - \mathbf{q}^{(1)}\| \leq  \epsilon' \times (d\frac{M_3}{(M_2)^{3/2} T})^{-1}$,  Inequality \eqref{eq:det2} of Lemma \ref{thm:determinant} gives
\be\label{eq:TV19}
1-\epsilon'< e^{-\epsilon'} \leq \left| \frac{\det\left(J\big(\mathcal{Q}^{\mathbf{q}^{(2)}}_T)\big|_{\mathbf{p}}\right)}{ \det \left(J\big(\mathcal{Q}^{\mathbf{q}^{(1)}}_T)\big|_{\mathbf{p}} \right)} \right |  \leq e^{\epsilon'} \leq \frac{1}{1-\epsilon'}.
\ee

Let $\rho_i$ be the density of $\mathcal{L}(\mathcal{Q}^{\mathbf{q}^{(i)}}_T(\mathfrak{p}_0))$ for $i \in\{1,2\}$.  Then Inequality \eqref{eq:TV20} implies that 

\be\label{eq:TV21} 
\rho_1(q) \geq (1-\epsilon')\cdot \left|\det \left(J\big(\mathcal{Q}^{\mathbf{q}^{(1)}}_T)\big|_{\mathbf{p}^\star} \right)\right|^{-1} \cdot \frac{1}{\mathrm{Vol}(P_0)}
\ee
for all $q \in \mathcal{Q}^{\mathbf{q}^{(1)}}_T(P_0))$.
Furthermore, Inequalities \eqref{eq:TV20} and \eqref{eq:TV19} together imply that
\be\label{eq:TV22}
\rho_2(q) \geq (1-\epsilon)\times (1-\epsilon')\cdot \left|\det \left(J\big(\mathcal{Q}^{\mathbf{q}^{(1)}}_T)\big|_{\mathbf{p}^\star} \right)\right|^{-1} \cdot \frac{1}{\mathrm{Vol}(P_0)}
\ee
for all $q \in \mathcal{Q}^{\mathbf{q}^{(2)}}_T(P_0))$.

Since $S_{\mathbf{q}^{(1)}} = \mathcal{Q}^{\mathbf{q}^{(1)}}_T(\mathbf{p}^\star) + (1-\frac{\epsilon}{d})\cdot\left[J\big(\mathcal{Q}^{\mathbf{q}^{(1)}}_T\big)\big|_{\mathbf{p}^\star}(P_0 -\mathbf{p}^\star)\right]$, we have the formula
\be\label{eq:TV23}
\mathrm{Vol}(S_{\mathbf{q}^{(1)}}) = (1-\frac{\epsilon}{d})^d\left|\det \left(J\big(\mathcal{Q}^{\mathbf{q}^{(1)}}_T\big)\big|_{\mathbf{p}^\star} \right)\right| \cdot \mathrm{Vol}(P_0)\\
\geq (1-\epsilon)\left|\det \left(J\big(\mathcal{Q}^{\mathbf{q}^{(1)}}_T\big)\big|_{\mathbf{p}^\star} \right)\right| \cdot \mathrm{Vol}(P_0).
\ee

Since the random variable $\mathfrak{p}_0$ is uniformly distributed on $P_0$, we have 
\be \label{eq:TV25}
\| \mathcal{L}(&\mathcal{Q}^{\mathbf{q}^{(2)}}_T(\mathfrak{p}_0)) - \mathcal{L}( \mathcal{Q}^{\mathbf{q}^{(1)}}_T(\mathfrak{p}_0)) \|_{\mathrm{TV}} \leq 1 - \int_{\mathbb{R}^d} \min( \rho_1(q), \rho_2(q)) \mathrm{d}q  \\
& = 1- \int_{\mathcal{Q}^{\mathbf{q}^{(1)}}_T(P_0)) \cap \mathcal{Q}^{\mathbf{q}^{(2)}}_T(P_0))} \min( \rho_1(q), \rho_2(q)) \mathrm{d}q \\
& \stackrel{{\scriptsize \textrm{Eq. }}\ref{eq:TV21},\, \ref{eq:TV22}}{\leq}  1- \int_{\mathcal{Q}^{\mathbf{q}^{(1)}}_T(P_0)) \cap \mathcal{Q}^{\mathbf{q}^{(2)}}_T(P_0))}  (1-\epsilon')^{2} \cdot \left|\det \left(J\big(\mathcal{Q}^{\mathbf{q}^{(1)}}_T\big)\big|_{\mathbf{p}^\star} \right)\right|^{-1} \cdot \frac{1}{\mathrm{Vol}(P_0)} \mathrm{d}q \\
& = 1- (1-\epsilon')^{2}\cdot \left|\det \left(J\big(\mathcal{Q}^{\mathbf{q}^{(1)}}_T\big)\big|_{\mathbf{p}^\star} \right)\right|^{-1} \cdot \frac{1}{\mathrm{Vol}(P_0)}  \cdot \mathrm{Vol}\left(\mathcal{Q}^{\mathbf{q}^{(1)}}_T(P_0)) \cap \mathcal{Q}^{\mathbf{q}^{(2)}}_T(P_0))\right) \\
& \stackrel{{\scriptsize \textrm{Lemma }}\ref{thm:overlap}}{\leq} 1- (1-\epsilon')^{2}\cdot \left|\det \left(J\big(\mathcal{Q}^{\mathbf{q}^{(1)}}_T\big)\big|_{\mathbf{p}^\star} \right)\right|^{-1} \cdot \frac{1}{\mathrm{Vol}(P_0)}  \cdot (1-\epsilon')\cdot \mathrm{Vol}(S_{\mathbf{q}^{(1)}}) \\
& \stackrel{{\scriptsize \textrm{Eq. }}\ref{eq:TV23}}{\leq}   1- (1-\epsilon)\cdot (1-\epsilon')^{3}\cdot \left|\det \left(J\big(\mathcal{Q}^{\mathbf{q}^{(1)}}_T\big)\big|_{\mathbf{p}^\star} \right)\right|^{-1} \cdot \frac{1}{\mathrm{Vol}(P_0)}  \cdot  \left|\det \left(J\big(\mathcal{Q}^{\mathbf{q}^{(1)}}_T\big)\big|_{\mathbf{p}^\star} \right)\right| \cdot \mathrm{Vol}(P_0)\\
&\leq  1- (1-\epsilon')^{4}.\\
\ee
This completes the proof of the lemma.
\end{proof}

Finally, we prove Lemma \ref{LemmaKernelContinuityMainLemma}: 

\begin{proof}[Proof of Lemma \ref{LemmaKernelContinuityMainLemma}]

Let the random variable $\mathfrak{p}_0(x)$ be distributed according to the uniform distribution on $C_x^\delta$, and define the random variable $\mathfrak{p}_0$ by $\mathfrak{p}_0 := \mathfrak{p}_0(\mathsf{p}^\star_0)$.

Recall that $\phi(x)$ is the density of the spherical Gaussian with identity covariance matrix, and note that $\|\nabla \phi(x) \| = (2\pi)^{-\frac{d}{2}} \|x\| e^{-\frac{1}{2}\|x\|^2}$. 
Therefore, for every $x\in \mathbb{R}^d$ we have: 

\be \label{eq:lattice1}
| \min_{y\in C_x^{2\delta}} \phi(y) - \max_{y\in C_x^{2\delta}} \phi(y)| &\leq 4 \delta \sqrt{d} \max_{y\in C_x^{2\delta}} \|\nabla \phi(y) \|\\
&\leq\begin{cases} 4 \delta \sqrt{d}\times  (2\pi)^{-\frac{d}{2}} (\|x\| +  4 \delta \sqrt{d}) e^{-\frac{1}{2}(\|x\|-  4 \delta \sqrt{d})^2} \quad \quad    \|x\| \geq  4 \delta \sqrt{d}\\ 4 \delta \sqrt{d}\times  (2\pi)^{-\frac{d}{2}} (\|x\| +  4 \delta \sqrt{d})  \quad \quad \quad \quad   \, \quad \quad  \quad \, \, \, \, 0\leq \|x\| <  4 \delta \sqrt{d}.
\end{cases}\\
\ee

Hence,

\be 
\mathbb{E}\bigg[ \big \| & \mathcal{L}(\overline{\mathfrak{p}_0} | \, \, \mathsf{p}^\star_0)-\mathcal{L}(\mathfrak{p}_0  | \, \, \mathsf{p}^\star_0) \big \|_{\mathrm{TV}} \bigg]\\
&\leq \sum_{z\in L_\delta}  \frac{| \min_{y\in C_z^\delta} \phi(y) - \max_{y\in C_z^\delta} \phi(y)|}{\Phi(C_z^\delta)} \times \Phi(C_z^\delta)\\
&\leq \int_{\mathbb{R}^d}  | \min_{y\in C_x^{2\delta}} \phi(y) - \max_{y\in C_x^{2\delta}} \phi(y)| \mathrm{d}x\\
&\leq \int_{\mathbb{R}^d}  4 \delta \sqrt{d}\times  (2\pi)^{-\frac{d}{2}} (\|x\| +  4 \delta \sqrt{d}) e^{-\frac{1}{2}(\|x\|-  4 \delta \sqrt{d})^2} \mathrm{d}x\\
& \stackrel{{\scriptsize \textrm{Eq. }}\ref{eq:lattice1}}{\leq}  \int_{0\leq \|x\| <  4 \delta \sqrt{d}}  4 \delta \sqrt{d}\times  (2\pi)^{-\frac{d}{2}} (\|x\| +  4 \delta \sqrt{d}) \mathrm{d}x\\
&\quad \quad+ \int_{\|x\| \geq  4 \delta \sqrt{d}}  4 \delta \sqrt{d}\times  (2\pi)^{-\frac{d}{2}} (\|x\| +  4 \delta \sqrt{d}) e^{-\frac{1}{2}(\|x\|-  4 \delta \sqrt{d})^2} \mathrm{d}x\\
&=  \int_0^{4 \delta \sqrt{d}}  4 \delta \sqrt{d}\times  (2\pi)^{-\frac{d}{2}} (r +  4 \delta \sqrt{d}) \times \mathrm{Vol}(r\mathbb{S}^{d-1}) \mathrm{d}r\\
&\quad \quad+ \int_{4 \delta \sqrt{d}}^\infty  4 \delta \sqrt{d}\times  (2\pi)^{-\frac{d}{2}} (r +  4 \delta \sqrt{d}) e^{-\frac{1}{2}(r-  4 \delta \sqrt{d})^2} \times \mathrm{Vol}(r\mathbb{S}^{d-1})  \mathrm{d}x\\
&= 4 \delta \sqrt{d} \mathrm{Vol}(\mathbb{S}^{d-1})  \int_0^{4 \delta \sqrt{d}} r^{d-1} \times (2\pi)^{-\frac{d}{2}} (r +  4 \delta \sqrt{d}) \mathrm{d}r\\
&\quad \quad+ 4 \delta \sqrt{d} \mathrm{Vol}(\mathbb{S}^{d-1})    \int_{4 \delta \sqrt{d}}^\infty r^{d-1}  \times  (2\pi)^{-\frac{d}{2}} (r +  4 \delta \sqrt{d}) e^{-\frac{1}{2}(r-  4 \delta \sqrt{d})^2} \mathrm{d}r\\
&\leq 4 \delta \sqrt{d} \mathrm{Vol}(\mathbb{S}^{d-1})  \times (4 \delta \sqrt{d})^{d-1} \times (2\pi)^{-\frac{d}{2}} (8 \delta \sqrt{d}) \times (4 \delta \sqrt{d})\\
&\quad \quad+ 4 \delta \sqrt{d} \mathrm{Vol}(\mathbb{S}^{d-1})    \int_{4 \delta \sqrt{d}}^\infty r^{d-1}  \times  (2\pi)^{-\frac{d}{2}} (r +  4 \delta \sqrt{d}) e^{-\frac{1}{2}(r-  4 \delta \sqrt{d})^2} \mathrm{d}r\\%
&= 2\mathrm{Vol}(\mathbb{S}^{d-1})  \times (4 \delta \sqrt{d})^{d+2} \times (2\pi)^{-\frac{d}{2}}\\
&\quad \quad+ 4 \delta \sqrt{d} \mathrm{Vol}(\mathbb{S}^{d-1})    \int_0^\infty (z+4 \delta \sqrt{d})^{d-1}  \times  (2\pi)^{-\frac{d}{2}} (z +  8 \delta \sqrt{d}) e^{-\frac{1}{2}z^2} \mathrm{d}z\\
&\leq 2\mathrm{Vol}(\mathbb{S}^{d-1}) (4 \delta \sqrt{d})^{d+2} (2\pi)^{-\frac{d}{2}}+ 4 \delta \sqrt{d} \mathrm{Vol}(\mathbb{S}^{d-1})    \int_0^\infty (z+8 \delta \sqrt{d})^{d}  \times  (2\pi)^{-\frac{d}{2}} e^{-\frac{1}{2}z^2} \mathrm{d}z\\
&=   2\mathrm{Vol}(\mathbb{S}^{d-1}) (4 \delta \sqrt{d})^{d+2} (2\pi)^{-\frac{d}{2}}\\
&\qquad \qquad \qquad \quad  + 4 \delta \sqrt{d} (2\pi)^{-\frac{d}{2}} \mathrm{Vol}(\mathbb{S}^{d-1}) \sqrt{2\pi}  \sum_{k=0}^d  {d \choose k} (8 \delta \sqrt{d})^{k} \int_0^\infty \frac{1}{\sqrt{2\pi}} z^{d-k} e^{-\frac{1}{2}z^2} \mathrm{d}z\\
&= 2 \frac{2\pi^{\frac{d}{2}}}{\Gamma(\frac{d}{2})} (4 \delta \sqrt{d})^{d+2} (2\pi)^{-\frac{d}{2}}+ 4 \delta \sqrt{d} (2\pi)^{-\frac{d}{2}} \frac{2\pi^{\frac{d}{2}}}{\Gamma(\frac{d}{2})} \sqrt{2\pi}  \sum_{k=0}^d  {d \choose k} (8 \delta \sqrt{d})^{k} \times 2^{\frac{d-k}{2}}\frac{\Gamma(\frac{d-k+1}{2})}{\Gamma(\frac{1}{2})}\\
&\leq 2 \frac{2}{\Gamma(\frac{d}{2})} (2\sqrt{2} \delta \sqrt{d})^{d+2}+ 8 \delta \sqrt{d} \sqrt{2\pi}  \sum_{k=0}^d  {d \choose k} (8 \delta \sqrt{d})^{k} \times \frac{\sqrt{d}}{\Gamma(\frac{1}{2})}\\
&\leq 2 \frac{2}{\sqrt{\pi}} (2\sqrt{2} \delta \sqrt{d})^{d+2}+ 8 \delta \sqrt{d} \sqrt{2\pi}  \sum_{k=0}^d  {d \choose k} (8 \delta \sqrt{d})^{k} \times 1^{d-k} \times \frac{\sqrt{d}}{\sqrt{\pi}}\\
&\stackrel{{\scriptsize \textrm{binomial theorem}}}{=} \frac{4}{\sqrt{\pi}} (2\sqrt{2} \delta \sqrt{d})^{d+2}+ 8 \delta d \sqrt{2} (1+8 \delta \sqrt{d})^d \\
\ee
where the second inequality holds since $C_z^\delta \subset C_x^{2\delta}$ for all $x\in C_z^\delta$, the second-last inequality holds since $\frac{\Gamma(\frac{d+1}{2})}{\Gamma(\frac{d}{2})}\leq \sqrt{d}$, and the last inequality holds snce $\Gamma(\frac{d}{2}) \geq \Gamma(\frac{1}{2}) = \sqrt{\pi}$

But 
$\delta \leq \epsilon' \times \min \left([8 d\sqrt{d}]^{-1} \, \,  , \, \, \left[16e\sqrt{2} d\right]^{-1}\right)$, so this bound implies
\be\label{eq:TV24}
\mathbb{E}\bigg[ \big \|\mathcal{L}(\overline{\mathfrak{p}_0} | \, \, \mathsf{p}^\star_0)-\mathcal{L}(\mathfrak{p}_0  | \, \, \mathsf{p}^\star_0) \big \|_{\mathrm{TV}} \bigg]  \leq \epsilon'.
\ee

Therefore, 
\be 
&\mathbb{E}\left[\, \,   \big \| \mathcal{L}(\mathcal{Q}^{\mathbf{q}^{(2)}}_T(\overline{\mathfrak{p}_0}) | \, \,  \mathsf{p}^\star_0)  - \mathcal{L}(\mathcal{Q}^{\mathbf{q}^{(1)}}_T(\overline{\mathfrak{p}_0})  | \, \,  \mathsf{p}^\star_0) \big \|_{\mathrm{TV}} \right]\\ &\leq \mathbb{E}\left[\, \,  \big \| \mathcal{L}(\mathcal{Q}^{\mathbf{q}^{(2)}}_T(\mathfrak{p}_0) | \, \,  \mathsf{p}^\star_0) - \mathcal{L}(\mathcal{Q}^{\mathbf{q}^{(1)}}_T(\mathfrak{p}_0)  | \, \,  \mathsf{p}^\star_0) \big \|_{\mathrm{TV}} + \|\mathcal{L}(\overline{\mathfrak{p}_0}  | \, \,  \mathsf{p}^\star_0)-\mathcal{L}(\mathfrak{p}_0  | \, \,  \mathsf{p}^\star_0 )\|_{\mathrm{TV}} \right]\\
&= \mathbb{E}\bigg[\, \,  \big \| \mathcal{L}(\mathcal{Q}^{\mathbf{q}^{(2)}}_T(\mathfrak{p}_0) | \, \,  \mathsf{p}^\star_0) - \mathcal{L}(\mathcal{Q}^{\mathbf{q}^{(1)}}_T(\mathfrak{p}_0)  | \, \,  \mathsf{p}^\star_0) \big \|_{\mathrm{TV}}\bigg] + \mathbb{E}\bigg[\|\mathcal{L}(\overline{\mathfrak{p}_0}  | \, \,  \mathsf{p}^\star_0)-\mathcal{L}(\mathfrak{p}_0  | \, \,  \mathsf{p}^\star_0 )\|_{\mathrm{TV}} \bigg]\\
&\stackrel{{\scriptsize \textrm{Eq. }}\ref{eq:TV24}}{\leq} \mathbb{E}\bigg[\, \,  \big \| \mathcal{L}(\mathcal{Q}^{\mathbf{q}^{(2)}}_T(\mathfrak{p}_0) | \, \,  \mathsf{p}^\star_0) - \mathcal{L}(\mathcal{Q}^{\mathbf{q}^{(1)}}_T(\mathfrak{p}_0)  | \, \,  \mathsf{p}^\star_0) \big \|_{\mathrm{TV}}\bigg] +  {\epsilon'}\\
&\stackrel{{\scriptsize \textrm{Lemma }}\ref{Thm:cube_overlap}}{\leq}  1- (1-{\epsilon'})^{4}  + {\epsilon'}\\
&\leq 5{\epsilon'}.
\ee
This completes the proof of the lemma.
\end{proof}

\subsection{Variant of Standard ``Drift and Minorization" Bound}

In the proofs of  Theorems \ref{ThmMainConcave} and \ref{ThmMainConcaveBulk}, we use the following small modification of the ``drift and minorization" bound in Theorem 5 of \cite{rosenthal1995minorization}:  

\begin{lemma} \label{LemmaDriftMinRep}
Let $P$ be the transition kernel of a Markov chain on state space $\mathbb{R}^{d}$ with stationary measure $\pi$. Assume that there exists measurable $S \subset \mathbb{R}^{2d}$, integer $k_{0} \in \mathbb{N}$, and constant $\epsilon > 0$ so that 
\be \label{IneqStarRep}
\sup_{ (X_{0},Y_{0}) = (x,y) \in S} \| P^{k_{0}}(x,\cdot) - P^{k_{0}}(y,\cdot) \|_{\mathrm{TV}} \leq 1 - \epsilon.
\ee 
Furthermore, assume that there exists a function $h \, : \, \mathbb{R}^{2d} \mapsto [1, \infty)$, a constant $\alpha > 1$, and a transition kernel $Q$ on $\mathbb{R}^{2d}$ with the following properties:
\begin{enumerate}
\item If $(x,y) \in \mathbb{R}^{2d}$ and $(X,Y) \sim Q((x,y),\cdot)$, then
\be 
X \sim P(x,\cdot), \text{ and } Y \sim P(y,\cdot).
\ee 
That is, $Q$ defines a one-step coupling of two Markov chains running according to $P$.
\item For all $(x,y) \notin S$ and $(X,Y) \sim Q((x,y),\cdot)$,
\be \label{IneqLyaRep}
\E[h(X,Y)] \leq \alpha^{-1} \, h(x,y).
\ee 
\end{enumerate} 
Finally, for $(x,y) \in \mathbb{R}^{2d}$, denote by $\{(X_{t},Y_{t})\}_{t \geq 0}$ a Markov chain evolving according to $Q$ and define
\be 
A = \sup_{(x,y) \in S} \E[h(X_{k_{0}},Y_{k_{0}})].
\ee 
Then, for all $j > 0$,
\be 
\| P^{k}(x,\cdot) - \pi(\cdot) \|_{\mathrm{TV}} \leq (1 - \epsilon)^{\lceil \frac{j}{k_{0}} \rceil} + \alpha^{-k + j k_{0} -1} A^{j-1} \E[h(x,Y)],
\ee 
where $Y \sim \pi$.
\end{lemma}

\begin{proof}
This result is very similar to Theorem 5 of \cite{rosenthal1995minorization}. Besides minor changes to notation, it is obtained by making the following substitutions in the statement of Theorem 5 of \cite{rosenthal1995minorization}:

\begin{enumerate}
\item In \cite{rosenthal1995minorization}, the set $S$ is assumed to be of the form $S = R \times R$ for some $R \subset \mathbb{R}^{d}$. We allow any measurable set $S$.
\item We have replaced the uniform minorization condition $(*)$ of \cite{rosenthal1995minorization} with the weaker condition given in Inequality \eqref{IneqStarRep}. 
\item In \cite{rosenthal1995minorization}, the coupling measure $Q((x,y),\cdot)$ is assumed to be the independent coupling $P(x,\cdot) \otimes P(y,\cdot)$.
\end{enumerate}

We assert that this result follows from \textit{exactly} the proof of Theorem 5 (and its preceding lemmas) in \cite{rosenthal1995minorization}, after making the following specific substitutions in the proof to account for the above substitutions:

\begin{enumerate}
\item The product set $R \times R$ should be replaced by $S$ every time it occurs. Similarly, the phrase ``$X^{(n)}$ and $Y^{(n)}$ are both in $R$" should be replaced by the phrase ``$(X^{(n)},Y^{(n)})$ is in $S$" when it occurs in the proof of Theorem 1.
\item References to condition $(*)$ should be replaced by references to Inequality \eqref{IneqStarRep} every time they occur throughout the proof.
\item In the main coupling construction given in Theorem 1 of \cite{rosenthal1995minorization}, the phrase ``if $X^{(n)}$ and $Y^{(n)}$ are not both in $R$, then simply update them independently, according to $P(X^{(n)},\cdot)$ and $P(Y^{(n)},\cdot)$ respectively" should be replaced by the phrase ``if $(X^{(n)}, Y^{(n)}) \notin S$, then simply update them according to $Q((X^{(n)},Y^{(n)}),\cdot)$."
\end{enumerate}

\end{proof}

\begin{remark} \label{RemThm12Rose}
We note that Theorem 12 of \cite{rosenthal1995minorization} remains true with condition $(*)$ replaced by Inequality \eqref{IneqStarRep} everywhere it appears, in both the proof and statement of the result. Since this is a simple change, we do not repeat the statement of that result.

We make one additional comment on this argument. We have asserted here that Theorem 5 of \cite{rosenthal1995minorization} remains true with condition $(*)$ replaced by the strictly weaker condition \eqref{IneqStarRep}. This may seem suspicious to a reader unfamiliar with \cite{rosenthal1995minorization} - why would the author use a stronger assumption than necessary? We point out that several stronger results that occur in other sections of \cite{rosenthal1995minorization} (including the regeneration time bounds in Section 4) \textit{do} require the full strength of $(*)$, and would not be true if it were replaced by condition \eqref{IneqStarRep}. Thus it is useful to concentrate on the condition $(*)$ throughout the paper.
\end{remark}

\section{Proof of Theorem  \ref{ThmMainConcaveBulk}} \label{SecThmMainConcaveBulk}

We prove  Theorem  \ref{ThmMainConcaveBulk}:

\begin{proof}[Proof of Theorem \ref{ThmMainConcaveBulk}]

Define the set
\be 
\mathcal{S} &= \{x \, : \, V(x) \leq \frac{40b}{a} \},
\ee 
the constant
\be 
\delta = \frac{1}{100} \, \min \left( [\frac{40}{9T\sqrt{M_2}}]^{-1} \, \, , \, \, [8 d\sqrt{d}]^{-1} \, \,  , \, \, \left[16e\sqrt{2} d\right]^{-1} \,\, , \, \, [d\frac{M_3}{(M_2)^{3/2} T}]^{-1} \, \, ,\, \, [10\frac{(M_2)^2}{d M_3}]^{-1}\right),
\ee 
and $m$ as in the statement of the theorem. Note that $\delta$ satisfies Equation \eqref{EqThm2Consts}.

Fix $x,y \in \mathcal{S}$. We denote by $\{X_{i}\}_{i \geq 0}$, $\{ Y_{i} \}_{i \geq 0}$ two Markov chains with transition kernel $K$ and starting points $X_{0} = x$, $Y_{0} = y$. We now describe a coupling of these two Markov chains (Figure \ref{fig:coupling}). Recall from Algorithm \ref{DefSimpleHMC} that the chains $\{X_{i}\}_{i \geq 0}$, $\{Y_{i}\}_{i \geq 0}$ can be determined by two i.i.d. sequences $\{{p}_i^{(x)}\}_{i \geq 0}$, $\{{p}_i^{(y)}\}_{i \geq 0}$ with distribution $\Phi_{1}$, so that 
\be 
X_{i+1} &= \mathcal{Q}_{T}^{X_{i}}({p}_i^{(x)}) \\ 
Y_{i+1} &= \mathcal{Q}_{T}^{Y_{i}}({p}_i^{(y)}). \\ 
\ee 
To define a coupling of the Markov chains $\{X_{i}\}_{i \geq 0}$, $\{ Y_{i} \}_{i \geq 0}$, it is enough to define a coupling of the \textit{update sequences} $\{{p}_i^{(x)}\}_{i \geq 0}$, $\{{p}_i^{(y)}\}_{i \geq 0}$. We need some additional notation before defining the full coupling.

For $x',y' \in \mathbb{R}^{d}$, define $K^{(2)}((x',y'),\cdot)$ to be the maximal coupling of the measures $K(x',\cdot)$ and $K(y',\cdot)$ constructed in Theorem 2.12 of \cite{den2012probability}. The details of this coupling are not important for this analysis; we simply state the two properties that are used in this proof:
\begin{enumerate}
\item For fixed $x',y' \in \mathbb{R}^{d}$ and $(X,Y) \sim K^{(2)}((x',y'),\cdot)$, we have:
\be \label{EqCouplingConsOpt}
\P[ X \neq Y ] \leq  \| K(x',\cdot) - K(y',\cdot) \|_{\mathrm{TV}}.
\ee 
This is the content of Theorem 2.12 of \cite{den2012probability}.
\item Since $K(x',\cdot)$ and $K(y',\cdot)$ both have densities, it is clear from the explicit construction in Theorem 2.12 of \cite{den2012probability} that the collection of distributions $\{K^{(2)}((x',y'),\cdot)\}_{(x',y') \in \mathbb{R}^{2d}}$ in fact defines a measurable probability kernel.
\end{enumerate} 

We now define our coupling of $\{{p}_i^{(x)}\}_{i \geq 0}$, $\{{p}_i^{(y)}\}_{i \geq 0}$. Roughly speaking, we set ${p}_i^{(x)} = {p}_i^{(y)}$ for all $i \neq m$, and for $i=m$ use the maximal coupling $K^{(2)}$ described above:

\begin{itemize}
\item For $i < m$, set ${p}_i^{(y)} = {p}_i^{(x)}$.
\item For $i = m$, begin by coupling $X_{m+1}, Y_{m+1}$ conditional on $\{ ({p}_{s}^{(y)}, {p}_{s}^{(x)})\}_{s < m}$ so that
\be 
\P[X_{m+1} \neq Y_{m+1} | X_{m} = x', Y_{m} = y'] \leq \| K(x',\cdot) - K(y',\cdot) \|_{\mathrm{TV}}. 
\ee 
A measurable coupling of $(X_{m+1},Y_{m+1})$ with this property exists by the properties of $K^{(2)}$ given immediately above, and in particular Inequality \eqref{EqCouplingConsOpt}. Since $X_{m+1}, Y_{m+1}$ are measurable functions of $(X_{m},p_{m}^{(x)})$ and $(Y_{m}, p_{m}^{(y)})$ respectively, the standard  disintegration theorem (see \textit{e.g.} Theorem 5.4 of \cite{kallenberg2006foundations}) implies that there exists coupling of the update sequence $(p_{m}^{(x)}, p_{m}^{(y)})$ conditional on $\{ ({p}_i^{(y)}, {p}_i^{(x)})\}_{s < m}$ so that
\be 
\P[\mathcal{Q}_{T}^{X_{m}}({p}_{m}^{(x)}) \neq \mathcal{Q}_{T}^{Y_{m}}({p}_{m}^{(y)}) | X_{m} = x', Y_{m} = y'] \leq  \| K(x',\cdot) - K(y',\cdot) \|_{\mathrm{TV}}. 
\ee 
We use this coupling of $({p}_{m}^{(x)}, {p}_{m}^{(y)})$.
\item For $i > m$, set ${p}_i^{(y)} = {p}_i^{(x)}$.
\end{itemize}

This completes the definition of our coupling. Define the events
\be 
\mathcal{E}_{1} &= \{ \forall \, 0 \leq i \leq m, \, V(X_{i}), \, V(Y_{i}) \leq \frac{4000mb}{a}  \} \\
\mathcal{E}_{2} &= \{ \forall \, 0 \leq i < m, \,  \|{p}_i^{(x)}\|  \leq \sqrt{ 2 \, d \, \log(200m) } \} \\
\mathcal{E} &= \mathcal{E}_{1} \cap \mathcal{E}_{2}.
\ee

Recall that, in the process of generating $\{X_{i},Y_{i}\}_{i=0}^{m}$ in Algorithm \ref{DefSimpleHMC},  Hamilton's Equation \eqref{EqHamiltonEquations} are solved $2m$ times. For $ i \in \{0,1,\ldots,m-1\}$, we define $(q_{t}^{(x,i)}, p_{t}^{(x,i)})_{0 \leq t \leq T}$ to be the solution to Equation \eqref{EqHamiltonEquations} with initial conditions $q_{0}^{(x,i)} = X_{i}$ and $p_{0}^{(x,i)} = {p}_i^{(x)}$; we define $(q_{t}^{(y,i)},p_{t}^{(y,i)})$ analogously.

Note that, on the event $\mathcal{E}$, all trajectories $\{ \{ (q_{t}^{(z,i)},p_{t}^{(z,i)})\}_{0 \leq t \leq T} \}_{\, z \in \{x,y\}, \, i \in \{0,1,\ldots,m-1\}}$ that occur while simulating the chains $\{X_{j}, Y_{j}\}_{j =0}^{m}$ have total energy satisfying
\be 
\sup_{0 \leq t \leq T, z \in \{x,y\}, i \in \{0,1,\ldots,m-1\}} H(q_{t}^{(z,i)},p_{t}^{(z,i)}) \leq \sup_{x \, : \, V(x) \leq \frac{4000mb}{a}} U(x) +  d \, \log(200m), \\
\ee 
where the Hamiltonian $H$ is as defined in Equation \eqref{EqHamilDef}. By Inequality \eqref{MixingBulkContainmentCond1}, this implies
\be 
\{(q_{t}^{(z,i)},p_{t}^{(z,i)})\}_{0 \leq t \leq T, \, z \in \{x,y\}, \, i \in \{0,1,\ldots,m-1\}} \subset \mathcal{X}.
\ee 
In particular, on the event $\mathcal{E}$, we can apply both Theorem \ref{ThmContractionConvexMainResult} and Lemma \ref{LemmaMinStrongLog}. Applying  Theorem \ref{ThmContractionConvexMainResult}, 
\be 
\E[ \| X_{m} - Y_{m} \| \, \mathbbm{1}_{\mathcal{E}}] &\leq \left(1 - \frac{1}{64} \left( \frac{m_{2}}{M_{2}} \right)^{2} \right)^{m} \, \|X_{0} - Y_{0} \| \\
&\leq \delta.
\ee 

Applying Lemma \ref{LemmaMinStrongLog}, 

\be 
\P[ \{ X_{m+1} \neq Y_{m+1} \} \cap \mathcal{E}] \leq \frac{1}{20}.
\ee 
So, we have 
\be 
\P[X_{m+1} \neq Y_{m+1}] \leq \frac{1}{20} + \P[\mathcal{E}^{c}]. 
\ee 
Applying Markov's inequality and then the drift condition given in Assumption \ref{AssumptionsDrift}, as well as the standard Gaussian tail bound,
\be \label{IneqConvexInBulk}
\P[X_{m+1} \neq Y_{m+1}] &\leq \frac{1}{20} + \P[\mathcal{E}^{c}] \\
&\leq \frac{1}{20} + \sum_{i=0}^{m-1} (\P[V(X_{i})> \frac{4000mb}{a}] + \P[V(Y_{i}) > \frac{4000mb}{a}]) + \sum_{i=0}^{m-1} \P[ \|{p}_i^{(x)}\| > \sqrt{ 2 \, d \, \log(200m)}] \\
&\leq \frac{1}{20} + \frac{(m+1)(a)}{4000 mb} \max_{0 \leq i \leq m} (\E[V(X_{i})] + \E[V(Y_{i})]) + 2m \, e^{-\frac{2d \log(200m)}{2d}} \\
&\leq \frac{1}{20} + \frac{(m+1)(a)}{4000 mb} (2 \frac{41 b}{a}) + \frac{1}{100} \\
&\leq \frac{1}{20} + \frac{1}{25} + \frac{1}{100} = \frac{1}{10}.
\ee

We now apply Theorem 12 of \cite{rosenthal1995minorization} to the thinned chain $\{X_{(m+1)i}\}_{i \geq 0}$, with function and constants
\be 
V = V , \, b = b , \, \lambda = 1 - a, \, d = \frac{40b}{a}, \,r = \frac{\log(1 + \frac{a}{3})}{\log(\frac{80b}{a^{2}})}.
\ee 
By Inequality \eqref{IneqConvexInBulk} and Remark \ref{RemThm12Rose}, we can take $\epsilon = \frac{1}{10}$.

In the notation of Theorem 12 of \cite{rosenthal1995minorization}, this results in the auxillary constants 
\be 
\alpha^{-1} &= \frac{1 + 2b + \lambda d}{1 + d} = 1 - \frac{38b}{1 + \frac{40b}{a}} \leq 1 - \frac{38}{41} a \\
A &= 1 + 2(\lambda d + b) \leq 1 + \frac{80b}{a}.
\ee 
Applying Theorem 12 of \cite{rosenthal1995minorization} (and Remark \ref{RemThm12Rose}) with these constants gives the bound 
\be 
\| \mathcal{L}(X_{(m+1)t}) - \pi \|_{\mathrm{TV}} &\leq (1 - \epsilon)^{rt} + (\alpha^{-(1-r)} A^{r})^{t} (1 + \frac{b}{a} + V(x)) \\
&\leq (\frac{9}{10})^{\frac{\log(1 + \frac{a}{3})}{\log(\frac{80b}{a^{2}})} \, t} + (1 - \frac{1}{3}a)^{t}  (1 + \frac{b}{a} + V(x)).
\ee 

This completes the proof of Inequality \eqref{MixingLogConcaveBulkMainConc1}. Inequality \eqref{SpectralLogConcaveBulkMainConc1} follows immediately from  Inequality \eqref{MixingLogConcaveBulkMainConc1} and an application of Theorem 2.1 of \cite{roberts1997geometric}.

\end{proof}

\section{Euler Integrators}

In this appendix, we analyze the first-order approximation to an HMC step defined in Algorithm \ref{alg:Approx_integrator} with oracle given by Algorithm \ref{alg:Integrator_first_order} and $k=1$. We use notation from those algorithms throughout this appendix. We also assume that Assumption \ref{AssumptionsConvexity} holds with $\mathcal{X} = \mathbb{R}^{d}$ throughout this appendix.

\begin{lemma} (first-order Euler method error) \label{thm:oracle}

If $\bigstar$ is the first-order Euler method of Algorithm \ref{alg:Integrator_first_order}, i.e.,
$q^\star_{\theta}(\mathbf{q},\mathbf{p}) = \mathbf{q}+\mathbf{p}\theta$ and  $p^\star_{\theta}(\mathbf{q},\mathbf{p}) = \mathbf{p} + \theta U'(\mathbf{q})$,
then
\be \label{LemmaOracle1} 
\|q^\star_{\theta}(\mathbf{q},\mathbf{p}) - q_\theta(\mathbf{q},\mathbf{p}) \| \leq \frac{1}{2}\theta^2 \frac{M_2}{\sqrt{m_2}} \sqrt{U(\mathbf{q}) + \frac{1}{2}\|\mathbf{p}\|^2}
\ee
and
\be \label{LemmaOracle2}
\|p^\star_{\theta}(\mathbf{q},\mathbf{p}) - p_\theta(\mathbf{q},\mathbf{p}) \| \leq \frac{1}{\sqrt{2}}\theta^2 M_2 \sqrt{U(\mathbf{q}) + \frac{1}{2}\|\mathbf{p}\|^2}.
\ee
\end{lemma}

\begin{proof}
For all $t \geq 0$, we have by conservation of energy that
\be 
U(q_{t}(\mathbf{q},\mathbf{p})) + \frac{1}{2} \| p_{t}(\mathbf{q},\mathbf{p}) \|^{2} = U(\mathbf{q}) + \frac{1}{2}\|\mathbf{p}\|^2
\ee 
and so we also have
\be \label{IneqConsEnergyDistCons}
\|q_t(\mathbf{q},\mathbf{p})\| \leq \frac{1}{\sqrt{m_2}}\sqrt{U(\mathbf{q}) + \frac{1}{2}\|\mathbf{p}\|^2}.
\ee

  This implies 

\be \| U'(q_t(\mathbf{q},\mathbf{p}))\| &\leq \bigg \| \int_0^{\|q_t(\mathbf{q},\mathbf{p})\|} D_{\frac{q_t(\mathbf{q},\mathbf{p})}{\|q_t(\mathbf{q},\mathbf{p})\|}}  U' \big|_{\ell_s(0, q_t(\mathbf{q},\mathbf{p}))} \mathrm{d}s \bigg \|\\
&\leq \int_0^{\|q_t(\mathbf{q},\mathbf{p})\|} M_2 \mathrm{d}s \\
&= M_2 \|q_t(\mathbf{q},\mathbf{p})\| \\
&\stackrel{{\scriptsize \textrm{Eq. }}\eqref{IneqConsEnergyDistCons}}{\leq} \frac{M_2}{\sqrt{m_2}}\sqrt{U(\mathbf{q}) + \frac{1}{2}\|\mathbf{p}\|^2} \quad \quad \forall t\geq 0.  
\ee

Therefore,
\[\|\mathbf{p} - p_t(\mathbf{q},\mathbf{p}) \| = \left \|\int_0^t U'(q_s(\mathbf{q},\mathbf{p})) \mathrm{d}s \right \| \leq \int_0^t \| U'(q_s(\mathbf{q},\mathbf{p}))\| \mathrm{d}s  \leq  t\times \frac{M_2}{\sqrt{m_2}} \sqrt{U(\mathbf{q}) + \frac{1}{2}\|\mathbf{p}\|^2},\]

for all $t\geq0$, and so
\be 
\|\mathbf{q}+\mathbf{p}\theta - q_\theta(\mathbf{q},\mathbf{p}) \| &= \left\|\int_0^\theta (\mathbf{p} - p_t(\mathbf{q},\mathbf{p})) \mathrm{d}t \right\|\\
&\leq \int_0^\theta \|\mathbf{p} - p_t(\mathbf{q},\mathbf{p}) \|\mathrm{d}t\\ 
&\leq \int_0^\theta t\times \frac{M_2}{\sqrt{m_2}} \sqrt{U(\mathbf{q}) + \frac{1}{2}\|\mathbf{p}\|^2} \mathrm{d}t\\ 
&= \frac{1}{2}\theta^2 \frac{M_2}{\sqrt{m_2}} \sqrt{U(\mathbf{q}) + \frac{1}{2}\|\mathbf{p}\|^2}.
\ee

This completes the proof of Inequality \eqref{LemmaOracle1}.

We now prove Inequality \eqref{LemmaOracle2}.  By the conservation of energy bound, $\|p_{t}(\mathbf{q},\mathbf{p})\| \leq \sqrt{2} \sqrt{U(\mathbf{q}) + \frac{1}{2}\|\mathbf{p}\|^2}$ for all $t\geq 0$,  so we have
\[ \|q_{t}(\mathbf{q},\mathbf{p}) - \mathbf{q}\| \leq t \, \sqrt{2} \sqrt{U(\mathbf{q}) + \frac{1}{2}\|\mathbf{p}\|^2}\]
for all $t \geq 0$. Applying this bound gives
\be 
\| U'(q_{t}(\mathbf{q},\mathbf{p})) - U'(\mathbf{q}) \| &= \left\| \int_{0}^{\|{q_t(\mathbf{q},\mathbf{p})- \mathbf{q}\|}} D_{\frac{q_{t}(\mathbf{q},\mathbf{p}) - \mathbf{q}}{\|q_{t}\mathbf{q},\mathbf{p}) - \mathbf{q}\|}} U' \big|_{\ell_s(\mathbf{q}, q_t(\mathbf{q},\mathbf{p}))} \mathrm{d}s \right \|\\
&\leq \int_{0}^{\|{q_t(\mathbf{q},\mathbf{p})- \mathbf{q}\|}} \left\| D_{\frac{q_{t}(\mathbf{q},\mathbf{p}) - \mathbf{q}}{\|q_{t}\mathbf{q},\mathbf{p}) - \mathbf{q}\|}} U' \big|_{\ell_s(\mathbf{q}, q_t(\mathbf{q},\mathbf{p}))} \right \| \mathrm{d}s\\
&\leq  \int_{0}^{\|{q_t(\mathbf{q},\mathbf{p})- \mathbf{q}\|}} M_2 \mathrm{d}s\\
&= M_2 \|q_t(\mathbf{q},\mathbf{p}) - \mathbf{q}\|\\
&\leq M_2 t \, \sqrt{2} \sqrt{U(\mathbf{q}) + \frac{1}{2}\|\mathbf{p}\|^2}.  
\ee 

Applying this bound to the quantity of interest,

\be 
\|\mathbf{p} + \theta U'(\mathbf{q}) - p_\theta(\mathbf{q},\mathbf{p}) \| &= \left \| \int_0^\theta (U'(q_t(\mathbf{q},\mathbf{p})) - U'(\mathbf{q})) \mathrm{d}t \right \|\\
&\leq \int_0^\theta \| U'(q_t(\mathbf{q},\mathbf{p})) - U'(\mathbf{q}) \| \mathrm{d}t\\
&\leq \int_0^\theta M_2 t \,\sqrt{2} \sqrt{U(\mathbf{q}) + \frac{1}{2}\|\mathbf{p}\|^2} \mathrm{d}t\\
&= \frac{1}{\sqrt{2}}\theta^2 M_2 \sqrt{U(\mathbf{q}) + \frac{1}{2}\|\mathbf{p}\|^2}.
\ee

This completes the proof of the lemma.
\end{proof}

\begin{lemma} \label{thm:Approx_integrator}
Let $(q_T^{\dagger \theta},  p_T^{\dagger \theta})$ be the first-order approximation to an HMC step defined in Algorithm \ref{alg:Approx_integrator}, with $k=1$ and oracle given by Algorithm \ref{alg:Integrator_first_order}.  For $T$, $\theta$ satisfying $7 \theta \leq T \leq \frac{1}{2\sqrt{2}} \frac{\sqrt{m_2}}{M_2}$ and $\frac{T}{\theta} \in \mathbb{N}$, the first-order Euler integrator  $(q_T^{\dagger \theta},  p_T^{\dagger \theta})$ gives error in the position of 
\be \label{eq:f20}
\|q_T^{\dagger \theta}(\mathbf{q},\mathbf{p}) - q_T(\mathbf{q},\mathbf{p})\| \leq 6\theta \times T \times \frac{M_2}{\sqrt{m_2}} \sqrt{H(\mathbf{q},\mathbf{p})}
\ee
and a maximum error in energy conservation of
\be \label{eq:f21}
|H(q_T^{\dagger \theta}(\mathbf{q},\mathbf{p}), p_T^{\dagger \theta}(\mathbf{q},\mathbf{p}))-  H(\mathbf{q},\mathbf{p})| \leq 7 \frac{\theta}{T} H(\mathbf{q},\mathbf{p}).
\ee
\end{lemma}

\begin{figure}[t]
\begin{center}
\includegraphics[trim={0 9cm 1.5cm 7cm}, clip, scale=0.75]{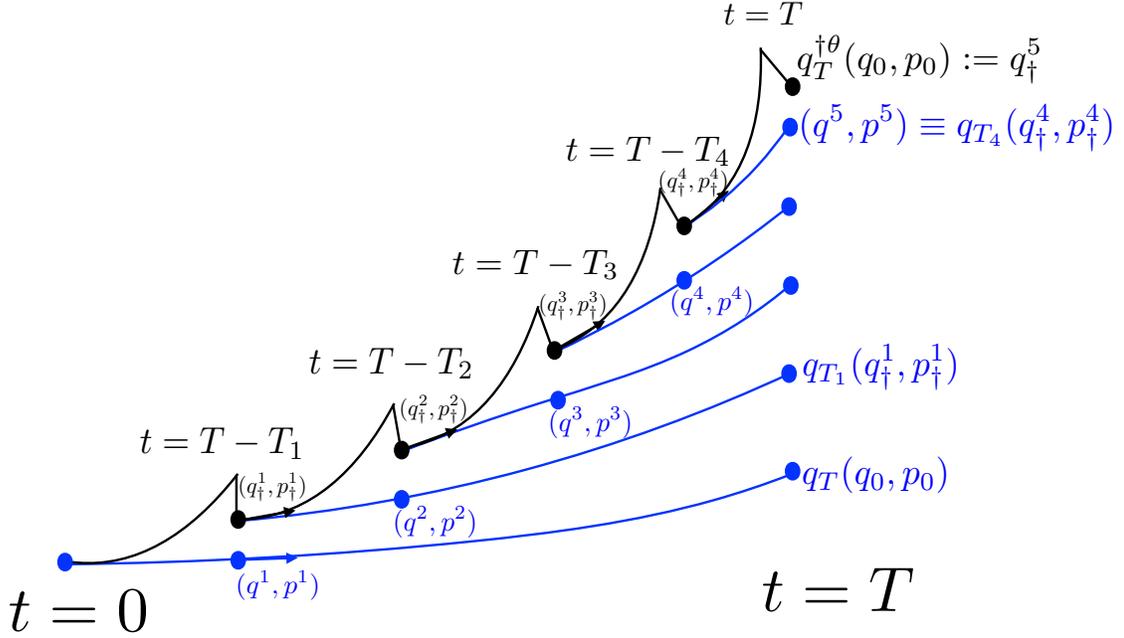}
\end{center}
\caption{This is an illustration of the proof of Lemma \ref{thm:Approx_integrator}.  Steps taken by the oracle $\bigstar$ (which are illustrated here by a saw-toothed path, with each ``tooth" representing a step) are in black.  The true Hamiltonian trajectories are blue curves.  Only the Hamiltonian trajectory $q_t(\mathbf{q},\mathbf{p})$ on the bottom belongs to the idealized HMC Markov chain.  We imagine the other Hamiltonian trajectories to help us bound the error.  The distance between $(q^i,p^i)$ (blue dot, blue arrow) and $(q_\dagger^i,p_\dagger^i)$ (black dot, black arrow) at each time $t=T-T_i$ are bounded because of our assumptions on the accuracy of the oracle $\bigstar$.  The distance between any blue dot at $t=T$ to the blue dot directly below it is bounded using part (1) of Lemma \ref{thm:bounds}.}\label{fig:Approx_integrator}
\end{figure}

\begin{proof}
In this proof we will use the notation of Algorithm \ref{alg:Approx_integrator}.  Define $q^{i+1} := q_{\theta}(q^i_\dagger,p^i_\dagger)$ and $p^{i+1} = p_{\theta}(q^i_\dagger,p^i_\dagger)$ for every $i \in 0,1,\ldots, \frac{T}{\theta}$.  Also set $q^0 = \mathbf{q}$ and $p^0 = \mathbf{p}$.
Set $E := H(\mathbf{q},\mathbf{p}) = U(\mathbf{q}) + \frac{1}{2}\|\mathbf{p}\|^2$.

We will now prove the following claim by induction:  For all $0\leq j \leq  \frac{T}{\theta} $, the inequality
\be \label{eq:inductive_assumption}
|H(q^j_\dagger, p^j_\dagger)-E| \leq j \times 7 \left(\frac{\theta}{T}\right)^2 E
\ee
is satisfied.  The case $i=0$ is obvious, since $H(q^0_\dagger, p^0_\dagger)= H(\mathbf{q},\mathbf{p}) = E$.  We now fix $i$ and assume that Inequality \eqref{eq:inductive_assumption} is satisfied for all $0 \leq j \leq i$; we then show that it is satisfied for $j= i+1$.

\textbf{Inductive assumption:}
Suppose that $|H(q^i_\dagger, p^i_\dagger)-E| \leq i \times 7 \left(\frac{\theta}{T}\right)^2 E$, and that $i \leq \frac{T}{\theta}-1$.\\

Since $i \leq \frac{T}{\theta}-1 \leq \frac{T}{\theta}$ and $\frac{\theta}{T} \leq \frac{1}{7}$, Inequality \eqref{eq:inductive_assumption} implies that
\be \label{eq:f5}
H(q^{i+1}, p^{i+1}) \stackrel{{\scriptsize \textrm{Conservation of Energy }}}{=} H(q^i_\dagger, p^i_\dagger) \leq E+ 7 \frac{\theta}{T} E \leq 2 E.
\ee

Then Lemma \ref{thm:oracle} and Inequality \eqref{eq:f5} together imply that

\be \label{eq:f10}
\|q^{i+1}_{\dagger} - q^{i+1}\| &\leq \theta^2 \frac{M_2}{\sqrt{m_2}} \sqrt{E}
\ee
and
\be \label{eq:f10b}
\|p^{i+1}_{\dagger} - p^{i+1}\| &\leq  \theta^2 M_2 \sqrt{E} .
\ee

But by Equation \ref{eq:f5},
\be \label{eq:f11}
\|p^{i+1}\| \leq 2 \sqrt{E}.
\ee

Therefore, by the triangle inequality, 
\be \label{eq:f15}
\|p^{i+1}_\dagger\|^2 - \|p^{i+1}\|^2 &\leq (\|p^{i+1}_\dagger - p^{i+1}\| + \|p^{i+1}\|)^2 - \|p^{i+1}\|^2\\
& = \|p^{i+1}_\dagger - p^{i+1}\|^2 + 2\|p^{i+1}_\dagger - p^{i+1}\| \times \|p^{i+1}\|\\
& \stackrel{{\scriptsize \textrm{Eq. }}\ref{eq:f11}}{\leq}  \|p^{i+1}_\dagger - p^{i+1}\|^2 + 2\|p^{i+1}_\dagger - p^{i+1}\| \times 2\sqrt{E}\\
& = \|p^{i+1}_\dagger - p^{i+1}\| \times (\|p^{i+1}_\dagger - p^{i+1}\| + 4\sqrt{E}) \\
&\stackrel{{\scriptsize \textrm{Eq. }}\ref{eq:f10b}}{\leq} \theta^2 M_2 \sqrt{E} \times (\theta^2 M_2 \sqrt{E} + 4\sqrt{E})\\
&= E \theta^2 M_2 \times (\theta^2 M_2 + 4).
\ee

Also, Equation \eqref{eq:f5}, together with Assumption \ref{AssumptionsConvexity}, implies that 
\be \label{eq:f12}
\|q^{i+1}\| \leq \frac{1}{\sqrt{m_2}}\sqrt{2E},
\ee
and thus (again by Assumption \ref{AssumptionsConvexity}) that
\be \label{eq:f13}
\|U'(q^{i+1})\| \stackrel{{\scriptsize \textrm{Assumption }}\ref{AssumptionsConvexity}}{\leq}  M_2 \|q^{i+1}\| \stackrel{{\scriptsize \textrm{Eq. }}\ref{eq:f12}}{\leq}  \frac{M_2}{\sqrt{m_2}}\sqrt{2 E}.
\ee

Therefore,
\be \label{eq:f14}
U(q^{i+1}_\dagger) - U(q^{i+1}) &\leq \max_{x \in \textrm{Convex Hull}(\{q^{i+1}, q^{i+1}_\dagger\})} \|U'(x)\| \times \|q^{i+1}_\dagger - q^{i+1}\| \\
& \stackrel{{\scriptsize \textrm{Assumption }}\ref{AssumptionsConvexity}}{\leq} \left(\|U'(q^{i+1})\| + M_2\|q^{i+1}_\dagger - q^{i+1}\|\right) \times \|q^{i+1}_\dagger - q^{i+1}\|\\
&\stackrel{{\scriptsize \textrm{Eq. }}\ref{eq:f13}}{\leq}  \left(\frac{M_2}{\sqrt{m_2}}\sqrt{2 E} + M_2\|q^{i+1}_\dagger - q^{i+1}\|\right) \times \|q^{i+1}_\dagger - q^{i+1}\|\\
&\stackrel{{\scriptsize \textrm{Eq. }}\ref{eq:f10}}{\leq} \left(\frac{M_2}{\sqrt{m_2}}\sqrt{2 E} + M_2\theta^2 \frac{M_2}{\sqrt{m_2}} \sqrt{E}\right) \times \theta^2 \frac{M_2}{\sqrt{m_2}} \sqrt{E}\\
&\leq \left(\sqrt{2} + M_2\theta^2 \right) \times \theta^2 \frac{M_2^2}{m_2} E.
\ee

Hence the total change in energy is bounded by
\be \label{eq:f16}
|H(q^{i+1}_\dagger, p^{i+1}_\dagger) - H(q^i_\dagger, p^i_\dagger)| &\stackrel{{\scriptsize \textrm{Conservation of Energy }}}{=} |H(q^{i+1}_\dagger, p^{i+1}_\dagger) - H(q^{i+1}, p^{i+1})|\\
&\stackrel{{\scriptsize \textrm{Eq. }}\ref{eq:f14}, \, \ref{eq:f15}}{\leq} \left(\sqrt{2} + M_2\theta^2 \right) \times \theta^2 \frac{M_2^2}{m_2} E +  \frac{1}{2}E \theta^2 M_2 \times (\theta^2  M_2 + 4)\\
&\leq  \left(4 + 1.5 M_2\theta^2 \right) \times \theta^2  \frac{M_2^2}{m_2} E\\
&\leq  \left(4 + 1.5 \right) \times \theta^2 \frac{M_2^2}{m_2} E\\
&=  5.5 [\frac{\theta}{T}\times T]^2 \frac{M_2^2}{m_2} E\\
&\leq  5.5 \left(\frac{\theta}{T}\times \frac{1}{2\sqrt{2}} \frac{\sqrt{m_2}}{M_2}\right)^2 \frac{M_2^2}{m_2} E\\
&\leq  7 \left(\frac{\theta}{T}\right)^2 E,
\ee
 where the second inequality is true since $0< \frac{M_2}{m_2}\leq 1$, and the third and fourth inequalities are true since $\theta \leq T \leq \frac{1}{2\sqrt{2}} \frac{\sqrt{m_2}}{M_2} \leq \frac{1}{\sqrt{M_2}}$.

Therefore, 
\be \label{eq:f17}
|H(q^{i+1}_\dagger, p^{i+1}_\dagger) - E| &\stackrel{{\scriptsize \textrm{Eq. }}\ref{eq:f16}}{\leq} |H(q^i_\dagger, p^i_\dagger) - E| + 7 \left(\frac{\theta}{T}\right)^2 E\\
&\stackrel{{\scriptsize \textrm{by inductive assumption}}}{\leq}  i \times 7 \left(\frac{\theta}{T}\right)^2 E + 7 \left(\frac{\theta}{T}\right)^2 E\\
&= (i+1) \times 7 \left(\frac{\theta}{T}\right)^2 E.
\ee
This completes the proof by induction of Inequality \eqref{eq:inductive_assumption} (and in particular this implies Inequality \eqref{eq:f21}).

Therefore, since $\frac{\theta}{T} \leq \frac{1}{7}$,
\begin{equation}\label{eq:f6}
|H(q^i_\dagger, p^i_\dagger)-E| \leq i\times7 \left(\frac{\theta}{T}\right)^2 E  \leq E, \quad \quad \forall \, 0\leq i \leq \frac{T}{\theta}.
\end{equation}

Therefore,  by Lemma \ref{thm:oracle} 
\be \label{eq:f18}
\|q^i_\dagger - q^i \| \leq  \theta^2  \frac{M_2}{\sqrt{m_2}} \sqrt{E} \quad \mathrm{and} \quad \|p^i_\dagger - p^i \| \leq \theta^2  M_2 \sqrt{E} \quad \quad \forall \, 0\leq i \leq \frac{T}{\theta}.
\ee

Define $T_i:= T- \theta \times i$ for all $i$. Therefore, since $T_{i+1} \leq T \leq \frac{1}{2\sqrt{2}} \frac{\sqrt{m_2}}{M_2} \leq \frac{1}{\sqrt{M_2}}$, for all $i\leq\frac{T}{\theta}-1$, we have by Lemma \ref{thm:bounds}

\be \label{eq:f19}
\|q_{T_{i+1}}(q^{i+1}_\dagger, p^{i+1}_\dagger) - q_{T_i}(q^i_\dagger, p^i_\dagger)\| &= \|q_{T_{i+1}}(q^{i+1}_\dagger, p^{i+1}_\dagger) - q_{T_{i+1}}(q^{i+1}, p^{i+1})\|\\
& \stackrel{{\scriptsize \textrm{Eq. }} \ref{thmboundsconc1}{\scriptsize \textrm{ of Lemma }}\ref{thm:bounds}}{\leq} \frac{1}{2}\left(\|q^{i+1}_\dagger - q^{i+1} \| + \frac{\|p^{i+1}_\dagger - p^{i+1} \|}{\sqrt{M_2}}\right)e^{T_{i+1}\sqrt{M_2}}\\
&\qquad \qquad + \frac{1}{2}\left(\|q^{i+1}_\dagger - q^{i+1} \| - \frac{\|p^{i+1}_\dagger - p^{i+1} \|}{\sqrt{M_2}}\right)e^{-T_{i+1}\sqrt{M_2}}\\
& \leq \frac{1}{2}\left(\|q^{i+1}_\dagger - q^{i+1} \| + \frac{\|p^{i+1}_\dagger - p^{i+1} \|}{\sqrt{M_2}}\right)e + \frac{1}{2}\left(\|q^{i+1}_\dagger - q^{i+1} \| - \frac{\|p^{i+1}_\dagger - p^{i+1} \|}{\sqrt{M_2}}\right)e^{0}\\
& \leq \left(\|q^{i+1}_\dagger - q^{i+1} \| + \frac{\|p^{i+1}_\dagger - p^{i+1} \|}{\sqrt{M_2}}\right)e\\
& \stackrel{{\scriptsize \textrm{Eq. }}\ref{eq:f18}}{\leq} \left(\theta^2 \frac{M_2}{\sqrt{m_2}} \sqrt{E} +  \frac{\theta^2 M_2 \sqrt{E}}{\sqrt{M_2}}\right)e \leq 6\theta^2 \frac{M_2}{\sqrt{m_2}} \sqrt{E},
\ee
where the second inequality is true since $0\leq T_{i+1} \leq \frac{1}{\sqrt{M_2}}$ and since the functions $e^t + e^{-t}$ and $e^t - e^{-t}$ are both nondecreasing in $t$ for $t\geq0$;  the fifth inequality is true since $\sqrt{m_2}\leq \sqrt{M_2}$. 

Therefore, since $q^{\frac{T}{\theta}}_\dagger = q_T^{\dagger \theta}(\mathbf{q},\mathbf{p})$,  $T_0=T$, and $(q_\dagger^0,p_\dagger^0) =(\mathbf{q},\mathbf{p})$,  by the triangle inequality (see Figure \ref{fig:Approx_integrator}) we have 
\be
\|q_T^{\dagger \theta}(\mathbf{q},\mathbf{p}) - q_T(\mathbf{q},\mathbf{p})\| &= \|q^{\frac{T}{\theta}}_\dagger - q_{T_0}(q_\dagger^0,p_\dagger^0)\|\\
&\leq \|q^{\frac{T}{\theta}}_\dagger - q_{T_{(\frac{T}{\theta}-1)}}(q^{\frac{T}{\theta}-1}_\dagger, p^{\frac{T}{\theta}-1}_\dagger) \|  + \sum_{i=0}^{\frac{T}{\theta}-2} \|q_{T_{i+1}}(q^{i+1}_\dagger, p^{i+1}_\dagger) - q_{T_i}(q^i_\dagger, p^i_\dagger)\|\\
&= \|q^{\frac{T}{\theta}}_\dagger - q^{\frac{T}{\theta}} \|  + \sum_{i=0}^{\frac{T}{\theta}-2} \|q_{T_{i+1}}(q^{i+1}_\dagger, p^{i+1}_\dagger) - q_{T_i}(q^i_\dagger, p^i_\dagger)\|\\
&\stackrel{{\scriptsize \textrm{Eq. }}\ref{eq:f18}}{\leq} \theta^2  \frac{M_2}{\sqrt{m_2}} \sqrt{E}  + \sum_{i=0}^{\frac{T}{\theta}-2} \|q_{T_{i+1}}(q^{i+1}_\dagger, p^{i+1}_\dagger) - q_{T_i}(q^i_\dagger, p^i_\dagger)\|\\
&\stackrel{{\scriptsize \textrm{Eq. }}\ref{eq:f19}}{\leq}  \theta^2  \frac{M_2}{\sqrt{m_2}} \sqrt{E}  + \sum_{i=0}^{\frac{T}{\theta}-2} 6\theta^2  \frac{M_2}{\sqrt{m_2}} \sqrt{E}\\
&\leq \frac{T}{\theta}\times 6\theta^2  \frac{M_2}{\sqrt{m_2}} \sqrt{E}.\\
\ee
This completes the proof of Inequality \eqref{eq:f20}.
\end{proof}

\section{Higher-Order Integrators} \label{AppHigherOrder}

Throughout this section, fix integers $d, \mathsf{m} \in \mathbb{N}$ with $\frac{d}{\mathsf{m}} \in \mathbb{N}$ and numerical integrator $\sharp$.  In this section we analyze a class of numerical integrators satisfying Condition \ref{condition:kth_order}, and show that the leapfrog integrator given in Algorithms \ref{alg:Integrator_leapfrog} and \ref{alg:Approx_integrator} with $k=2$ is included in this class.

We set some notation. Recall that, for every  $x \in \mathbb{R}^d$ and all $i \in \{1, \ldots \frac{d}{\mathsf{m}}\}$, we define $x^{(i)} := (x[\mathsf{m}(i-1)+1], \ldots, x[\mathsf{m}i])$. We will consider potentials $U \, : \, \mathbb{R}^{d} \mapsto \mathbb{R}$ satisfying the separability assumption \ref{assumption:product_measure_potential}. Recall that, when  Assumption \ref{assumption:product_measure_potential} holds, Hamilton's equations are separable, so we have
\be \label{eq:separabililty_ideal}
q_T(\mathbf{q}, \mathbf{p}) [\mathsf{m}(i-1)+1, \ldots, \mathsf{m}i ] &= q_T^{(i)  }(\mathbf{q}^{(i)}, \mathbf{p}^{(i)})\\
p_T(\mathbf{q}, \mathbf{p})  [\mathsf{m}(i-1)+1, \ldots, \mathsf{m}i ] &= p_T^{(i)  }(\mathbf{q}^{(i)}, \mathbf{p}^{(i)}) \quad \quad \forall i \in \{1, \ldots \frac{d}{\mathsf{m}}\}
\ee
for every $\mathbf{q}, \mathbf{p} \in \mathbb{R}^d$.
Also under Assumption \ref{assumption:product_measure_potential}, by Condition \ref{DefSepCond} we have 
\be \label{eq:separabililty_sharp}
q_T^{\sharp \theta}(\mathbf{q}, \mathbf{p}) [\mathsf{m}(i-1)+1, \ldots, \mathsf{m}i ] &= q_T^{(i) \sharp \theta}(\mathbf{q}^{(i)}, \mathbf{p}^{(i)})\\
p_T^{\sharp \theta}(\mathbf{q}, \mathbf{p})  [\mathsf{m}(i-1)+1, \ldots, \mathsf{m}i ] &= p_T^{(i) \sharp \theta}(\mathbf{q}^{(i)}, \mathbf{p}^{(i)}) \quad \quad \forall i \in \{1, \ldots \frac{d}{\mathsf{m}}\}
\ee

for every $\mathbf{q}, \mathbf{p} \in \mathbb{R}^d$.

We will restrict our attention to numerical integrators satisfying property \ref{condition:kth_order}. As motivation, we note that the second-order leapfrog integrator does satisfy this property:

\begin{lemma} (leapfrog error bounds, Prop. 5.3 and 5.4 of \cite{HMC_optimal_tuning}) \label{thm:leapfrog}
Suppose that $U_i:\mathbb{R}^\mathsf{m} \rightarrow \mathbb{R}^{+}$ satisfies Assumption \ref{assumption:leapfrog} for $\mathsf{a}=4$ and $\mathsf{b} = 3$. Then the leapfrog integrator defined in Algorithms \ref{alg:Integrator_leapfrog} and \ref{alg:Approx_integrator} with $k=2$ satisfies Condition \ref{condition:kth_order} with constant $\mathsf{c} = 2$. 
\end{lemma}

We now proceed to bound the error for the integrator $\sharp$. We begin by collecting the main assumptions of this section into an easy-to-use lemma:

\begin{lemma} \label{thm:leapfrog2}
Let $U$ and $U_1,U_2, \ldots, U_{\frac{d}{\mathsf{m}}}$ satisfy Assumptions \ref{assumption:product_measure_potential}, \ref{assumption:leapfrog} and let $\sharp$ satisfy Condition \ref{condition:kth_order}. Then

\[\|q_T^{\sharp \theta }(\mathbf{q},\mathbf{p}) - q_T(\mathbf{q},\mathbf{p})\| \leq \sqrt{\sum_{i=1}^\frac{d}{\mathsf{m}} \left[ \theta \mathsf{K} ( H_i^{\mathsf{c}}(\mathbf{q}^{(i)},\mathbf{p}^{(i)}) +1)\right]^2} \]
and
\[|H(q_T^{\sharp \theta}(\mathbf{q},\mathbf{p}), p_T^{\sharp \theta}(\mathbf{q},\mathbf{p}))-  H(\mathbf{q},\mathbf{p})| \leq \sum_{i=1}^\frac{d}{\mathsf{m}}  \theta \mathsf{K} (H_i^{2\mathsf{c}}(\mathbf{q}^{(i)},\mathbf{p}^{(i)}) +1) .\]
\end{lemma}
\begin{proof}
This is an immediate consequence of Condition \ref{condition:kth_order}, Assumption \ref{assumption:product_measure_potential}, and Equations \ref{eq:separabililty_ideal} and \ref{eq:separabililty_sharp}.
\end{proof}

In order to analyze Metropolis-adjusted HMC (Algorithm \ref{alg:Metropolis}) with the higher-order integrator $\sharp$, we must first analyze unadjusted HMC (Algorithm \ref{alg:Unadjusted}) with a modified ``toy" integrator $\diamondsuit$.  Towards this end, we will define a family of  ``good sets"
\be  \label{eq:good_set}
\mathsf{G}  = \{(q,p) \in \mathbb{R}^d\times \mathbb{R}^d :  \max_{1\leq i \leq \frac{d}{\mathsf{m}}}\|q^{(i)}\|< \mathsf{g}_\infty, \max_{1\leq i \leq \frac{d}{\mathsf{m}}}\|p^{(i)}\| < \mathsf{g}_\infty, ||p|| > \mathsf{g}_2 \},
\ee
depending on two constants $\mathsf{g}_2, \mathsf{g}_\infty$ that are to be fixed later.  Let $\clubsuit$ be the numerical integrator given by Algorithm \ref{alg:Approx_integrator} with oracle given by the Euler integrator of Algorithm \ref{alg:Integrator_first_order} and $k=1$. We now define the modified ``toy" integrator $\diamondsuit$ for which we will analyze Algorithm \ref{alg:Unadjusted}:
 \begin{defn} \label{defn:toy}
Define the toy integrator $\diamondsuit$ associated with numerical integrator $\sharp$ by the equations
\be \label{eq:toy}
(q^{\diamondsuit \theta}_{T}(q,p), p^{\diamondsuit \theta}_{T}(q,p)) := \begin{cases}(q^{\sharp \theta}_{T}(q,p), p^{\sharp \theta}_{T}(q,p)) \quad \quad \quad \quad  \textrm{      if         } (q,p) \in \mathsf{G}\\
(q^{\clubsuit \theta}_{T}(q,p), p^{\clubsuit \theta}_{T}(q,p))  \quad \quad \quad \, \, \, \textrm{   otherwise.}
\end{cases}
\ee
\end{defn}

\begin{lemma} \label{thm:leapfrog3}
Let $U$ and $U_1,U_2, \ldots, U_{\frac{d}{\mathsf{m}}}$ satisfy Assumptions \ref{assumption:product_measure_potential}, \ref{assumption:leapfrog} and let $\sharp$ satisfy Condition \ref{condition:kth_order}.
Fix  $\mathsf{g}_2, \mathsf{g}_\infty$ and let $(\mathbf{q},\mathbf{p}) \in \mathsf{G}$.  Then there exists a constant $\mathsf{K}'$ that depends only on $\mathsf{B}$, $\mathsf{K}$, $T$, $M_2$, $\mathsf{c}$ and $\mathsf{g}_{\infty}$, and is polynomial in $\mathsf{g}_\infty$, such that
\be \label{IneqAppBLemm41} 
\|q_T^{\sharp \theta }(\mathbf{q},\mathbf{p}) - q_T(\mathbf{q},\mathbf{p})\| \leq \sqrt{\frac{d}{\mathsf{m}}} \times \theta \mathsf{K}' 
\ee 
and
\be \label{IneqAppBLemm42}
|H(q_T^{\sharp \theta}(\mathbf{q},\mathbf{p}), p_T^{\sharp \theta}(\mathbf{q},\mathbf{p}))-  H(\mathbf{q},\mathbf{p})| \leq  \frac{d}{\mathsf{m}} \times  \theta \mathsf{K}'.
\ee
\end{lemma}

\begin{proof}

Applying Lemma \ref{thm:leapfrog2}, and noting that $H_i(\mathbf{q}^{(i)},\mathbf{p}^{(i)}) \leq M_2\|\mathbf{q}^{(i)}\|^2 + \frac{1}{2}\|\mathbf{p}^{(i)}\|^2$ for every $i$ by Assumption \ref{AssumptionsConvexity}, we have 
\be 
\|q_T^{\sharp \theta }(\mathbf{q},\mathbf{p}) - q_T(\mathbf{q},\mathbf{p})\| &\leq \sqrt{\sum_{i=1}^\frac{d}{\mathsf{m}} \left[ \theta \mathsf{K} ( H_i^{\mathsf{c}}(\mathbf{q}^{(i)},\mathbf{p}^{(i)}) +1)\right]^2} \\
&\leq \sqrt{\sum_{i=1}^\frac{d}{\mathsf{m}} \left[ \theta \mathsf{K} (( M_2 \mathsf{g}_{\infty}^2 + \frac{1}{2}\mathsf{g}_{\infty}^2)^{\mathsf{c}} +1)\right]^2} \\
&\leq \sqrt{\frac{d}{\mathsf{m}}} \, \theta \times (\mathsf{K} ( (M_2 \mathsf{g}_{\infty}^2 + \frac{1}{2}\mathsf{g}_{\infty}^2)^{\mathsf{c}} +1)).
\ee 
Thus, Inequality \eqref{IneqAppBLemm41} holds for any $\mathsf{K}' \geq \mathsf{K}_{1} \equiv  (\mathsf{K} ( (M_2 \mathsf{g}_{\infty}^2 + \frac{1}{2}\mathsf{g}_{\infty}^2)^{\mathsf{c}} +1))$, where $\mathsf{K}_{1}$ satisfies the conclusions of the lemma. Applying Lemma \ref{thm:leapfrog2} again, essentially the same calculation shows that Inequality \eqref{IneqAppBLemm42} holds for a similar constant $\mathsf{K}_{2}$. Taking $\mathsf{K}' = \max(\mathsf{K}_{1},\mathsf{K}_{2})$ completes the proof.
\end{proof}

\begin{lemma} \label{thm:leapfrog4}
Let $U$ and $U_1,U_2, \ldots, U_{\frac{d}{\mathsf{m}}}$ satisfy Assumptions \ref{assumption:product_measure_potential}, \ref{assumption:leapfrog}  and let $\sharp$ satisfy Condition \ref{condition:kth_order}.
Fix  $\mathsf{g}_2, \mathsf{g}_\infty$ and let $(\mathbf{q},\mathbf{p}) \in \mathsf{G}$.  Then there exists a constant $\mathsf{K}''$ that depends only on $\mathsf{B}$, $\mathsf{K}$, $T$, $M_2$, $\mathsf{c}$, $\mathsf{m}$ and $\mathsf{g}_{\infty}$, and is polynomial in $\mathsf{g}_\infty$ and $\frac{\sqrt{d}}{\mathsf{g}_2}$,  such that
\[\|q_T^{\sharp \theta }(\mathbf{q},\mathbf{p}) - q_T(\mathbf{q},\mathbf{p})\| \leq \theta \mathsf{K}''  \sqrt{H(\mathbf{q},\mathbf{p})}\]
and
\[|H(q_T^{\sharp \theta}(\mathbf{q},\mathbf{p}), p_T^{\sharp \theta}(\mathbf{q},\mathbf{p}))-  H(\mathbf{q},\mathbf{p})| \leq  \theta \mathsf{K}'' H(\mathbf{q},\mathbf{p}).\]
\end{lemma}

\begin{proof}

Set $\mathsf{K}'' = \max \left( \frac{\sqrt{d}}{\mathsf{g}_2}  \sqrt{\frac{2}{\mathsf{m}}} \mathsf{K}', \, \, \left(\frac{\sqrt{d}}{\mathsf{g}_2}  \sqrt{\frac{2}{\mathsf{m}}}\right)^2 \mathsf{K}'\right)$.  By Lemma \ref{thm:leapfrog3} we have
\be
\|q_T^{\sharp \theta }(\mathbf{q},\mathbf{p}) - q_T(\mathbf{q},\mathbf{p})\| &\stackrel{{\scriptsize \textrm{Lemma }}\ref{thm:leapfrog3}}{\leq} \sqrt{\frac{d}{\mathsf{m}}} \times \theta \mathsf{K}'\\
 &\leq \frac{\|\mathbf{p}\|}{\mathsf{g}_2} \times \sqrt{\frac{d}{\mathsf{m}}} \times \theta \mathsf{K}'\\
 &\leq \theta \times \left[\frac{\sqrt{d}}{\mathsf{g}_2}  \sqrt{\frac{2}{\mathsf{m}}} \mathsf{K}' \right]\times \sqrt{H(\mathbf{q},\mathbf{p})}\\
 &= \theta \times \mathsf{K}'' \times \sqrt{H(\mathbf{q},\mathbf{p})},
\ee
where the second inequality holds since  $(\mathbf{q},\mathbf{p}) \in \mathsf{G}$ implies that $ \frac{\|\mathbf{p}\|}{\mathsf{g}_2}\geq 1$, and the third inequality holds since $\frac{1}{2}\| \mathbf{p}\|^2 \leq H(\mathbf{q},\mathbf{p})$.

By Lemma \ref{thm:leapfrog3} we also have
\be
|H(q_T^{\sharp \theta}(\mathbf{q},\mathbf{p}), p_T^{\sharp \theta}(\mathbf{q},\mathbf{p}))-  H(\mathbf{q},\mathbf{p})| &\stackrel{{\scriptsize \textrm{Lemma }}\ref{thm:leapfrog3}}{\leq}   \frac{d}{\mathsf{m}} \times  \theta \mathsf{K}'\\
 &\leq \frac{\|\mathbf{p}\|^2}{(\mathsf{g}_2)^2} \times \frac{d}{\mathsf{m}} \times \theta \mathsf{K}'\\
 &\leq \theta \times \left[\left(\frac{\sqrt{d}}{\mathsf{g}_2}  \sqrt{\frac{2}{\mathsf{m}}}\right)^2 \mathsf{K}' \right]\times H(\mathbf{q},\mathbf{p})\\
 &= \theta \times \mathsf{K}'' \times H(\mathbf{q},\mathbf{p}),
\ee
where the second inequality holds since  $(\mathbf{q},\mathbf{p}) \in \mathsf{G}$ implies that $ \frac{\|\mathbf{p}\|}{\mathsf{g}_2}\geq 1$, and the third inequality holds since $\frac{1}{2}\| \mathbf{p}\|^2 \leq H(\mathbf{q},\mathbf{p})$.

\end{proof}

Combining Lemmas \ref{thm:leapfrog4} and \ref{thm:Approx_integrator}, we get error bounds for the modified ``toy" integrator  $\diamondsuit$ of Equation \eqref{eq:toy}:

\begin{lemma} (modified ``toy" integrator error bounds) \label{thm:leapfrog5}

Let $U$ and $U_1,U_2, \ldots, U_{\frac{d}{\mathsf{m}}}$ satisfy Assumptions \ref{assumption:product_measure_potential}, \ref{assumption:leapfrog} and let $\sharp$ satisfy Condition \ref{condition:kth_order}.
Fix  $\mathsf{g}_2, \mathsf{g}_\infty$ and let $(\mathbf{q},\mathbf{p}) \in \mathbb{R}^{d}\times \mathbb{R}^{d}$.  Then there exists a constant 
\be \label{DefKTriple}
\mathsf{K}''' = \mathsf{K}'''(\mathsf{B}, \mathsf{K}, T, M_2, m_{2}, \mathsf{m}, \mathsf{c}, \mathsf{g}_{\infty}, \frac{\mathsf{g}_{2}}{\sqrt{d}})
\ee 
that depends polynomially on $\mathsf{g}_\infty$ and $\frac{\sqrt{d}}{\mathsf{g}_2}$, such that
\[\|q_T^{\diamondsuit \theta }(\mathbf{q},\mathbf{p}) - q_T(\mathbf{q},\mathbf{p})\| \leq \theta \times \frac{6T}{\sqrt{m_2}} \times \mathsf{K}'''  \sqrt{H(\mathbf{q},\mathbf{p})}\]
and
\[|H(q_T^{\diamondsuit \theta}(\mathbf{q},\mathbf{p}), p_T^{\diamondsuit \theta}(\mathbf{q},\mathbf{p}))-  H(\mathbf{q},\mathbf{p})| \leq  \theta \times \frac{7}{T M_2}\times \mathsf{K}''' H(\mathbf{q},\mathbf{p}).\]
\end{lemma}
\begin{proof}
For $(p,q) \in \mathsf{G}$, the bounds follow from Lemma  \ref{thm:leapfrog4}. For $(p,q) \notin \mathsf{G}$, the bounds follow from Lemma \ref{thm:Approx_integrator}.
\end{proof}

\section{Mixing of Approximate HMC} \label{SecAppendixMixApprox}

In this section, we prove the main bounds used in the proof of Theorem \ref{ThmMainApprox}.  We first need the generic bounds in the following subsection: 

\subsection{Perturbation Bounds for Markov Chains}
We give some simple bounds on small perturbations of Markov chains. We begin by recalling the definition of the \textit{trace} of a Markov chain on a set:

\begin{defn} [Trace Chain]
Let  $K$ be the transition kernel of an ergodic Markov chain on state space $\Omega$ with stationary measure $\mu$, and let $S \subset \Omega$ be a subset with $\mu(S) > 0$. Let $\{X_{t}\}_{t \geq 0}$ be a Markov chain evolving according to $K$, and iteratively define
\be 
c_{0} &= \inf \{t \geq 0 \, : \, X_{t} \in S \} \\
c_{i+1} &= \inf \{t > c_{i} \, : \, X_{t} \in S \}.
\ee 
Then 
\be \label{EqTraceCoup}
\hat{X}_{t} = X_{c_{t}}, \quad t \geq 0
\ee 
is the \textit{trace of $\{X_{t}\}_{t \geq 0}$ on $S$}. Note that $\{\hat{X}_{t}\}_{t \geq 0}$ is a Markov chain with state space $S$, and so this procedure also defines a transition kernel with state space $S$. We call this kernel the \textit{trace of the kernel $K$ on $S$}.
\end{defn}

 Recall the following bound from \cite{mangoubi2016rapid}:

\begin{lemma} \label{ThmGenericApproxErrorBound}
Let $K$ be a transition kernel on $\mathbb{R}^{d}$ with unique stationary measure $\mu$ and \textit{contraction coefficient} $\kappa > 0$ satisfying

\be \label{SimpAss2}
W_{1}(K(x,\cdot),K(y,\cdot)) \leq (1 - \kappa) \|x-y\|
\ee  
for all $x, y \in \mathbb{R}^{d}$. Let $Q$ be a transition kernel on  $\mathbb{R}^{d}$. Assume that
\be \label{SimpAss1}
\sup_{x \in \mathbb{R}^{d}} W_{1}(K(x,\cdot), Q(x,\cdot)) < \delta
\ee 
for some fixed $\delta \geq 0$. Then $Q$ satisfies 
\be \label{WassContConc1}
W_{1}(\pi_{1}\, Q^\mathcal{I}, \pi_{2} \, K^\mathcal{I}) \leq (1 - \kappa)^\mathcal{I} \, W_1(\pi_{1},\pi_{2}) + \frac{\delta}{\kappa}
\ee 
for all measures $\pi_{1},\pi_{2}$ and $\mathcal{I} \in \mathbb{N}$. Furthermore, if $Q$ is ergodic with stationary measure $\nu$ , then
\be \label{WassContConc2}
W_{1}(\mu, \nu) \leq \frac{\delta}{ \kappa}.
\ee 
\end{lemma}

We use this to prove the following bound, which we can use if the approximation of $K$ by $Q$ is not uniformly good:

\begin{lemma}\label{ThmGenericApproxErrorBound2}
Let $K$ be a transition kernel on metric space $\mathbb{R}^{d}$ with unique stationary measure $\mu$ and \textit{contraction coefficient} $\kappa > 0$ satisfying

\be \label{SimpAss2}
W_{1}(K(x,\cdot),K(y,\cdot)) \leq (1 - \kappa) \| x - y \|
\ee  
for all $x, y \in \mathbb{R}^{d}$. Let $Q$ be a transition kernel on  $\mathbb{R}^{d}$. Fix $\lambda>0$ and define $V(x) \equiv e^{\lambda \|x\|}$. 
 Assume that there exists $0 < \alpha < 1$, $0 \leq \beta < \infty$ so that
\be \label{IneqContractionForQ}
Q(x,\cdot)[V] &\leq (1 - \alpha) V(x) + \beta, \\
K(x,\cdot)[V] &\leq (1 - \alpha) V(x) + \beta \quad \quad \forall x \in \mathbb{R}^d.
\ee

Assume that there exists some $\frac{4 \beta}{\alpha} < C < \infty$ and $\delta > 0$ so that

\be \label{SimpAss1}
\sup_{x \, : \, V(x) \leq C} W_{1}(K(x,\cdot), Q(x,\cdot)) < \delta.
\ee 

Then $Q$ satisfies 
\be \label{WassContConc3}
W_{1}(Q^\mathcal{I}(x,\cdot), \mu) \leq \frac{2 \log(C)}{\lambda} \, (1 - \kappa)^{s}  + \frac{\delta}{\kappa} + \frac{\beta (s+1)}{\lambda \alpha C}\left(8+2\log(C)\right)
\ee 
for all  $0 \leq s \leq \mathcal{I} \in \mathbb{N}$ and all $x$ satisfying $V(x) \leq \frac{\beta}{\alpha}$. Furthermore, if $Q$ is ergodic with stationary measure $\nu$,
\be \label{WassContConc4}
W_{1}(\mu, \nu) \leq \inf_{s \in \{0,1,\ldots\}} \frac{2 \log(C)}{\lambda} \, (1 - \kappa)^{s} + \frac{\delta}{\kappa} + \frac{\beta (s+1)}{\lambda \alpha C}\left(8+2\log(C)\right).
\ee 
\end{lemma}

\begin{proof}

We make some initial estimates. Let $\{X_\mathcal{I}\}_{\mathcal{I} \geq 0}$ be a Markov chain evolving according to $Q$ and started at a (possibly random) point $X_{0}$ satisfying $\mathbb{E}[V(X_0)] \leq \frac{\beta}{\alpha}$. By Inequality \eqref{IneqContractionForQ},
\be  \label{eq:Markov_inequality1}
\mathbb{E}[V(X_\mathcal{I})] & \leq (1 - \alpha) \mathbb{E}[V(X_{\mathcal{I}-1})] + \beta \\
&\leq \ldots \\
&\leq (1 - \alpha)^\mathcal{I} \mathbb{E}[V(X_{0})] + \frac{\beta}{\alpha} \\
&\leq \frac{2 \beta}{\alpha} \quad \quad \forall \mathcal{I}\in \mathbb{N}.
\ee

By Markov's inequality, then,
\be \label{eq:Markov_inequality2}
\mathbb{P}[\sup_{\mathcal{I}-s \leq h \leq \mathcal{I}} V(X_{h}) > r] \leq \frac{2 \beta (s+1)}{\alpha r}
\ee 
for any fixed integers $0 \leq s \leq \mathcal{I}$ and any $r>0$. Rewriting this,

\be \label{eq:Markov_inequality5a}
\mathbb{P}[\sup_{\mathcal{I}-s \leq h \leq \mathcal{I}} e^{\lambda \|X_{h}\|} > r] \leq \frac{2 \beta (s+1)}{\alpha r}  \quad \quad \forall \, r>0,
\ee
and so
\be  \label{eq:Markov_inequality5}
\mathbb{P}[\sup_{\mathcal{I}-s \leq h \leq \mathcal{I}} \|X_{h}\| > r] \leq \frac{2 \beta (s+1)}{\alpha} e^{-\lambda r}  \quad \quad \forall r>0, \,
\ee 
which gives 
\be \label{eq:Markov_inequality3}
\mathbb{E}[(\sup_{\mathcal{I}-s \leq h \leq \mathcal{I}} &\|X_{h}\|) \times \mathbbm{1}\{ \sup_{\mathcal{I}-s \leq h \leq \mathcal{I}} V(X_{h})\geq C\}] \\
&= \mathbb{E}[(\sup_{\mathcal{I}-s \leq h \leq \mathcal{I}} \|X_{h}\|) \times \mathbbm{1}\{ \sup_{\mathcal{I}-s \leq h \leq \mathcal{I}} \|X_{h}\|\geq \frac{1}{\lambda}\log(C)\}]\\
& \stackrel{{\scriptsize \textrm{Eq. }}\ref{eq:Markov_inequality5}}{\leq} \int_{\frac{1}{\lambda}\log(C)}^\infty  \frac{2 \beta (s+1)}{\alpha} e^{-\lambda r} \d r\\
&= \frac{2 \beta (s+1)}{\lambda \alpha C}.
\ee

Also,  by Jensen's inequality, Equation \eqref{eq:Markov_inequality1} implies that,
\be
e^{\lambda \mathbb{E}[\|X_\mathcal{I}\|]} \leq \mathbb{E}[e^{\lambda \|X_\mathcal{I}\|}] = \mathbb{E}[V(X_\mathcal{I})]  \stackrel{\textrm{Eq. \ref{eq:Markov_inequality1}}}{\leq} \frac{2\beta}{\alpha},
\ee
and so
\be \label{eq:Markov_inequality6}
\mathbb{E}[\|X_\mathcal{I}\|] \leq \lambda^{-1}\log(\frac{2\beta}{\alpha}).
\ee

Now let $\{Y_\mathcal{I}\}_{\mathcal{I} \geq 0}$ be a Markov chain evolving according to $K$ and started at $Y_{0} \sim \mu$.  Then by Inequality \eqref{IneqContractionForQ},
\be
\mathbb{E}[V(Y_0)] = \mathbb{E}[V(Y_1)] \leq (1-\alpha)\mathbb{E}[V(Y_0)] + \beta,
\ee
which implies that $\mathbb{E}[V(Y_0)] \leq \frac{\beta}{\alpha}$.  Therefore, by the same argument we used to show Inequalities \eqref{eq:Markov_inequality5a}, \eqref{eq:Markov_inequality3}, and \eqref{eq:Markov_inequality6} (but replacing $K$ with $Q$ and  $X_{h}$ with $Y_{h}$), we must have, respectively, 
\be \label{eq:Markov_inequality4}
\mathbb{E}[\sup_{\mathcal{I}-s \leq h \leq \mathcal{I}} \|Y_{h}\| \times \mathbbm{1}\{ \sup_{\mathcal{I}-s \leq h \leq \mathcal{I}} V(Y_{h})\geq C\}] \leq  \frac{2 \beta (s+1)}{\lambda \alpha C}
\ee
and
\be \label{eq:Markov_inequality5b}
\mathbb{P}[\sup_{\mathcal{I}-s \leq h \leq \mathcal{I}} e^{\lambda \|Y_{h}\|} > r] \leq \frac{2 \beta (s+1)}{\alpha r}  \quad \quad \forall r>0
\ee
and
\be \label{eq:Markov_inequality6b}
\mathbb{E}[\|Y_\mathcal{I}\|] \leq \lambda^{-1}\log(\frac{2\beta}{\alpha}).
\ee

We now prove Inequality \eqref{WassContConc3} using an explicit coupling. Fix integers $0 \leq s \leq \mathcal{I} < \infty$. Let $\{X_{h}\}_{{h}=0}^{\mathcal{I}-s}$, $\{Y_{h}\}_{{h}=0}^{\mathcal{I}-s}$ be Markov chains with transition kernels $Q$ and $K$ respectively, and starting points $X_{0} = x$ and $Y \sim \mu$. We begin to construct our coupling of these two chains by allowing them to evolve independently over the time interval $\{0,1,\ldots,\mathcal{I}-s\}$.

Next, let $\tilde{Q}$, $\tilde{K}$ be the traces of $Q$, $K$ on the set $\{z \in \mathbb{R}^{d} \, : \, V(z) \leq C\}$. Fix $\gamma > 0$.  By Lemma \ref{ThmGenericApproxErrorBound}, it is possible to couple two Markov chains  $\{ \tilde{X}_i \}_{i \geq 0}$, $\{ \tilde{Y}_i \}_{i \geq 0}$ with transition kernels $\tilde{Q}$, $\tilde{K}$ and starting points  
\be \label{EqStartingPointTracesFin}
\tilde{X}_{0} = \begin{cases}
X_{\mathcal{I}-s}, \qquad V(X_{\mathcal{I}-s}) \leq C \\
0, \, \, \, \,  \qquad V(X_{\mathcal{I}-s}) > C \\
\end{cases}
\ee 

\be 
\tilde{Y}_{0} = \begin{cases}
Y_{\mathcal{I}-s}, \,  \qquad V(Y_{\mathcal{I}-s}) \leq C \\
 0, \, \, \,  \quad \qquad V(Y_{\mathcal{I}-s}) > C \\
\end{cases}
\ee  
so that
\be \label{IntermediateTraceCoup}
\E[\| \tilde{X}_{s} - \tilde{Y}_{s} \|] \leq (1 - \kappa)^{s} \sup_{x,y \, : \, V(x), \, V(y) \leq C} \, \|x-y\| + \frac{\delta}{\kappa} + \gamma.
\ee
Next, couple $\{X_{h}\}_{{h}=\mathcal{I}-s}^\mathcal{I}$ to $\{\tilde{X}_{h}\}_{{h}=0}^{s}$ (respectively, couple $\{Y_{h}\}_{{h}=\mathcal{I}-s}^\mathcal{I}$ to $\{\tilde{Y}_{h}\}_{{h}=0}^{s}$) according to the natural coupling of a Markov chain to its trace on a set (that is, the coupling in Equation \eqref{EqTraceCoup}). Defining $\tau_X = \min \{ {h} \geq \mathcal{I} -s \, : \, V(X_{h}) > C\}$,  $\tau_Y = \min \{ {h} \geq \mathcal{I} -s \, : \, V(Y_{h}) > C\}$, and $\tau = \min \{ \tau_X, \tau_Y \}$, we have:
\be 
\E[\| X_\mathcal{I} - Y_\mathcal{I} \|] &= \E[\| X_\mathcal{I} - Y_\mathcal{I} \| \, \mathbbm{1}_{\tau >\mathcal{I}}] + \E[\| X_\mathcal{I} - Y_\mathcal{I} \| \, \mathbbm{1}_{\tau \leq \mathcal{I}}] \\
&\leq \E[\| X_\mathcal{I} - Y_\mathcal{I} \| \, \mathbbm{1}_{\tau >\mathcal{I}}] + \E[(\| X_\mathcal{I}\| + \|Y_\mathcal{I} \|) \, \mathbbm{1}_{\tau \leq \mathcal{I}}] \\
&\leq \E[\| X_\mathcal{I} - Y_\mathcal{I} \| \, \mathbbm{1}_{\tau >\mathcal{I}}] + \E[(\| X_\mathcal{I}\| + \|Y_\mathcal{I} \|) \, \mathbbm{1}_{\tau_X \leq \mathcal{I}}] + \E[(\| X_\mathcal{I}\| + \|Y_\mathcal{I} \|) \, \mathbbm{1}_{\tau_Y \leq \mathcal{I}}] \\
&\leq \E[\| X_\mathcal{I} - Y_\mathcal{I} \| \, \mathbbm{1}_{\tau >\mathcal{I}}] + \E[(\| X_\mathcal{I}\| + \|Y_\mathcal{I} \|) \, \mathbbm{1}_{\tau_X \leq \mathcal{I}}\times (\mathbbm{1}_{\tau_Y \leq \mathcal{I}} +\mathbbm{1}_{\tau_Y >\mathcal{I}})]\\
&\quad \quad+ \E[(\| X_\mathcal{I}\| + \|Y_\mathcal{I} \|) \, \mathbbm{1}_{\tau_Y \leq \mathcal{I}}\times (\mathbbm{1}_{\tau_X \leq \mathcal{I}} +\mathbbm{1}_{\tau_X >\mathcal{I}})] \\
&\leq \E[\| X_\mathcal{I} - Y_\mathcal{I} \| \, \mathbbm{1}_{\tau >\mathcal{I}}] + \E[\| X_\mathcal{I}\| \, \mathbbm{1}_{\tau_X \leq \mathcal{I}}] + \E[\|Y_\mathcal{I}\| \, \mathbbm{1}_{\tau_Y \leq \mathcal{I}}] + \E[\|Y_\mathcal{I}\| \, \mathbbm{1}_{\tau_X \leq \mathcal{I}} \times \mathbbm{1}_{\tau_Y >\mathcal{I}}] \\
&\qquad \qquad \qquad+ \E[\| Y_\mathcal{I}\| \, \mathbbm{1}_{\tau_Y \leq \mathcal{I}}] + \E[\|X_\mathcal{I}\| \, \mathbbm{1}_{\tau_X \leq \mathcal{I}}] + \E[\|X_\mathcal{I}\| \, \mathbbm{1}_{\tau_Y \leq \mathcal{I}} \times \mathbbm{1}_{\tau_X >\mathcal{I}}] \\
&\leq \E[\| X_\mathcal{I} - Y_\mathcal{I} \| \, \mathbbm{1}_{\tau >\mathcal{I}}] + \E[\| X_\mathcal{I}\| \, \mathbbm{1}_{\tau_X \leq \mathcal{I}}] + \E[\|Y_\mathcal{I}\| \, \mathbbm{1}_{\tau_Y \leq \mathcal{I}}] + \E[\frac{1}{\lambda}\log(C) \, \mathbbm{1}_{\tau_X \leq \mathcal{I}} \times \mathbbm{1}_{\tau_Y >\mathcal{I}}] \\
&\qquad \qquad \qquad+ \E[\| Y_\mathcal{I}\| \, \mathbbm{1}_{\tau_Y \leq \mathcal{I}}] + \E[\|X_\mathcal{I}\| \, \mathbbm{1}_{\tau_X \leq \mathcal{I}}] + \E[\frac{1}{\lambda}\log(C) \, \mathbbm{1}_{\tau_Y \leq \mathcal{I}} \times \mathbbm{1}_{\tau_X >\mathcal{I}}] \\
&\leq  \E[\| \tilde{X}_{s} - \tilde{Y}_{s} \| \mathbbm{1}_{\tau >\mathcal{I}}]  + 2\mathbb{E}[\sup_{\mathcal{I}-s \leq h \leq \mathcal{I}} \|X_{h}\| \times \mathbbm{1}\{ \sup_{\mathcal{I}-s \leq h \leq \mathcal{I}} V(X_{h})\geq C\}]  +  \frac{1}{\lambda}\log(C) \, \mathbb{P}[\sup_{\mathcal{I}-s \leq h \leq \mathcal{I}} V(X_{h})\geq C\}]  \\
&+ 2\mathbb{E}[\sup_{\mathcal{I}-s \leq h \leq \mathcal{I}} \|Y_{h}\| \times \mathbbm{1}\{ \sup_{\mathcal{I}-s \leq h \leq \mathcal{I}} V(Y_{h})\geq C\}] +  \frac{1}{\lambda}\log(C) \times \mathbb{P}[\sup_{\mathcal{I}-s \leq h \leq \mathcal{I}} V(Y_{h})\geq C\}]   \\
&\stackrel{\textrm{Eq. \ref{eq:Markov_inequality5a},\ref{eq:Markov_inequality5b}}}{\leq}  \E[\| \tilde{X}_{s} - \tilde{Y}_{s} \| \mathbbm{1}_{\tau >\mathcal{I}}]  + 2\mathbb{E}[\sup_{\mathcal{I}-s \leq h \leq \mathcal{I}} \|X_{h}\| \times \mathbbm{1}\{ \sup_{\mathcal{I}-s \leq h \leq \mathcal{I}} V(X_{h})\geq C\}] \\
&+ 2\mathbb{E}[\sup_{\mathcal{I}-s \leq h \leq \mathcal{I}} \|Y_{h}\| \times \mathbbm{1}\{ \sup_{\mathcal{I}-s \leq h \leq \mathcal{I}} V(Y_{h})\geq C\}] +  2\lambda^{-1}\log(C) \times \frac{2 \beta (s+1)}{\alpha C}.\\
\ee

Applying Inequalities \eqref{eq:Markov_inequality3}, \eqref{eq:Markov_inequality4}, and \eqref{IntermediateTraceCoup},
\be 
\E[\| X_\mathcal{I} - Y_\mathcal{I} \|] &\leq  (1 - \kappa)^{s} \sup_{x,y \, : \, V(x), \, V(y) \leq C} \, \|x-y\| + \frac{\delta}{\kappa} + \frac{8 \beta (s+1)}{\lambda \alpha C} +  2\lambda^{-1}\log(C) \times \frac{2 \beta (s+1)}{\alpha C} + \gamma\\
&\leq  (1 - \kappa)^{s} \frac{2 \log(C)}{\lambda} + \frac{\delta}{\kappa} + \frac{\beta (s+1)}{\lambda \alpha C}\left(8+2\log(C)\right)  + \gamma,
\ee 
where the second line uses the fact $\sup_{x,y \, : \, V(x), \, V(y) \leq C} \|x-y\| \leq  \frac{2 \log(C)}{\lambda}$.  Since $\gamma > 0$ was arbitrary, this completes the proof of Inequality \eqref{WassContConc3}.

To prove Inequality \eqref{WassContConc4}, we note by the ergodicity of $Q$ and Inequality \eqref{WassContConc3} that, for fixed $x$ satisfying $V(x) \leq \frac{\beta}{\alpha}$,
\be 
W_{1}(\nu, \mu) &= \lim_{\mathcal{I} \rightarrow \infty} W_{1}(Q^\mathcal{I}(x,\cdot), \mu) \\
&\leq \lim_{\mathcal{I} \rightarrow \infty} \inf_{0 \leq s \leq \mathcal{I}}  \left[\frac{2 \log(C)}{\lambda} (1 - \kappa)^{s} + \frac{\delta}{\kappa} + \frac{\beta (s+1)}{\lambda \alpha C}\left(8+2\log(C)\right)\right] \\
&= \inf_{s \in \{0,1,\ldots\}} \left[ \frac{2 \log(C)}{\lambda} (1 - \kappa)^{s} + \frac{\delta}{\kappa} + \frac{\beta (s+1)}{\lambda \alpha C}\left(8+2\log(C)\right)\right], 
\ee 
which completes the proof.
\end{proof}

\subsection{Bounds on unadjusted HMC Algorithm} \label{AppSubAppBoundUnad}

We now prove a mixing bound for the unadjusted HMC algorithm described in Algorithm \ref{alg:Unadjusted}. We set some notation to be used throughout this section. We always fix a target potential $U$ that satisfies Assumption \ref{AssumptionsConvexity} with $\mathcal{X} = \mathbb{R}^{d}$. We also fix integration time $0 \leq T \leq \frac{1}{2\sqrt{2}} \frac{\sqrt{m_2}}{M_2}$ 
 and let $K$ be the transition kernel defined in Algorithm \ref{DefSimpleHMC} with these parameters. We also fix a parameter $0 \leq \theta \leq \frac{T}{5}$, and a numerical integrator $\dagger$ satisfying the following condition:

 \begin{cond} \label{cond:integrator}
 We say that a numerical integrator $\dagger$ satisfies Condition \ref{cond:integrator} with constant $\mathsf{A}$ and parameters as above if it satisfies the inequalities
\be \label{IneqCondIntPost}
\|q_T^{\dagger \theta}(\mathbf{q},\mathbf{p}) - q_T(\mathbf{q},\mathbf{p})\| \leq 6\theta \times T \times \frac{\mathsf{A}}{\sqrt{m_2}} \sqrt{H(\mathbf{q},\mathbf{p})}
\ee 
and 
\be \label{IneqCondIntEnergy}
|H(q_T^{\dagger \theta}(\mathbf{q},\mathbf{p}), p_T^{\dagger \theta}(\mathbf{q},\mathbf{p}))-  H(\mathbf{q},\mathbf{p})| \leq 7 \frac{\theta}{T} H(\mathbf{q},\mathbf{p}) \times \frac{\mathsf{A}}{M_2}
\ee
for all $\mathbf{q}, \mathbf{p} \in \mathbb{R}^d$.
\end{cond}

Let $Q$ be the transition kernel defined by Algorithm \ref{alg:Unadjusted} and these parameters. We use the notation from Algorithms \ref{DefSimpleHMC} and \ref{alg:Unadjusted} in our analyses of these algorithms. We begin with coarse bounds on the behaviour of the HMC chain:

\begin{lemma}\label{ThmGenericApproxErrorBound3}
Set notation as above. Fix $ C'>0$ and $0 < \lambda \leq \frac{1}{32}\sqrt{2m_2}$ and define $V(x) \equiv e^{\lambda \|x\|}$ for all $x\in  \mathbb{R}^d$.  Let $\kappa:= \frac{1}{8} (\sqrt{m}_2T)^2$ and fix $\theta>0$ such that $6\theta \, T \, \frac{\mathsf{A}}{\sqrt{m_2}} \sqrt{M_2} < \kappa$. Then for all $x \in \mathbb{R}^{d}$ we have the drift condition 
\be \label{eq:w6}
K(x,\cdot)[V] &\leq (1 - \alpha) V(x) + \beta, \\
Q(x,\cdot)[V] &\leq (1 - \alpha) V(x) + \beta,
\ee

where
\be 
\alpha &= 1-  (1+e^{-\frac{d}{8}})e^{\lambda\left(-\kappa + 6\theta \, T \, \frac{\mathsf{A}}{\sqrt{m_2}} \sqrt{M_2}\right)\, C' +  \lambda\left(1 + 6\theta \, T \, \mathsf{A}\right) \, \sqrt{\frac{d}{m_{2}}}}\\
\beta &=  e^{\lambda C'} \, (1 - \alpha).
\ee

Moreover,
\be \label{eq:w7}
\|q_T^{\dagger \theta}(x,\mathbf{p})\| &\leq  \left(1-\kappa + 6\theta \, T \, \frac{\mathsf{A}}{\sqrt{m_2}} \sqrt{M_2}\right)\,  \|x\| +  \left(1 + 6\theta \, T \,\mathsf{A} \right) \, \frac{\|\mathbf{p}\|}{\sqrt{2 m_{2}}}
\ee
for any $x,\mathbf{p} \in \mathbb{R}^d$. 
\end{lemma}

\begin{remark} \label{RemGenericApproxErrorBound3}
Note that this result remains true if $U$ satisfies only the condition in Inequality \eqref{IneqCondIntPost} of Assumption \ref{cond:integrator}, but does not satisfy the condition in Inequality \eqref{IneqCondIntEnergy}. This can be seen by simply reading the proof, which does not reference Inequality \eqref{IneqCondIntEnergy}. We use this fact in the proof of Lemma \ref{thm:MH1}.
\end{remark}

\begin{proof}

We begin by coupling two Markov chains evolving according to Algorithm \ref{DefSimpleHMC} to a Markov chain evolving according to Algorithm \ref{alg:Unadjusted}.

Let $\mathbf{p}_0, \mathbf{p}_1, \ldots$ be independent $\mathcal{N}(0,I_d)$ Gaussians. Recalling the forward mapping representation \eqref{EqForwardMappingRep} of Algorithm \ref{DefSimpleHMC}, we set initial conditions $X_0 = x$, $Y_0 = 0$, and inductively define 
\be 
X_{i+1} &= \mathcal{Q}_{T}^{X_{i}}(\mathbf{p}_{i}) \\
Y_{i+1} &= \mathcal{Q}_{T}^{Y_{i}}(\mathbf{p}_{i}) \\
\ee 
for $i \in \mathbb{N}$. This gives a coupling of two copies $\{X_{i}\}_{i \geq 0}$, $\{Y_{i}\}_{i \geq 0}$ of the idealized HMC chain defined in Algorithm \ref{DefSimpleHMC}.  We couple these two chains to a Markov chain $(X_{0}', X_{1}',\ldots)$ generated by Algorithm \ref{alg:Unadjusted} with starting point $X_0'=X_0 = x$ by inductively setting
\be 
X_{i+1}' = q_{T}^{\dagger \theta}(X_{i}', \mathbf{p}_{i})
\ee 
for all $i \geq 0$. Note that the resulting Markov chain evolves according to Algorithm \ref{alg:Unadjusted}.

Define $\mathfrak{p}_0=\mathbf{p}_0$, so $\mathfrak{p}_0 \sim \mathcal{N}(0,I_d)$. By Theorem \ref{ThmContractionConvexMainResult} we have
\be
\label{eq:w1}
\|X_1-Y_1\| \leq (1-\kappa) \, \|X_0-Y_0\| = (1-\kappa) \, \|X_0\|.
\ee

Also, by conservation of energy for Hamilton's equations,
\be
U(Y_1) \leq H(Y_0,\mathbf{p}_0) = 0 + \frac{1}{2}\|\mathfrak{p}_0\|^{2}. 
\ee
But by strong convexity of $U$ (Assumption \ref{AssumptionsConvexity}) $m_2 \|Y_1\|^2 \leq U(Y_1)$, so
\be \label{eq:Y1}
\|Y_1\| \leq \frac{1}{\sqrt{m_2}}\sqrt{U(Y_1)} \leq  \frac{\|\mathfrak{p}_0\|}{\sqrt{2m_2}}.
\ee
Therefore, combining Inequalities \eqref{eq:w1} and \eqref{eq:Y1}, 
\be
\|X_1\| \leq \|X_1-Y_1\| + \|Y_1\| \leq (1-\kappa)\, \|X_0\| + \|Y_1\| \leq (1-\kappa)\, \|X_0\| +  \frac{\|\mathfrak{p}_0\|}{\sqrt{2m_2}}
\ee
and so,
\be \label{eq:w10}
e^{\lambda\|X_1\|} \leq e^{\lambda(1-\kappa)\times \|X_0\|} e^{\lambda\frac{\|\mathfrak{p}_0\|}{\sqrt{2m_2}}}.
\ee

Inequality \eqref{eq:w10} implies that 

\be
e^{\lambda\|X_1\|} \leq  \begin{cases}
(1-\alpha')e^{\lambda\|X_0\|}, \, \,  \quad  \qquad \|X_0\| \geq C'\\
\beta', \, \, \qquad \quad \qquad \quad \qquad \|X_0\| \leq C', 
\end{cases} 
\ee
where $\alpha' = 1-e^{-\lambda\kappa C' + \lambda\frac{\|\mathfrak{p}_0\|}{\sqrt{2m_2}}}$ and $\beta' = e^{\lambda(1-\kappa)C'+\lambda\frac{\|\mathfrak{p}_0\|}{\sqrt{2m_2}}}$ are random variables that depend on $\mathfrak{p}_{0}$. Thus, for any $X_{0}$,
\be \label{eq:c1}
e^{\lambda\|X_1\|} \leq (1-\alpha')e^{\lambda\|X_0\|} + \beta'.
\ee

Now,  $\|\mathfrak{p}_0\|^2 \sim \chi^2_d$, so by the Hanson-Wright concentration inequality (see, for instance, \cite{hanson1971bound, rudelson2013hanson}), we have
\be
\mathbb{P}(\|\mathfrak{p}_0\|^2>s+d)  \leq e^{-\frac{s}{8}} \quad \quad  \textrm{ for } s> d.
\ee
Therefore, for any $\gamma$ such that $0< \gamma<\frac{1}{16}\sqrt{2m_2}$, 
\be
\mathbb{P}\left[e^{\gamma\frac{\|\mathfrak{p}_0\|}{\sqrt{2m_2}}} > e^{\gamma\frac{\sqrt{s+d}}{\sqrt{2m_2}}}\right] \leq e^{-\frac{s}{8}} \quad \quad  \textrm{ for } s> d
\ee
and so
\be \label{eq:c4}
\mathbb{P}\left[e^{\gamma\frac{\|\mathfrak{p}_0\|}{\sqrt{2m_2}}} > r \right] \leq e^{-\frac{1}{8}((\frac{\sqrt{2m_2}}{\gamma}\log(r))^2-d)} \quad \quad  \textrm{ for } r> e^{\gamma\frac{\sqrt{2d}}{\sqrt{2m_2}}}.
\ee

However, for $r > e^{\gamma\frac{\sqrt{2d}}{\sqrt{2m_2}}}$, we have
\be \label{eq:c5}
e^{-\frac{1}{8}((\frac{\sqrt{2m_2}}{\gamma}\log(r))^2-d)}= e^{\frac{1}{8}d} [e^{\log(r)}]^{-\frac{1}{8}(\frac{\sqrt{2m_2}}{\gamma})^2\log(r)} &= e^{\frac{1}{8}d} r^{-\frac{1}{8}(\frac{\sqrt{2m_2}}{\gamma})^2\log(r)}\\
&\leq e^{\frac{1}{8}d} r^{-\frac{1}{8}\frac{\sqrt{2m_2}}{\gamma}\sqrt{2d}}.
\ee
Inequalities \eqref{eq:c4} and \eqref{eq:c5} together give
\be \label{eq:c3}
\mathbb{P}\left[e^{\gamma\frac{\|\mathfrak{p}_0\|}{\sqrt{2m_2}}} > r \right] \leq e^{\frac{1}{8}d} r^{-\frac{1}{8}\frac{\sqrt{2m_2}}{\gamma}\sqrt{2d}} \quad \quad  \textrm{ for } r> e^{\gamma\frac{\sqrt{2d}}{\sqrt{2m_2}}},
\ee
and hence
\be
\mathbb{E}\left[e^{\gamma\frac{\|\mathfrak{p}_0\|}{\sqrt{2m_2}}} \right] &\stackrel{\textrm{Eq. \ref{eq:c3}}}{\leq}  e^{\gamma\frac{\sqrt{2d}}{\sqrt{2m_2}}} + \int_{e^{\gamma\frac{\sqrt{2d}}{\sqrt{2m_2}}}}^\infty e^{\frac{1}{8}d} r^{-\frac{1}{8}\frac{\sqrt{2m_2}}{\gamma}\sqrt{2d}}  \d r\\
&= e^{\gamma\frac{\sqrt{2d}}{\sqrt{2m_2}}} +  \frac{1}{{\frac{1}{8}\frac{\sqrt{2m_2}}{\gamma}\sqrt{2d}-1}}e^{-\frac{1}{8}d+\gamma\frac{\sqrt{2d}}{\sqrt{2m_2}}}\\
&\leq (1+e^{-\frac{d}{8}})\, e^{\gamma\frac{\sqrt{2d}}{\sqrt{2m_2}}}\\
\ee
where the inequality uses the fact that $\gamma \leq \frac{1}{16}\sqrt{2m_2}$.  So we have 
\be \label{eq:c2}
\mathbb{E}\left[e^{\gamma\frac{\|\mathfrak{p}_0\|}{\sqrt{2m_2}}} \right] \leq  (1+e^{-\frac{d}{8}})e^{\gamma\frac{\sqrt{2d}}{\sqrt{2m_2}}} \quad \quad \textrm{ for } 0<\gamma \leq\frac{1}{16}\sqrt{2m_2}.\\
\ee

Therefore,
\be \label{eq:Jensen}
\mathbb{E}[e^{\lambda\|X_1\|}] &\stackrel{\textrm{Eq. \ref{eq:c1}}}{\leq} \mathbb{E}[(1-\alpha')e^{\lambda\|X_0\|}] + \mathbb{E}[\beta'] \\
&= \mathbb{E}[e^{-\lambda\kappa C' + \lambda\frac{\|\mathfrak{p}_0\|}{\sqrt{2m_2}}}e^{\lambda\|X_0\|}]  + \mathbb{E}[e^{\lambda(1-\kappa)C'+\lambda\frac{\|\mathfrak{p}_0\|}{\sqrt{2m_2}}}]\\
&\stackrel{\textrm{Eq. \ref{eq:c2}}}{\leq}   (1+e^{-\frac{d}{8}})e^{-\lambda\kappa C'} e^{\lambda\frac{\sqrt{2d}}{\sqrt{2m_2}}}e^{\lambda\|X_0\|} +  (1+e^{-\frac{d}{8}})e^{\lambda(1-\kappa)C'}e^{\lambda\frac{\sqrt{2d}}{\sqrt{2m_2}}},
\ee
where the assumption of Inequality  \eqref{eq:c2} is satisfied because $0<\lambda \leq \frac{1}{32}\sqrt{2m_2}<\frac{1}{16}\sqrt{2m_2}$.

Next, since $6\theta \, T \, \frac{\mathsf{A}}{\sqrt{m_2}}>0$, we have
\be \label{eq:alpha_max}
1- \alpha &=(1+e^{-\frac{d}{8}})e^{\lambda\left(-\kappa + 6\theta \, T \, \frac{\mathsf{A}}{\sqrt{m_2}} \sqrt{M_2}\right)\, C' +  \lambda\left(1 + 6\theta \, T \, \mathsf{A}\right) \, \sqrt{\frac{d}{m_{2}}}}\\
&\geq (1+e^{-\frac{d}{8}})e^{-\lambda\kappa C'+ \lambda\frac{\sqrt{2d}}{\sqrt{2m_2}}}
\ee
 and  
\be \label{eq:beta_max}
\beta &=  e^{\lambda C'} \, (1 - \alpha) \\
&\geq (1+e^{-\frac{d}{8}})e^{\lambda(1-\kappa)C' + \lambda\frac{\sqrt{2d}}{\sqrt{2m_2}}}.
\ee

Therefore, substituting Inequalities \eqref{eq:alpha_max} and  \eqref{eq:beta_max} into Inequality \eqref{eq:Jensen}, we get
\[\mathbb{E}[V(X_1)] \leq (1-\alpha) V(X_0)  + \beta.\]
and hence
\begin{equation}
K(x,\cdot)[V] \leq (1-\alpha)\times V(x) + \beta.
\end{equation}
 This proves the first line of Ineqality \eqref{eq:w6}. We now prove Inequality \eqref{eq:w7}, before finally proving the second line of Inequality \eqref{eq:w6}. By Condition \ref{cond:integrator}, for every $\mathbf{p} \in \mathbb{R}^d$ we have 
\be \label{eq:w11}
\|q_T^{\dagger \theta}(x,\mathbf{p}) - q_T(x,\mathbf{p})\| &\leq 6\theta \, T \, \frac{\mathsf{A}}{\sqrt{m_2}} \sqrt{H(X_0,\mathbf{p})}\\
&\stackrel{\textrm{Assumption } \ref{AssumptionsConvexity}}{\leq} 6\theta \, T \, \frac{\mathsf{A}}{\sqrt{m_2}} \sqrt{M_2\|X_0\|^2 + \frac{1}{2}\|\mathbf{p}\|^2}\\
&\leq 6\theta \, T \, \frac{\mathsf{A}}{\sqrt{m_2}} \left[\sqrt{M_2\|X_0\|^2} + \sqrt{\frac{1}{2}\|\mathbf{p}\|^2}\right],\\
\ee
Also, by the same calculation that was used to prove Inequality \eqref{eq:w10}, we have
\be \label{eq:w10b}
e^{\lambda\|q_T(x,\mathbf{p})\|} \leq e^{\lambda(1-\kappa)\times \|x\|} e^{\lambda\frac{\|\mathbf{p}\|}{\sqrt{2m_2}}} \quad \quad \forall \mathbf{p}\in \mathbb{R}^d.
\ee

And so for all  $\mathbf{p} \in \mathbb{R}^d$ we have,
\be
\|q_T^{\dagger \theta}(x,\mathbf{p})\| &\leq \| q_T(x,\mathbf{p})\|+  \|q_T^{\dagger \theta}(x,\mathbf{p}) - q_T(x,\mathbf{p})\|\\
&\stackrel{\textrm{Eq. \ref{eq:w10b}, \ref{eq:w11}}}{\leq}  (1-\kappa)\, \|x\| +  \frac{\|\mathbf{p}\|}{\sqrt{2m_2}} + 6\theta \, T \, \frac{\mathsf{A}}{\sqrt{m_2}} \left[\sqrt{M_2}\|x\| + \frac{1}{\sqrt{2}}\|\mathbf{p}\| \right]\\
&\leq  \left(1-\kappa + 6\theta \, T \, \frac{\mathsf{A}}{\sqrt{m_2}} \sqrt{M_2}\right)\, \|x\| +  \left(\frac{1}{\sqrt{2m_2}} + \frac{1}{\sqrt{2}}6\theta \, T \, \frac{\mathsf{A}}{\sqrt{m_2}} \right) \, \|\mathbf{p}\|.
\ee

This proves Inequality \eqref{eq:w7}. Inequality \eqref{eq:w7} in turn implies that  
\be
\|X_1'\| \leq  \left(1-\kappa + 6\theta \, T \, \frac{\mathsf{A}}{\sqrt{m_2}} \sqrt{M_2}\right)\, \|X_0\| +  \left(1 + 6\theta \, T \, \mathsf{A} \right) \, \frac{\|\mathfrak{p}_0\|}{\sqrt {2 m_{2}}}
\ee
and so
\be
e^{\lambda\|X_1'\|} \leq  e^{\lambda\|X_0\|}e^{\lambda\left(-\kappa + 6\theta \, T \, \frac{\mathsf{A}}{\sqrt{m_2}} \sqrt{M_2}\right)\, \|X_0\| +  \lambda\left(1+ 6\theta \, T \, \mathsf{A} \right) \, \frac{\|\mathfrak{p}_0\|}{\sqrt{2 m_{2}}}}.
\ee
Since $6\theta \, T \, \frac{\mathsf{A}}{\sqrt{m_2}} \sqrt{M_2} < \kappa$, this implies 
\be
e^{\lambda\|X'_1\|} \leq \begin{cases}
(1-\alpha'')e^{\lambda\|X_0\|}, \qquad \, \, \|X_0\| \geq C'\\
e^{\lambda\|X'_1\|} \leq \beta'' \qquad \qquad \|X_0\|  \leq C',
\end{cases}
\ee
where 
\be 
\alpha'' &:=1-e^{\lambda\left(-\kappa + 6\theta \, T \, \frac{\mathsf{A}}{\sqrt{m_2}} \sqrt{M_2}\right)\, C' +  \lambda\left(1 + 6\theta \, T \, \mathsf{A} \right) \, \frac{\|\mathfrak{p}_0\|}{\sqrt{2 m_{2}}}} \\
\beta''&:= e^{\lambda\left(1-\kappa + 6\theta \, T \, \frac{\mathsf{A}}{\sqrt{m_2}} \sqrt{M_2}\right)\, C' +   \lambda\left(1 + 6\theta \, T \, \mathsf{A} \right) \, \frac{\|\mathfrak{p}_0\|}{\sqrt{2 m_{2}}}}
\ee
are random variables that depend on $\mathfrak{p}_{0}$. Thus, 
\be
e^{\lambda\|X_1'\|} \leq  (1-\alpha'')e^{\lambda\|X_0\|} + \beta''
\ee
for any $X_{0}$, and so
\be \label{eq:Jensen2}
\mathbb{E}&[e^{\lambda\|X_1'\|}] \leq \mathbb{E}[(1-\alpha'')e^{\lambda\|X_0\|} + \beta''] \\
&= \mathbb{E}[1-\alpha'']\, e^{\lambda\|X_0\|} + \mathbb{E}[\beta'']\\
&=\mathbb{E}\left[e^{\lambda\left(-\kappa + 6\theta \, T \, \frac{\mathsf{A}}{\sqrt{m_2}} \sqrt{M_2}\right)\, C' +   \lambda\left(1 + 6\theta \, T \, \mathsf{A} \right) \, \frac{\|\mathfrak{p}_0\|}{\sqrt{2 m_{2}}}}\right]\, e^{\lambda\|X_0\|}\\&\qquad\qquad+ \mathbb{E}\left[e^{\lambda\left(1-\kappa + 6\theta \, T \, \frac{\mathsf{A}}{\sqrt{m_2}} \sqrt{M_2}\right)\, C' +   \lambda\left(1 + 6\theta \, T \, \mathsf{A} \right) \, \frac{\|\mathfrak{p}_0\|}{\sqrt{2 m_{2}}}}\right]\\
&\stackrel{\textrm{Eq. }\ref{eq:c2}}{\leq} (1+e^{-\frac{d}{8}})e^{\lambda\left(-\kappa + 6\theta \, T \, \frac{\mathsf{A}}{\sqrt{m_2}} \sqrt{M_2}\right)\, C' +  \lambda\left(\frac{1}{\sqrt{2m_2}} + \frac{1}{\sqrt{2}}6\theta \, T \, \frac{\mathsf{A}}{\sqrt{m_2}} \right) \, \sqrt{2d}}\, e^{\lambda\|X_0\|}
\\&\qquad\qquad+  (1+e^{-\frac{d}{8}})e^{\lambda\left(1-\kappa + 6\theta \, T \, \frac{\mathsf{A}}{\sqrt{m_2}} \sqrt{M_2}\right)\, C' +  \lambda\left(\frac{1}{\sqrt{2m_2}} + \frac{1}{\sqrt{2}}6\theta \, T \, \frac{\mathsf{A}}{\sqrt{m_2}} \right) \, \sqrt{2d}},
\ee
where Inequality \eqref{eq:c2} is applied in the fourth line with $\gamma = \lambda\left(1 + 6\theta \, T \, \mathsf{A} \right)$  (note that the condition on $\gamma$ is satisfied, since $6\theta \, T \, \mathsf{A} \leq 6\theta \, T \, \frac{\mathsf{A}}{\sqrt{m_2}} \sqrt{M_2} < \kappa < 1$ implies $\gamma < 2\lambda \leq 2 \times \frac{1}{32} \sqrt{2m_2} = \frac{1}{16}\sqrt{2m_2}$). 
Therefore, since $X_0 = X_0' = x$, 
\be 
\mathbb{E}[V(X_1')] \leq (1-\alpha) V(X'_0)  + \beta,
\ee
and hence 
\begin{equation}
Q(x,\cdot)[V] \leq (1-\alpha)\times V(x) + \beta.
\end{equation}
This proves the second line of Inequality \eqref{eq:w6}, completing the proof of the lemma.

\end{proof}

We show that we can use the approximation bound in Lemma \ref{ThmGenericApproxErrorBound2}: 

\begin{lemma}\label{ThmGenericApproxErrorBound4}
Set notation and parameters as in the statement of Lemma \ref{ThmGenericApproxErrorBound3}.  Then, using the same notation, the assumptions \eqref{SimpAss2}, \eqref{IneqContractionForQ} and \eqref{SimpAss1}  of  Lemma \ref{ThmGenericApproxErrorBound2} hold for any choice of $C > 0$ and the choice 
\be 
\delta = 6\theta \, T \, \frac{\mathsf{A}}{\sqrt{m_2}} \, \left(\sqrt{M_2} \lambda^{-1}\log(C) + \frac{\sqrt{d}}{\sqrt{2}} \right). \\
\ee 
\end{lemma} 

\begin{proof}
Inequalities \eqref{SimpAss2} and \eqref{IneqContractionForQ} follow immediately from Theorem \ref{ThmContractionConvexMainResult} and Lemma \ref{ThmGenericApproxErrorBound3}, respectively.  So we need only prove Inequality \eqref{SimpAss1}. Fix $x \in \mathbb{R}^d$ and let $\mathfrak{p}_0 \sim \mathcal{N}(0,I_d)$ be a standard spherical gaussian.  Set $X_1:= q_T(x,\mathfrak{p}_0)$ and $X_1':= q_T^{\dagger \theta}(x,\mathfrak{p}_0)$, so that $X_1\sim K(x,\cdot)$ and $X_1'\sim Q(x,\cdot)$. By Condition \ref{cond:integrator},
\be \label{eq:X1_error}
\|X'_1-X_1\| = \|q_T^{\dagger \theta}(x,\mathfrak{p}_0) - q_T(x,\mathfrak{p}_0)\|\leq 6\theta \, T \, \frac{\mathsf{A}}{\sqrt{m_2}} \sqrt{H(x,\mathfrak{p}_0)}.
\ee

Therefore,  
\be
\sup_{x \, : \, V(x) \leq C} W_1 (K(x,\cdot), Q(x,\cdot)) &\leq \sup_{x \, : \, V(x) \leq C} W_2 (K(x,\cdot), Q(x,\cdot))\\
&\stackrel{\textrm{Eq. } \ref{eq:X1_error}}{\leq} \sup_{x \, : \,e^{\lambda \|x\|} \leq C} \mathbb{E}\left[\left(6\theta \times T \times \frac{\mathsf{A}}{\sqrt{m_2}} \sqrt{H(x,\mathfrak{p}_0)}\right)^2\right]^{\frac{1}{2}}\\
&\stackrel{\textrm{Assumption }\ref{AssumptionsConvexity}}{\leq} \sup_{x \, : \, \|x\| \leq \lambda^{-1}\log(C)} \mathbb{E}\left[\left(6\theta \times T \times \frac{\mathsf{A}}{\sqrt{m_2}} \sqrt{M_2\|x\|^2 + \frac{1}{2}\|\mathfrak{p}_0\|^2}\right)^2\right]^{\frac{1}{2}}\\
&= \mathbb{E}\left[\left(6\theta \times T \times \frac{\mathsf{A}}{\sqrt{m_2}} \sqrt{M_2 (\lambda^{-1}\log(C))^2 + \frac{1}{2}\|\mathfrak{p}_0\|^2}\right)^2\right]^{\frac{1}{2}}\\
&= 6\theta \times T \times \frac{\mathsf{A}}{\sqrt{m_2}} \times \left(M_2 (\lambda^{-1}\log(C))^2 + \frac{1}{2}d \right)^{\frac{1}{2}}\\
&\leq 6\theta \times T \times \frac{\mathsf{A}}{\sqrt{m_2}} \times \left(\sqrt{M_2} \lambda^{-1}\log(C) + \frac{\sqrt{d}}{\sqrt{2}} \right).
\ee
This completes the proof of the lemma.
\end{proof}

We conclude with a bound on the approximation error of $Q$ after many steps:

\begin{lemma}\label{thm:unadjusted_wass}

Set notation and parameters as in Lemma \ref{ThmGenericApproxErrorBound3}, fix $C > \frac{4 \beta}{\alpha}$, and let 

\[\delta =  6\theta \, T \, \frac{\mathsf{A}}{\sqrt{m_2}} \, \left(\sqrt{M_2} \lambda^{-1}\log(C) + \frac{1}{\sqrt{2}}\sqrt{d} \right).\]
Then $Q$ satisfies 
\begin{equation}
W_{1}(Q^\mathcal{I}(x,\cdot), \pi) \leq (1 - \kappa)^{s} \, 2\lambda^{-1}\log(C) + \frac{\delta}{\kappa} + \frac{\beta (s+1)}{\lambda \alpha C}\left(8+2\log(C)\right)
\end{equation} 
for all  $0 \leq s \leq \mathcal{I} \in \mathbb{N}$ and all $x$ satisfying $V(x) \leq \frac{\beta}{\alpha}$. Furthermore, if $Q$ is ergodic with stationary measure $\nu$, 
\begin{equation}
W_{1}(\pi, \nu) \leq \inf_{s \in \{0,1,\ldots\}} \left[ (1 - \kappa)^{s}\times 2\lambda^{-1}\log(C) + \frac{\delta}{\kappa} + \frac{\beta (s+1)}{\lambda \alpha C}\left(8+2\log(C)\right)\right].
\end{equation} 
\end{lemma}
\begin{proof}
Set $V(x)\equiv e^{\lambda \|x\|}$ for all $x \in \mathbb{R}^d$.  The proof now follows by applying Lemmas \ref{ThmGenericApproxErrorBound2} and \ref{ThmGenericApproxErrorBound4}, with constants given in the statement of Lemma \ref{ThmGenericApproxErrorBound4}.
\end{proof}

Define the function
\be \label{EqDefGammaFunc}
\Gamma(a,b) := \left[ 2e^{\frac{1}{32}\left(4-2a \right)\, \frac{16}{a}\log(1+e^{-\frac{b}{8}}) -\frac{7}{8}}\right]
\ee  for $a, b > 0$. We restate Lemma \ref{thm:unadjusted_wass} in a form that will be easier to refer to in subsequent appendices:

\begin{lemma}\label{thm:unadjusted_wass2}
Choose $\epsilon >0$. Fix notation as in Lemma \ref{ThmGenericApproxErrorBound3}, with the additional constraints $0 \leq T \leq \frac{1}{2\sqrt{2}} \frac{\sqrt{m_2}}{M_2}$ and  $\lambda = \frac{\kappa}{64} \sqrt{\frac{m_{2}}{d}}$.
Set
\be 
s &= \frac{1}{\kappa}\log\left(\frac{24\log(15000 \, \Gamma(\kappa,d) \, \lambda^{-1}\epsilon^{-1}\kappa^{-2}) +24\log(\kappa^{-1})}{\lambda \epsilon}\right) \\
C &= (15000 \, \Gamma(\kappa,d) \, \lambda^{-1}\epsilon^{-1}(s+1) \kappa^{-2} +4)^2\\
\theta &\leq \frac{\kappa \epsilon \sqrt{m_2}}{18T \mathsf{A}(\sqrt{M_2} \lambda^{-1}\log(C) + \frac{1}{\sqrt{2}}\sqrt{d})} \\
C' &= (1+\frac{16}{\kappa}\log (1+e^{-\frac{d}{8}})) \, \frac{8}{\kappa} \sqrt{\frac{d}{m_{2}}}.
\ee  
Then
\begin{equation} \label{WassContConc5}
W_{1}(Q^\mathcal{I}(x,\cdot), \pi) \leq \epsilon
\end{equation} 
for all $\mathcal{I} \geq s$ and all $x$ satisfying $\|x\|  \leq \frac{\sqrt{d}}{\sqrt{m_2}}$. Furthermore, if $Q$ is ergodic with stationary measure $\nu$, 
\begin{equation} \label{WassContConc6}
W_{1}(\pi, \nu) \leq \epsilon.
\end{equation} 
\end{lemma}

\begin{proof}
By our assumption about the value of $\theta$, and recalling that $0<\kappa <1$ and $0<m_2\leq M_2$, we have 
\be\label{eq:w12}
6\theta \, T \, \frac{\mathsf{A}}{\sqrt{m_2}} \sqrt{M_2} \leq \frac{1}{2}\kappa
\ee
and
\be\label{eq:w12b}
6\theta \, T \, \frac{\mathsf{A}}{\sqrt{m_2}} \leq \frac{1}{2\sqrt{M_2}}\kappa \leq \frac{1}{\sqrt{m_2}}.
\ee

Hence by Inequalities \eqref{eq:w12} and \eqref{eq:w12b}, we have after some algebraic manipulation

\be 
1- \alpha &= (1+e^{-\frac{d}{8}})e^{\lambda\left(-\kappa + 6\theta \, T \, \frac{\mathsf{A}}{\sqrt{m_2}} \sqrt{M_2}\right)\, C' +  \lambda\left(1 + 6\theta \, T \, \mathsf{A}\right) \, \sqrt{\frac{d}{m_{2}}}} \\
&\leq (1+e^{-\frac{d}{8}})e^{-\frac{1}{2}\lambda\kappa C' +  \lambda\left(\frac{1}{\sqrt{2m_2}} + \frac{1}{\sqrt{2m_2}}\right) \times \sqrt{2d}}\\
&=  (1+e^{-\frac{d}{8}})e^{-\frac{1}{2}\lambda\kappa \times(1+\frac{16}{\kappa} \log (1+e^{-\frac{d}{8}})) \times \frac{4}{\kappa} \times \left(\frac{1}{\sqrt{2m_2}} + \frac{1}{\sqrt{2m_2}}\right) \times \sqrt{2d} +  \lambda\left(\frac{1}{\sqrt{2m_2}} + \frac{1}{\sqrt{2m_2}}\right) \times \sqrt{2d}}\\
&\leq e^{-\frac{1}{32}\kappa}.
\ee
Also by Inequalities \eqref{eq:w12} and \eqref{eq:w12b},
\be 
\beta &=  e^{\lambda C'} (1 - \alpha) \\
&\leq  (1+e^{-\frac{d}{8}})e^{\lambda\left(1-\frac{1}{2}\kappa \right)\times C' +2 \lambda \sqrt{\frac{d}{m_{2}}}}\\
&=  (1+e^{-\frac{d}{8}})e^{\lambda\left(1-\frac{1}{2}\kappa \right)\times (1+\frac{16}{\kappa}\log (1+e^{-\frac{d}{8}}))\times \frac{4}{\kappa} \times \left(\frac{1}{\sqrt{2m_2}} + \frac{1}{\sqrt{2m_2}}\right) \times \sqrt{2d} +\lambda\left(\frac{1}{\sqrt{2m_2}} + \frac{1}{\sqrt{2m_2}}\right) \times \sqrt{2d}}\\
&=  (1+e^{-\frac{d}{8}})e^{\frac{1}{32}\left(4-2\kappa \right)\times(1+\frac{16}{\kappa}\log (1+e^{-\frac{d}{8}}))+\frac{1}{32}\kappa}\\
&=  (1+e^{-\frac{d}{8}})e^{\frac{1}{32}\left(4-2\kappa \right)\times\frac{16}{\kappa}\log (1+e^{-\frac{d}{8}})} \times e^{\frac{1}{32}(4-\kappa)}\\
&\leq 2e^{\frac{1}{32}\left(4-2\kappa \right)\times\frac{16}{\kappa}\log (1+e^{-\frac{d}{8}})} \times e^{\frac{1}{32}(4-\kappa)}.
\ee

Combining these two calculations, 

\be
\frac{\beta}{\alpha} &\leq \frac{ 2e^{\frac{1}{32}\left(4-2\kappa \right)\times\frac{16}{\kappa}\log(1+e^{-\frac{d}{8}})} \times e^{\frac{1}{32}(4-\kappa)}}{1-e^{-\frac{1}{32}\kappa}}\\
&= \left[2 e^{\frac{1}{32}\left(4-2\kappa \right)\times\frac{16}{\kappa}\log(1+e^{-\frac{d}{8}}) -\frac{7}{8}}\right] \times \frac{e^{1-\frac{1}{32}\kappa}}{1-e^{-\frac{1}{32}\kappa}}\\
&= \Gamma(\kappa,d)\times \frac{e^{1-\frac{1}{32}\kappa}}{1-e^{-\frac{1}{32}\kappa}}.
\ee
Since $T \leq \frac{1}{2\sqrt{2}} \frac{\sqrt{m_2}}{M_2}$, we have $0 < \frac{1}{32}\kappa \leq \frac{1}{32}\times \frac{1}{16} < \frac{1}{3}$ and so this implies
\be \label{eq:w13}
\frac{\beta}{\alpha} &\leq \Gamma(\kappa,d) \times \frac{e^{1-\frac{1}{32}\kappa}}{1-e^{-\frac{1}{32}\kappa}} \leq \Gamma(\kappa,d) \times \frac{1}{(\frac{1}{32}\kappa)^{2}}.
\ee
Note that $C > 4$, so we have the easy inequality
 \be 
C &= (15000 \, \Gamma(\kappa,d) \, \lambda^{-1}\epsilon^{-1}(s+1) \kappa^{-2} +4)^2 \\
&\geq 15000 \, \Gamma(\kappa,d) \, \lambda^{-1}\epsilon^{-1}(s+1) \kappa^{-2} \, \left(8+ \log\left(C\right)\right)
\ee 
Applying Inequality \eqref{eq:w13}, this implies
\be \label{eq:w13'}
\frac{\beta (s+1)}{\lambda \alpha C}\left(8+2\log(C)\right) \leq \frac{1}{3}\epsilon.
\ee

Our assumption on $\theta$ also implies that
\be\label{eq:w14}
\delta \leq \frac{1}{3} \kappa\epsilon.
\ee
Finally, our assumption about $s$ implies that $s \geq \frac{1}{\kappa}\log\left(\frac{6 \log(C)}{\lambda \epsilon}\right)$ and hence that 
\be\label{eq:w15}
(1 - \kappa)^{s} \times 2\lambda^{-1}\log(C) \leq \frac{1}{3} \epsilon.
\ee

Therefore, Inequalities \eqref{eq:w13'}, \eqref{eq:w14}, and \eqref{eq:w15} together imply that 
\be\label{eq:w16}
(1 - \kappa)^{s} \times 2\lambda^{-1}\log(C) + \frac{\delta}{\kappa} + \frac{\beta (s+1)}{\lambda \alpha C}\left(8+2\log(C)\right) \leq \frac{1}{3}\epsilon +  \frac{1}{3}\epsilon + \frac{1}{3}\epsilon = \epsilon.
\ee

Inequality \eqref{eq:w13} also implies that $C > 4 \frac{\beta}{\alpha}$.  Moreover, since $0<\kappa \leq \frac{1}{16}$, our choice of $\lambda$ satisfies $0< \lambda \leq \frac{1}{32}\sqrt{2m_2}$. Therefore, by Lemma \ref{thm:unadjusted_wass} and Inequality \eqref{eq:w16}, for all $x$ satisfying $V(x) \leq \frac{\beta}{\alpha}$, we have
\be
W_{1}(Q^\mathcal{I}(x,\cdot), \pi) \leq \epsilon
\ee 
and
\be
W_{1}(\pi, \nu) \leq \epsilon.
\ee

This completes the proof of the desired mixing bounds for starting point $x$ satisfying $V(x) \leq \frac{\beta}{\alpha}$. To complete the proof of the theorem, then, it is enough to show that $V(x) \leq \frac{\beta}{\alpha}$ for all $x$ satisfying $\|x\|  \leq \sqrt{\frac{d}{m_{2}}}$.  To do so, we must find a lower bound for $\frac{\beta}{\alpha}$.

By Inequality \eqref{eq:beta_max} and the trivial bounds $\alpha \leq 1$ and $\lambda, (1- \kappa), C' \geq 0$, we have 

\be 
\frac{\beta}{\alpha} &\geq \beta \\
&\geq (1+e^{-\frac{d}{8}})e^{\lambda\frac{\sqrt{2d}}{\sqrt{2m_2}}}e^{\lambda(1-\kappa)C'} \\
&\geq e^{\lambda\frac{\sqrt{d}}{\sqrt{m_2}}}.
\ee 

Hence, $V(x) \leq \frac{\beta}{\alpha}$ if
\be
e^{\lambda \|x\|} \leq e^{\lambda\frac{\sqrt{d}}{\sqrt{m_2}}}.
\ee
That is, 
\be \label{IneqStartCondSuffCond}
\{x \, : \, \|x\| \leq \sqrt{\frac{d}{m_{2}}} \} \subset \{x \, : \, V(x) \leq \frac{\beta}{\alpha} \}.
\ee
This completes the proof.
\end{proof}

\section{Unadjusted HMC and Metropolis-Adjusted HMC} \label{SecMH}

We first set notation to be used throughout this appendix. Fix constants $\mathsf{m}, d \in \mathbb{N}$ and a potential function $U$ that satisfies Assumptions \ref{assumption:product_measure_potential}, \ref{assumption:leapfrog} and \ref{AssumptionsConvexity} with $\mathcal{X} = \mathbb{R}^{d}$. Also fix an integrator $\sharp$ that satisfies condition \ref{condition:kth_order}.

In order to define the remaining objects used in this section, we need to fix a collection of constants that satisfy the following equalities and inequalities; for ease of reference we collect them here.  First, we allow $0 < \epsilon, \, \epsilon''' < e^{-1}$. Recalling the definition of $\Gamma$ from Equation \eqref{EqDefGammaFunc}, set: 
\be
T&\leq \frac{1}{2 \sqrt{2}} \frac{\sqrt{m_{2}}}{M_{2}} \\
\kappa &=  \frac{1}{8} (\sqrt{m}_2T)^2\\
\lambda &= \frac{\kappa}{64} \sqrt{\frac{m_{2}}{\mathsf{m}}}\\
s &= \frac{1}{\kappa}\log\left(\frac{24\log(15000  \, \lambda^{-1}\epsilon^{-1}\kappa^{-2} \, \Gamma(\kappa,\mathsf{m})) +24\log(\kappa^{-1})}{\lambda \epsilon}\right) \\
\mathcal{I} &\geq s \\
C &= (15000 \, \lambda^{-1}\epsilon^{-1}(s+1) \kappa^{-2} \, \Gamma(\kappa,\mathsf{m}) +4)^2\\
\mathsf{g}_\infty &= \max \left[ \lambda^{-1} \log\left(\frac{2048\, \mathcal{I}\, \Gamma(\kappa,\mathsf{m})}{\kappa^{2}}, \,\frac{d \log(\frac{1}{\epsilon'''})}{\mathsf{m} \epsilon'''}    \right), \, \, \, \sqrt{2\mathsf{m} - 8\log \left(\frac{\mathsf{m} \epsilon'''}{\mathcal{I} d \, \log(\frac{1}{\epsilon'''})} \right)} \right]\\
\mathsf{g}_2 &=  \sqrt{d-2\sqrt{d}\log^{\frac{1}{2}} \left(\frac{\mathcal{I}\log(\frac{1}{\epsilon'''})}{\epsilon'''} \right)}. \\
\theta &\leq \min \left( \frac{\kappa \epsilon \sqrt{m_2}}{18T \mathsf{A}(\sqrt{M_2} \lambda^{-1}\log(C) + \frac{1}{\sqrt{2}}\sqrt{\mathsf{m}})}, \, \,  \frac{\mathsf{m} \, M_{2} T}{7 d \, \mathsf{g}_\infty^2 \, \mathsf{A} (M_2 + \frac{1}{2})} \,   \log\left([1-\frac{\epsilon'''}{(\mathcal{I}+1) \log(\frac{1}{\epsilon'''})}]^{-1}\right)  \right) \\
\mathsf{A} &= \max(\mathsf{K}''', M_2)\\ 
C' &= (1+\frac{16}{\kappa}\log (1+e^{-\frac{\mathsf{m}}{8}}))\, \frac{8 \sqrt{\mathsf{m}}}{\kappa \sqrt{m_{2}}} \\
\alpha &= 1-  (1+e^{-\frac{\mathsf{m}}{8}})e^{\lambda\left(-\kappa + 6\theta \, T \, \frac{\mathsf{A}}{\sqrt{m_2}} \sqrt{M_2}\right)\, C' +  \lambda\left(1 + 6\theta \, T \, \mathsf{A} \right) \, \sqrt{\frac{\mathsf{m}}{m_{2}}} }\\
\beta &= e^{\lambda C'} (1 - \alpha),
\ee
where $\mathsf{K}'''$ is as in Equation \eqref{DefKTriple}. Also define the function $V(x):=e^{\lambda \|x\|}$. Note that this notation agrees with Lemma \ref{thm:unadjusted_wass2} of Appendix \ref{SecAppendixMixApprox}, but we replace every instance of the dimension $d$ with $\mathsf{m}$. Finally, let  $\mathsf{G}$ be the ``good set" defined in Equation \eqref{eq:good_set}, and let $\diamondsuit$ be the modified ``toy" integrator given in Definition  \ref{defn:toy} with these parameters. We note that, by Lemma \ref{thm:leapfrog5}, $\diamondsuit$ satisfies Condition \ref{cond:integrator} for this choice of $\mathsf{A}$.

We now define a coupling of two Markov chains $\{X_{i}',X_{i}''\}_{ i \geq 0}$, where marginally $\{X_{i}'\}_{i \geq 0}$ evolves according to Algorithm \ref{alg:Unadjusted} with parameters as above and numerical integrator $\dagger = \diamondsuit$, and  $\{X_{i}''\}_{i \geq 0}$ evolves according to Algorithm \ref{alg:Metropolis} with parameters as above and numerical integrator $\dagger = \sharp$. As mentioned immediately following Algorithm \ref{DefSimpleHMC}, these algorithms define a \textit{random mapping representation} of the associated HMC schemes. Thus, in order to define a coupling of these chains, it is enough to specify the following information:

\begin{enumerate}
\item the distribution of the starting points $X_{0}', X_{0}''$, \textit{and}
\item a coupling of the sequence of random momentum variables $\{\textbf{p}_{i}\}_{i \geq 0}$ chosen in Step 2 of each of these algorithms, \textit{and} 
\item a coupling of the sequence of random variables $\{ b_{i} \}_{i \geq 0}$ chosen in Step 4 of Algorithm \ref{alg:Metropolis}.
\end{enumerate}  

We make the following simple choice of coupling:
\begin{enumerate}
\item We fix a single $x \in \mathbb{R}^{d}$ satisfying $\|x \| \leq \sqrt{\frac{d}{m_{2}}}$ and set $X_{0}'= X_{0}'' = x.$
\item We draw a single i.i.d. sequence $\{\textbf{p}_{i}\}_{i \geq 0}$ of standard Gaussian random variables, and use this single update sequence for both Markov chains.
\item We sample the i.i.d. sequence $\{b_{i} \}_{i \geq 0}$ independently of the sequence $\{\textbf{p}_{i}\}_{i \geq 0}$ .
\end{enumerate}

We observe that since the potential is of the form given in Assumption \ref{assumption:product_measure_potential}, The unadjusted chain can be written in the form

\[X'_i = (X'^{(1)}_i, X'^{(2)}_i, \ldots, X'^{(\frac{d}{\mathsf{m}})}_i)^\top \]

where $X'^{(1)}_i, X'^{(2)}_i, \ldots, X'^{(\frac{d}{\mathsf{m}})}_i$ are independent Markov chains with state space $\mathbb{R}^{\mathsf{m}}$.

Finally, let $\mathfrak{t}_G = \inf \{ h \, : \, (X'_h, \textbf{p}_{h}) \notin \mathsf{G}\}$ be the exit time of the unadjusted chain from $\mathsf{G}$, and let $\mathfrak{t}_{\mathrm{reject}} = \inf \{ h \geq 0 \, : \, X_{h+1}'' = X_{h}''\}$ be the first time that Algorithm \ref{alg:Metropolis} rejects a proposal. Notice that, from our couplings, we have $X'_i = X''_i$ for all $i \leq \min(\mathfrak{t}_G, \mathfrak{t}_{\mathrm{reject}})$.

Keeping this notation for the remainder of the appendix, we begin to state the main results.

\begin{lemma} \label{thm:MH1}
Fix $j \in \{1, \ldots, \frac{d}{\mathsf{m}}\}$, $\epsilon'' > 0$ and 
\be
R = \lambda^{-1} \log\left( 2048 (\epsilon'')^{-1} \, \mathcal{I} \, \kappa^{-2} \, \Gamma(\kappa,\mathsf{m}) \right).
\ee
Then
\be
\mathbb{P}[\sup_{0 \leq h \leq \mathcal{I}-1} \|X'^{(j)}_{h}\| \geq R] \leq \epsilon''.
\ee
\end{lemma}

\begin{proof}

By Inequality \eqref{eq:w6} of Lemma \ref{ThmGenericApproxErrorBound3} (and Remark \ref{RemGenericApproxErrorBound3}), we have for ${h} \in \mathbb{N}$
\be
\mathbb{E}[V(X'^{(j)}_{{h}})] & \leq (1 - \alpha) \mathbb{E}[V(X'^{(j)}_{{h}-1})] + \beta \\
&\leq \ldots \\
&\leq (1 - \alpha)^{{h}} V(X'^{(j)}_{0}) + \frac{\beta}{\alpha} \\
&\leq \frac{2 \beta}{\alpha},
\ee 
where in the last line we use Inequality \eqref{IneqStartCondSuffCond}. By Markov's inequality, then,
\be 
\mathbb{P}[\sup_{0 \leq h \leq \mathcal{I}-1} V(X'^{(j)}_{h}) \geq r] \leq \frac{2 \beta \mathcal{I}}{\alpha  r}.
\ee 
Applying this bound with $r=R$, we have by Inequality  \eqref{eq:w13} that
\be
\mathbb{P}[\sup_{0 \leq h \leq \mathcal{I}-1} \|X'^{(j)}_{h}\| \geq R] &\leq \frac{2 \beta \mathcal{I}}{\alpha} e^{-\lambda R} \\
&= \epsilon'' \, \times \frac{\beta}{\alpha} \times (\frac{1}{32} \kappa)^{2} \times \Gamma(\kappa, \mathsf{m})^{-1} \\
&\stackrel{{\scriptsize \textrm{Eq. }} \ref{eq:w13}}{\leq} \epsilon''. 
\ee 

\end{proof}

Recall the definition of the ``good set"
\be
\mathsf{G}  = \{(q,p) \in \mathbb{R}^d\times \mathbb{R}^d :  \max_{1\leq i \leq \frac{d}{\mathsf{m}}}\|q^{(i)}\|< \mathsf{g}_\infty, \max_{1\leq i \leq \frac{d}{\mathsf{m}}}\|p^{(i)}\| < \mathsf{g}_\infty, ||p|| > \mathsf{g}_2 \}.
\ee

The next three Lemmas show that, with high probability, a chain evolving according Algorithm \ref{alg:Unadjusted} using the ``toy" integrator \eqref{eq:toy} will stay in $\mathsf{G}$ for many steps:

\begin{lemma} \label{thm:MH2'}
With notation as above,
\be
\mathbb{P}[\inf_{0 \leq h \leq \mathcal{I}-1} \|\mathbf{p}_{h}\| \leq \mathsf{g}_2] \leq \frac{\epsilon'''}{\log(\frac{1}{\epsilon'''})}.
\ee
\end{lemma}
\begin{proof}
$\|\mathbf{p}_{h}\| \sim \chi_d$, the Chi-square distribution with $d$ degrees of freedom. Thus, by the Hanson-Wright inequality (see Lemma 1 of \cite{laurent2000adaptive}),
\be
\mathbb{P}[d- \|\mathbf{p}_{h}\|^2 \geq 2\sqrt{d}\sqrt{r}] \leq e^{-r} \quad \quad \forall r>0.
\ee
Rearranging and reparameterizing, this gives 
\be \label{eq:Hanson1}
\mathbb{P}[\|\mathbf{p}_{h}\| \leq r] \leq e^{-\frac{1}{4d}(d-r^2)^2} \quad \quad \forall \, 0<r<\sqrt{d},
\ee
so that
\be \label{eq:Hanson2}
\mathbb{P}[\inf_{0 \leq h \leq \mathcal{I}-1} \|\mathbf{p}_{h}\| \leq r] &\leq \sum_{{h}=0}^{\mathcal{I}-1} \mathbb{P}(\|\mathbf{p}_{h}\| \leq r)\\
& \stackrel{{\scriptsize \textrm{Eq. }}\ref{eq:Hanson1}}{\leq} \sum_{{h}=0}^{\mathcal{I}-1} e^{-\frac{1}{4d}(d-r^2)^2}\\
& = \mathcal{I}e^{-\frac{1}{4d}(d-r^2)^2} \quad \quad \forall 0<r<\sqrt{d}.\\
\ee
Inequality \eqref{eq:Hanson2} immediately implies the desired bound
\be
\mathbb{P}[\inf_{0 \leq h \leq \mathcal{I}-1} \|\mathbf{p}_{h}\| \leq \mathsf{g}_2] \leq \frac{\epsilon'''}{\log(\frac{1}{\epsilon'''})}.
\ee
\end{proof}

\begin{lemma} \label{thm:MH2}
With notation as above,
\be
\mathbb{P}[\sup_{1\leq i\leq \frac{d}{\mathsf{m}}} \sup_{0 \leq h \leq \mathcal{I}-1} \|X'^{(i)}_{h}\| \geq \mathsf{g}_\infty] \leq \frac{\epsilon'''}{\log(\frac{1}{\epsilon'''})}.
\ee

\end{lemma}
\begin{proof}
Let $\epsilon'' = \frac{\mathsf{m}}{d} \times \frac{\epsilon'''}{\log(\frac{1}{\epsilon'''})}$, and let $R =\lambda^{-1} \log\left(\frac{2048\, \mathcal{I}\, \Gamma(\kappa,\mathsf{m})}{\kappa^{2}}, \,\frac{d \log(\frac{1}{\epsilon'''})}{\mathsf{m} \epsilon'''}    \right) \leq \mathsf{g}_\infty$. Then

\be
\mathbb{P}[\sup_{1\leq i\leq \frac{d}{\mathsf{m}}} \sup_{0 \leq h \leq \mathcal{I}-1} \|X'^{(i)}_{h}\| \geq \mathsf{g}_\infty] &=\mathbb{P}\left[\bigcup_{1\leq i\leq \frac{d}{\mathsf{m}}} \left\{\sup_{0 \leq h \leq \mathcal{I}-1} \|X'^{(i)}_{h}\| \geq \mathsf{g}_\infty\right\}\right] \\
&\leq \sum_{i=1}^{\frac{d}{\mathsf{m}}} \mathbb{P}[\sup_{0 \leq h \leq \mathcal{I}-1} \|X'^{(i)}_{h}\| \geq \mathsf{g}_\infty]\\
&\leq \sum_{i=1}^{\frac{d}{\mathsf{m}}} \mathbb{P}[\sup_{0 \leq h \leq \mathcal{I}-1} \|X'^{(i)}_{h}\| \geq R]\\
&\stackrel{{\scriptsize \textrm{Lemma }}\ref{thm:MH1}}{\leq} \sum_{i=1}^{\frac{d}{\mathsf{m}}} \frac{\mathsf{m}}{d} \times \frac{\epsilon'''}{\log(\frac{1}{\epsilon'''})}\\
&=\frac{\epsilon'''}{\log(\frac{1}{\epsilon'''})}.
\ee
\end{proof}

\begin{lemma} \label{LemmaMH5}
With notation as above,
\be
\mathbb{P}[\sup_{1\leq i\leq \frac{d}{\mathsf{m}}} \sup_{0 \leq h \leq \mathcal{I}-1} \|\mathbf{p}^{(i)}_{h}\| \geq \mathsf{g}_\infty] \leq \frac{\epsilon'''}{\log(\frac{1}{\epsilon'''})}.
\ee
\end{lemma}
\begin{proof}
Noting  $\|\mathbf{p}^{(i)}_h\|^2 \sim \chi^2_{\mathsf{m}}$, we have by the Hanson-Wright concentration inequality (see \textit{e.g.} \cite{hanson1971bound, rudelson2013hanson}):
\be
\mathbb{P}[\|\mathbf{p}^{(i)}_h\|^2 \geq x+\mathsf{m}]  \leq e^{-\frac{x}{8}} \quad \quad  \textrm{ for } x> \mathsf{m},
\ee
so
\be
\mathbb{P}[\|\mathbf{p}^{(i)}_h\| \geq r]  \leq e^{-\frac{r^2-\mathsf{m}}{8}} \quad \quad  \textrm{ for } r> \sqrt{2\mathsf{m}},
\ee
and so
\be \label{eq:Hanson3}
\mathbb{P}[\|\mathbf{p}^{(i)}_h\| \geq \mathsf{g}_\infty]  \leq \frac{\mathsf{m}}{\mathcal{I} d} \times \frac{\epsilon'''}{\log(\frac{1}{\epsilon'''})},
\ee
for all $h,i$. Therefore,
\be
\mathbb{P}[\sup_{1\leq i\leq \frac{d}{\mathsf{m}}} \sup_{0 \leq h \leq \mathcal{I}-1} \|\mathbf{p}^{(i)}_{h}\| \geq \mathsf{g}_\infty] &\leq \sum_{{h}=0}^{\mathcal{I}-1} \sum_{i=1}^{\frac{d}{\mathsf{m}}}  \mathbb{P}[\|\mathbf{p}^{(i)}_h\| \geq \mathsf{g}_\infty]\\
&\stackrel{\textrm{Eq. } \ref{eq:Hanson3}}{\leq} \sum_{{h}=0}^{\mathcal{I}-1} \sum_{i=1}^{\frac{d}{\mathsf{m}}} \frac{\mathsf{m}}{\mathcal{I} d} \times \frac{\epsilon'''}{\log(\frac{1}{\epsilon'''})}\\
&=  \frac{\epsilon'''}{\log(\frac{1}{\epsilon'''})}.
\ee

\end{proof}

The following Lemma shows that the rejection probability is small inside the set $\mathsf{G}$: 

\begin{lemma} \label{thm:rejection}
With notation as above, 
\be
\mathbb{P}[\{ \mathfrak{t}_{\mathrm{reject}} \leq \mathcal{I}\} \cap \{\mathfrak{t}_{G} > \mathcal{I} \}] < \frac{\epsilon'''}{\log(\frac{1}{\epsilon'''})}.
\ee
\end{lemma}
\begin{proof}
Whenever $(X_i'',\mathbf{p}_i) \in \mathsf{G}$, we have $\|X_i'' \|, \|\mathbf{p}_i\| \leq \sqrt{\frac{d}{\mathsf{m}}}\mathsf{g}_\infty$, and so
\be \label{eq:rejection1}
H(X_i'',\mathbf{p}_i) &\stackrel{\textrm{Assumption } \ref{AssumptionsConvexity}}{\leq} M_2 \|X_i'' \|^2 + \frac{1}{2} \|\mathbf{p}_i\|^2\\
& \leq (M_2 + \frac{1}{2})\frac{d}{\mathsf{m}}\mathsf{g}_\infty^2
\ee

Define the change in energy $\Delta \hat{E}(X_i'', \mathbf{p}_i)  := H(q_T^{\sharp \theta}(X_{i}'', \mathbf{p}_i), p_T^{\sharp \theta}(X_{i}'',\mathbf{p}_i))-  H(X_{i}'',\mathbf{p}_i)$. Recall that $\diamondsuit(q,p) = \sharp(q,p)$ for $(q,p) \in \mathsf{G}$, and furthermore that $\diamondsuit$ satisfies Condition \ref{cond:integrator} for our choice of $\mathsf{A}$ by Lemma \ref{thm:leapfrog5}. Thus, for $ (X_{i}'',\mathbf{p}_i) \in \mathsf{G}$, we can apply  Condition \ref{cond:integrator} and Inequality \eqref{eq:rejection1} to find:
\be \label{eq:rejection2}
\Delta \hat{E}(X_{i}'', \mathbf{p}_i) &\leq 7 \frac{\theta}{T} H(X_{i}'',\mathbf{p}_i) \times \frac{\mathsf{\mathsf{A}}}{M_2}\\
&\stackrel{\textrm{Eq. } \ref{eq:rejection1}}{\leq} 7 \frac{\theta}{T} (M_2 + \frac{1}{2})\frac{d}{\mathsf{m}}\mathsf{g}_\infty^2 \times \frac{\mathsf{\mathsf{A}}}{M_2}
\\ &\leq \log\left([1-\frac{\epsilon'''}{(\mathcal{I}+1)\log(\frac{1}{\epsilon'''})}]^{-1}\right).
\ee
Defining 
\be \label{EqRejTime}
\mathcal{T}_{\mathrm{reject}} = \{i \geq 0 \,: \, X_{i+1}'' = X_{i}''\}
\ee to be the collection of times at which the proposed step is rejected. We then have:
\be
\mathbb{P}[\{ \mathfrak{t}_{\mathrm{reject}} \leq \mathcal{I}\} \cap \{\mathfrak{t}_{G} > \mathcal{I} \}]  
&\leq \sum_{i=0}^{\mathcal{I}} \mathbb{P}[\{ i \in \mathcal{T}_{\mathrm{reject}}\} \cap\{(X_{i}'',\mathbf{p}_i) \in \mathsf{G}\}]\\
&\leq \sum_{i=0}^{\mathcal{I}} \mathbb{P}[i \in \mathcal{T}_{\mathrm{reject}} | (X_{i}'',\mathbf{p}_i) \in \mathsf{G}]\\
&= \sum_{i=0}^{\mathcal{I}} ( 1-\mathbb{P}[i \notin \mathcal{T}_{\mathrm{reject}}| (X_{i}'',\mathbf{p}_i) \in \mathsf{G}] )\\
&=\sum_{i=0}^{\mathcal{I}} \E[ 1-\min(e^{-\hat{E}(X_{i}'', \mathbf{p}_i)}, \, \, 1) | (X_{i}'',\mathbf{p}_i) \in \mathsf{G}]\\
&\stackrel{\textrm{Eq. } \ref{eq:rejection2}}{\leq} \sum_{i=0}^{\mathcal{I}} \frac{\epsilon'''}{(\mathcal{I}+1)\log(\frac{1}{\epsilon'''})}\\
&= \frac{\epsilon'''}{\log(\frac{1}{\epsilon'''})}.
\ee

\end{proof}

In the rest of this section, let $E_i = H(X''_i,\mathbf{p}_i)$ be the energy at the beginning of step $i$ for every $i \geq 0$. We now derive inequalities that will allow us to bound the Wasserstein mixing profile of the Metropolis adjusted chain $X''_0, X''_1, \ldots$, beginning with a weak bound on the total energy:

\begin{lemma}\label{thm:MH3}
\be
\mathbb{P}[E_i> y]  \leq  2(i+1)e^{-\frac{1}{8}(\frac{1}{2(i+1)}y-d)} \quad \quad  \textrm{ for } y> \max\left(4(i+1)d, M_2 \|X''_0\|^2\right).
\ee
\end{lemma}
\begin{proof}
Define $\Delta E_i := E_{i+1}-E_{i}$ and recall the set $\mathcal{T}_{\mathrm{reject}}$ of times at which a proposal is rejected, as defined in Equation \eqref{EqRejTime}. If proposal $i$ is accepted (\textit{i.e.} $i \notin \mathcal{T}_{\mathrm{reject}}$), then $\Delta E_i$ is the sum of the change in energy when refreshing the momentum and the numerical integrator error in conservation of energy.  If proposal $i$ is rejected (\textit{i.e.} $i \notin \mathcal{T}_{\mathrm{reject}}$), then $\Delta E_i$ is just the change in energy when refreshing the momentum.  Therefore,

\be \label{eq:d1}
\Delta E_i &\leq  \frac{1}{2}\|\mathbf{p}_{i+1}\|^2 + \Delta \hat{E}_i \times \mathbbm{1}\{i \notin \mathcal{T}_{\mathrm{reject}}\}
\ee
where $\Delta \hat{E}_i = H(q_T^{\sharp \theta}(X''_i,\mathbf{p}_i), p_T^{\sharp \theta}(X''_i,\mathbf{p}_i)) - H(X''_i,\mathbf{p}_i)$.

Noting that $\mathbb{P}[i \notin \mathcal{T}_{\mathrm{reject}}| \Delta \hat{E}_{i}]  = \min(1, e^{-\Delta \hat{E}_i})$,
\be\label{eq:d2}
\mathbb{P}[\Delta \hat{E}_i \times \mathbbm{1}\{i \notin \mathcal{T}_{\mathrm{reject}}\} > y] &= \mathbb{P}[i \notin \mathcal{T}_{\mathrm{reject}} | \Delta \hat{E}_i>y]\, \mathbb{P}[\Delta \hat{E}_i>y]\\
&\leq \mathbb{P}[i \notin \mathcal{T}_{\mathrm{reject}}| \Delta \hat{E}_i>y]\\
&\leq e^{-y} \quad \quad \forall y \geq 0.
\ee
This gives 

\be \label{eq:MH1}
\mathbb{P}[\Delta E_i > y] &\stackrel{\textrm{Eq. } \ref{eq:d1}}{\leq} \mathbb{P}\left[\frac{1}{2}\|\mathbf{p}_i\|^2 + \Delta \hat{E}_i\times \mathbbm{1}\{i \notin \mathcal{T}_{\mathrm{reject}}\} > y\right]\\
&\leq \mathbb{P}\left[\left \{\frac{1}{2}\|\mathbf{p}_i\|^2>\frac{1}{2}y\right\}\cup\bigg\{\Delta \hat{E}_i\times \mathbbm{1}\{i \notin \mathcal{T}_{\mathrm{reject}}\} > \frac{1}{2}y\bigg\}\right]\\
&\stackrel{\textrm{Eq. } \ref{eq:d2}}{\leq} \mathbb{P}[\frac{1}{2}\|\mathbf{p}_i\|^2>\frac{1}{2}y] + e^{-\frac{1}{2}y} \quad \quad \forall y \geq 0.\\
\ee

Noting  $\|\mathbf{p}_i\|^2 \sim \chi^2_d$, we have by the Hanson-Wright concentration inequality (see, \textit{e.g.}, \cite{hanson1971bound, rudelson2013hanson}):
\be
\mathbb{P}[\|\mathbf{p}_i\|^2>x+d]  \leq e^{-\frac{x}{8}} \quad \quad  \textrm{ for } x> d,
\ee
so 
\be\label{eq:d3}
\mathbb{P}[\frac{1}{2}\|\mathbf{p}_i\|^2>\frac{1}{2}y]  \leq e^{-\frac{1}{8}(y-d)}  \quad \quad  \textrm{ for } y> 2d.
\ee
So by Inequalities \eqref{eq:MH1} and \eqref{eq:d3},
\be \label{eq:d4}
\mathbb{P}[\Delta E_i > y] \leq e^{-\frac{1}{8}(y-d)} + e^{-\frac{1}{2}y} \quad \quad  \textrm{ for } y> 2d.
\ee

Therefore, 
\be \label{eq:d5}
\mathbb{P}[E_i-E_0>\frac{i}{i+1}y] &=\mathbb{P}[\sum_{j=0}^{i-1} \Delta E_j > \frac{i}{i+1}y]\\
&\leq \sum_{j=0}^{i-1}\mathbb{P}[\Delta E_j > \frac{1}{i+1}y]\\
&\stackrel{{\scriptsize \textrm{Eq.  }}\ref{eq:d4}}{\leq} i\times(e^{-\frac{1}{8}(\frac{1}{i+1}y-d)} + e^{-\frac{1}{2(i+1)}y}) \quad \quad  \textrm{ for } y> 2(i+1)d.
\ee

By Assumption \ref{AssumptionsConvexity}, 
\be \label{eq:d_X0}
U(X''_0) \stackrel{\textrm{Assumption } \ref{AssumptionsConvexity}}{\leq} M_2 \|X''_0\|^2 \leq M_2(\frac{\sqrt{d}}{\sqrt{m_2}})^2 = M_2 \frac{d}{m_2}.
\ee
But $E_0 = U(X''_0) + \frac{1}{2}\|\mathbf{p}_0\|^2$, so by Inequality \eqref{eq:d3} and Assumption \ref{AssumptionsConvexity},
\be \label{eq:d6}
\mathbb{P}(E_0> M_2 \|X''_0\|^2 + \frac{1}{i+1}y)  \leq e^{-\frac{1}{8}(\frac{2}{i+1}y-d)}  \quad \quad  \textrm{ for } y> (i+1)d.
\ee

Therefore, 
\be
\mathbb{P}[E_i> M_2 \|X''_0\|^2 + y]  &\leq \mathbb{P}[E_0> M_2 \|X''_0\|^2 + \frac{1}{i+1}y] + \mathbb{P}[E_i-E_0>\frac{i}{i+1}y]\\
&\stackrel{{\scriptsize \textrm{Eq.  }}\ref{eq:d5}, \ref{eq:d6}}{\leq} e^{-\frac{1}{8}(\frac{2}{i+1}y-d)}  +i\times(e^{-\frac{1}{8}(\frac{1}{i+1}y-d)} + e^{-\frac{1}{2(i+1)}y}) \quad \quad  \textrm{ for } y> 2(i+1)d.
\ee

Hence, for $y> \max\left(2(i+1)d, M_2 \|X''_0\|^2\right)$, we have
\be
\mathbb{P}[E_i> 2y]  &\leq  e^{-\frac{1}{8}(\frac{2}{i+1}y-d)}  +i\times(e^{-\frac{1}{8}(\frac{1}{i+1}y-d)} + e^{-\frac{1}{2(i+1)}y})\\
&\leq  (i+1)\times(e^{-\frac{1}{8}(\frac{1}{i+1}y-d)} + e^{-\frac{1}{2(i+1)}y})\\
&\leq  2(i+1)e^{-\frac{1}{8}(\frac{1}{i+1}y-d)}.
\ee

This completes the proof.
\end{proof}

\begin{lemma}\label{thm:MH4}
With notation as above,
\be
\mathbb{P}(\|X''_\mathcal{I}\|>z) \leq 2(\mathcal{I}+1) \times e^{-\frac{1}{8}(\frac{m_2}{2(\mathcal{I}+1)} z^2-d)} \quad \quad  \textrm{ for } \quad  z> \max\left(\sqrt{\frac{4}{m_2}(\mathcal{I}+1)d}, \, \, \sqrt{\frac{M_2}{m_2}\|X''_0\|^2}\right).
\ee
\end{lemma}

\begin{proof}
By Assumption \ref{AssumptionsConvexity},
\be
m_2\|X''_\mathcal{I}\|^2 \leq U(X''_\mathcal{I}) \leq E_\mathcal{I} .
\ee
Therefore,
\be
\mathbb{P}(\|X''_\mathcal{I}\|>z)  &= \mathbb{P}(m_2\|X''_\mathcal{I}\|^2>m_2 z^2)\\ 
&\leq \mathbb{P}(E_\mathcal{I}>m_2 z^2)\\
&\stackrel{\textrm{Lemma } \ref{thm:MH3}}{\leq} 2(\mathcal{I}+1) \times e^{-\frac{1}{8}(\frac{1}{2(\mathcal{I}+1)}m_2 z^2-d)} \quad \quad  \textrm{ for } m_2 z^2> \max\left(4(\mathcal{I}+1)d, M_2 \|X''_0\|^2\right).\\
\ee

\end{proof}

\begin{lemma} \label{thm:badset}
With notation as above,
\be \label{eq:badset3}
\mathbb{P}[\min( \mathfrak{t}_{\mathrm{reject}}, \mathfrak{t}_{G}) \leq \mathcal{I}]  \leq  4 \times \frac{\epsilon'''}{\log(\frac{1}{\epsilon'''})}
\ee
and
\be \label{eq:badset4}
\mathbb{E}[\|X''_\mathcal{I}\| \times \mathbbm{1}_{\min( \mathfrak{t}_{\mathrm{reject}}, \mathfrak{t}_{G}) \leq \mathcal{I}}] &\leq \sqrt{32}\frac{(\mathcal{I}+1)^{1.5}}{\sqrt{d}\sqrt{m_2}}(\epsilon''')^{d}\\
&\quad \quad+ 3\epsilon''' \times \max\left(\sqrt{\frac{32}{m_2}(\mathcal{I}+1)d}, \, \, \sqrt{\frac{M_2}{m_2}}\|X''_0\|\right)\\
\ee
\end{lemma}
\begin{proof}

Recalling that $X_{i}' = X_{i}''$ for all $i \leq \min( \mathfrak{t}_{\mathrm{reject}}, \mathfrak{t}_{G})$, we write

\be \label{eq:badset2}
\P[\min( &\mathfrak{t}_{\mathrm{reject}}, \mathfrak{t}_{G}) \leq \mathcal{I}] \\
&= \P[\{  \max_{1 \leq i \leq \frac{d}{\mathsf{m}}} \sup_{0 \leq h \leq \mathcal{I}-1} \|X_{h}'^{(i)}\| \geq \mathsf{g}_\infty \} \cup \{ \inf_{0 \leq h \leq \mathcal{I}-1} \| \mathbf{p}_{h} \| \leq \mathsf{g}_{2} \} \cup \{  \sup_{1\leq i\leq \frac{d}{\mathsf{m}}} \sup_{0 \leq h \leq \mathcal{I}-1} \|\mathbf{p}^{(i)}_{h}\| \geq \mathsf{g}_\infty \} \cup \{\mathfrak{t}_{\mathrm{reject}} \leq \mathcal{I}\}] \\
&\leq  \P[  \max_{1 \leq i \leq \frac{d}{\mathsf{m}}} \sup_{0 \leq h \leq \mathcal{I}-1} \|X_{h}'^{(i)}\| \geq \mathsf{g}_\infty ] +\P[ \inf_{0 \leq h \leq \mathcal{I}-1} \| \mathbf{p}_{h} \| \leq \mathsf{g}_{2}  ] + \P[\sup_{1\leq i\leq \frac{d}{\mathsf{m}}} \sup_{0 \leq h \leq \mathcal{I}-1} \|\mathbf{p}^{(i)}_{h}\| \geq \mathsf{g}_\infty ] \\
& \qquad + \P[\{\mathfrak{t}_{\mathrm{reject}} \leq \mathcal{I}\} \cap \{ \mathfrak{t}_{G} > \mathcal{I} \}] \\
&\stackrel{\textrm{Lemma } \ref{thm:rejection}}{\leq} \P[  \max_{1 \leq i \leq \frac{d}{\mathsf{m}}} \sup_{0 \leq h \leq \mathcal{I}-1} \|X_{h}'^{(i)}\| \geq \mathsf{g}_\infty ] +\P[ \inf_{0 \leq h \leq \mathcal{I}-1} \| \mathbf{p}_{h} \| \leq \mathsf{g}_{2}  ] + \P[\sup_{1\leq i\leq \frac{d}{\mathsf{m}}} \sup_{0 \leq h \leq \mathcal{I}-1} \|\mathbf{p}^{(i)}_{h}\| \geq \mathsf{g}_\infty] \\
& \qquad + \frac{\epsilon'''}{\log(\frac{1}{\epsilon'''})} \\
&\stackrel{\textrm{Lemma} \ref{thm:MH2}}{\leq}  \frac{\epsilon'''}{\log(\frac{1}{\epsilon'''})} +\P[ \inf_{0 \leq h \leq \mathcal{I}-1} \| \mathbf{p}_{h} \| \leq \mathsf{g}_{2}  ] + \P[\sup_{1\leq i\leq \frac{d}{\mathsf{m}}} \sup_{0 \leq h \leq \mathcal{I}-1} \|\mathbf{p}^{(i)}_{h}\| \geq \mathsf{g}_\infty] + \frac{\epsilon'''}{\log(\frac{1}{\epsilon'''})} \\
&\stackrel{\textrm{Lemma } \ref{thm:MH2'}}{\leq}  \frac{\epsilon'''}{\log(\frac{1}{\epsilon'''})} + \frac{\epsilon'''}{\log(\frac{1}{\epsilon'''})} + \P[\sup_{1\leq i\leq \frac{d}{\mathsf{m}}} \sup_{0 \leq h \leq \mathcal{I}-1} \|\mathbf{p}^{(i)}_{h}\| \geq \mathsf{g}_\infty] + \frac{\epsilon'''}{\log(\frac{1}{\epsilon'''})} \\
&\stackrel{\textrm{Lemma } \ref{LemmaMH5}}{\leq}  \frac{\epsilon'''}{\log(\frac{1}{\epsilon'''})} + \frac{\epsilon'''}{\log(\frac{1}{\epsilon'''})} + \frac{\epsilon'''}{\log(\frac{1}{\epsilon'''})} + \frac{\epsilon'''}{\log(\frac{1}{\epsilon'''})}. \\
\ee

This completes the proof of Inequality \eqref{eq:badset3}.

Next, set $\Lambda := \max\left(\sqrt{\frac{32}{m_2}(\mathcal{I}+1)d}, \, \, \sqrt{\frac{M_2}{m_2}}\|X''_0\|\right)  \sqrt{\log(\frac{1}{\epsilon'''})}$ and  $\Psi := \frac{4\epsilon'''}{\log(\frac{1}{\epsilon'''})}\times \Lambda$. We have

\be
 \mathbb{E}[\|X''_\mathcal{I}\| \times &\mathbbm{1}_{\min( \mathfrak{t}_{\mathrm{reject}}, \mathfrak{t}_{G}) \leq \mathcal{I}}] \stackrel{{\scriptsize \textrm{Eq. }}\ref{eq:badset3}}{\leq}  \int_\Lambda^\infty \mathbb{P}(\|X''_\mathcal{I}\|>z) \mathrm{d}z + \Psi\\
&\stackrel{{\scriptsize \textrm{Lemma }}\ref{thm:MH4}}{\leq}   \int_{\Lambda}^\infty 2(\mathcal{I}+1) e^{-\frac{1}{8}(\frac{m_2}{2(\mathcal{I}+1)} z^2-d)}\mathrm{d}z + \Psi\\
&\leq  \int_{\Lambda}^\infty \frac{z}{\Lambda} 2(\mathcal{I}+1) e^{-\frac{1}{8}(\frac{m_2}{2(\mathcal{I}+1)} z^2-d)}\mathrm{d}z + \Psi\\
&=   -\frac{16}{\Lambda}(\mathcal{I}+1)\frac{\mathcal{I}+1}{m_2}e^{-\frac{m_2}{16(\mathcal{I}+1)} z^2}\, e^{\frac{d}{8}} \bigg|_\Lambda^\infty + \Psi\\
&\leq   \frac{32}{\Lambda}\frac{(\mathcal{I}+1)^2}{m_2}e^{-2d\log(\frac{1}{\epsilon'''})}\, e^{\frac{d}{8}}+ \Psi\\
&\leq   \frac{32}{\Lambda}\frac{(\mathcal{I}+1)^2}{m_2}e^{-d\log(\frac{1}{\epsilon'''})}+ \Psi\\
&=   \frac{32}{\Lambda}\frac{(\mathcal{I}+1)^2}{m_2}(\epsilon''')^{d}+ \Psi\\
&\leq   \sqrt{32}\frac{(\mathcal{I}+1)^{1.5}}{\sqrt{d}\sqrt{m_2}}(\epsilon''')^{d}+ \Psi\\
&\leq  \sqrt{32}\frac{(\mathcal{I}+1)^{1.5}}{\sqrt{d}\sqrt{m_2}}(\epsilon''')^{d}+ 4\epsilon''' \times \max\left(\sqrt{\frac{32}{m_2}(\mathcal{I}+1)d}, \, \, \sqrt{\frac{M_2}{m_2}}\|X''_0\|\right)  \\
\ee
where the third inequality holds since $\frac{z}{\Lambda} \geq 1$ inside the bounds of integration, and the fifth and seventh inequalities hold since $\log(\frac{1}{\epsilon'''})>1$.  This completes the proof of Inequality \eqref{eq:badset4}.
\end{proof}

\section{Proof of Theorem \ref{ThmDriftHMC} } \label{AppDriftCond}

In order to prove Theorem \ref{ThmDriftHMC}, we need the following very rough bound on the distance that can be travelled by solutions to Hamilton's equations:

\begin{lemma} \label{LemmaSmallDisplacement}
Assume that there exists some $0 < C < \infty$ so that 
\be \label{LemmaSmallDisplacementAssump}
\| U'(q) \| \leq C \| q \|
\ee  for all $q \in \mathbb{R}^{d}$. Let $(q_{t},p_{t})$ be solutions to Hamilton's equations \eqref{EqHamiltonEquations} with initial conditions $(q_{0},p_{0}) = (\mathbf{q},\mathbf{p}) \in \mathbb{R}^{2d}$. Then for all $t \geq 0$,
\be 
\| q_t - \mathbf{q} \| \leq  \frac{1}{2C} \left( e^{-\sqrt{C}t} (e^{\sqrt{C}t} - 1)(\sqrt{C} \|\mathbf{p}\| (e^{\sqrt{C} t} + 1) + C \|\mathbf{q}\|(e^{\sqrt{C}t} - 1)) \right).
\ee  
\end{lemma}

\begin{proof}

Note that
\be 
f(t) \equiv \frac{1}{2C} \left( e^{-\sqrt{C}t} (e^{\sqrt{C}t} - 1)(\sqrt{C} \|\mathbf{p}\| (e^{\sqrt{C} t} + 1) + C \|\mathbf{q}\|(e^{\sqrt{C}t} - 1)) \right)
\ee 
is a solution to the system of equations 
\be 
f(0) &= 0 \\
f'(0) &= \|\mathbf{p}\| \\
f''(t) &= C f(t) + C \|\mathbf{q}\|, \qquad t \geq 0.
\ee

By Hamilton's equations \eqref{EqHamiltonEquations} and Inequality \eqref{LemmaSmallDisplacementAssump},
\be 
\frac{d^{2}}{dt^{2}} \| q_t - q_0 \| &\leq \| U'(q_t) \| \\
&\leq C \| q_t \| \\
&\leq C \| q_t - q_0 \| + C \| q_0 \|.
\ee 
We also have 
\be 
\frac{d}{dt} \bigg|_{t=0} \| q_t - q_0 \| &\leq \| p_0 \|.
\ee 
By Lemma \ref{LemmaOdeComp}, this implies 
\be 
\| q_t - q_0\| &\leq f(t) \\
&=  \frac{1}{2C} \left( e^{-\sqrt{C}t} (e^{\sqrt{C}t} - 1)(\sqrt{C} \|\mathbf{p}\| (e^{\sqrt{C} t} + 1) + C \|\mathbf{q}\|(e^{\sqrt{C}t} - 1)) \right)
\ee 
for all $t \geq 0$. This completes the proof.
\end{proof}

The following lemma gives us the main bounds in the proof of  Theorem \ref{ThmDriftHMC}:

\begin{lemma} [Lyapunov Function for Hamiltonian Dynamics: Gaussian-Like Tails] \label{LemmaLyapunovHamDynBasic}
Let $U$ satisfy Assumption \ref{AssumptionsConvexity} with $\mathcal{X} = \mathbb{R}^{d}$. Fix an initial position $(\mathbf{q},\mathbf{p}) \in \mathbb{R}^{2d}$ that satisfies
\be \label{IneqLyapMainLemmDistAssumption}
\frac{1}{\sqrt{2 m_{2}}} \|\mathbf{p}\| -  \frac{m_{2}^{2}}{512 \, M_{2}^{2}} \| \mathbf{q} \| \leq -1
\ee  
and let $(q_t,p_t)_{t \geq 0}$ be a solution to Equation \eqref{EqHamiltonEquations} with initial conditions $q_0 = \mathbf{q}$, $p_0 = \mathbf{p}$. Then for $T = \frac{1}{2\sqrt{2}} \frac{\sqrt{m_{2}}}{M_{2}}$,
\be \label{IneqLyapPreLemmaConc1}
e^{ \| q_T \|} \leq e^{-1} \, e^{\| \mathbf{q} \|}
\ee 
and also
\be \label{IneqLyapPreLemmaConc2}
\inf_{0 \leq s \leq t} \| q_s \| \geq  \frac{\|\mathbf{q}\|}{2} - \frac{\|\mathbf{p}\|}{2 \sqrt{M_{2}}}.
\ee 
\end{lemma}

\begin{proof}
Let $(\alpha(t), \beta(t))_{t \geq 0}$ be a solution to Equation \eqref{EqHamiltonEquations} with initial conditions $\alpha(0) = 0$, $\beta(0) = \mathbf{p}$. By Theorem \ref{ThmContractionConvexMainResult}, 

\be \label{IneqDriftInitialContraction}
\| \alpha(T) - q_T \| &\leq \left[1- \frac{1}{64} (\sqrt{m}_{2} T)^2 \right] \times \| \alpha(0) - \mathbf{q} \| \\
&= (1 - \frac{1}{64} m_{2} T^{2}) \|  \mathbf{q} \|.
\ee 

By conservation of energy for Hamilton's equations,
\be 
\frac{1}{2} \| \beta(T) \|^{2} + U(\alpha(T)) =\frac{1}{2} \| \beta(0) \|^{2} + U(\alpha(0)) = \frac{1}{2} \| \beta(0) \|^{2},
\ee 
so
\be \label{IneqDriftTrivialEnergy}
 U(\alpha(T)) \leq \frac{1}{2} \| \beta(0) \|^{2}.
\ee 
By our assumptions, for $q \in \mathbb{R}^{d}$ we have 
\be 
U(q) \geq m_{2} \|q \|^{2}.
\ee 
Combining this with Inequality \eqref{IneqDriftTrivialEnergy}, we have 
\be 
m_{2} \| \alpha(T) \|^{2} \leq \frac{1}{2} \| \beta(0) \|^{2},
\ee  
so that
\be \label{IneqSimpleEnergyBound}
\| \alpha(T) \| \leq \frac{1}{\sqrt{2 m_{2}}} \| \beta(0) \|.
\ee 
Combining this with Inequality \eqref{IneqDriftInitialContraction} and then the assumption \eqref{IneqLyapMainLemmDistAssumption}, 
\be 
\| q_T \| &\leq \| \alpha(T) \| + \| \alpha(T) - q_T \| \\
&\leq \frac{1}{\sqrt{2 m_{2}}} \| \mathbf{p} \| + (1 - \frac{1}{64} m_{2} T^{2}) \| \mathbf{q} \| \\
&= \frac{1}{\sqrt{2 m_{2}}} \| \mathbf{p} \| + (1 - \frac{m_{2}^{2}}{512\, M_{2}^{2}}) \| \mathbf{q} \| \\
&\leq \| \mathbf{q} \| - 1.
\ee 
We conclude that 
\be 
e^{ \| q_T \|} \leq e^{-1} e^{ \|q_0\|},
\ee 
completing our proof of Inequality \eqref{IneqLyapPreLemmaConc1}.

We prove Inequality \eqref{IneqLyapPreLemmaConc2} by an application of Lemma \ref{LemmaSmallDisplacement}, which gives 
\be 
\| q_s \| &\geq \| q_0 \| - \| q_0 - q_s \| \\
&\geq \| \mathbf{q} \| -  \frac{1}{2M_{2}} \left( e^{-\sqrt{M_{2}}s} (e^{\sqrt{M_{2}}s} - 1)(\sqrt{M_{2}} \|\mathbf{p}\| (e^{\sqrt{M_{2}} s} + 1) + M_{2} \|\mathbf{q}\|(e^{\sqrt{M_{2}}s} - 1)) \right) \\
&\geq \|\mathbf{q}\|(1 - \frac{1}{2}( e^{\sqrt{M_{2}}s} - 1)) - \frac{\|\mathbf{p}\|}{2 \sqrt{M_{2}}} (e^{\sqrt{M_{2}}s}- e^{-\sqrt{M_{2}}s}) 
\ee 
for all $s \geq 0$. We note that the RHS is monotone increasing in $s$. Thus, setting $s = T$ and using the bounds $0 \leq \frac{m_{2}}{M_{2}} \leq 1$ and $\sqrt{M_{2}} T \leq \frac{1}{2 \sqrt{2}}$,

\be 
\| q_T \| &\geq \|\mathbf{q}\|(1 - \frac{1}{2}( e^{\sqrt{M_{2}}T} - 1)) - \frac{\|\mathbf{p}\|}{2 \sqrt{M_{2}}} (e^{\sqrt{M_{2}}T}- e^{-\sqrt{M_{2}}T})  \\
&\geq \|\mathbf{q}\|(1 - \frac{1}{2}( e^{\frac{1}{2 \sqrt{2}}} - 1)) - \frac{\|\mathbf{p}\|}{2 \sqrt{M_{2}}} (e^{\frac{1}{2 \sqrt{2}}}- e^{-\frac{1}{2 \sqrt{2}}}) \\
&\geq \frac{\|\mathbf{q}\|}{2} - \frac{\|\mathbf{p}\|}{2 \sqrt{M_{2}}},
\ee

completing the proof.

\end{proof}

We apply this lemma to prove Theorem \ref{ThmDriftHMC}:

\begin{proof} [Proof of Theorem \ref{ThmDriftHMC}]

For $\mathbf{q} \in \mathbb{R}^{d}$, define the associated ``good" set to be 
\be 
\mathcal{G}(\mathbf{q}) = \{ \mathbf{p} \in \mathbb{R}^{d} \, : \, \| \mathbf{p} \| \leq \min ( c_{1} \| \mathbf{q}\| - c_{3} , c_{2} \|\mathbf{q}\| - c_{4} ) \},
\ee 
where
\be 
c_{1} = \frac{\sqrt{2} m_{2}^{2.5}}{256 M_{2}^{2}}, \, c_{2} =  2 \max(1,\sqrt{M_{2}}) , \,c_{3} = \sqrt{2m_{2}}, c_{4}= 2 C \sqrt{M_{2}}.
\ee 
Roughly speaking, the condition $\mathbf{q} \in \mathcal{G}(\mathbf{q})$ guarantees that the solution to Hamilton's equations $(q_{t},p_{t})_{t=0}^{T}$ with initial conditions $(\mathbf{q},\mathbf{p})$ 
\begin{itemize}
\item will stay outside of the set $\{x \, : \, V(x) \leq C \}$ (by Inequality \eqref{IneqLyapPreLemmaConc2} of Lemma \ref{LemmaLyapunovHamDynBasic}), \textit{and}
\item will drift towards the origin (by Inequality \eqref{IneqLyapPreLemmaConc1} of Lemma \ref{LemmaLyapunovHamDynBasic}).
\end{itemize}

We now make this precise. Fix $X_{0} = x \in \mathbb{R}^{d}$ as in the statement of the theorem, let $p \sim \Phi_{1}$, and set $X_{1} = \mathcal{Q}_{T}^{(X_{0})}(p)$; note that $X_{1}$ has the distribution required by the statement of the theorem.  By Lemma \ref{LemmaLyapunovHamDynBasic}, we have
\be \label{IneqContractionRecall1}
e^{\|X_{1}\|} \, \mathbbm{1}_{p \in \mathcal{G}(X_{0})} \leq e^{-1} e^{\|X_{0}\|}.
\ee 

Mimicking the calculation leading to Inequality \eqref{IneqSimpleEnergyBound}, we also have for \textit{any} initial position and velocity the deterministic inequality
\be 
m_{2} \| X_{1} \|^{2} \leq M_{2} \|X_{0}\|^{2} + \frac{1}{2} \| p \|^{2},
\ee 
so that 
\be \label{IneqContractionRecall2}
\|X_{1}\| \leq \sqrt{\frac{M_{2}}{m_{2}}} \|X_{0}\| + \frac{1}{\sqrt{2 m_{2}}} \|p \|.
\ee 

Combining Inequalities \eqref{IneqContractionRecall1} and \eqref{IneqContractionRecall2},

\be 
\E[e^{\|X_{1}\|} | X_{0}] &= \E[e^{\|X_{1}\|} \, \mathbbm{1}_{p \in \mathcal{G}(X_{0})}| X_{0}] + \E[e^{\|X_{1}\|} \, \mathbbm{1}_{p \notin \mathcal{G}(X_{0})}| X_{0}] \\
&\leq e^{-1} e^{\|X_{0}\|} + \E[e^{\|X_{1}\|} \, \mathbbm{1}_{p \notin \mathcal{G}(X_{0})}| X_{0}] \\
&\leq  e^{-1} e^{\|X_{0}\|} + e^{\frac{\sqrt{M_{2}}}{m_{2}} \| X_{0}\| } \int_{p \notin \mathcal{G}(X_{0})} e^{\frac{1}{\sqrt{2 m_{2}}} \| p\|} d \Phi_{1}(p). \\
\ee 

Defining $A = \sup_{x \in \mathbb{R}^{d}} e^{\frac{\sqrt{M_{2}}}{m_{2}} \|x\|} \int_{p \notin \mathcal{G}(x)} e^{\frac{1}{\sqrt{2 m_{2}}} \| p\|} d \Phi_{1}(p) < \infty$, this implies 

\be 
\E[e^{\|X_{1}\|} | X_{0}] \leq  e^{-1} e^{\|X_{0}\|} + A.
\ee 

This completes the proof of Inequality \eqref{IneqDriftConcMain} with no bound on $A$. We will now bound $A$.

For $R \geq 2 \ell$, let $X \sim \Phi_{1}$. We then have the bound
\be \label{SillyConstBound}
\int_{\|p \| > R} e^{\ell \|p \|} d \Phi_{1}(p) &= \frac{1}{\sqrt{2 \pi}} \int_{\|x \| > R} e^{\ell \|x \|} e^{-\frac{\|x\|^{2}}{2}} dx \\
&\leq  \frac{1}{\sqrt{2 \pi}} \int_{\|x \| > R} e^{\ell \|x \|} e^{-\frac{(\|x - 2 \ell\| + \ell)^{2}}{2}} dx \\
&= \frac{1}{\sqrt{2 \pi}} \int_{\|x \| > R} e^{2 \ell^{2}} e^{-\frac{\|x - 2 \ell \|^{2}}{2}} dx \\
&= e^{2 \ell^{2}} \P[X > R - 2 \ell] \\
&\leq 2 e^{2 \ell^{2}} e^{-\frac{(R-2 \ell)^{2}}{2d}}.
\ee

For all $\ell, R$ (and in particular for $R < 2 \ell$), we have the trivial bound 
\be \label{SillyConstBound2}
\int_{\|p \| > R} e^{\ell \|p \|} d \Phi_{1}(p) &\leq \int e^{\ell \|p \|} d \Phi_{1}(p) \\
&= e^{d \,\frac{\ell^{2}}{2}}.
\ee

Thus, defining 

\be 
\log(B_{1}(x)) &= \log(2) + \frac{\sqrt{M_{2}}}{m_{2}}x +  \frac{1 }{m_{2}} - \frac{1}{2d} \left( c_{1}x - c_{3} - \sqrt{\frac{2}{m_{2}}} \right)^{2} \\
\log(B_{2}(x)) &= \log(2) + \frac{\sqrt{M_{2}}}{m_{2}}  x + \frac{1 }{m_{2}} -\frac{1}{2d} \left( c_{2} x - c_{4} - \sqrt{\frac{2}{m_{2}}} \right)^{2} \\
\log(B_{3}) &= \frac{\sqrt{M_{2}}}{m_{2} c_{1}} \left( c_{3} + \sqrt{\frac{2}{m_{2}}} \right)  + \frac{d}{4 m_{2}} \\
\log(B_{4}) &= \frac{\sqrt{M_{2}}}{m_{2} c_{2}} \left( c_{4} + \sqrt{\frac{2}{m_{2}}} \right) + \frac{d}{4 m_{2}} \\
\ee 
for $x \in \mathbb{R}^{+}$, we have by Inequalities \eqref{SillyConstBound} and \eqref{SillyConstBound2}
\be \label{IneqImmA}
A &=  \sup_{x \in \mathbb{R}^{d}} e^{\frac{\sqrt{M_{2}}}{m_{2}} \| x\| } \int_{p \notin \mathcal{G}(x)} e^{\frac{1}{\sqrt{2 m_{2}}} \| p\|} d \Phi_{V}(p) \\
&\leq \max(\sup_{x \in \mathbb{R}^{+}} \max( B_{1}(x),B_{2}(x)), B_{3}, B_{4}).
\ee 
Optimizing by hand,  

\be 
\sup_{x \in \mathbb{R}^{+}}  \log(B_{1}(x)) \leq \log(2) + \frac{1}{m_{2}} - \frac{1}{2d}(c_{3} + \sqrt{\frac{2}{m_{2}}})^{2} - \frac{ (\frac{\sqrt{M_{2}}}{m_{2}} + \frac{1}{d}c_{1}(c_{3} - d \sqrt{\frac{2}{m_{2}}}))^{2}}{ c_{1}^{2}} 
\ee 
and 
\be 
\sup_{x \in \mathbb{R}^{+}}  \log(B_{2}(x)) \leq \log(2) + \frac{1}{m_{2}} - \frac{1}{2d}(c_{4} + \sqrt{\frac{2}{m_{2}}})^{2} - \frac{ (\frac{\sqrt{M_{2}}}{m_{2}} + \frac{1}{d}c_{2}(c_{4} - d \sqrt{\frac{2}{m_{2}}}))^{2}}{ c_{2}^{2}}.
\ee 
Combining this with Inequality \eqref{IneqImmA} and expanding the definition of these constants (and using the relationships $0 < m_{2} \leq M_{2} < \infty$ and $d \geq 1$ to remove some terms that cannot possibly be the largest) completes the proof of Inequality \eqref{IneqDriftConcConst}.
\end{proof}

\section{Proofs of Main Results for Approximate HMC Dynamics} \label{AppProofsMainApproxRes}

In this section, we prove the remainder of the main results: Theorems \ref{ThmMainApprox}, \ref{ThmMainApprox_leapfrog_unadjusted}, and \ref{ThmMainApprox_leapfrog2}. The proofs in this section are short, and all of them essentially use the same strategy: with appropriately chosen parameters, numerical implementations of HMC inherit some of the good mixing properties of the ``ideal" HMC algorithm. Although the proofs are short, they rely crucially on the definitions and estimates obtained in Appendices \ref{AppHigherOrder}, \ref{SecAppendixMixApprox} and \ref{SecMH}. We suggest that all three of those appendices should be read before this, but give a quick guide to the most relevant material from those appendices:

\begin{itemize}
\item Appendix  \ref{AppHigherOrder} is very short. The critical content is the definition of the ``toy" integrator (Definition \ref{defn:toy}) and the bound in Lemma \ref{thm:leapfrog5} on the error of this integrator.
\item Appendix \ref{SecAppendixMixApprox} is quite lengthy and contains many generic MCMC bounds. These bounds are the core of the proofs of our main results; this appendix largely explains how to ``glue them together." Having said that, most of this appendix can be skipped on a first reading. We suggest that the reader review the introduction of Appendix \ref{AppSubAppBoundUnad} and the statements of Lemmas \ref{ThmGenericApproxErrorBound3} (which gives a drift condition for generic approximate HMC algorithms) and Lemma \ref{thm:unadjusted_wass2} (which gives a Wasserstein mixing bound for generic approximate HMC algorithms). Read the notation guide at the start of Appendix \ref{AppSubAppBoundUnad} before reading the lemmas.
\item Appendix \ref{SecMH} is of medium length and contains many specific bounds related to Algorithms \ref{alg:Unadjusted} and \ref{alg:Metropolis}; these bounds will allow us to apply the results of Appendix \ref{SecAppendixMixApprox}. We suggest that the reader review the introduction of Appendix \ref{SecMH} and the statements of Lemmas \ref{thm:MH2'}, \ref{LemmaMH5} and \ref{thm:badset} of Appendix \ref{SecMH} before continuing. Read the notation guide at the start of Appendix \ref{SecMH} before reading the lemmas.
\end{itemize}

We set notation that will be used throughout this appendix. Set $\Gamma = \Gamma(\cdot, \cdot)$ as in Equation \eqref{EqDefGammaFunc} and $\mathsf{K}'''$ as in Equation \eqref{DefKTriple}. Fix constants satisfying:

\be 
T &\leq \frac{1}{2 \sqrt{2}} \frac{\sqrt{m_{2}}}{M_{2}} \\
\kappa &=  \frac{1}{8} (\sqrt{m}_2T)^2\\
\mathsf{A} &= \max(\mathsf{K}', \mathsf{K}''', M_2).\\
\ee 

Also define the following functions:

\be 
\lambda_{n} & =  \frac{\kappa}{64} \sqrt{\frac{m_{2}}{n}}  \\
s_{n,\epsilon'} &= \frac{1}{\kappa}\log\left(\frac{24\log(15000 \, \Gamma(\kappa, n) \,  \kappa^{-2}) +24(\lambda_{n} \epsilon') \,\log(\kappa^{-1}) }{(\lambda_{n} \epsilon')^{2}} \right)\\
C_{n,\epsilon'} &= ( \frac{15000 \,(s_{n,\epsilon'}+1) \, \Gamma(\kappa, n) }{  \lambda_{n} \, \epsilon'  \, \kappa^{2}} +4)^2\\
\mathsf{g}_\infty^{(1)}(n) &=  \lambda_{\mathsf{m}}^{-1} \log\left(2048 \frac{n}{\mathsf{m}}  \frac{\log(\frac{1}{\epsilon'''})}{\epsilon'''}  \frac{\mathcal{I} \, \Gamma(\kappa,\mathsf{m})}{  \kappa^{2}}\right) \\
\mathsf{g}_\infty^{(2)}(n) &= \sqrt{2\mathsf{m}+ 8\log \left(\frac{\mathcal{I} n}{\mathsf{m}}  \frac{\log(\frac{1}{\epsilon'''})}{\epsilon'''} \right)} \\
\mathsf{g}_\infty^{(3)} & =  \frac{1}{\lambda_{\mathsf{m}}} (\frac{2 - \kappa}{\kappa}\log(2) -\frac{7}{8}) + \frac{1}{\lambda_{\mathsf{m}}}\log(\frac{6144}{\kappa^{2}} \frac{\mathcal{I} d}{\mathsf{m}\epsilon})\\
\mathsf{g}_2(n) &=  \sqrt{n-2\sqrt{n}\log^{\frac{1}{2}} \left(\frac{\mathcal{I}\log(\frac{1}{\epsilon'''})}{\epsilon'''} \right)},\\
\ee 

where the domains of these functions are: 

\be 
\epsilon''' >0, \quad n \in \{d, \mathsf{m}\}, \quad \epsilon' > 0.
\ee

\begin{proof} [Proof of Theorem \ref{ThmMainApprox}]

Set notation as in Lemma \ref{thm:unadjusted_wass2} and let
\be 
\theta_{0} = \frac{\kappa \epsilon \sqrt{m_2}}{18T M_2(\sqrt{M_2} \lambda_{d}^{-1}\log(C) + \frac{1}{\sqrt{2}}\sqrt{d})}.
\ee 
This theorem is an immediate consequence of Lemma \ref{thm:unadjusted_wass2}, if we set the value of the variable ``$\mathsf{A}$" in the statement of Lemma \ref{thm:unadjusted_wass2} to be equal to $M_2$.

\end{proof}

\begin{proof} [Proof of Theorem \ref{ThmMainApprox_leapfrog_unadjusted}]

Fix $\mathcal{I}\in \mathbb{N}$. We fix the constants

\be
\epsilon''' &= \frac{\mathsf{m}}{3d}\epsilon, \quad \epsilon' = \epsilon^{2} \sqrt{\frac{\mathsf{m}}{d}}, \quad s = s_{\mathsf{m},\epsilon'}, \quad C = C_{\mathsf{m},\epsilon'}, \\
\mathsf{g}_{\infty} &= \max(\mathsf{g}_{\infty}^{(1)}(\mathsf{m}),\mathsf{g}_{\infty}^{(2)}(\mathsf{m}), \mathsf{g}_{\infty}^{(3)}), \quad \mathsf{g}_{2} = \mathsf{g}_{2}(\mathsf{m}) \\
\theta_{0} &= \min( \frac{\kappa \sqrt{\frac{\mathsf{m}}{d}} \epsilon^2 \sqrt{m_2}}{18T \mathsf{A}(\sqrt{M_2} \lambda_{\mathsf{m}}^{-1}\log(C) + \frac{1}{\sqrt{2}}\sqrt{\mathsf{m}})} \, , \, \frac{1}{2}\epsilon^2 (\mathcal{I} \sqrt{\frac{d}{\mathsf{m}}} \mathsf{A})^{-1} ),\\
\ee

and also fix $0 < \theta < \theta_{0}$.

Let $\diamondsuit$ be the ``toy" numerical integrator defined in Equation \eqref{eq:toy} with these parameters and base numerical integrator $\sharp$, where we recall that $\sharp$ is the numerical integrator in the statement of Theorem \ref{ThmMainApprox_leapfrog_unadjusted}.  Let $\Lambda$ be the transition kernel defined by Algorithm \ref{alg:Unadjusted} with numerical integrator  $\diamondsuit$ and all other parameters as in the statement of Theorem \ref{ThmMainApprox_leapfrog_unadjusted}. Let $K$ be the transition kernel associated with Algorithm \ref{DefSimpleHMC} with parameters as in the statement of Theorem \ref{ThmMainApprox_leapfrog_unadjusted}.

We now define a coupling of four Markov chains  $\{X_{i}'\}_{i \geq 0} \sim \Lambda$, $\{\hat{X}_{i}'\}_{i \geq 0} \sim Q_{\theta}$ and $\{X_{i}\}_{i \geq 0}, \, \{Y_{i}\}_{i \geq 0} \sim K$. As mentioned immediately following Algorithm \ref{DefSimpleHMC}, all of the transition kernels given in this paper are stated in terms of a \textit{random mapping representation} of the associated HMC schemes. Thus, in order to define a coupling of these chains, it is enough to specify the following information:

\begin{enumerate}
\item the distribution of the starting points $X_{0}', \hat{X}_{0}', X_{0}, Y_{0}$, \textit{and}
\item a coupling of the sequence of random momentum variables $\{\mathbf{p}_{i}\}_{i \geq 0}$ chosen in Step 2 of each of these algorithms.
\end{enumerate}  

We make the following simple choice of coupling: 
\begin{enumerate}
\item We fix a single $x \in \mathbb{R}^{d}$ satisfying $\|x \| \leq \sqrt{\frac{d}{m_{2}}}$ and set $X_0 = X_{0}'= \hat{X}_{0}' = x.$ Then sample $Y_{0} \sim \pi$.
\item We draw a single i.i.d. sequence $\{\mathbf{p}_{i}\}_{i \geq 0}$ of standard Gaussian random variables, and use this single update sequence for all four Markov chains.
\end{enumerate}

For $1 \leq j \leq \frac{d}{m}$, let  $\mathfrak{t}_G^{(j)} = \inf \{ h \, : \, (X'^{(j)}_h, \mathbf{p}^{(j)}_{h}) \notin \mathsf{G}\}$ be the exit time of the ``toy" chain from the following ``good set" $\mathsf{G}$:
 \be
\mathsf{G}  = \{(q,p) \in \mathbb{R}^{\mathsf{m}}\times \mathbb{R}^{\mathsf{m}} :  \|q\|< \mathsf{g}_\infty, \|p\| < \mathsf{g}_\infty, ||p|| > \mathsf{g}_2 \}.
\ee
 
Notice that, if $\mathfrak{t}_{G}^{(j)} > \mathcal{I}$, then $X'^{(j)}_h  = \hat{X}'^{(j)}_h$ for all $h \leq \mathcal{I}$. Thus, if $\min_{1 \leq j \leq \frac{d}{m}}( \mathfrak{t}_{G}^{(j)}) > \mathcal{I}$, we have $X'_h = \hat{X}'_h$ for all $h \leq \mathcal{I}$. 
 
Fix $C' = (1+\frac{16}{\kappa}\log (1+e^{-\frac{\mathsf{m}}{8}})) \, \frac{8}{\kappa} \sqrt{\frac{\mathsf{m}}{m_{2}}}$, define $V(x) \equiv e^{\lambda_{\mathsf{m}} \|x\|}$, and let 

 \be 
\alpha &= 1-  (1+e^{-\frac{\mathsf{m}}{8}})e^{\lambda_{\mathsf{m}}\left(-\kappa + 6\theta \, T \, \frac{\mathsf{A}}{\sqrt{m_2}} \sqrt{M_2}\right)\, C' +  \lambda_{\mathsf{m}} \left(1 + 6\theta \, T \, \mathsf{A}\right) \, \sqrt{\frac{\mathsf{m}}{m_{2}}}}\\
\beta &=  e^{\lambda_{\mathsf{m}} C'} \, (1 - \alpha).
\ee

By Inequality \eqref{eq:w13}, 
\be \label{eq:prokhorov5}
\frac{\beta}{\alpha}  \leq \frac{1024}{\kappa^{2}} \, \Gamma(\kappa,d).
\ee

For every $j \in \{1, \ldots, \frac{d}{\mathsf{m}}\}$, let $\Lambda^{(j)}$ and $Q^{(j)}_{\theta}$ be the transition kernels of $\{X_{i}'^{(j)}\}_{i \geq 0}$ and $\{\hat{X}_{i}'^{(j)}\}_{i \geq 0}$ respectively. Let $\mu^{(j)}$ and $\hat{\mu}^{(j)}$ be their stationary distributions. By Inequality \eqref{eq:w6}, 
\be \label{IneqQuickLyap}
\Lambda^{(j)}(x,\cdot)[V] \leq (1 - \alpha) V(x) + \beta,
\ee
and by Inequality \eqref{IneqStartCondSuffCond},
\be \label{IneqQuickStart} 
V(X_{0}'^{(j)}) \leq V(X_{0}') \leq \frac{\beta}{\alpha}.
\ee

Inequalities \eqref{IneqQuickLyap} and \eqref{IneqQuickStart} allow us to apply the generic Inequality \eqref{eq:Markov_inequality2} for Lyapunov functions, giving 
\be
\mathbb{P}[\sup_{0 \leq h \leq \mathcal{I}-1} V(X'^{(j)}_{h}) \geq r] \leq \frac{2 \beta \mathcal{I}}{\alpha r} \quad \quad \forall r>0.
\ee 
Reparameterizing, we have
\be \label{eq:prokhorov1}
\mathbb{P}[\sup_{0 \leq h \leq \mathcal{I}-1} \|X'^{(j)}_{h}\| \geq r] \leq \frac{2 \beta \mathcal{I}}{\alpha e^{\lambda_{\mathsf{m}} r}} \quad \quad \forall r>0.
\ee

Applying this with $r = \mathsf{g}_{\infty}$, we obtain after some basic algebra:
 \be \label{eq:prokhorov2}
\mathbb{P}[\sup_{0 \leq h \leq \mathcal{I}-1} \|X'^{(j)}_{h}\| \geq \mathsf{g}_\infty] \stackrel{{\scriptsize \textrm{Eq. }}\ref{eq:prokhorov5}}{\leq} \epsilon \,\frac{\mathsf{m}}{3d}.
\ee 

By Lemmas \ref{thm:MH2'} and \ref{LemmaMH5}, we also have 
\be \label{eq:prokhorov3}
\mathbb{P}[\inf_{0 \leq {h} \leq \mathcal{I}-1} \|\mathbf{p}^{(j)}_{h}\| \leq \mathsf{g}_2] &\leq \frac{\epsilon'''}{\log(\frac{1}{\epsilon'''})} \leq \epsilon \, \frac{\mathsf{m}}{3d}\\
\mathbb{P}[\sup_{0 \leq {h} \leq \mathcal{I}-1} \|\mathbf{p}^{(j)}_{h}\| \geq \mathsf{g}_\infty] &\leq \frac{\epsilon'''}{\log(\frac{1}{\epsilon'''})} \leq  \epsilon \, \frac{\mathsf{m}}{3d}.
\ee
Hence, Inequalities \eqref{eq:prokhorov2} and \eqref{eq:prokhorov3} together imply that 
\be \label{eq:prokhorov4}
\mathbb{P}[\mathfrak{t}_{G}^{(j)} \leq \mathcal{I}] &\leq \mathbb{P}[\sup_{0 \leq h \leq \mathcal{I}-1} \|X'^{(j)}_{h}\| \geq \mathsf{g}_\infty] + \mathbb{P}[\inf_{0 \leq {h} \leq \mathcal{I}-1} \|\mathbf{p}^{(j)}_{h}\| \leq \mathsf{g}_2]\\
& \quad \quad + \mathbb{P}[\sup_{0 \leq {h} \leq \mathcal{I}-1} \|\mathbf{p}^{(j)}_{h}\| \geq \mathsf{g}_\infty]\\
&\leq \epsilon \, \frac{\mathsf{m}}{d}.
\ee
 

Recall that $\{Y_{i}\}_{i \geq 0} \sim K$ is a Markov chain started at $Y_{0} \sim \pi$, its stationary measure. Repeatedly applying Lemma \ref{thm:leapfrog3} and  Theorem \ref{ThmContractionConvexMainResult} (in a manner somewhat similar to the argument described in Figure \ref{fig:Approx_integrator}) gives
\be \label{eq:noC1}
\mathbb{E}\left[\|\hat{X}'_\mathcal{I} - Y_\mathcal{I}\| \times \mathbbm{1}\{\mathcal{I} < \min_{1 \leq j \leq \frac{d}{\mathsf{m}}} (\mathfrak{t}_{G}^{(j)})\}\right] &\leq \mathbb{E}\bigg[\left( \|\hat{X}'_\mathcal{I} - X_\mathcal{I}\| + \|X_\mathcal{I} - Y_\mathcal{I}\| \right)\times \mathbbm{1}\{\mathcal{I} < \min_{1 \leq j \leq \frac{d}{\mathsf{m}}} (\mathfrak{t}_{G}^{(j)})\}\bigg]\\
&\stackrel{{\scriptsize \textrm{ Th. }}\ref{ThmContractionConvexMainResult}}{\leq}  \mathbb{E}\left[\left( \|\hat{X}'_\mathcal{I} - X_\mathcal{I}\| + (1-\kappa)^{\mathcal{I}} \times \|X_0 - Y_0\| \right)\times \mathbbm{1}\{\mathcal{I} < \min_{1 \leq j \leq \frac{d}{\mathsf{m}}} (\mathfrak{t}_{G}^{(j)})\}\right]\\
&\stackrel{{\scriptsize \textrm{Lemma }}\ref{thm:leapfrog3}, \textrm{ Th. }\ref{ThmContractionConvexMainResult}}{\leq} \mathbb{E}\bigg[\left(\sqrt{\frac{d}{\mathsf{m}}} \times \theta \mathsf{A} +  \|\hat{X}'_{\mathcal{I}-1} - X_{\mathcal{I}-1}\|+ (1-\kappa)^\mathcal{I} \times \|X_0 - Y_0\| \right)\\
& \qquad \qquad \qquad \qquad \times \mathbbm{1}\{\mathcal{I} < \min_{1 \leq j \leq \frac{d}{\mathsf{m}}} (\mathfrak{t}_{G}^{(j)})\}\bigg]\\
&\stackrel{{\scriptsize \textrm{Lemma }}\ref{thm:leapfrog3}, \textrm{ Th. }\ref{ThmContractionConvexMainResult}}{\leq} \mathbb{E}\bigg[\left(2\times \sqrt{\frac{d}{\mathsf{m}}}
 \times \theta \mathsf{A} +  \|\hat{X}'_{\mathcal{I}-2} - X_{\mathcal{I}-2}\|+ (1-\kappa)^{\mathcal{I}} \times \|X_0 - Y_0\| \right)\\
 &\qquad \qquad \qquad \qquad \times \mathbbm{1}\{\mathcal{I} < \min_{1 \leq j \leq \frac{d}{\mathsf{m}}} (\mathfrak{t}_{G}^{(j)})\}\bigg]\\
&\leq \ldots \\
&\stackrel{{\scriptsize \textrm{Lemma }}\ref{thm:leapfrog3}, \textrm{ Th. }\ref{ThmContractionConvexMainResult}}{\leq} \mathbb{E}\bigg[\left(\mathcal{I} \times \sqrt{\frac{d}{\mathsf{m}}} \times \theta \mathsf{A}+ (1-\kappa)^{\mathcal{I}} \times \|X_0 - Y_0\| \right)\\
&\qquad \qquad \qquad \qquad \times \mathbbm{1}\{\mathcal{I} < \min_{1 \leq j \leq \frac{d}{\mathsf{m}}} (\mathfrak{t}_{G}^{(j)})\}\bigg]\\
&\leq \mathcal{I} \times \sqrt{\frac{d}{\mathsf{m}}} \times \theta \mathsf{A}+ (1-\kappa)^{\mathcal{I}} \times \mathbb{E}[\|X_0 - Y_0\|]\\
&\leq \mathcal{I} \times \sqrt{\frac{d}{\mathsf{m}}} \times \theta \mathsf{A}+ (1-\kappa)^{\mathcal{I}} \times (\|X_0\|+ \mathbb{E}[\|Y_0\|])\\
&\leq \mathcal{I} \times \sqrt{\frac{d}{\mathsf{m}}} \times \theta \mathsf{A}+ (1-\kappa)^{\mathcal{I}} \times (\frac{\sqrt{d}}{\sqrt{m_2}}+ \frac{\sqrt{d}}{\sqrt{m_2}})\\
&\leq \frac{1}{2} \epsilon^2 +  \frac{1}{2} \epsilon^2\\
&= \epsilon^{2}.
\ee

Therefore, recalling that the Prokhorov distance is bounded above by the square-root of the Wasserstein distance (see Theorem 2 of \cite{gibbs2002choosing}), we have
\be
\mathsf{Prokh}(Q_{\theta}^{\mathcal{I}}(x,\cdot), \pi) &\leq \sqrt{ \mathbb{E}\left[\|\hat{X}'_{\mathcal{I}} - Y_{\mathcal{I}}\| \times \mathbbm{1}\{\mathcal{I}< \min_{1 \leq j \leq \frac{d}{\mathsf{m}}} (\mathfrak{t}_{G}^{(j)})\}\right] } + \mathbb{P}[ \min_{1 \leq j \leq \frac{d}{\mathsf{m}}} (\mathfrak{t}_{G}^{(j)}) \leq \mathcal{I}] \\
&\stackrel{{\scriptsize \textrm{Eq. }}\ref{eq:noC1}}{\leq} \sqrt{\epsilon^2} + \mathbb{P}[ \min_{1 \leq j \leq \frac{d}{\mathsf{m}}} (\mathfrak{t}_{G}^{(j)}) \leq \mathcal{I}]\\
& \leq  \epsilon +  \sum_{j=1}^{\frac{d}{\mathsf{m}}} \mathbb{P}[\mathfrak{t}_{G}^{(j)} \leq \mathcal{I}]\\
&\stackrel{{\scriptsize \textrm{Eq. }}\ref{eq:prokhorov4}}{\leq} \epsilon  +  \epsilon \, \sum_{j=1}^{\frac{d}{\mathsf{m}}} \frac{\mathsf{m}}{d} \leq 2 \epsilon.
\ee
This completes the proof of the theorem.

\end{proof}

\begin{proof} [Proof of Theorem \ref{ThmMainApprox_leapfrog2}]

We will compare $Q_{\theta}$ to a closely related chain defined in Appendix \ref{SecMH}.  We begin by fixing some constants that we will need to apply the results in earlier appendices: 

\be 
\epsilon''' &=  \frac{1}{4}\epsilon, \quad \epsilon' = \epsilon^{2}, \quad s = s_{d,\epsilon'}, \quad C = C_{d,\epsilon'}, \quad \mathsf{g}_{\infty} = \max(\mathsf{g}_{\infty}^{(1)}(d),\mathsf{g}_{\infty}^{(2)}(d)), \quad \mathsf{g}_2 = \mathsf{g}_2(d) \\
\theta_{0} &= \min \left( \frac{\kappa \epsilon^2 \sqrt{m_2}}{18T \mathsf{A}(\sqrt{M_2} \lambda_d^{-1}\log(C) + \frac{1}{\sqrt{2}}\sqrt{d})}, \, \, \, \,   \log\left([1-\frac{\epsilon'''}{\mathcal{I}\log(\frac{1}{\epsilon'''})}]^{-1}\right)\times  \left[\frac{7}{T} (M_2 + \frac{1}{2})\frac{d}{\mathsf{m}}\mathsf{g}_\infty^2 \times \frac{\mathsf{\mathsf{A}}}{M_2}\right]^{-1}\right).
\ee

Also fix $0 < \theta < \theta_{0}$.  Let  $\diamondsuit$ be the ``toy" numerical integrator defined in Equation \eqref{eq:toy} with  constants $0< \mathsf{g}_{2},\mathsf{g}_{\infty}<\infty$ as above and numerical integrator $\dagger$, the numerical integrator given in the statement of Theorem \ref{ThmMainApprox_leapfrog2}. Let $L$ be the transition kernel defined by Algorithm \ref{alg:Metropolis} with numerical integrator  $\diamondsuit$ and all other parameters as in the statement of the theorem.

Fix $\theta$ to satisfy $0 < \theta < \theta_{0}$. We now define a coupling of the two Markov chains  $\{X_{i}'\}_{i \geq 0} \sim L$ and $\{X_{i}''\}_{i \geq 0} \sim Q_{\theta}$. As mentioned immediately following Algorithm \ref{DefSimpleHMC}, these algorithms define a \textit{random mapping representation} of the associated HMC schemes. Thus, in order to define a coupling of these chains, it is enough to specify the following information:

\begin{enumerate}
\item the distribution of the starting points $X_{0}', X_{0}''$, \textit{and}
\item a coupling of the sequence of random momentum variables $\{\mathbf{p}_{i}\}_{i \geq 0}$ chosen in Step 2 of Algorithm \ref{alg:Metropolis}, \textit{and}
\item a coupling of the sequence of random variables $\{b_{i}\}_{i \geq 0}$ chosen in Step 4 of Algorithm \ref{alg:Metropolis}.
\end{enumerate}  

We make the following simple choice of coupling:
\begin{enumerate}
\item We fix a single $x \in \mathbb{R}^{d}$ satisfying $\|x \| \leq \sqrt{\frac{d}{m_{2}}}$ and set $X_{0}'= X_{0}'' = x.$
\item We draw a single i.i.d. sequence $\{\mathbf{p}_{i}\}_{i \geq 0}$ of standard Gaussian random variables, and use this single update sequence for both Markov chains.
\item We draw a single i.i.d. sequence $\{b_{i}\}_{i \geq 0}$  of uniform random variables, independently of $\{\mathbf{p}_{i}\}_{i \geq 0}$,  and use this single update sequence for both Markov chains.
\end{enumerate}

Let 
\be 
\mathfrak{t}_G = \inf \{ i \, : \, (X'_i, \mathbf{p}_{i}) \notin \mathsf{G}\}
\ee 
be the exit time of the unadjusted chain from the ``good set" $\mathsf{G}$ defined in Equation \eqref{eq:good_set}, and let 
\be 
\mathfrak{t}_{\mathrm{reject}} = \inf\{ h \geq 0 \, : \, X_{h+1}'' = X_{h}''\}
\ee be the first time that Algorithm \ref{alg:Metropolis} rejects a proposal. Notice that, from the definitions in Equation \eqref{eq:toy}, Algorithm \ref{alg:Unadjusted}, and Algorithm \ref{alg:Metropolis}, we have $X'_i = X''_i$ for all $i < \min(\mathfrak{t}_G, \mathfrak{t}_{\mathrm{reject}})$.

Recalling that the Prokhorov distance is bounded above by the square-root of the Wasserstein distance (see Theorem 2 of \cite{gibbs2002choosing}), and also that $X_{h}' = X_{h}''$ for all $0 \leq h < \min(\mathfrak{t}_{G}, \mathfrak{t}_{\mathrm{reject}})$, we have for all $\mathcal{I} > s$

\be
\mathsf{Prokh}(Q_{\theta}^{\mathcal{I}}(x,\cdot), \pi) &\leq \sqrt{W_{1}(L^{\mathcal{I}}(x,\cdot), \pi)} + \mathbb{P}[ \min(\mathfrak{t}_{G}, \mathfrak{t}_{\mathrm{reject}}) \leq \mathcal{I}]
\\
& \stackrel{{\scriptsize \textrm{Lemma }}\ref{thm:badset}}{\leq} \sqrt{W_{1}(L^{\mathcal{I}}(x,\cdot), \pi)} +  4 \, \frac{\epsilon'''}{\log(\frac{1}{\epsilon'''})}\\
& \stackrel{{\scriptsize \textrm{Lemma }}\ref{thm:unadjusted_wass2}}{\leq} \sqrt{\epsilon^2} +  \epsilon.\\\\
&= 2 \epsilon.
\ee

This completes the proof of Inequality \eqref{ProkhContConcMH}, and the theorem.

\end{proof}

\end{document}